\documentclass[10pt]{article}
\usepackage{amsmath,amssymb,amsthm}
\usepackage{graphicx,subfigure,float,url,color}
\usepackage{mathrsfs}
\usepackage[colorlinks=true]{hyperref}
\usepackage{pdfsync}
\usepackage{enumitem,dsfont}
\usepackage{bbm}

\def\R{\mathrm{I\kern-0.21emR}}
\def\N{\mathrm{I\kern-0.21emN}}

\renewcommand{\div}{\mathrm{div}}
\newcommand{\Lip}{\operatorname{Lip}}
\newcommand{\Hol}{\operatorname{Hol}}
\newcommand{\supp}{\mathrm{supp}}
\newcommand{\card}{\mathrm{card}}
\newcommand{\diam}{\mathrm{diam}}

\renewcommand{\geq}{\geqslant}
\renewcommand{\leq}{\leqslant}

\newtheorem{theorem}{Theorem}[section]
\newtheorem{proposition}{Proposition}[section]
\newtheorem{corollary}{Corollary}[section]

\newtheorem{lemma}{Lemma}[section]
\newtheorem{example}{Example}[section]
\theoremstyle{definition}\newtheorem{remark}{Remark}[section]

\title{Mean field, hydrodynamic and graph limits for deterministic interacting particle systems: a survey with quantitative estimates}

\author{
Thierry Paul\footnote{CNRS Laboratoire Ypatia des Sciences Math\'ematiques LYSM, Rome, Italy (\texttt{thierry.paul@sorbonne-universite.fr}).}
\and
Emmanuel Tr\'elat\footnote{Sorbonne Universit\'e, CNRS, Universit\'e Paris Cit\'e, Laboratoire Jacques-Louis Lions (LJLL), F-75005 Paris, France (\texttt{emmanuel.trelat@sorbonne-universite.fr}).}
}

\date{}

\begin{document}

\maketitle

\begin{abstract}
We present a unified framework, with quantitative estimates, for deterministic interacting particle systems whose pairwise interactions may depend on heterogeneous labels. Heterogeneity is kept at every level by adding a frozen label variable $x\in\Omega$ to the state. Within this framework we compare several limiting procedures: the direct continuum / graph limit, the mean field limit yielding a Vlasov equation on the extended space of labels and states, the Liouville lift of the particle system together with propagation of chaos through marginals of arbitrary order, and the hydrodynamic moment closures. We give a common language for these limits and identify precisely where the various passages commute and where they do not; in particular, we separate the continuum / graph limit equation from the classical hydrodynamic Euler equations and characterize when the former arises as a moment closure of the latter (linearity in $(\xi,\xi')$ or monokinetic ansatz). Along the way, we prove quantitative convergence estimates for the graph limit and for the passages from particles or Liouville to Vlasov, and we discuss the limitations of the framework, in particular concerning singular kernels and stochastic dynamics. The paper is written as a survey with original contributions, with an emphasis on estimates, examples, and a clear delineation of scope.
\end{abstract}

\tableofcontents

\section{Introduction}\label{sec_intro}

\subsection{Purpose and first examples}\label{sec_purpose}
This paper is a survey, with original quantitative contributions, on the passage from finite deterministic particle systems to continuum equations. We focus on systems in which agents may be heterogeneous: the interaction between agents $i$ and $j$ is allowed to depend on their labels. A central point of the paper is that this heterogeneity can be kept in the limiting equations by adding a label variable $x\in\Omega$ whose dynamics is frozen, thereby restoring a form of exchangeability in the extended state space.

The systematic study of large-$N$ limits of particle systems has a long history, going back to Hartree~\cite{hartree} for quantum systems and to Vlasov~\cite{vlasov} for plasmas, and continuing through the modern theory of mean field limits and propagation of chaos (see, e.g., \cite{Gallagher_BAMS2019, Golse_2016, Jabin_KRM2014, MischlerMouhot_IM2013, Spohn_book1991, Sznitman}).
At the microscopic scale, the basic model considered here is
\begin{equation}\label{eqagent}
\dot\xi_i(t) = \frac{1}{N} \sum_{j=1}^N G_{ij}^N(t,\xi_i(t),\xi_j(t)) ,
\qquad i\in\{1,\ldots,N\},
\end{equation}
where $\xi_i(t)\in\R^d$ stands for various parameters describing the behavior of the $i^\textrm{th}$ agent and $G_{ij}^N:\R\times\R^d\times\R^d\rightarrow\R^d$ models the interaction between agents $i$ and $j$. Throughout the paper, all systems are deterministic. We assume that the dependence on the labels is continuously embedded in a kernel $G(t,x,x',\xi,\xi')$ so that 
$$
G(t,x_i^N,x_j^N,\xi,\xi') = G_{ij}^N(t,\xi,\xi')
$$
for suitable labels $x_i^N\in\Omega$. This assumption is made precise in Assumption \ref{G}.

Dynamics of the form \eqref{eqagent} are used in a wide range of problems, ranging from flocking and swarming in biology, traffic flows and social dynamics, to fluid mechanics and quantum systems (see, e.g., \cite{AlbiPareschi_AML2013, BellomoBellouquidNietoSoler_DCDS-B2014, CarrilloChoi_ARMA2021, HaTadmor_KRM2008, HegselmannKrause, MotschTadmor_SIREV2014}, among many others). Three families of examples will guide our presentation:

\smallskip
\textbullet\ The linear \emph{Hegselmann--Krause} opinion model~\cite{HegselmannKrause}
\begin{equation*}
\dot\xi_i(t) = \frac{1}{N} \sum_{j=1}^N\sigma_{ij}(\xi_j(t)-\xi_i(t)),
\end{equation*}
where $(\sigma_{ij}^N)$ is an $N$-by-$N$ matrix of weights. Setting $\sigma_{ij}^N=\sigma(x_i^N,x_j^N)$ for some graphon $\sigma:\Omega^2\to\R$, this leads to the continuum / graph limit equation
\begin{equation*}
\partial_t y(t,x) = \int_\Omega \sigma(x,x') \big(y(t,x')-y(t,x)\big) \, dx' .
\end{equation*}
This is a paradigmatic example of a system whose interaction depends explicitly on labels.

\smallskip
\textbullet\ The \emph{Kuramoto} oscillator model on networks, in which each agent carries its own intrinsic frequency; this is the canonical example of synchronization dynamics.

\smallskip
\textbullet\ Second-order models such as \emph{Cucker--Smale} or Hamiltonian systems, in which $\xi=(q,p)\in\R^r\times\R^r$ and the interaction kernel reads, respectively,
\begin{equation*}
G(\xi,\xi')=\begin{pmatrix} p \\ a(\Vert q-q'\Vert )(p'-p) \end{pmatrix} ,
\qquad
G(\xi,\xi')=\begin{pmatrix} p \\ -\nabla V(q-q')\end{pmatrix} ,
\end{equation*}
for some influence function $a$ or potential $V$. For these systems, the kinetic and hydrodynamic interpretations are classical~\cite{HaTadmor_KRM2008,FigalliKang_APDE2019}.

\smallskip
The terminology used in this paper is the following. We call
\begin{equation*}
\partial_t y(t,x) = \int_\Omega G(t,x,x',y(t,x),y(t,x'))\, d\nu(x')
\end{equation*}
the \emph{continuum / graph limit equation}, abbreviated as the \emph{CGL equation}. We use the abbreviation CGL throughout for brevity.
We reserve the name \emph{Euler equation} for classical hydrodynamic equations, such as the pressureless Euler systems obtained by taking moments of a kinetic equation. The two are sometimes equal and sometimes not: one of the aims of the paper is to clarify under which assumptions the CGL equation can also be interpreted as a hydrodynamic closure of a Vlasov equation. 

\subsection{Main contributions}\label{sec_main_contributions}
The paper is meant to be read as a structured map of the micro--meso--macro passages, with several quantitative statements. Its main contributions are the following.
\begin{enumerate}[leftmargin=*,label=$\bf (\arabic*)$]
\item \textbf{A common framework for distinguishable agents.} We formulate a unified setting for deterministic systems of heterogeneous agents by adjoining a label variable $x\in\Omega$ with frozen dynamics. In the extended space $\Omega\times\R^d$, empirical measures recover a form of exchangeability even when the original particle dynamics is not invariant under permutations of the variables $\xi_i$ alone. This is the key device making mean-field-type tools available in the heterogeneous setting.
\item \textbf{Quantitative particle-to-CGL passage.} We give a direct passage from \eqref{eqagent} to the CGL equation based on Riemann sums and tagged partitions, with $L^\infty$ estimates (Theorems \ref{thm_estim_graph} and \ref{thm_estim_graph_2}). This recovers and refines previous graph-limit results for network dynamics.
\item \textbf{Vlasov equation on $\Omega\times\R^d$.} We derive a Vlasov equation on the extended space and prove existence, uniqueness and Dobrushin-type stability estimates in Wasserstein distance (Theorem \ref{thm_vlasov}). This provides a mean field description that preserves the label distribution.
\item \textbf{Liouville lift and quantitative propagation of chaos.} We lift the particle dynamics to the Liouville equation and obtain quantitative propagation of chaos for marginals of arbitrary order, with two natural choices of initial data, empirical and semi-empirical (Theorems \ref{thm_CV_liouville_empirical} and \ref{thm_CV_liouville}).
\item \textbf{Identification of the Vlasov-to-CGL closure.} We identify two cases in which the passage from Vlasov to CGL is closed: when $G$ is linear in $(\xi,\xi')$, and when one restricts to monokinetic measures $\mu=\nu\otimes\delta_{y(\cdot)}$. In the general case, the moments form an open hierarchy and no closed CGL equation can be obtained from Vlasov alone. This delineates the genuine difference between the graph limit and the hydrodynamic limit.
\end{enumerate}

\subsection{A map of the limiting procedures}\label{sec_bilan}
The article is organized around the following three levels. Here and in the sequel, the interaction mapping $G$ satisfies Assumption \ref{G}, stated below.
\begin{itemize}
\item The microscopic model, which is the particle system
\begin{equation}\label{micro_model}
\dot\xi_i^N(t) = \frac{1}{N} \sum_{j=1}^N G(t,x_i^N,x_j^N,\xi_i^N(t),\xi_j^N(t)),\qquad i\in\{1,\ldots,N\}.
\end{equation}
When extending this system by setting $\dot x_i^N(t)=0$, in some sense we perform an extension of the particle system to the phase space.
\item The mesoscopic model, which is the (kinetic) Vlasov equation
\begin{equation}\label{meso_model}
\partial_t\mu + \div_\xi(\mathcal{X}[\mu]\mu) = 0
\end{equation}
where 
$$
\mathcal{X}[\mu](t,x,\xi) = \int_{\Omega\times\R^d} G(t,x,x',\xi,\xi') \, d\mu(x',\xi')
$$ 
for all $(t,x,\xi) \in\R\times\Omega\times\R^d $, obtained by mean field limit.
\item The macroscopic model, which is the continuum / graph limit equation
\begin{equation}\label{macro_model}
\partial_t y(t,x) = (A(t,y(t)))(x) = \int_{\Omega} G(t,x,x',y(t,x),y(t,x')) \, d\nu(x')
\end{equation}
where $\nu\in\mathcal{P}(\Omega)$, obtained by graph limit.
\end{itemize}
Additionally, we have also considered the Liouville equation, 
\begin{equation}\label{liouville_model}
\partial_t\rho^N + \div_\Xi(Y^N\rho^N) = 0
\end{equation}
where $Y^N$ is the vector field on $\R^{dN}$ representing the system of all particles.

Figure \ref{fig_embeddings} illustrates the relationships investigated in the paper.

\begin{figure}[h]
\begin{center}
\resizebox{14.5cm}{!}{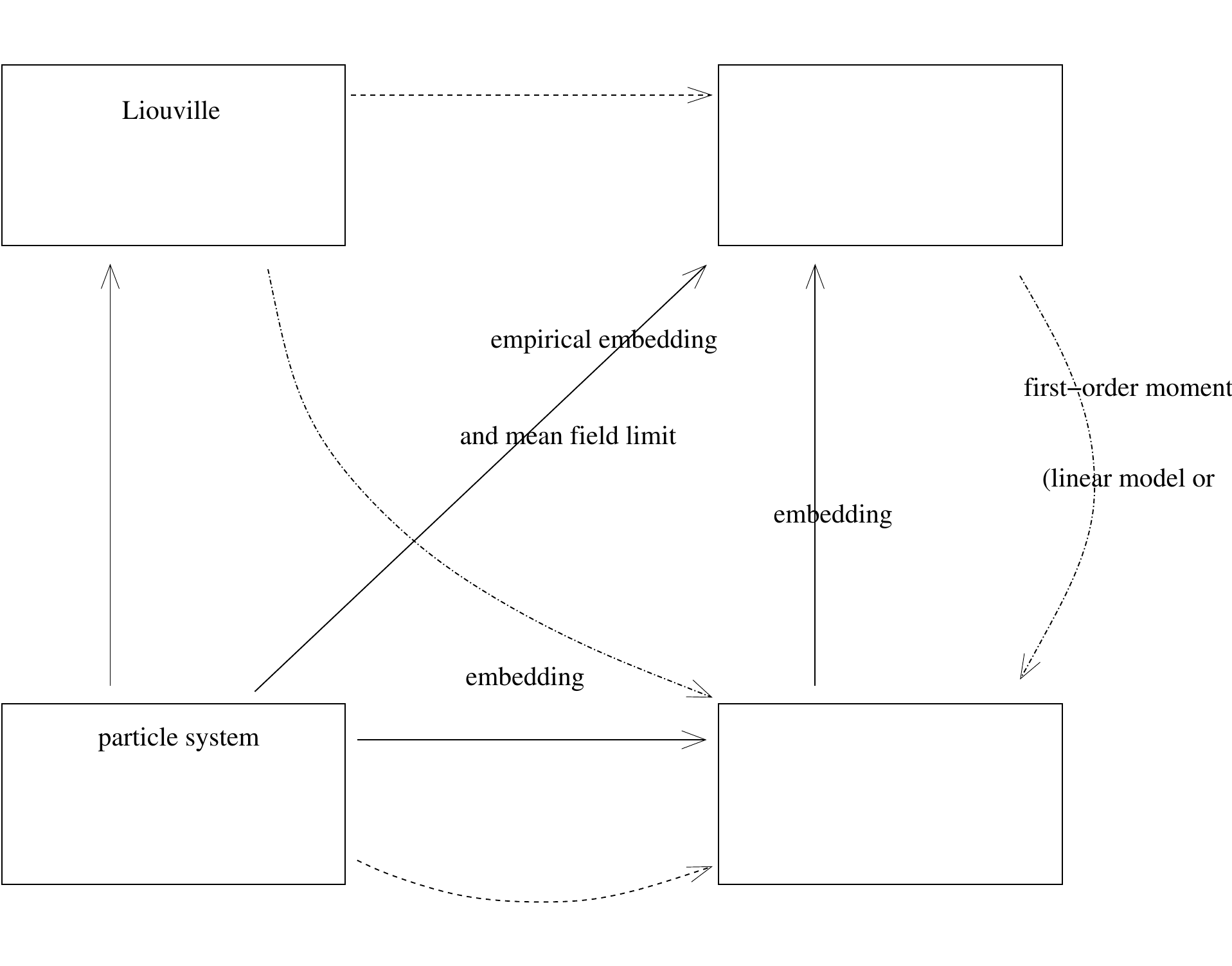}
\end{center}
\caption{Relationships between particle (microscopic) system, Liouville (probabilistic) equation, Vlasov (mesoscopic, mean field) equation, CGL (macroscopic, graph limit) equation.
We do not write the superscript $N$ in the various formulas for better readability.}
\label{fig_embeddings}
\end{figure}

\paragraph{Particle to Liouville.}
Any solution $\Xi^N(\cdot)$ of the particle system \eqref{micro_model} can be embedded as a Dirac measure $\rho^N(\cdot)=\delta_{X^N}\otimes\delta_{\Xi^N(\cdot)}$ that is a solution of the Liouville equation \eqref{liouville_model}.

\paragraph{Particle to Vlasov.}
By Proposition \ref{prop_empirical_x}, any solution $\Xi^N(\cdot)$ of the particle system \eqref{micro_model} can be embedded to an empirical measure 
$$
\mu(\cdot)=\mu^e_{(X^N,\Xi^N(\cdot))}=\frac{1}{N}\sum_{i=1}^N\delta_{x_i^N}\otimes\delta_{\xi_i^N(\cdot)}
$$
that is a solution of the Vlasov equation \eqref{meso_model}. Conversely if an empirical measure $\mu(\cdot)=\mu^e_{(X^N,\Xi^N(\cdot))}$ (with distinct points) is a solution of the Vlasov equation \eqref{meso_model} then $\Xi^N(\cdot)$ must be a solution of \eqref{micro_model}. 

In this context, the mean field limit consists of taking the limit $N\rightarrow+\infty$.

\paragraph{Particle to CGL.}
Any solution $\Xi^N(\cdot)$ of the particle system \eqref{micro_model} can be embedded to a solution of the general nonlinear CGL equation \eqref{macro_model} by using an empirical measure $\nu$ (see Remark \ref{rem_embedding_particle_Euler}).

Alternatively and more substantively, to pass from the microscopic to the macroscopic scale, by Theorems \ref{thm_estim_graph} and \ref{thm_estim_graph_2}, one can take the graph limit of the particle system (Riemann sum theorem) and thus obtain the CGL equation, with estimates of convergence as $N\rightarrow+\infty$. 

\paragraph{Liouville to Vlasov.}
By Theorems \ref{thm_CV_liouville_empirical} or \ref{thm_CV_liouville}, one can recover the solutions of the Vlasov equation \eqref{meso_model} from those of the Liouville equation \eqref{liouville_model}, for some appropriate initial conditions $\rho(0)$, by taking marginals and taking the limit $N\rightarrow+\infty$.

\paragraph{CGL to Vlasov.}
By Proposition \ref{prop_monokinetic} in Section \ref{sec_nu-monokinetic}, given any $\nu\in\mathcal{P}(\Omega)$ and any solution $t\mapsto y(t,\cdot)$ of the CGL equation \eqref{macro_model}, the $\nu$-monokinetic measure mapping $t\mapsto\mu(t) = \mu^\nu_{y(t,\cdot)} = \nu\otimes \delta_{y(t,\cdot)}$ defined by \eqref{def_monokinetic_measure} is a solution of the Vlasov equation \eqref{meso_model}.
This embedding from the macroscopic to the mesoscopic scale is general and is valid for the mean field $\mathcal{X}[\mu]$ defined by \eqref{def_mean_field} and for the nonlinear operator $A$ defined by \eqref{def_A_general}.

\paragraph{Vlasov equation to CGL equation.}
Here, and only here, we assume, first, that $G$ is linear with respect to $(\xi,\xi')$ (as is the case for the Hegselmann--Krause model).
Proposition \ref{prop1} says that, given any solution $t\mapsto\mu(t)$ of the Vlasov equation \eqref{meso_model}, defining $\nu=\pi_*\mu(t)$ (marginal of $\mu(t)$, which does not depend on $t$), the moment mapping $t\mapsto y(t,\cdot)$ of order $1$, defined by $y(t,x)=\int_{\R^d}\xi\, d\mu_{t,x}(\xi)$, is a solution of the CGL equation \eqref{macro_model} (which is linear in this case).

As discussed in Section \ref{sec_nu-monokinetic}, there is a second way, still not general, to go from Vlasov to CGL, by assuming that the solution $\mu(\cdot)$ of the Vlasov equation is $\nu$-monokinetic. In this case, its moment $y$ of order $1$ is a solution of the nonlinear CGL equation \eqref{Euler_general}.

This projection from the mesoscopic to the macroscopic scale is not general because, in general, $y$ does not satisfy a closed equation.

\paragraph{Liouville equation to CGL equation.}
Proposition \ref{prop_liouville_to_cgl_moment} and its Corollary in Section \ref{sec_liouville_to_euler} show how to derive CGL from Liouville, for specific initial conditions $\rho^N(0)$, by taking an adequate moment of $\rho^N(t)$ (in a suitable bounded Lipschitz dual norm) and then passing to the limit $N\rightarrow+\infty$.

\medskip

Finally, all relationships displayed above are general (i.e., valid for a general interaction mapping $G$) except the transition from the mesoscopic (kinetic, mean field) model to the macroscopic (CGL) model, which is valid if $G$ is linear with respect to $(\xi,\xi')$ but does not hold in general.
The graph limit procedure is of a different nature and relies on the Riemann sum theorem (see Section \ref{sec_graphlimit}). 

Notably, the above relationships suggest that the mesoscopic level should not be viewed as strictly intermediate between the microscopic and macroscopic ones. 

\medskip

We emphasize the following important novelty of our article:
\vskip 0.3cm
\begin{center}\boxed{\begin{array}{c}
\mbox{\bf a Vlasov-type mesoscopic equation for heterogeneous agent systems,}\\
\mbox{and the identification of the continuum / graph limit equation}\\
\mbox{\bf as a hydrodynamic closure in the linear or monokinetic regimes.}
\end{array}
}
\end{center}

\subsection{Microscopic viewpoint: family of particle systems}\label{sec_micro}
Let $d\in\N^*$ be fixed. Throughout the paper, we consider an arbitrary norm $\Vert\cdot\Vert$ on $\R^d$.
At the microscopic level, given any $N\in\N^*$, we consider a system of $N$ interacting ``particles" or ``agents" $\xi^N_i(t)\in\R^d$, called the \emph{particle system} (or \emph{multiagent system}), of dynamics
\begin{equation}\label{system_particles_ij}
\boxed{
\dot\xi^N_i(t) = \frac{1}{N} \sum_{j=1}^N G_{ij}^N \left( t,\xi^N_i(t),\xi^N_j(t) \right) ,
\qquad i\in\{1,\ldots,N\}
}
\end{equation}
where $G_{ij}^N:\R\times\R^d\times\R^d\rightarrow\R^d$ stands for the interaction between the particles $i$ and $j$. The dot stands for the time derivative.
The most usual case, widely treated in the existing literature, is when $G_{ij}^N=G$:
in this case, the particles are indistinguishable (or, exchangeable in the probabilistic language), reflecting the fact that the dynamics are invariant under permutations of the $\xi^N_i$.
We show here that there is no difficulty to treat the more general situation where the particles are distinguishable and the interactions depend on the agents. In \eqref{system_particles_ij}, $G_{ij}^N$ depends on $i,j,N$.

\smallskip

Throughout the paper, we make the following crucial assumption: 
\begin{enumerate}[label=$\bf(G)$]
\item\label{G}
There exist a complete metric space $(\Omega,\mathrm{d}_\Omega)$ and a \emph{continuous} mapping 
$$
\begin{array}{rcl}
G:\R\times\Omega\times\Omega\times\R^d\times\R^d & \rightarrow & \R^d \\
(t,x,x',\xi,\xi') & \mapsto & G(t,x,x',\xi,\xi') 
\end{array}
$$
locally Lipschitz with respect to $(\xi,\xi')$ uniformly with respect to $(t,x,x')$ on any compact subset of $\R\times\Omega\times\Omega$, 
such that, for every $N\in\N^*$, there exist $x^N_1,\ldots,x^N_N$ in $\Omega$ such that
\begin{equation}\label{interpolation_G}
G \left( t,x^N_i,x^N_j,\xi,\xi' \right) = G_{ij}^N(t,\xi,\xi') \qquad \forall t\in\R\qquad \forall \xi,\xi'\in\R^d\qquad \forall i,j\in\{1,\ldots,N\} .
\end{equation}
\end{enumerate}
Under Assumption \ref{G}, for every $N\in\N^*$ the particle system \eqref{system_particles_ij} is equivalently written as
\begin{equation}\label{particle_system}
\boxed{
\begin{aligned}
\dot x^N_i(t) &= 0 \\
\dot\xi^N_i(t) &= \frac{1}{N} \sum_{j=1}^N G \left( t,x^N_i,x^N_j,\xi^N_i(t),\xi^N_j(t) \right) ,
\qquad i\in\{1,\ldots,N\}
\end{aligned}
}
\end{equation}
The variables $x^N_i\in\Omega$ are parameters, and a usual way to treat parameters in differential equations is to treat them as state variables whose dynamics are zero, whence the dynamics $\dot x^N_i(t)=0$ above.
For each index $i$, the variable $x^N_i$ can be seen as the ``label" (type, name, color) of the agent $i$, used to distinguish it from the others.

\smallskip

Assumption \ref{G} (in particular, \eqref{interpolation_G}) amounts to a continuous interpolation of the mappings $G_{ij}^N$. 
The continuity assumption includes the idea of the existence of a limit system as $N\rightarrow+\infty$. In some sense, this assumption is unavoidable: indeed, if $G$ were not required to be continuous, then completely different systems \eqref{system_particles_ij} could be considered as $N$ varies and then no limit (at least, in a strong sense) for large $N$ could exist. Note anyway that, interestingly, the authors of \cite{JabinPoyatoSoler} do not assume \ref{G}, but in order to pass to the mean field limit they make another assumption of uniform boundedness on their dynamics in order to have a weak star limit. However at the limit the distinguishability of particles is lost. In contrast, in our paper we want to obtain strong (mean field, hydrodynamic, graph) limits and to preserve distinguishability at the limit.

Note that Assumption \ref{G} implies that the Lipschitz constants of the mappings $G_{ij}^N$ are uniformly bounded (with respect to $i,j,N$) on any compact.

In Assumption \ref{G}, the complete metric space $\Omega$ used for the parameters $x^N_i$ is arbitrary. 
For instance we can take $\Omega=[0,1]$, but we allow for more general sets, in view of deriving on $\Omega$ some interesting classes of PDEs (see Section \ref{sec_approx_PDE}).

The choice of the possible values of the $x^N_i$ is not imposed in Assumption \ref{G}. If one wishes moreover to fix some precise points $x^N_i$, such as the natural ones $x^N_i=\frac{i}{N}$ when $\Omega=[0,1]$, often used in numerical analysis, then having \eqref{interpolation_G} satisfied requires some compatibility conditions on the mappings $G_{ij}^N$. 

\smallskip

In the above framework, the classical case studied in the existing literature, where particles are indistinguishable, is when the mapping $G$ does not depend on $(x,x')$. 

\medskip

Setting $X^N=(x^N_1,\ldots,x^N_N)\in\Omega^N$, the system \eqref{particle_system} can also be written in the form
\begin{equation}\label{system_Y}
\boxed{
\dot\Xi^N(t) = Y^N\big(t,X^N,\Xi^N(t)\big)
}
\end{equation}
where $\Xi^N(t)=(\xi^N_1(t),\ldots,\xi^N_N(t))$.
Here and in what follows, the time-dependent vector field $Y^N(t,X,\cdot)$ on $(\R^d)^N$, depending on the parameter $X\in\Omega^N$, is defined by
\begin{equation}\label{def_Y}
Y^N(t,X,\cdot)=\big(Y^N_1(t,X,\cdot),\ldots,Y^N_N(t,X,\cdot)\big)
\end{equation}
with 
\begin{equation}\label{def_Yi}
Y^N_i(t,X,\Xi) = \frac{1}{N} \sum_{j=1}^N G(t,x_i,x_j,\xi_i,\xi_j) \qquad \forall i\in\{1,\ldots,N\}
\end{equation}
for all $t\in\R$, $X=(x_1,\ldots,x_N)\in\Omega^N$ and $\Xi=(\xi_1,\ldots,\xi_N)\in(\R^d)^N$.
We denote by $(\Phi^N(t,X,\cdot))_{t\in I}$ ($I\subset\R$) the local-in-time flow of diffeomorphisms of $\R^{dN}$ generated by the time-dependent vector field $Y^N(t,X,\cdot)$: this flow, called the \emph{particle flow}, is parametrized by $X\in\Omega^N$. 
We have $\Xi^N(t) = \Phi^N(t,X^N,\Xi^N(0))$ for every $t\in I$. 

\begin{lemma}[Uniform maximal time]\label{lem_Tmax}
For any compact subset $K$ of $\Omega\times\R^d$, there exists $T_{\max}(K)\in(0,+\infty]$ such that, for any $N\in\N^*$, for any $(X,\Xi(0))\in K^N$,\footnote{With a slight abuse of notation, $(X^N,\Xi^N(0))\in K^N$ means that $(x_i^N,\xi_i^N(0))\in K$ for every $i\in\{1,\ldots,N\}$.} there exists a unique solution $t\mapsto \Xi^N(t) = \Phi^N(t,X^N,\Xi^N(0))$ of \eqref{system_Y} on $[0,T_{\max}(K))$, of parameter $X^N$ and of initial condition $\Xi^N(0)$ at $t=0$, and of class $\mathscr{C}^1$ with respect to $t$.
Moreover, for any $T\in[0,T_{\max}(K))$, there exists a compact subset $K_T\subset\R^d$, depending on $T$ but not on $N$, such that $\xi_i^N(t)\in K_T$ for every $i\in\{1,\ldots,N\}$, every $t\in[0,T]$, and every $N\in\N^*$.
\end{lemma}

Lemma \ref{lem_Tmax} shows that, given a compact set $K$ of initial conditions, the time $T_{\max}(K)$ is uniform with respect to $N\in\N^*$, and that, given any $T\in(0,T_{\max}(K))$, any solution of \eqref{system_Y} on $[0,T]$, starting in $K$ at $t=0$, is contained in a compact set that depends on $T$ but not on $N$. 

Lemma \ref{lem_Tmax} follows directly from the usual proof of the Picard-Lindel\"of (Cauchy-Lipschitz) theorem by a fixed point argument (see \cite[Chapter II]{Hartman}), using Assumption \ref{G}, noting that, for every $T>0$, on $[0,T]\times K^N$ the vector field $Y^N$ is uniformly bounded with respect to $N$ and is Lipschitz with respect to $\Xi$ uniformly with respect to $(t,X)$ on any compact, with a Lipschitz constant that is uniform with respect to $N$. Note that, for a given $N\in\N^*$, the maximal time of definition of the solution $t\mapsto \Phi^N(t,X,\Xi(0))$ may be larger than $T_{\max}(K)$; what is important in the lemma is the uniform bound below with respect to $N$. 

Of course, if $G$ is globally Lipschitz with respect to $(\xi,\xi')\in\R^d\times\R^d$, uniformly with respect to $(t,x,x')$ on any compact subset of $[0,+\infty)\times\Omega\times\Omega$, then $T_{\max}(K)=+\infty$ for any compact $K\subset\Omega\times\R^d$.
But our framework is more general and allows for superlinearities.

We next give some examples covered by this general framework.

\subsection{Examples}\label{sec_ex_particle}
\paragraph{First-order systems.}
General first-order systems of the form
\begin{equation}\label{first-order_particle}
\dot\xi^N_i(t) = F_i^N(t,\xi^N_i(t)) + \frac{1}{N} \sum_{j=1}^N K_{ij}^N(t,\xi^N_i(t),\xi^N_j(t)), \qquad i\in\{1,\ldots,N\},
\end{equation}
can be written as \eqref{system_particles_ij} with $G_{ij}^N(t,\xi,\xi') = F_i^N(t,\xi)+K_{ij}^N(t,\xi,\xi')$. Assumption \ref{G} is satisfied if there exist a set $\Omega$ and sufficiently regular mappings $F$ and $K$ such that $F( t,x^N_i,x^N_j,\xi,\xi' ) = F_i^N(t,\xi)$ and $K( t,x^N_i,x^N_j,\xi,\xi' ) = K_{ij}^N(t,\xi,\xi')$ as in \eqref{interpolation_G}. 

\medskip\noindent
-- A first meaningful example is the linear Hegselmann--Krause first-order consensus system (see \cite{HegselmannKrause}), modeling for instance the propagation of opinions (studied in \cite{BoudinSalvaraniTrelat_SIMA2022}), of dynamics
\begin{equation}\label{HK_particle}
\dot\xi^N_i(t) = \frac{1}{N} \sum_{j=1}^N \sigma_{ij}^N \big(\xi^N_j(t)-\xi^N_i(t)\big), \qquad i\in\{1,\ldots,N\},
\end{equation}
with constant interaction coefficients $\sigma_{ij}^N\geq 0$ (not necessarily symmetric). 
Assumption \ref{G} requires that there exist a set $\Omega$ (for example, but not necessarily, $\Omega=[0,1]$) and a continuous function $\sigma$ on $\Omega^2$ such that, for every $N\in\N^*$, there exist distinct points $x^N_1,\ldots,x^N_N$ in $\Omega$ such that $\sigma(x^N_i,x^N_j)=\sigma_{ij}^N$. 
The graph interpretation, which is particularly relevant here, will be commented in Section \ref{sec_graphlimit}.
We have then $G(t,x,x',\xi,\xi')=\sigma(x,x')(\xi'-\xi)$ for all $(t,x,x',\xi,\xi') \in\R\times\Omega^2\times(\R^d)^2$.

More general models can be considered, with interaction coefficients $\sigma_{ij}$ depending on $t$ and on the $\xi_i$ (see the survey \cite{MotschTadmor_SIREV2014} and see the recent Transformers particle model studied in \cite{GeshkovskiLetrouitPolyanskiyRigollet}).

\medskip\noindent
-- A second interesting example is the Kuramoto model
\begin{equation}\label{kuramoto_particle}
\dot\xi^N_i(t) = \alpha_i^N + \frac{1}{N}\sum_{j=1}^N \sigma_{ij}^N\sin(\xi^N_j(t)-\xi^N_i(t)), \qquad i\in\{1,\ldots,N\},
\end{equation}
where $d=1$, $\xi_i^N(t)\in\R$ is the phase of the oscillator $i$, $\alpha_i^N\in\R$ is its frequency and $\sigma_{ij}^N\in\R$ is an interaction coefficient between oscillators $i$ and $j$. 
This system was introduced in \cite{Kuramoto_1975} in view of studying synchronization of interacting oscillators. 
To write the particle system \eqref{kuramoto_particle} in the form \eqref{particle_system}, now two parameters (labels) are required for each particle, one standing for the frequency and the other for the interaction as in the previous example. 
We set $\Omega = \R\times[0,1]$ and for every $x\in\Omega$ we denote by $x=(\alpha,\beta)\in\Omega$ the two coordinates of $x$. 
Assumption \ref{G} is satisfied if there exists a continuous function $\sigma$ on $[0,1]^2$ satisfying $\sigma(\beta^N_i,\beta^N_j)=\sigma_{ij}$ as in \eqref{interpolation_G}, and we have then $G(t,x,x',\xi,\xi') = \alpha + \sigma(\beta,\beta')\sin(\xi'-\xi)$ (where $x=(\alpha,\beta)$ and $x'=(\alpha',\beta')$).

\medskip\noindent
-- Consider again the general system \eqref{first-order_particle}, but where now $F_i^N=F$ and $K_{ij}^N=K$ do not depend on $i,j,N$, with $F,K\in\mathscr{C}^1(\R^d,\R^d)$. In this case \eqref{first-order_particle} becomes
\begin{equation}\label{FK_particle}
\dot\xi^N_i(t) = F(\xi^N_i(t)) + \frac{1}{N} \sum_{j=1}^N K(\xi^N_i(t)-\xi^N_j(t)), \qquad i\in\{1,\ldots,N\},
\end{equation}
which is a much used particle model (see \cite{Jabin_KRM2014}).
Assumption \ref{G} is satisfied and $G(t,x,x',\xi,\xi') = F(\xi)+K(\xi-\xi')$, not depending on $(t,x,x')$: this is an indistinguishable case.
Often, $K=-\nabla V$ where $V$ is an interaction potential (that we consider here to be regular).

\paragraph{Second-order systems.} 
Setting $d=2r$ and denoting $\xi=(q,p)\in\R^r\times\R^r$, general second-order systems of the form
\begin{equation}\label{second-order_particle}
\dot q^N_i(t) = p^N_i(t), \qquad
\dot p^N_i(t) = \frac{1}{N} \sum_{j=1}^N b_{ij}^N(t,q^N_i(t),p^N_i(t),q^N_j(t),p^N_j(t)), \qquad i\in\{1,\ldots,N\},
\end{equation}
can be written as \eqref{system_particles_ij} with $G_{ij}^N(t,\xi,\xi') = (p, b_{ij}^N(t,\xi,\xi'))$. Assumption \ref{G} is satisfied if there exist a set $\Omega$ and a sufficiently regular mapping $b$ interpolating all mapping $b_{ij}^N$ as stated in \eqref{interpolation_G}, i.e., $b( t,x^N_i,x^N_j,\xi,\xi' ) = b_{ij}^N(t,\xi,\xi')$.

Here, $q$ is a position and $p$ is a speed or a momentum.
It is important to note that the variable $q$ should not be confused with the variable $x\in\Omega$ that is used here to designate the label of a particle.

\medskip\noindent
-- A famous example of second-order dynamics is the Cucker--Smale model (see \cite{CuckerSmale})
\begin{equation}\label{CS_particle}
\dot q^N_i(t) = p^N_i(t), \qquad
\dot p^N_i(t) = \frac{1}{N} \sum_{j=1}^N a(\Vert q^N_j(t)-q^N_i(t)\Vert) (p^N_j(t)-p^N_i(t)), \qquad i\in\{1,\ldots,N\},
\end{equation}
where $a\in\mathscr{C}^1(\R)$.
Assumption \ref{G} is satisfied with $G=(G_q,G_p)$ where $G_q(t,x,x',\xi,\xi')=p$ and $G_p(t,x,x',\xi,\xi')=a(\Vert q'-q\Vert) (p'-p)$, not depending on $(t,x,x')$: this is an indistinguishable case.

Many variants of that model are covered by our framework, for instance the potential $a$ may depend on $i$ and $j$, and other terms can be added to the dynamics of $p_i$, for instance self-propulsion and attraction-repulsion forces (like in \cite{CarrilloChoi_ARMA2021}); in this case, defining a set $\Omega$ is required. 

\medskip\noindent
-- Many second-order particle systems studied in the literature, modeling Newtonian dynamics of $N$ particles interacting through a pairwise force $K$ (typically derived from a potential), are of the form \eqref{second-order_particle} with $b_{ij}^N(t,\xi,\xi')=K(\xi,\xi')$, not depending on $i,j,N$, with $K\in\mathscr{C}^1(\R^d,\R^d)$, yielding
\begin{equation}\label{second_order_K}
\dot q^N_i(t) = p^N_i(t), \qquad
\dot p^N_i(t) = \frac{1}{N} \sum_{j=1}^N K(q^N_i(t),q^N_j(t)), \qquad i\in\{1,\ldots,N\} .
\end{equation}
Assumption \ref{G} is satisfied with $G(t,x,x',\xi,\xi') = (p,K(q,q'))$, not depending on $(t,x,x')$: this is an indistinguishable case.
Note that, when $K=-\nabla V$ for some potential function, the above particle system stands for the classical $N-$body problem in Hamiltonian form (see next for more general Hamiltonian cases), with the Hamiltonian function given by 
$H(q_1,p_1,\ldots,q_N,p_N) = \frac{1}{2} \sum_{i=1}^N \Vert p_i\Vert^2 + \frac{1}{2N} \sum_{i,j=1}^N V(q_i,q_j)$ (with $V$ symmetric).

\paragraph{Hamiltonian systems.}
Still with $d=2r$ and $\xi=(q,p)\in\R^r\times\R^r$, given any $N\in\N^*$, consider the Hamiltonian function
\begin{equation}\label{ham_particle}
H^N(q_1,p_1,\ldots,q_N,p_N) = \sum_{j=1}^N h_j^N(q_j,p_j) + \frac{1}{N} \sum_{j,k=1}^N h_{jk}^N(q_j,p_j,q_k,p_k)
\end{equation}
for some $\mathscr{C}^1$ functions $h_j^N$ and $h_{jk}^N$. The Hamiltonian system of $N$ particles, given by $\dot q_i = \frac{\partial H}{\partial p_i}$, $\dot p_i = - \frac{\partial H}{\partial q_i}$ for $i\in\{1,\ldots,N\}$, can be written as \eqref{system_particles_ij} with
$$
G_{ij}^N(t,\xi,\xi') = 
\begin{pmatrix}
\phantom{-}\partial_2 h_i^N(q,p) + \partial_2 h_{ij}^N(q,p,q',p') + \partial_4 h_{ji}^N(q',p',q,p) \\[1mm]
-\partial_1 h_i^N(q,p) - \partial_1 h_{ij}^N(q,p,q',p') - \partial_3 h_{ji}^N(q',p',q,p) 
\end{pmatrix}
$$
where $\partial_k$ denotes the partial derivative with respect to the $k^\textrm{th}$-variable. 

Having Assumption \ref{G} satisfied requires at least that the Hamiltonians $h_j^N$ and $h_{jk}^N$ be uniformly (wrt $j,k,N$) locally Lipschitz. 
Note that the Hamiltonian $H^N$ defined by \eqref{ham_particle} involves sums of ``single" (noninteracting) and of ``pairwise" Hamiltonians, but not of ``triple-wise" or more.

Many classical Hamiltonian systems of $N$ particles are written as above with Hamiltonians not depending on $j,k,N$, for instance in quantum mechanics (see \cite{Golse_2016}) or in geometric mechanics (e.g., point-vortex systems on Riemannian manifolds, geodesic flows of $N$-body type).
An example, where Assumption \ref{G} is satisfied, 
used to model systems of fermions confined in a magnetic field, is when 
$h_j^N(q_j,p_j) = V(q_j) + \frac{1}{2}\Vert p_j-A(q_j)\Vert^2$ for some confining potential $V\in\mathscr{C}^1(\R^d)$ and some magnetic potential vector $A\in\mathscr{C}^1(\R^d,\R^d)$,
and $h_{jk}(q_j,p_j,q_k,p_k) = W(\Vert q_j-q_k\Vert)$ for some pairwise interaction potential $W\in\mathscr{C}^1(\R^d)$.
In this case, we have 
\begin{equation}\label{fermion_particle}
G(t,x,x',\xi,\xi') = \Big( p-A(q) , -\nabla V(q)+dA(q).(p-A(q))-\partial_1 W(q,q')-\partial_2 W(q',q) \Big) .
\end{equation}

\begin{remark}[On the wording ``indistinguishability"]
In the literature, a dynamical system $\dot z(t) = X(t,z(t))$ in $\R^n$ is said to be ``indistinguishable", or ``exchangeable" in the probabilistic wording, if it is invariant under permutations in the following sense: given any $z_0\in\R^n$, denoting by $t\mapsto z(t,z_0)$ the unique solution on some interval $I$ of the system such that $z(0,z_0)=z_0$, we have $z(t,\sigma(z_0))=\sigma(z(t,z_0))$ for every $t\in I$, for every permutation $\sigma\in\mathfrak{S}_n$. Equivalently, the vector field $X$ is invariant under the action of the permutation, i.e., $\sigma_*X(t,\cdot)=X(t,\cdot)$ for every $t$.

\smallskip
{\bf (A)} General particle systems in $\R^{dN}$ of the form \eqref{system_particles_ij} are not indistinguishable in general because the interaction mapping $G_{ij}^N$ depends on $i$ and $j$ (but they are indistinguishable if $G_{ij}^N=G$): the dynamics are not invariant under permutations $\sigma\in\mathfrak{S}_{dN}$ acting on $\Xi=(\xi_1,\ldots,\xi_N)$.
This is the standard wording used in the literature to describe the distinguishability or indistinguishability of systems of particles, and we will follow this wording throughout the article.

\smallskip
{\bf (B)} In Section \ref{sec_micro} we have introduced a set of labels $x\in\Omega$, distinguishing particles, and we have done the fundamental assumption \ref{G}. In this context, the particle system \eqref{system_particles_ij} (which is, in general, distinguishable) has been rewritten as \eqref{particle_system} or equivalently as \eqref{system_Y}, by augmenting the state space to $\Omega^N\times\R^{dN}$. But then, in this augmented form, the system \eqref{system_Y} is \emph{always} indistinguishable in the sense that it is invariant under permutations $\sigma\in\mathfrak{S}_{dN}$ acting simultaneously on $X=(x_1,\ldots,x_N)$ and on $\Xi=(\xi_1,\ldots,\xi_N)$.
Hence, in some way, we recover indistinguishability in the new state space $\Omega^N\times\R^{dN}$. 

Despite the slight ambiguity, throughout the paper, we will continue to use the wording described in (A).
\end{remark}

\subsection{Reader's guide and assumptions at a glance}\label{sec_objectives}
Section \ref{sec_micro} introduces the labelled particle system and explains how the usual indistinguishable setting is recovered when $G$ does not depend on $(x,x')$. Section \ref{sec_ex_particle} gives the main examples. Section \ref{sec_graphlimit} proves the direct particle-to-CGL convergence. Section \ref{sec_vlasov} derives the Vlasov equation by empirical measures. Section \ref{sec_liouville} studies the Liouville lift and its marginals. Section \ref{sec_hydro} studies the hydrodynamic moment viewpoint and explains why the Vlasov-to-CGL passage is not automatic. Section \ref{sec_further_comments} provides further comments and perspectives, and Appendix \ref{sec_appendix} collects technical tools.

The following table summarizes the assumptions most often used in the paper and the corresponding outputs.
\begin{center}
\begin{tabular}{p{3.6cm}p{4.6cm}p{6.1cm}}
\hline
Assumption & Where it is used & Output \\
\hline
$G$ continuous in $(x,x')$ and locally Lipschitz in $(\xi,\xi')$ & Basic framework, Assumption \ref{G} & Well-posed particle systems on a uniform time interval; topological particle-to-CGL and particle-to-Vlasov convergence. \\
$G$ locally H\"older or Lipschitz in $(x,x',\xi,\xi')$ & Quantitative graph and mean field estimates & Rates for Riemann-sum errors and Wasserstein stability estimates. \\
Tagged partitions of $(\Omega,\nu)$ & Direct CGL limit and semi-empirical approximations & Explicit approximation of $\nu$ and rates depending on the mesh exponent $r$. \\
Linearity of $G$ in $(\xi,\xi')$ & Moment equation of order $1$ & Closure of the first moment and identification of the resulting equation with the CGL equation. \\
Monokinetic ansatz $\mu=\nu\otimes\delta_{y(\cdot)}$ & Vlasov-to-CGL passage & Equivalence between the Vlasov equation restricted to monokinetic measures and the nonlinear CGL equation. \\
\hline
\end{tabular}
\end{center}

\subsection{Scope, limitations and possible extensions}\label{sec_scope_limitations}
The paper is deliberately restricted to deterministic finite systems and regular interaction kernels. This already covers many network and collective dynamics models, but it excludes several important directions.

First, singular kernels, such as Coulomb, Poisson or point-vortex type interactions, are outside the scope of the present analysis. They require compactness, modulated energy, cut-off or stability arguments that are very different from the ODE and Wasserstein estimates used here (see, e.g., \cite{HaurayJabin_ARMA2007, JabinWang_IM2018, Serfaty}).

Second, stochastic particle systems are not treated. Mean field limits with noise often lead to McKean-Vlasov or kinetic Fokker-Planck equations, while the compatibility between stochasticity and graph limits is more delicate. The deterministic diagram of Figure \ref{fig_embeddings} should therefore be viewed as a reference map, not as a stochastic result.

Third, Assumption \ref{G} is a structural interpolation assumption on the heterogeneous interactions. It is strong enough to preserve labels and to obtain strong limits. It is not the only possible way to treat non-exchangeable systems; weaker compactness approaches may exist, but they generally lose part of the label information or give weaker convergence.

Finally, the Vlasov-to-CGL passage is not a general theorem. In the absence of linearity or a monokinetic ansatz, the first moment does not satisfy a closed equation. This obstruction is part of the message of the paper: the direct graph limit and the hydrodynamic limit agree only under specific closure mechanisms.

\subsection{General notations}\label{sec_general_notations}
Let $(E,\mathrm{d}_E)$ be a Polish space.

\paragraph{H\"older and Lipschitz mappings.}
Let $U$ be a subset of $E$. Let $k\in\N^*$ and let $\Vert\cdot\Vert$ be a norm on $\R^k$.
Given any $\alpha\in(0,1]$, we denote by $\mathscr{C}^{0,\alpha}(U,\R^k)$ the set of all continuous mappings $g\in \mathscr{C}^0(U,\R^k)$ that are \emph{$\alpha$-H\"older continuous} (with respect to the norm $\Vert\ \Vert$), meaning that
$$
\Hol_\alpha(g) = \sup_{\substack{y,y'\in U \\ y\neq y'}} \frac{\Vert g(y)-g(y')\Vert}{\mathrm{d}_E(y,y')^\alpha} < +\infty. 
$$
When $\alpha=1$, we speak of a Lipschitz mapping and we denote $\Lip(g)=\Hol_1(g)$. 
When $U$ is compact, $\mathscr{C}^{0,\alpha}(U,\R^k)$ is a Banach space endowed with the norm
$$
\Vert g\Vert_{\mathscr{C}^{0,\alpha}(U,\R^k)} = \max_{y\in U}\Vert g(y)\Vert + \Hol_\alpha(g) .
$$
When $k=1$ and $\alpha=1$, we denote $\Lip(U) = \mathscr{C}^{0,1}(U,\R)$.

\paragraph{Probability Radon measures.}
We denote by $\mathcal{P}(E)$ the set of probability Radon measures on $E$. 
We also consider $\mathcal{P}_c(E)$, $\mathcal{P}^{ac}(E)$, where the subscript $c$ means ``with compact support" and the superscript $ac$ means ``absolutely continuous with respect to a Lebesgue measure" (in the case where $E$ is equipped with a Lebesgue measure), and for every $p\geq 1$ the set $\mathcal{P}_p(E)$ stands for the set of all $\mu\in\mathcal{P}(E)$ that have a finite moment of order $p$, i.e., $\int_E \mathrm{d}_E(y_0,y)^p\, d\mu(y) < +\infty$ where $y_0\in E$ is arbitrary.
Given any Borel mapping $\phi:E\rightarrow F$ where $F$ is another Polish space and given any $\mu\in\mathcal{P}(E)$, the image (or pushforward) of $\mu$ under $\phi$ is $\phi_*\mu=\mu\circ\phi^{-1}$.

We denote by $\mathscr{C}^0(E)$ the set of continuous functions on $E$ and by $\mathscr{C}^0_c(E)$ the set of continuous functions of compact support on $E$.
When $E$ is a smooth manifold, we adopt similar notations for the set $\mathscr{C}^\infty(E)$ of smooth functions on $E$.
We recall that the topological dual $(\mathscr{C}^0_c(E))'$ (resp., $(\mathscr{C}^0(E))'$) is the set of all Radon measures on $E$ (resp., with compact support). Endowed with the total variation norm $\Vert\ \Vert_{TV}$ which is the dual norm, it is a Banach space.

Throughout the paper, $\delta_\star$ is the Dirac measure at $\star$.

\paragraph{Wasserstein distance.}
Given any $p\geq 1$, the Wasserstein distance $W_p(\mu_1,\mu_2)$ of order $p$ between two probability measures $\mu_1,\mu_2\in\mathcal{P}(E)$, with respect to the distance $\mathrm{d}_E$, is defined as the infimum of the Monge-Kantorovich cost $\int_{E^2} \mathrm{d}_E(y_1,y_2)^p \, d\Pi(y_1,y_2)$ over the set of probability measures $\Pi\in\mathcal{P}(E^2)$ coupling $\mu_1$ with $\mu_2$, i.e., whose marginals on the two copies of $E$ are $\mu_1$ and $\mu_2$:
\begin{equation}\label{def_Wp}
W_p(\mu_1,\mu_2) = \inf\left\{ \left( \int_{E^2} \mathrm{d}_E(y_1,y_2)^p \, d\Pi(y_1,y_2) \right)^{1/p} \ \mid\ \Pi\in\mathcal{P}(E^2),\ (\pi_1)_*\Pi=\mu_1, \ (\pi_2)_*\Pi=\mu_2 \right\}
\end{equation}
where $\pi_1:E^2\rightarrow E$ and $\pi_2:E^2\rightarrow E$ are the canonical projections defined by $\pi_1(y_1,y_2)=y_1$ and $\pi_2(y_1,y_2)=y_2$ for all $(y_1,y_2)\in E\times E$.
Equivalently,
\begin{equation}\label{def_Wp_randomlaws}
W_p(\mu_1,\mu_2) = \inf\left\{ \Big( \mathbb{E} \, \mathrm{d}_E(Y_1,Y_2)^p \Big)^{1/p}   \ \mid\ \textrm{law}(Y_1)=\mu_1,\ \textrm{law}(Y_2)=\mu_2 \right\}
\end{equation}
where the infimum is taken over all possible random variables $Y_1$ and $Y_2$ (defined on a same probability space, with values in $E$) having the laws $\mu_1$ and $\mu_2$ respectively.
Then, $W_p$ is a distance on $\mathcal{P}_p(E)$, which metrizes the weak convergence in $\mathcal{P}_p(E)$ in the following sense: given $\mu\in\mathcal{P}_p(E)$ and given a sequence $(\mu_j)_{j\in\N^*}$ in $\mathcal{P}_p(E)$, we have $W_p(\mu_j,\mu)\rightarrow 0$ as $j\rightarrow+\infty$ if and only if $\int_{E} f\, d\mu_j\rightarrow\int_{E} f\, d\mu$ for every continuous bounded function $f$ on $E$ and $\int_{E} \mathrm{d}_E(y_0,y)^p\, d\mu_j(y)\rightarrow \int_{E} \mathrm{d}_E(y_0,y)^p\, d\mu(y)$ as $j\rightarrow+\infty$ for some (and thus any) $y_0\in E$ (see \cite[Chapter 5, Section 5.2]{Santambrogio_2015} or \cite[Theorem 6.9]{Villani_2009}), if and only if $\int_{E} f\, d\mu_j\rightarrow\int_{E} f\, d\mu$ for every continuous function $f$ on $E$ such that $\vert f(y)\vert\leq C(1+\mathrm{d}_E(y_0,y)^p)$ for every $y\in E$, for some $C>0$ and some (and thus any) $y_0\in E$ (see \cite[Theorem 7.12]{Villani_2003}).
It can be noted that, given any subset $K\subset E$ of finite diameter, we have 
\begin{equation}\label{inegWp1Wp2}
1\leq p_1\leq p_2 \ \Rightarrow\ 
W_{p_1}(\mu_1,\mu_2) \leq W_{p_2}(\mu_1,\mu_2) \leq \mathrm{diam}_E(K)^{1-p_1/p_2} W_{p_1}(\mu_1,\mu_2)^{p_1/p_2}
\end{equation}
for all $\mu_1,\mu_2\in\mathcal{P}_c(E)$ of compact support contained in $K$ (see \cite[Chapter 5]{Santambrogio_2015}), where $\diam_E(K)$ is the supremum of all $\mathrm{d}_E(y,y')$ over all possible $y,y'\in K$.

For $p=1$, the duality formula for the Kantorovich-Rubinstein distance (see \cite[Chapter 5]{Villani_2009}) gives the equivalent definition
\begin{equation}\label{def_W1}
W_1(\mu_1,\mu_2)=\sup\left\{\int_{E}f\,d(\mu_1-\mu_2)\ \mid\ f\in\Lip(E),\ \Lip(f)\leq 1\right\} ,
\end{equation}
valid for all $\mu_1,\mu_2\in\mathcal{P}_1(E)$.

For $p=+\infty$, we set $W_\infty(\mu_1,\mu_2) = \lim_{p\rightarrow+\infty} W_p(\mu_1,\mu_2)$ (see \cite[Chapter 5, Section 5.5.1]{Santambrogio_2015}).

Note that the infimum in \eqref{def_Wp}, as well as in \eqref{def_Wp_randomlaws}, is a minimum (i.e., there exists an optimal coupling) and that the supremum in \eqref{def_W1} is a maximum (see \cite[Chapters 4 and 5]{Villani_2009} or \cite[Chapter 3, Section 3.1.1]{Santambrogio_2015}).


\paragraph{Disintegration.}
In this paper, we are going to consider measures on $\Omega\times\R^d$, for $d\in\N^*$ (and on $\Omega^k\times(\R^d)^k$ for $k\in\N^*$), where $(\Omega,\mathrm{d}_\Omega)$ is a complete metric space and $\R^d$ is endowed with an arbitrary norm $\Vert\cdot\Vert$. 
We endow $\Omega\times\R^d$ with the distance $\mathrm{d}_{\Omega\times\R^d} = \mathrm{d}_\Omega + \mathrm{d}_{\R^d}$ where $\mathrm{d}_{\R^d}$ is the distance on $\R^d$ induced by the norm $\Vert\cdot\Vert$.

Denoting by $\pi:\Omega\times\R^d\rightarrow\Omega$ the canonical projection, given any $\mu\in\mathcal{P}(\Omega\times\R^d)$, in the sequel we will always denote by $\nu$ the nonnegative probability Radon measure on $\Omega$ defined as the image (pushforward) of $\mu$ under $\pi$, 
\begin{equation}\label{def_nu}
\nu = \pi_*\mu = \mu\circ\pi^{-1} ,
\end{equation}
that is also the marginal of $\mu$ on $\Omega$. 
Note 
that, since $\pi$ is continuous, $\supp(\nu)=\overline{\pi(\supp(\mu))}$.
By disintegration of $\mu$ with respect to $\nu$, there exists a family $(\mu_x)_{x\in\Omega}$ of probability Radon measures on $\R^d$ (uniquely defined $\nu$-almost everywhere) such that $\mu=\int_{\Omega}\mu_x\, d\nu(x)$, i.e., 
$$
\int_{\Omega\times\R^d} h(x,\xi)\, d\mu(x,\xi) = \int_{\Omega} \int_{\R^d} h(x,\xi)\, d\mu_x(\xi)\, d\nu(x)
$$
for every Borel measurable function $h:\Omega\times\R^d\rightarrow[0,+\infty)$ (see, e.g., \cite{Bogachev}).
Moreover, we set $\mu_x=0$ whenever $x\in\Omega\setminus\supp(\nu)$.

When $\Omega$ is a smooth manifold, if $\mu\in\mathcal{P}^{ac}(\Omega\times\R^d)$ with a density $f\in L^1(\Omega\times\R^d)$, i.e., $\frac{d\mu}{dx\, d\xi}(x,\xi) = f(x,\xi)$, then $\nu$ is absolutely continuous, of density $\frac{d\nu}{dx}(x)=\int_{\R^d}f(x,\xi)\, d\xi$, and for $\nu$-almost every $x\in\Omega$ the probability measure $\mu_x$ has the density $\frac{d\mu_x}{d\xi}(\xi) = \frac{f(x,\xi)}{\int_{\R^d}f(x,\xi')\, d\xi'}$.

Given any $\mu^1,\mu^2\in\mathcal{P}_1(\Omega\times\R^d)$ having the same marginal $\nu$ on $\Omega$, we define
\begin{equation}\label{def_L1Wp}
L^1_\nu W_p(\mu^1,\mu^2) = \int_{\Omega} W_p(\mu^1_x,\mu^2_x)\, d\nu(x) .
\end{equation}
Obviously, $L^1_\nu W_p$ is a distance on the subset denoted $\mathcal{P}_p^\nu(\Omega\times\R^d)$ of elements of $\mathcal{P}_p(\Omega\times\R^d)$ having the same marginal $\nu$. Note that $W_1(\mu^1,\mu^2)\leq L^1_\nu W_1(\mu^1,\mu^2)$ for all $\mu_1,\mu_2\in\mathcal{P}_1^\nu(\Omega\times\R^d)$.\footnote{Indeed, $\int_{\Omega\times\R^d}f\, d(\mu^1-\mu^2) = \int_{\Omega} \int_{\R^d} f(x,\xi)\, d(\mu^1_x-\mu^2_x)\, d\nu(x) \leq \int_{\Omega} \Lip(f(x,\cdot))\, W_1(\mu^1_x,\mu^2_x)\, d\nu(x)$ for every $f\in\Lip(\Omega\times\R^d)$, and if $\Lip(f)\leq 1$ then $\Lip(f(x,\cdot))\leq 1$ for every $x\in\Omega$. Then, take the supremum over all $f$.}

\paragraph{Tagged partitions.}
Let $\nu\in\mathcal{P}(\Omega)$. 
We say that $(\mathcal{A}^N,X^N)_{N\in\N^*}$ is a family of \emph{tagged partitions} of $\Omega$ associated with $\nu$ if 
$\mathcal{A}^N=(\Omega^N_1,\ldots,\Omega^N_N)$ is a $N$-tuple of disjoint subsets $\Omega^N_i\subset\Omega$ such that \begin{equation}\label{def_tagged}
\Omega = \bigcup_{i=1}^N\Omega^N_i\qquad \textrm{with}\qquad
\nu(\Omega^N_i)=\frac{1}{N}
\quad \textrm{and}\quad
\diam_\Omega(\Omega^N_i)\leq \frac{C_\Omega}{N^r}\qquad \forall i\in\{1,\ldots,N\},
\end{equation}
for some $C_\Omega>0$ and $r>0$ not depending on $N$,
and $X^N=(x^N_1,\ldots,x^N_N)$ is a $N$-tuple of points $x^N_i\in\Omega^N_i$. 
Here, $\diam_\Omega(\Omega^N_i)$ is the supremum of all $\mathrm{d}_\Omega(x,x')$ over all possible $x,x'\in\Omega^N_i$.

Families of tagged partitions always exist when $\Omega$ is a compact $n$-dimensional smooth manifold with or without boundary and $\nu$ is a Lebesgue measure on $\Omega$, with $r=1/n$. 
For instance, when $\Omega=[0,1]$, we take $\Omega^N_i=[a^N_i,a^N_{i+1})$ for some subdivision $0=a^N_1<a^N_2<\cdots<a^N_{N+1}=1$ satisfying \eqref{def_tagged}; when $d\nu(x)=dx$, a natural choice is $a^N_i=\frac{i-1}{N}$, and $x^N_i=a^N_i$ or $\frac{a^N_i+a^N_{i+1}}{2}$, for every $i\in\{1,\ldots,N\}$ (and then $C_\Omega=1$ and $r=1$ in this case). When $\Omega$ is a compact domain of $\R^n$, a family of tagged partitions is obtained by considering a family of meshes, as classically done in numerical analysis, with $r=1/n$.

The concept of tagged partition is used in Riemann (and more generally, Henstock-Kurzweil) integration theory. 
We refer to \cite{Fremlin} for (much more) general results.
A real-valued function $f$ on $\Omega$, of compact support, is said to be $\nu$-Riemann integrable if it is bounded, $\nu$-measurable, and if, for any family $(\mathcal{A}^N,X^N)_{N\in\N^*}$ of tagged partitions, we have
\begin{equation}\label{CV_Riemann_sum_1}
\sum_{i=1}^N \int_{\Omega_i^N} \vert f(x)-f(x^N_i) \vert\, d\nu(x) = \mathrm{o}(1) 
\end{equation}
and thus
\begin{equation}\label{CV_Riemann_sum_2}
\int_{\Omega} f\, d\nu = \frac{1}{N} \sum_{i=1}^N f(x^N_i) + \mathrm{o}(1)
\end{equation}
as $N\rightarrow +\infty$.
A function $f$ of essential compact support on $\Omega$ is $\nu$-Riemann integrable if and only if $f$ is bounded and continuous $\nu$-almost everywhere on $\Omega$.

\section{From microscopic to macroscopic scale: the continuum / graph limit}\label{sec_graphlimit}
In this section we explore the point of view of Riemann sums, in order to derive error estimates mainly resulting from the discrepancy between an integral and a Riemann sum, building on the concept of graph limit introduced in \cite{Medvedev_SIMA2014}.

\subsection{Continuum / graph limit equation}\label{sec_euler_equation}
Given any $\nu\in\mathcal{P}(\Omega)$, we define the nonlinear operator
$A:\R\times L^\infty_\nu(\Omega,\R^d)\rightarrow L^\infty_\nu(\Omega,\R^d)$ (depending on $\nu$) by
\begin{equation}\label{def_A_general}
\boxed{
A(t,y)(x) = \int_{\Omega} G(t,x,x',y(x),y(x'))\, d\nu(x')
}
\end{equation}
(recall that $G$ satisfies Assumption \ref{G})
for every $t\in\R$ and for every $y\in L^\infty_\nu(\Omega,\R^d)$.
We consider the continuum / graph limit equation
\begin{equation}\label{Euler_general}
\boxed{
\partial_t y(t,\cdot) = A(t,y(t,\cdot))
}
\end{equation}
It is a nonlinear (nonlocal) integro-differential equation. 

We will see in Section \ref{sec_hydro} the interpretation of $y(t,x)$ as a ``velocity field" (moment of order $1$ of the solution of the Vlasov equation).

\begin{theorem}[Existence and uniqueness for the continuum / graph limit equation \eqref{Euler_general}]\label{thm_euler}
Assume that $\Omega$ is compact. Let $\nu\in\mathcal{P}(\Omega)$ and let $y^0\in L^\infty_\nu(\Omega,\R^d)$.
We denote by $K'=\mathrm{ess.im}(y^0)$ its essential range (it is a compact subset of $\R^d$) and we set $K=\Omega\times K'$ (compact).
There exists a unique solution $t\mapsto y(t,\cdot)\in L^\infty_\nu(\Omega,\R^d)$ on $[0,T_{\max}(K))$ (where $T_{\max}(K)$ is given by Lemma \ref{lem_Tmax}) of the nonlinear continuum / graph limit equation \eqref{Euler_general} such that $y(0,\cdot)=y^0(\cdot)$, of class $\mathscr{C}^1$ with respect to $t$.

Moreover, if $y^0\in \mathscr{C}^0(\Omega,\R^d)$ then $y(t,\cdot)\in \mathscr{C}^0(\Omega,\R^d)$ for every $t\in [0,T_{\max}(K))$.
\end{theorem}

Local-in-time existence and uniqueness for the continuum / graph limit equation \eqref{Euler_general} follow directly from the Picard-Lindel\"of (Cauchy-Lipschitz) theorem applied in the Banach space $L^\infty_\nu(\Omega,\R^d)$, but we prefer to see Theorem \ref{thm_euler} as a consequence of Theorem \ref{thm_vlasov} (existence and uniqueness for Vlasov equations) and of Proposition \ref{prop_monokinetic} (monokinetic measures), as it will be made precise in Remark \ref{rem_Tmax_Euler} in Section \ref{sec_nu-monokinetic}. For the last statement of Theorem \ref{thm_euler}, we also refer to Theorem \ref{thm_estim_graph} and to its proof (see Appendix \ref{app_proof_thm_estim_graph}) for possible variants. 

\begin{remark}\label{rem_Omega_not_compact}
When $\Omega$ is not compact, the above result remains true provided that there exists a compact subset $\Omega_1$ of $\Omega$ such that $G(t,x,x',\xi,\xi')=0$ for every $x\in\Omega\setminus\Omega_1$ and all $(t,x',\xi,\xi')\in\R\times\Omega\times\R^d\times\R^d$, and the initial condition $y^0$ for the CGL equation satisfies $y^0(x)=0$ for $\nu$-almost every $x\in\Omega\setminus\Omega_1$. Indeed, in this case the solution of the continuum / graph limit equation is supported in $\Omega_1$. Alternatively, we can also assume that $\supp(\nu)\subset\Omega_1$.
\end{remark}

\begin{remark}\label{rem_euler_general_indistinguishable}
When $G$ does not depend on $(x,x')$ (and thus, particles are indistinguishable), it makes sense anyway to consider the continuum limit equation \eqref{Euler_general}, with a solution $y(t,x)$ depending on $x\in\Omega$. Although the particles are indistinguishable, the set of labels $\Omega$ may be seen as a way to ``enforce" distinguishability at the level of the CGL equation, by assigning to each particle a label that is an element of $\Omega$. As we will see in Section \ref{sec_ex_euler}, such continuum limit equations do not seem to have been studied in the literature in the indistinguishable case.
Note that distinguishability is made possible because we take an initial condition $y^0(\cdot)$ depending on $x\in\Omega$ in a nontrivial way. In contrast, if $y^0(\cdot)\equiv y^0\in\R^d$ is constant, then $y(t,\cdot)\equiv y(t)\in\R^d$ does not depend on $x$ (this follows from Remark \ref{rem_lem_xxprime} in Appendix \ref{app_proof_thm_estim_graph}) and the continuum limit equation becomes the differential equation $\dot y(t) = G(t,y(t),y(t))$ in $\R^d$, which is much less meaningful.
\end{remark}

\begin{remark}[``Empirically embedding" the particle system to the continuum / graph limit equation]\label{rem_embedding_particle_Euler}
In this remark, we assume that $\nu = \nu^e_{X^N}=\frac{1}{N}\sum_{j=1}^N \delta_{x^N_j}$. The operator $A$ defined by \eqref{def_A_general} is then given by $A(t,y)(x) = \frac{1}{N}\sum_{j=1}^N G(t,x,x_j^N,y(x),y(x_j^N))$ for every $t\in\R$ and every $y\in L^\infty_\nu(\Omega,\R^d)$. Consequently:
\begin{itemize}[leftmargin=3.5mm,parsep=1mm,itemsep=1mm,topsep=1mm]
\item If $t\mapsto\Xi^N(t)=(\xi^N_1(t),\ldots,\xi^N_N(t))$ is a solution of the particle system \eqref{particle_system} then, defining $y(t,x)=\xi^N_i(t)$ if $x=x^N_i$ for $i\in\{1,\ldots,N\}$ and $0$ otherwise, $t\mapsto y(t,\cdot)$ is a solution of the continuum / graph limit equation \eqref{Euler_general}. 
\item Conversely, if $t\mapsto y(t,\cdot)$ is a solution of the continuum / graph limit equation \eqref{Euler_general} then, defining $\xi^N_i(t)=y(t,x^N_i)$ for $i\in\{1,\ldots,N\}$, $t\mapsto\Xi^N(t)=(\xi^N_1(t),\ldots,\xi^N_N(t))$ is a solution of the particle system \eqref{particle_system}. Note however that we may have $y(t,x)\neq 0$ for $x\notin\{x^N_1,\ldots,x^N_N\}$.
\end{itemize}
\end{remark}

The above empirical embedding is rather tautological. It becomes much more meaningful to fix a probability measure $\nu\in\mathcal{P}(\Omega)$ and to approximate the solutions of the continuum / graph limit equation, in a quantitative sense, by solutions of the particle system. The rough idea is to approximate the integral in the continuum / graph limit equation
$$
\partial_t y(t,x) = \int_{\Omega} G(t,x,x',y(t,x),y(t,x'))\, d\nu(x')
$$
by a Riemann sum, so that
$$
\partial_t y(t,x_i^N) \simeq \frac{1}{N} \sum_{j=1}^N G(t,x_i^N,x_j^N,y(t,x_i^N),y(t,x_j^N))
$$
for $N$ sufficiently large, and then, comparing with \eqref{particle_system}, it is expected that $\xi_i^N(t) \simeq y(t,x_i^N)$ for every $i\in\{1,\ldots,N\}$, where $t\mapsto\Xi^N(t)=(\xi^N_1(t),\ldots,\xi^N_N(t))$ is a solution of the particle system \eqref{particle_system}, for appropriate initial conditions.
This is done in detail in Section \ref{sec_estimates_graph_limit} hereafter. 

The application of the Riemann sum theorem is actually at the core of the notion of \emph{graph limit} used in \cite{Medvedev_SIMA2014} to pass to the continuum limit in nonlocally coupled dynamical networks (see also the recent papers \cite{AyiPouradierDuteil_JDE2021, BiccariKoZuazua_M3AS2019, BonnetPouradierDuteilSigalotti_M3AS2022, BoudinSalvaraniTrelat_SIMA2022, EspositoPatacchiniSchlichtingSlepcev_ARMA2021, JabinPoyatoSoler}). Obtaining error estimates is then quite easy by developing standard numerical analysis arguments, which consist of estimating the discrepancy between an integral and approximating Riemann sums. This is the contents of the proofs of Theorems \ref{thm_estim_graph} and \ref{thm_estim_graph_2} hereafter.

The terminology ``\emph{graph limit}" refers to the graph interpretation of some classes of particle systems, like, very typically, the opinion propagation model given in Example \ref{HK_particle} (see Section \ref{sec_ex_euler} for its graph limit): in this example, for any $N\in\N^*$, to the matrix of coefficients $\sigma_{ij}^N$ is associated a directed graph whose vertices are the indices $i\in\{1,\ldots,N\}$ and which has an edge from $i$ to $j$ if $\sigma_{ij}^N>0$. In this context, under Assumption \ref{G}, the function $\sigma$ which satisfies $\sigma(x_i^N,x_j^N)=\sigma_{ij}^N$ is referred to as a \emph{graphon} and is the ``continuum limit" of the graph as $N\rightarrow+\infty$. This is why the continuum / graph limit equation can also be called the graph limit of the system of particles. In \cite{Medvedev_ARMA2014, Medvedev_SIMA2014}, for appropriate choices of interaction coefficients, the system \eqref{HK_particle} is interpreted as a nonlinear heat equation on a graph. The graph interpretation may be particularly relevant when wanting to prove, for instance, consensus results by exploiting the connectivity properties of the graph, as in \cite{BoudinSalvaraniTrelat_SIMA2022}; we also mention \cite{EspositoPatacchiniSchlichtingSlepcev_ARMA2021} for exploiting the graph structure and \cite{JabinPoyatoSoler} for the related mean field context. 
We stress anyway that, as said above, from the analysis point of view, taking the graph limit mainly consists of taking the limit in a Riemann sum, as in \eqref{CV_Riemann_sum_2}. This is thanks to this ``numerical analysis" viewpoint that we can easily derive general error estimates, as shown hereafter.

\subsection{Convergence estimates for the particle-to-CGL passage}\label{sec_estimates_graph_limit}
Throughout this section, we assume that $\Omega$ is compact.
Let $\nu\in\mathcal{P}(\Omega)$.
We consider the general nonlinear continuum / graph limit equation \eqref{Euler_general}, with the nonlinear operator $A$ defined by \eqref{def_A_general}. 
Recall that $G$ satisfies Assumption \ref{G}.

We also assume that there exists a family $(\mathcal{A}^N,X^N)_{N\in\N^*}$ of tagged partitions associated with $\nu$ satisfying \eqref{def_tagged} (see Section \ref{sec_general_notations}), with $\mathcal{A}^N=(\Omega^N_1,\ldots,\Omega^N_N)$ and $X^N=(x^N_1,\ldots,x^N_N)$. 

We state two theorems whose roles are complementary and whose distinction is worth emphasizing. Theorem \ref{thm_estim_graph} compares the empirical reconstruction $y^N(t,\cdot)$ obtained from the particle solution \emph{starting from sampled data} $\xi_i^N(0)=y^0(x_i^N)$, with the solution $y(t,\cdot)$ of the CGL equation \emph{starting from the continuous datum} $y^0$. It thus quantifies the discrepancy between the discrete and the continuum dynamics when both are launched from the ``same" continuous initial profile.
Theorem \ref{thm_estim_graph_2}, by contrast, takes \emph{any} discrete initial data $\Xi^N(0)$ and compares the resulting empirical reconstruction with the solution $y_N(t,\cdot)$ of the CGL equation initialized from the corresponding piecewise-constant function $y^N(0,\cdot)$. The two theorems thus answer two different questions: ``how well does the particle system approximate a smooth limiting profile?" (Theorem \ref{thm_estim_graph}) and ``how stable is the particle-to-CGL passage when the initial data are themselves piecewise constant?" (Theorem \ref{thm_estim_graph_2}). The latter is what is needed in order to validate semi-empirical approximation procedures used later in the paper.

\begin{theorem}\label{thm_estim_graph}
Let $y^0$ be a bounded and $\nu$-almost everywhere continuous function on $\Omega$ (thus, $\nu$-Riemann integrable), with values in $\R^d$.
\\
On the one part, we consider the unique solution $t\mapsto y(t,\cdot)\in L^\infty(\Omega,\R^d)$ on $[0,T_{\max}(K))$ of the (nonlinear) continuum / graph limit equation \eqref{Euler_general} such that $y(0,\cdot) = y^0(\cdot)$, where $K=\Omega\times\mathrm{ess.im}(y^0)$ (compact) and $\mathrm{ess.im}(y^0)\subset\R^d$ is the essential range of $y^0$.
\\
On the other, for any $N\in\N^*$, we consider the unique solution $t\mapsto\Xi^N(t)=(\xi^N_1(t),\ldots,\xi^N_N(t))\in\R^{dN}$ on $[0,T_{\max}(K))$ of the particle system \eqref{particle_system} such that $\xi^N_i(0)=y^0(x^N_i)$ for every $i\in\{1,\ldots,N\}$,
and we set
\begin{equation}\label{def_y_Xi}
y^N(t,x) = \sum_{i=1}^N \xi^N_i(t) \, \mathds{1}_{\Omega^N_i}(x) \qquad\forall (t,x)\in\R\times\Omega
\end{equation}
where $\mathds{1}_{\Omega^N_i}$ is the characteristic function of $\Omega^N_i$, defined by $\mathds{1}_{\Omega^N_i}(x)=1$ if $x\in\Omega^N_i$ and $0$ otherwise.
\begin{itemize}[leftmargin=3.5mm]
\item For every $t\in[0,T_{\max}(K))$, $y(t,\cdot)$ is bounded and continuous $\nu$-almost everywhere on $\Omega$, with the same continuity set as $y^0$, 
and
\begin{equation}\label{estim_graph_1gen}
\Vert y(t,\cdot) - y^N(t,\cdot) \Vert_{L^\infty(\Omega,\R^d)} = \mathrm{o}(1)
\end{equation}
as $N\rightarrow+\infty$, where the remainder term $\mathrm{o}(1)$ is uniform with respect to $t$ on any compact interval of $[0,T_{\max}(K))$. In particular,
\begin{equation}\label{estim_graph_1}
\max_{i\in\{1,\ldots,N\}} \Vert y(t,x^N_i) - \xi^N_i(t) \Vert = \mathrm{o}(1) . 
\end{equation}
\item Assume that there exists $\alpha\in(0,1]$ such that $y^0\in \mathscr{C}^{0,\alpha}(\Omega,\R^d)$ 
and $G$ is locally $\alpha$-H\"older continuous with respect to $(x,x',\xi,\xi')$ (uniformly with respect to $t$ on any compact).
Then, for every $t\in[0,T_{\max}(K))$, we have $y(t,\cdot)\in \mathscr{C}^{0,\alpha}(\Omega,\R^d)$ with 
\begin{equation}\label{holalpha}
\Hol_\alpha(y(t,\cdot)) \leq e^{tL_y(t)}\left( 1+\Hol_\alpha(y(0,\cdot)) \right) 
\end{equation}
and, for every $N\in\N^*$,
\begin{equation}\label{estim_graph_2}
\max_{i\in\{1,\ldots,N\}} \Vert y(t,x^N_i) - \xi^N_i(t) \Vert \leq  \frac{C_\Omega^\alpha}{N^{r\alpha}} \left(1+\Hol_\alpha(y^0) \right) e^{2tL_y^N(t)} 
\end{equation}
and actually,
\begin{equation}\label{estim_graph_3}
\Vert y(t,\cdot) - y^N(t,\cdot) \Vert_{L^\infty(\Omega,\R^d)} \leq 2\frac{C_\Omega^\alpha}{N^{r\alpha}} \left(1+\Hol_\alpha(y^0) \right) e^{2tL_y^N(t)} 
\end{equation}
\end{itemize}
where $C_\Omega$ is given by \eqref{def_tagged}. The constant $L_y^N(t)$ in \eqref{estim_graph_2} and \eqref{estim_graph_3} is defined by
\begin{equation}\label{hypGL1}
L_y^N(t) = \max_{0\leq\tau\leq t} \Hol_\alpha (G(\tau,\cdot,\cdot,\cdot,\cdot)_{\vert \Omega^2\times S_y^N(\tau)^2})
+ \max_{\substack{x,x'\in\Omega \\ 0\leq\tau\leq t}}\Lip(G(\tau,x,x',\cdot,\cdot)_{\vert S_y^N(\tau)^2})  
\end{equation}
where $S_y^N(\tau)\subset\R^d$ is the (compact) convex closure of all $y(\tau,x)$ for $x\in\Omega$ 
and all $\xi^N_i(\tau)$ for $i\in\{1,\ldots,N\}$.
The constant $L_y(t)$ in \eqref{holalpha} is defined as $L_y^N(t)$ but with $S_y^N(\tau)$ replaced by $S_y(\tau)$ that is the convex closure of all $y(\tau,x)$ for $x\in\Omega$, i.e., like $S_y^N(\tau)$ but without the $\xi^N_i(\tau)$. 
We have $L_y(t)\leq L_y^N(t)$.
\end{theorem}

Theorem \ref{thm_estim_graph} is proved in Appendix \ref{app_proof_thm_estim_graph}.
Note that, by Lemma \ref{lem_Tmax}, given any $T\in[0,T_{\max}(K))$, the sets $S_y^N(t)$ and thus the scalars $L_y^N(t)$ are uniformly bounded with respect to $t\in[0,T]$ and to $N\in\N^*$.

\begin{remark}\label{rem_thm_estim_graph_noncompact}
Having in mind Remark \ref{rem_Omega_not_compact}, 
Theorem \ref{thm_estim_graph} can be extended to the case where $\Omega$ is not compact, under the following additional assumptions:
\begin{itemize}
\item the family of tagged partitions is such that the points $x^N_i$ remain in a compact subset of $\Omega$;
\item the initial condition $y^0$ is of compact essential support;
\item the set $S_y^N(\tau)\subset\Omega\times\R^d$ is defined as the compact closure of all $(x,y(\tau,x))$ for $x\in\mathrm{ess\,supp}(y(\tau,\cdot))$ (essential support) and all $(x^N_i,\xi^N_i(\tau))$ for $i\in\{1,\ldots,N\}$.
\end{itemize}
The above assumptions imply that $y(t,\cdot)$ is of compact essential support, for every $t\geq 0$, and that $L_y^N(t)$ is well defined. 
\end{remark}

\begin{theorem}\label{thm_estim_graph_2}
Let $K'$ be a compact subset of $\R^d$.
Given any $N\in\N^*$, let $\Xi^N_0\in (K')^N$. We set $K=\Omega\times K'$.

On the one part, we consider the unique solution $t\mapsto\Xi^N(t)=(\xi^N_1(t),\ldots,\xi^N_N(t))\in\R^{dN}$ on $[0,T_{\max}(K))$ of the particle system \eqref{particle_system} such that $\Xi^N(0)=\Xi^N_0$, and we define $y^N(t,x) $ by \eqref{def_y_Xi}.

On the other part, we consider the unique solution $t\mapsto y_N(t,\cdot)\in L^\infty(\Omega,\R^d)$ on $[0,T_{\max}(K))$ of the continuum / graph limit equation \eqref{Euler_general} such that $y_N(0,\cdot) = y^N(0,\cdot)$ (i.e., $y_N(0,x)=\xi^N_i(0)$ if $x\in\Omega^N_i$).
Then, for every $t\in[0,T_{\max}(K))$,
\begin{equation}\label{estim_graph_1gen_2}
\Vert y_N(t,\cdot) - y^N(t,\cdot) \Vert_{L^\infty(\Omega,\R^d)} = \mathrm{o}(1)
\end{equation}
as $N\rightarrow+\infty$, where the remainder term $\mathrm{o}(1)$ is uniform with respect to $t$ on any compact interval of $[0,T_{\max}(K))$.

If moreover $G$ is locally $\alpha$-H\"older continuous with respect to $(x,x',\xi,\xi')$ (uniformly with respect to $t$ on any compact), then, for every $N\in\N^*$ and every $t\in[0,T_{\max}(K))$,
\begin{equation}\label{estim_graph_2_2}
\Vert y_N(t,\cdot) - y^N(t,\cdot) \Vert_{L^\infty(\Omega,\R^d)} \leq 2\frac{C_\Omega^\alpha}{N^{r\alpha}} e^{2tL_{y_N}^N(t)}  \qquad \forall t\geq 0, 
\end{equation}
where $L_{y_N}^N(t)$ is defined by \eqref{hypGL1} (with $y$ replaced by $y_N$).
\end{theorem}

Theorem \ref{thm_estim_graph_2} is proved in Appendix \ref{app_proof_thm_estim_graph_2}.
Note that, by Lemma \ref{lem_Tmax}, given any $T\in[0,T_{\max}(K))$, the scalars $L_{y_N}^N(t)$ are uniformly bounded with respect to $t\in[0,T]$ and to $N\in\N^*$.

Note that, in particular, taking $x=x_i^N$ in \eqref{estim_graph_2_2}, we have
$$
\max_{i\in\{1,\ldots,N\}} \Vert y_N(t,x^N_i) - \xi^N_i(t) \Vert\leq 2\frac{C_\Omega^\alpha}{N^{r\alpha}} e^{2tL_{y_N}^N(t)}  ,
$$
which improves the estimates obtained in \cite{AyiPouradierDuteil_JDE2021}.

\begin{remark}
The trivial case where $y^0(\cdot)\equiv y^0\in\R^d$, mentioned in Remark \ref{rem_euler_general_indistinguishable}, corresponds in the framework of Theorems \ref{thm_estim_graph} and \ref{thm_estim_graph_2} to taking $\xi_i^N(0)$ not depending on $i$ (equivalently, $\xi_i^N(t)$ not depending on $i$, for every $t$), which is the case where all particles coincide. 
\end{remark}

\begin{remark}
In Theorems \ref{thm_estim_graph} and \ref{thm_estim_graph_2}, we have first established the convergence results \eqref{estim_graph_1gen} and \eqref{estim_graph_1gen_2}, under the sole Assumption \ref{G}, i.e., $G$ is continuous with respect to $(x,x',\xi,\xi')$ and locally Lipschitz with respect to $(\xi,\xi')$. Under the additional assumption that $G$ is locally H\"older with respect to $(x,x',\xi,\xi')$, we derive the convergence estimates \eqref{estim_graph_2} and \eqref{estim_graph_2_2}. As mentioned in Section \ref{sec_objectives}, Assumption \ref{G} leads to topological convergence results, and reinforcing Assumption \ref{G} is compelling to obtain convergence rates.
\end{remark}

\subsection{Examples}\label{sec_ex_euler}
Following the examples of Section \ref{sec_ex_particle}, we now give the corresponding continuum / graph limit equation.

\medskip\noindent
-- For the Hegselmann--Krause (opinion propagation) model \eqref{HK_particle}, under Assumption \ref{G} we have $G(t,x,x',\xi,\xi')=\sigma(x,x')(\xi'-\xi)$ and the continuum / graph limit equation is
\begin{equation}\label{HK_euler}
\partial_t y(t,x) = \int_{\Omega} \sigma(x,x') (y(t,x')-y(t,x)) \, d\nu(x').
\end{equation}
In this case, the operator $A$ (defined by \eqref{def_A_general}) is linear and is given by
$$
(Ay)(x) = \int_{\Omega} \sigma(x,x') (y(x')-y(x))\, d\nu(x')
\qquad\forall y\in L^\infty_\nu(\Omega,\R^d)
$$
(see \cite{BiccariKoZuazua_M3AS2019}). 
Extended to $L^2_\nu(\Omega,\R^d)$, it is a Hilbert-Schmidt operator. The spectral study of $A$ has been done in \cite{BoudinSalvaraniTrelat_SIMA2022} (with $\nu$ the Lebesgue measure) in view of deriving consensus results.
The graph interpretation is particularly meaningful in this example, and the function $\sigma$ (assumed to exist) is called a graphon (see \cite{AyiPouradierDuteil_2023} for a recent study and see \cite{BonnetPouradierDuteilSigalotti_M3AS2022} for a time-varying case).

\medskip\noindent
-- For the Kuramoto particle system \eqref{kuramoto_particle}, under Assumption \ref{G} we have $\Omega=\R\times[0,1]$ and $G(t,x,x',\xi,\xi') = \alpha + \sigma(\beta,\beta')\sin(\xi'-\xi)$ (where $x=(\alpha,\beta)$ and $x'=(\alpha',\beta')$), and the continuum / graph limit equation is
$$
\partial_t y(t,x) = \alpha + \int_{\Omega} \sigma(\beta,\beta') \sin(y(t,x')-y(t,x)) \, d\nu(x').
$$
The graph limit operator at the right-hand side of the above equation is introduced in \cite{ChibaMedvedev} although, in that reference, the authors focus on the study of the mean field limit (see Section \ref{sec_ex_vlasov}).

\medskip\noindent
-- For the first-order system \eqref{FK_particle}, we have $G(t,x,x',\xi,\xi')=F(\xi)+K(\xi-\xi')$ and the continuum / graph limit equation is
\begin{equation}\label{FK_euler}
\partial_t y(t,x) = F(y(t,x)) + \int_\Omega K(y(t,x)-y(t,x')) \, d\nu(x').
\end{equation}
In this case, the operator $A$ (defined by \eqref{def_A_general}) is nonlinear, nonlocal, and does not depend on $t$. 
To the best of our knowledge, the equation \eqref{FK_euler} has not been studied in the literature.

\medskip\noindent
-- For the Cucker--Smale dynamics \eqref{CS_particle}, setting $y=(y_1,y_2)\in\R^r\times\R^r$, the continuum / graph limit equation is 
\begin{equation}\label{CS_euler}
\begin{split}
\partial_t y_1(t,x) &= y_2(t,x), \\
\partial_t y_2(t,x) &= \int_\Omega a(\Vert y_1(t,x')-y_1(t,x)\Vert) (y_2(t,x')-y_2(t,x))\, d\nu(x') .
\end{split}
\end{equation}
Similarly, the equation \eqref{CS_euler} does not appear to have been considered before.

As discussed in Remark \ref{rem_euler_general_indistinguishable}, thanks to the set $\Omega$ we have in some sense ``enforced" distinguishability. If one takes an initial condition that is constant with respect to $x$ then $y_1(t,x)=y_1(t)$ and $y_2(t,x)=y_2(x)$ do not depend on $x$ and we have $\dot y_1(t)=y_2(t)$ and $\dot y_2(t)=0$, which is the familiar fact that the derivative of the average position is the average velocity, and that, in absence of an external force, the average velocity is conserved. 

The Cucker--Smale dynamics \eqref{CS_particle} is a second-order system. For second-order dynamics, we will see in Section \ref{sec_second_order_hydrodynamics} a different definition of CGL equation which gives rise to interesting 
(and already known and studied)
dynamics.

\medskip\noindent
-- For the (indistinguishable) second-order dynamics \eqref{second_order_K}, similarly to the Cucker--Smale example, the continuum limit equation is
\begin{equation*}
\begin{split}
\partial_t y_1(t,x) &= y_2(t,x), \\
\partial_t y_2(t,x) &= \int_\Omega K(y_1(t,x),y_1(t,x')) \, d\nu(x') .
\end{split}
\end{equation*}

\section{From microscopic to mesoscopic scale I: mean field through empirical measures {\normalsize (from ODEs to Vlasov)}}\label{sec_vlasov}
Within the Lagrangian viewpoint, the $N$ particles at time $t$ are embedded as Dirac masses into the space of Radon measures, and their corresponding average, the empirical measure, converges by the mean field limit procedure, as $N\rightarrow+\infty$, to a probability Radon measure $\mu(t)$ on $\Omega\times\R^d$ satisfying the Vlasov equation. When $\mu(t)$ is absolutely continuous with respect to a Lebesgue measure, its density $f(t,x,\xi)$ represents the density of particles with label $x$ and state $\xi$ at time $t$.

\subsection{Vlasov equation}\label{sec_vlasov_def}
Given any $\mu\in\mathcal{P}_c(\Omega\times\R^d)$, we define $\nu=\pi_*\mu$ by \eqref{def_nu} (marginal of $\mu$ on $\Omega$) and we define the \emph{mean field}, also called \emph{interaction kernel}, as the non-local time-dependent vector field on $\R^d$, parametrized by $x\in\Omega$, given by
\begin{equation}\label{def_mean_field}
\boxed{
\begin{aligned}
\mathcal{X}[\mu](t,x,\xi) &= \int_{\Omega\times\R^d} G(t,x,x',\xi,\xi') \, d\mu(x',\xi')  \\
&= \int_{\Omega} \int_{\R^d} G(t,x,x',\xi,\xi') \, d\mu_{x'}(\xi')\, d\nu(x')  \qquad \forall (t,x,\xi)\in\R\times\Omega\times\R^d 
\end{aligned}
}
\end{equation}
(recall that $G$ satisfies Assumption \ref{G}).
Note that $\mathcal{X}[\mu](t,x,\xi)$ is the expectation of $G(t,x,x',\xi,\xi')$ with respect to the measure $\mu$, performed with respect to $(x',\xi')\in\Omega\times\R^d$ (see Appendix \ref{app_mean_field} for more details and consequences of that definition).

We consider the \emph{Vlasov (or continuity) equation}
\begin{equation}\label{vlasov}
\boxed{
\partial_t\mu + \div_\xi(\mathcal{X}[\mu]\mu) = 0 
}
\end{equation}
where the divergence\footnote{Recall that $\div(\mathcal{X}\mu)=L_\mathcal{X}\mu$ (Lie derivative of the measure $\mu$) is the measure defined by $\langle L_\mathcal{X}\mu,f\rangle = -\langle\mu,L_\mathcal{X}f\rangle = -\int_{\R^d} \mathcal{X}.\nabla f\, d\mu$ for every $f\in \mathscr{C}^\infty_c(\R^d)$.} acts only with respect to $\xi$. 
It is a nonlocal transport equation because the velocity field $\mathcal{X}[\mu]$ defined by \eqref{def_mean_field} is nonlocal. 

\begin{remark}\label{rem_notdependt}
Given any solution $t\mapsto\mu(t)$ on $[0,T]$ of the Vlasov equation \eqref{vlasov} (see further for the rigorous definition of a solution), the total mass $\mu(t)(\Omega\times\R^d)$ is constant with respect to $t$, i.e., $\mu(t)$ is a probability measure for every $t\in[0,T]$. Also, the marginal $\nu = \pi_*\mu(t)$ does not depend on $t$, because the Vlasov equation can be written as $\partial_t\mu + L_{\mathcal{X}[\mu]}\mu=0$ with the Lie derivative acting with respect to the variable $\xi$, and we have $\pi_* L_{\mathcal{X}[\mu]}=0$. 
\end{remark}

Disintegrating $\mu_t=\mu(t)$ as $\mu_t = \int_{\Omega} \mu_{t,x}\, d\nu(x)$ with respect to its marginal $\nu=\pi_*\mu_t$ on $\Omega$ (which does not depend on $t$ by Remark \ref{rem_notdependt}), by uniqueness $\nu$-almost everywhere of the disintegration, \eqref{vlasov} is equivalent to 
\begin{equation}\label{vlasov_xfixed}
\partial_t\mu_{t,x} + \div_\xi(\mathcal{X}[\mu_t](t,x,\cdot)\mu_{t,x}) = 0
\end{equation}
for $\nu$-almost every $x\in\Omega$.
Note that the time evolution of $\mu_{t,x}$ depends on the whole $\mu_0$ and not only on $\mu_{0,x}$, since $\mathcal{X}[\mu_t]$ involves an integral over all possible $x'\in\Omega$.

Therefore, the Vlasov equation \eqref{vlasov} can be thought of as an infinite number (if $\Omega$ has an infinite number of elements) of coupled Vlasov equations \eqref{vlasov_xfixed}. The most standard case studied in the literature corresponds to a measure $\mu$ not depending on $x$. 

\medskip
Given any interval $I\subset\R$, let $\mathscr{C}^0(I,\mathcal{P}_c(\Omega\times\R^d))$ be the Banach space of continuous mappings $t\in I\mapsto\mu(t)\in\mathcal{P}_c(\Omega\times\R^d)$, with $\mathcal{P}_c(\Omega\times\R^d)$ endowed with the weak topology (metrized by the Wasserstein distance $W_p$, for any $p\in[1,+\infty)$, as recalled in Section \ref{sec_general_notations}).

We define $\mathscr{C}^0_{\mathrm{comp}}(I,\mathcal{P}_c(\Omega\times\R^d))$ as the set of all $\mu\in\mathscr{C}^0(I,\mathcal{P}_c(\Omega\times\R^d))$ that are \emph{equi-compactly supported} on any compact interval of $I$, meaning that for any $t_1,t_2\in I$, there exists a compact subset $K\subset\Omega\times\R^d$ such that $\supp(\mu(t))\subset K$ for every $t\in[t_1,t_2]$. There exist elements of $\mathscr{C}^0(I,\mathcal{P}_c(\Omega\times\R^d))$ that are not equi-compactly supported on any compact interval of $I$ (for instance, if $I=[0,T]$, take $\mu(t) = (1-e^{-1/t}) \delta_0 + e^{-1/t} \delta_{1/t}$).

In view of obtaining existence and uniqueness of solutions of the Vlasov equation \eqref{vlasov}, we consider the following concept of solution.
Assuming that $0\in I$, by definition, a solution $t\mapsto\mu(t)$ of \eqref{vlasov} on $I$ such that $\mu(0)=\mu_0\in\mathcal{P}_c(\Omega\times\R^d)$ is an element $\mu\in\mathscr{C}^0_{\mathrm{comp}}(I,\mathcal{P}_c(\Omega\times\R^d))$ such that,
denoting $\mu_t = \mu(t)$,\footnote{Note that, seeing $\mu$ as a measure on $I\times\Omega\times\R^d$, the marginal of $\mu$ on $I$ is the Lebesgue measure and the disintegration of $\mu$ is $\mu=\int_I \mu_t\, dt$.}
for every $g\in \mathscr{C}^\infty_c(\Omega\times\R^d)$, the function $t\mapsto\int g\, d\mu_t$ is absolutely continuous on $I$ and
\begin{multline}\label{vlasov_meaning}
\int_{\Omega\times\R^d} g(x,\xi)\, d\mu_t(x,\xi) = \int_{\Omega\times\R^d} g(x,\xi)\, d\mu_0(x,\xi) \\
+ \int_0^t \int_{\Omega\times\R^d} \int_{\Omega\times\R^d} \left\langle \nabla_\xi g(x,\xi), G(\tau,x,x',\xi,\xi') \right\rangle\, d\mu_\tau(x',\xi')\, d\mu_\tau(x,\xi)\, d\tau 
\end{multline}
for almost every $t\in I$.

\begin{theorem}[Existence, uniqueness and stability properties for the Vlasov equation \eqref{vlasov}]\label{thm_vlasov}
Recalling Assumption \ref{G}, let $p\in[1,+\infty)$ be arbitrary.
\begin{enumerate}[leftmargin=*,label=$\bf (\Alph*)$]
\item\label{A} Given any $\mu_0\in\mathcal{P}_c(\Omega\times\R^d)$, setting $T_0 = T_{\max}(\supp(\mu_0))$ (given by Lemma \ref{lem_Tmax}),
there exists a unique solution $\mu\in \mathscr{C}^0_{\mathrm{comp}}([0,T_0),\mathcal{P}_c(\Omega\times\R^d))$ of the Vlasov equation \eqref{vlasov} (in the sense \eqref{vlasov_meaning}) such that $\mu(0)=\mu_0$. 
Moreover, $t\mapsto\mu(t)$ is locally Lipschitz with respect to $t$ for the distance $W_p$, 
and we have 
\begin{equation}\label{mu_image_varphi}
\mu(t) = \varphi_{\mu_0}(t)_* \mu_0 ,
\end{equation}
which is a notation meaning that $\mu_{t,x} = \varphi_{\mu_0}(t,x,\cdot)_* \mu_{0,x}$ for every $t\in[0,T_0)$ and $\nu$-almost every $x\in\Omega$, and where $t\mapsto\varphi_{\mu_0}(t,x,\cdot)$ is the unique solution (Vlasov flow) of 
\begin{equation}\label{vlasov_flow}
\partial_t\varphi_{\mu_0}(t,x,\cdot) = \mathcal{X}[\mu(t)](t,x,\cdot) \circ \varphi_{\mu_0}(t,x,\cdot) 
\end{equation}
such that $\varphi_{\mu_0}(0,x,\cdot)=\mathrm{id}_{\R^d}$ for $\nu$-almost every $x\in\Omega$.
Moreover, if $\mu_0\in\mathcal{P}_c^{ac}(\Omega\times\R^d)$ then $\mu(t)\in\mathcal{P}_c^{ac}(\Omega\times\R^d)$ for every $t\in[0,T_0)$.
Furthermore:
\begin{enumerate}[label=$\bf (A_{\arabic*})$]
\item\label{A1} Any solution of \eqref{vlasov} depends continuously on its initial condition $\mu(0)\in\mathcal{P}_c(\Omega\times\R^d)$ for the weak topology in the following sense: 
given any compact subset $K\subset\Omega\times\R^d$,
given any $\mu(0)\in\mathcal{P}_c(\Omega\times\R^d)$ such that $\supp(\mu(0))\subset K$,
given any (equi-compactly supported) sequence of measures $\mu^k(0)\in\mathcal{P}_c(\Omega\times\R^d)$ such that $\supp(\mu^k(0))\subset K$ for every $k\in\N^*$, if $\mu^k(0)$ converges weakly to $\mu(0)$ (equivalently, $W_p(\mu^k(0),\mu(0))\rightarrow 0$) as $k\rightarrow+\infty$, then $\mu^k(t)$ converges weakly to $\mu(t)$ (equivalently, $W_p(\mu^k(t),\mu(t))\rightarrow 0$) as $k\rightarrow+\infty$, uniformly on any compact interval of $[0,T_{\max}(K))$.
\item\label{A2} For all solutions $\mu^1,\mu^2\in\mathscr{C}^0_{\mathrm{comp}}([0,T],\mathcal{P}_c(\Omega\times\R^d))$ of \eqref{vlasov} (for some $T>0$) such that $\mu^1(0),\mu^2(0)\in\mathcal{P}_c^\nu(\Omega\times\R^d)$ have the same marginal $\nu$ on $\Omega$, 
setting\footnote{Note that $S_{\mu^1,\mu^2}(t)$ is compact, that $\varphi_{\mu^1_0}(t,\supp(\mu^1_0)) = \supp(\mu^1_t)$ and $\supp(\mu^1(t))\cup\supp(\mu^2(t))\subset S_{\mu^1,\mu^2}(t)$.}
$$
S_{\mu^1,\mu^2}(\tau) = \supp(\nu)\times \big( \varphi_{\mu^1_0}(\tau,\supp(\mu^1_0)\cup\supp(\mu^2_0)) \cup \supp(\mu^2(\tau)) \big)
$$
and defining
\begin{equation}\label{def_C}
C_{\mu^1,\mu^2}(t) = \exp\bigg( 2 \int_0^t \underset{(x,\xi),(x',\xi')\in S_{\mu^1,\mu^2}(\tau)}{\mathrm{ess\,sup}} \Vert (\partial_\xi G,\partial_{\xi'}G)(\tau,x,x',\xi,\xi')\Vert \, d\tau\bigg) ,
\end{equation}
we have
\begin{equation}\label{L1Wpmu1mu2}
L^1_\nu W_p(\mu^1(t),\mu^2(t))\leq C_{\mu^1,\mu^2}(t) \, L^1_\nu W_p(\mu^1(0),\mu^2(0)) \qquad \forall t\in[0,T]
\end{equation}
(where $L^1_\nu W_p$ is defined by \eqref{def_L1Wp}).
\end{enumerate}
\item\label{B} Assume moreover that $G$ is locally Lipschitz with respect to $(x,x',\xi,\xi')$ uniformly with respect to $t$ on any compact interval.
For all solutions $\mu^1(\cdot),\mu^2(\cdot)\in\mathscr{C}^0_{\mathrm{comp}}([0,T],\mathcal{P}_c(\Omega\times\R^d))$ of \eqref{vlasov} (for some $T>0$),
setting
$$
S_{\mu^1,\mu^2}(\tau) = \big(\supp(\nu_1)\cup\supp(\nu_2)\big)\times \big( \varphi_{\mu^1_0}(\tau,\supp(\mu^1_0)\cup\supp(\mu^2_0)) \cup \supp(\mu^2(\tau)) \big)
$$
and defining
\begin{equation}\label{def_C_totalLip}
C_{\mu^1,\mu^2}(t) = \exp\bigg( 2 \int_0^t  \Lip( G(\tau,\cdot,\cdot,\cdot,\cdot)_{\vert S_{\mu^1,\mu^2}(\tau)^2} )  \, d\tau \bigg) ,
\end{equation}
we have
\begin{equation}\label{W1mu1mu2}
W_p(\mu^1(t),\mu^2(t))\leq C_{\mu^1,\mu^2}(t) \, W_p(\mu^1(0),\mu^2(0)) \qquad \forall t\in[0,T] .
\end{equation}
\end{enumerate}
\end{theorem}

Theorem \ref{thm_vlasov} is proved in Appendix \ref{app_proof_thm_vlasov}.
The statement \ref{B} of Theorem \ref{thm_vlasov} is a slight extension, with parameter $x$, of \cite[Theorem 2.3]{PiccoliRossiTrelat_SIMA2015} (see also \cite{PiccoliRossi_ACAP2013, PiccoliRossi_ARMA2014, PiccoliRossi_ARMA2016}) where it is assumed that $G$ is globally Lipschitz. 
Without parameter $x$, we recover the famous stability estimate obtained by Dobrushin in \cite{Dobrushin} (see Corollary \ref{cor_vlasov_indistinguishable} further).
To the best of our knowledge, statement \ref{A} is new. Note that, in \ref{A2}, the initial measures $\mu_1(0)$ and $\mu_2(0)$ are required to have the same marginal (and thus, equivalently, $\mu^1(t)$ and $\mu^2(t)$ have the same marginal for any $t$).
By contrast, in \ref{A1} and in \ref{B}, the measures under consideration are not assumed to have the same marginal. In \ref{A1}, the weak convergence $\mu^k(0)\rightharpoonup\mu(0)$ implies the weak convergence $\nu^k\rightharpoonup\nu$ of marginals but it is wrong in general that $\mu^k_x(0)\rightharpoonup\mu_x(0)$ for $x\in\Omega$.

In the statement \ref{B}, the assumption that $G$ is locally Lipschitz with respect to $(x,x',\xi,\xi')$ is much stronger than  \ref{G}: for the Hegselmann--Krause model \eqref{HK_particle} (resp., the Cucker--Smale model \eqref{CS_particle}), this requires $\sigma$ (resp., $a$) to be locally Lipschitz. In general, requiring that $G$ be locally Lipschitz with respect to $(x,x')$ is not a natural assumption for the particle system \eqref{particle_system}.
Note that, under this stronger assumption, the unique solution $\mu(\cdot)$ in \ref{A} is locally Lipschitz with respect to $t$ for the Wasserstein distance $W_1$.

Finally, in the literature, $G$ is usually assumed to be globally Lipschitz. Here, under the weaker assumption \ref{G}, we have a maximal time of definition of $\mu$ depending on the compact support of $\mu_0$, according to Lemma \ref{lem_Tmax}.
Note that, when $G$ is bounded, we can consider in Theorem \ref{thm_vlasov} measures that are not of compact support.

\paragraph{Particular case where $G$ does not depend on $(x,x')$.}
When $G$ does not depend on $(x,x')$, particles are indistinguishable: this is the classical case that has been much studied in the existing literature. 
We now show how this is recovered from our more general framework. 
Given any measure $\mu\in\mathcal{P}(\Omega\times\R^d)$, we define $\bar\mu\in\mathcal{P}(\R^d)$ as the image of $\mu$ under the projection of $\Omega\times\R^d$ onto $\R^d$, that is,
\begin{equation}\label{def_barmu}
\int_{\R^d} f(\xi)\, d\bar\mu(\xi) = \int_{\Omega\times\R^d} f(\xi)\, d\mu(x,\xi) = \int_\Omega \int_{\R^d} f(\xi)\, d\mu_x(\xi)\, d\nu(x)
\end{equation}
for every Borel measurable function $f:\R^d\rightarrow[0,+\infty)$. 
Since $G$ does not depend on $(x,x')$, the mean field $\mathcal{X}[\mu]$ defined by \eqref{def_mean_field} does not depend on $x$ and we have $\mathcal{X}[\mu](t,x,\xi)=\bar{\mathcal{X}}[\bar\mu](t,\xi)$ where the mean field $\bar{\mathcal{X}}[\bar\mu]$ is defined by
\begin{equation}\label{meanfield_indistinguishable}
\bar{\mathcal{X}}[\bar\mu](t,\xi) = \int_{\R^d} G(t,\xi,\xi')\, d\bar\mu(\xi') \qquad \forall (t,\xi)\in\R\times\R^d .
\end{equation}
Accordingly, since the projection onto $\R^d$ commutes with $\partial_t$ and with the divergence with respect to $\xi$, it follows that, if $t\mapsto\mu(t)$ is a solution of the Vlasov equation \eqref{vlasov} then $t\mapsto\bar\mu(t)$ is a solution of the Vlasov equation (without dependence on $x$)
\begin{equation}\label{vlasov_without_x}
\partial_t\bar\mu + \div(\bar{\mathcal{X}}[\bar\mu]\bar\mu) = 0 .
\end{equation}
We have the following corollary of Theorem \ref{thm_vlasov}, already well known in the existing literature (famous Dobrushin estimate, see \cite{Dobrushin}).

\begin{corollary}\label{cor_vlasov_indistinguishable}
Let $p\in[1,+\infty)$ be arbitrary. 
Given any $\bar\mu_0\in\mathcal{P}_c(\R^d)$, there exists a unique solution $\bar\mu\in \mathscr{C}^0_{\mathrm{comp}}([0,T_{\max}(\supp(\bar\mu_0))),\mathcal{P}_c(\R^d))$ of the Vlasov equation \eqref{vlasov_without_x}, locally Lipschitz with respect to $t$ for the distance $W_p$, such that $\bar\mu(0)=\bar\mu_0$, and we have 
$$
\bar\mu(t) = \varphi_{\bar\mu_0}(t,\cdot)_* \bar\mu_0 
$$
where $t\mapsto\varphi_{\bar\mu_0}(t,\cdot)$ is the unique solution of 
$$
\partial_t\varphi_{\bar\mu_0}(t,\cdot) = \bar{\mathcal{X}}[\mu(t)](t,\cdot) \circ \varphi_{\bar\mu_0}(t,\cdot)
$$
such that $\varphi_{\bar\mu_0}(0,\cdot)=\mathrm{id}_{\R^d}$.
Moreover, if $\bar\mu_0\in\mathcal{P}_c^{ac}(\R^d)$ then $\bar\mu(t)\in\mathcal{P}_c^{ac}(\R^d)$ for every $t\in\R$.
Furthermore, we have 
\begin{equation}\label{W1mu1m2_indistinguishable}
W_p(\bar\mu^1(t),\bar\mu^2(t))\leq C_{\bar\mu^1,\bar\mu^2}(t) \, W_p(\bar\mu^1(0),\bar\mu^2(0)) \qquad \forall t\in[0,T]
\end{equation}
for all solutions $\bar\mu^1(\cdot)$ and $\bar\mu^2(\cdot)$ of \eqref{vlasov_without_x} on $[0,T]$ (for some $T>0$) such that $\bar\mu^1(0),\bar\mu^2(0)\in\mathcal{P}_c(\R^d)$.
Here, $C_{\bar\mu^1,\bar\mu^2}(t)$ is defined by \eqref{def_C} or \eqref{def_C_totalLip} (without dependence on $x$).
\end{corollary}

\begin{proof}
Let $\bar\nu$ be an arbitrary probability measure on $\Omega$.
Given any $\bar\mu\in\mathcal{P}(\R^d)$, we define $\mu\in\mathcal{P}(\Omega\times\R^d)$ by $\mu=\bar\nu\otimes\bar\mu$: the marginal of $\mu$ on $\Omega$ is $\bar\nu$ and the disintegration of $\mu=\int_{\Omega}\mu_x\, d\bar\nu(x)$ with respect to $\bar\nu$ is given by $\mu_x=\bar\mu$ if $x\in\supp(\bar\nu)$ and $\mu_x=0$ if $x\in\Omega\setminus\supp(\bar\nu)$.

This embedding allows us to recover Corollary \ref{cor_vlasov_indistinguishable} as a consequence of Theorem \ref{thm_vlasov}. Indeed, 
obviously, $\bar\mu(\cdot)$ is a solution of the Vlasov equation \eqref{vlasov_without_x} without dependence on $x$ if and only if $\mu(\cdot)=\bar\nu\otimes\bar\mu(\cdot)$ is a solution of the Vlasov equation \eqref{vlasov}.
This gives the first part of the corollary.

To obtain \eqref{W1mu1m2_indistinguishable}, it suffices to take $\bar\nu=\delta_{\bar x}$ for some $\bar x\in\Omega$ and to note that $W_p(\bar\mu^1,\bar\mu^2) = W_p(\bar\nu\otimes\bar\mu^1,\bar\nu\otimes\bar\mu^2)$. 
Then, \eqref{W1mu1m2_indistinguishable} follows from \eqref{L1Wpmu1mu2} or from \eqref{W1mu1mu2}.
\end{proof}

\subsection{Relationship between the particle system and the Vlasov equation}\label{sec_relationship_particle_vlasov}
For every $N\in\N^*$, given any $X^N=(x^N_1,\ldots,x^N_N)\in\Omega^N$ and any $\Xi^N=(\xi^N_1,\dots,\xi^N_N)\in\R^{dN}$, we define the \emph{empirical measure} $\mu^e_{(X^N,\Xi^N)} \in \mathcal{P}(\Omega\times\R^d)$ corresponding to $(X^N,\Xi^N)$ by
\begin{equation}\label{def_empirical_measure}
\mu^e_{(X^N,\Xi^N)} = \frac{1}{N}\sum_{i=1}^N \delta_{x^N_i}\otimes\delta_{\xi^N_i} .
\end{equation}
The disintegration of $\mu^e_{(X^N,\Xi^N)}$ with respect to its marginal $\nu^e_{\Xi^N} = \pi_* \mu^e_{(X^N,\Xi^N)} = \frac{1}{N}\sum_{i=1}^N\delta_{x^N_i}$ on $\Omega$ (that is itself an empirical measure corresponding to $X^N$) gives the family of conditional measures defined by $(\mu^e_{(X^N,\Xi^N)})_x=\delta_{\xi^N_i}$ if $x=x^N_i$ and $0$ otherwise. 

The relationship between the particle system \eqref{particle_system} and the Vlasov equation \eqref{vlasov} is given by the result below.
To simplify the notation, hereafter we denote $\mu^e_N = \mu^e_{(X^N,\Xi^N)}$ and $\nu^e_N=\nu^e_{\Xi^N}$.

\begin{proposition}\label{prop_empirical_x}
If $t\mapsto\Xi^N(t)=(\xi^N_1(t),\ldots,\xi^N_N(t)) \in \R^{dN}$ is a solution on $[0,T]$ (for some $T>0$) of the particle system \eqref{system_Y} with parameter $X^N=(x^N_1,\ldots,x^N_N)\in\Omega^N$, then 
$$
t\mapsto\mu^e_N(t) = \varphi_{\mu^e_N(0)}(t)_* \, \mu^e_N(0) = \frac{1}{N}\sum_{i=1}^N \delta_{x^N_i}\otimes\delta_{\xi^N_i(t)}
$$
is a solution of the Vlasov equation \eqref{vlasov} on $[0,T]$. The converse is true if all $x^N_i$ are distinct and all $\xi^N_i(t)$ are distinct.

Actually, $t\mapsto\Xi^N(t)$ is solution  on $[0,T]$ of \eqref{system_Y} with parameter $X^N$ if and only if
\begin{equation}\label{link_particle_meanfield}
\xi^N_i(t) = \varphi_{\mu^e_N(0)} \left( t,x^N_i,\xi^N_i(0) \right) \qquad \forall t\in[0,T]\qquad \forall i\in\{1,\ldots,N\}.
\end{equation}
\end{proposition}

\begin{proof}
The Vlasov equation \eqref{vlasov} is written as $\partial_t\mu + L_{\mathcal{X}[\mu]}\mu=0$ with the Lie derivative acting with respect to the variable $\xi$. Hence, setting $X^N(t)=(x^N_1(t),\ldots,x^N_N(t))$ and $\Xi^N(t)=(\xi^N_1(t),\ldots,\xi^N_N(t))$, the mapping $t\mapsto\mu^e_N(t)$ is a solution of the Vlasov equation \eqref{vlasov} if and only if, for any $g\in \mathscr{C}^\infty_c(\Omega\times\R^d)$, we have 
$$
\langle \partial_t\mu^e_N + L_{\mathcal{X}[\mu^e_N]}\mu^e_N, g\rangle=0,
$$
i.e.,
$$
0 = \frac{1}{N} \sum_{i=1}^N \bigg( \frac{d}{dt} g(x^N_i(t),\xi^N_i(t)) - \partial_\xi g(x^N_i(t),\xi^N_i(t)). \frac{1}{N} \sum_{j=1}^N G(t,x^N_i(t),x^N_j(t),\xi^N_i(t),\xi^N_j(t)) \bigg)
$$
which is satisfied if $t\mapsto(X^N,\Xi^N(t))$ is a solution of \eqref{particle_system}. If all $x^N_i$ are distinct and all $\xi^N_i(t)$ are distinct, the converse is obtained by taking $g$ localized around $(x^N_i,\xi^N_i(t))$.

To obtain the second part of the proposition, we note that 
$$
\mathcal{X}[\mu^e_N](t,x,\xi) = \frac{1}{N}\sum_{j=1}^N G(t,x,x^N_j,\xi,\xi^N_j)
$$
and thus $\mathcal{X}[\mu^e_N](t,x^N_i,\xi^N_i) = Y_i(t,X^N,\Xi^N)$ for every $i\in\{1,\ldots,N\}$.
Therefore, \eqref{system_Y} is equivalent to $\dot\xi^N_i(t)=\mathcal{X}[\mu^e_N(t)](t,x^N_i,\xi^N_i(t))$ for every $i\in\{1,\ldots,N\}$.
Besides, by definition of $t\mapsto\varphi_{\mu^e_N(0)}(t,x^N_i,\cdot)$ (given in \ref{A} in Theorem \ref{thm_vlasov}), we have
$$
\partial_t \varphi_{\mu^e_N(0)}(t,x^N_i,\xi^N_i(0)) = \mathcal{X}[\mu^e_N(t)](t,x^N_i,\varphi_{\mu^e_N(0)}(t,x^N_i,\xi^N_i(0)))
$$
with $\varphi_{\mu^e_N(0)}(0,x^N_i,\xi^N_i(0)) = \xi^N_i(0)$.
Then, \eqref{link_particle_meanfield} follows by Cauchy uniqueness.
\end{proof}

As a consequence of the statements \ref{A1} and \ref{B} of Theorem \ref{thm_vlasov} and of Proposition \ref{prop_empirical_x}, we have the following corollary (the last part of which is already well known in the indistinguishable case).

\begin{corollary}\label{cor_empirical_X}
Let $K$ be a compact subset of $\Omega\times\R^d$. Let $p\in[1,+\infty)$ be arbitrary. Let $\mu_0\in\mathcal{P}_c(K)$ 
and let $t\mapsto \mu(t) = \varphi_{\mu_0}(t,\cdot,\cdot)_* \mu_0$ be the solution on $[0,T_{\max}(K))$ of the Vlasov equation \eqref{vlasov} such that $\mu(0)=\mu_0$. Besides, for every $N\in\N^*$, let $(X^N,\Xi^N_0)\in K^N$ 
be such that the empirical measure $\mu^e_N(0) = \frac{1}{N}\sum_{i=1}^N \delta_{x^N_i}\otimes\delta_{\xi^N_i(0)}$ converges weakly (equivalently, in Wasserstein distance $W_p$) to $\mu_0$ as $N\rightarrow+\infty$ (see Appendix \ref{app_kreinmilman} for general results). For every $N\in\N^*$, let $t\mapsto\Xi^N(t)$ be the solution on $[0,T_{\max}(K))$ of the particle system \eqref{system_Y} with parameter $X^N$ such that $\Xi^N(0)=\Xi^N_0$. 

Then, the empirical measure $\mu^e_N(t) = \frac{1}{N}\sum_{i=1}^N \delta_{x^N_i}\otimes\delta_{\xi^N_i(t)}$ converges weakly (equivalently, in Wasserstein distance $W_p$) to $\mu(t)$ as $N\rightarrow+\infty$, uniformly with respect to $t$ on any compact interval of $[0,T_{\max}(K))$.

If moreover $G$ is locally Lipschitz with respect to $(x,x',\xi,\xi')$ (uniformly with respect to $t$ on any compact), then 
$$
W_p(\mu(t),\mu^e_N(t)) \leq C_{\mu,\mu^e_N}(t) \, W_p(\mu_0,\mu^e_N(0))
$$
for every $t\in[0,T_{\max}(K))$ (with $C_{\mu,\mu^e_N}(t)$ defined by \eqref{def_C}).
\end{corollary}

Lemmas \ref{lem_kreinmilman} and \ref{lem_villani} in Appendix \ref{app_kreinmilman} provide general results ensuring that 
$W_p(\mu_0,\mu^e_N(0))\rightarrow 0$ 
as $N\rightarrow+\infty$, and Lemma \ref{lem_rate_CV_empirical} gives an estimate of convergence, at rate $\frac{1}{N^{r/p}}$, within the framework of tagged partitions.

\begin{remark}\label{rem_mesure_SE}
Alternatively, instead of empirical measures, we may also consider \emph{semi-empirical measures}: setting 
$$
(\mu_0)^{se}_{X^N} = \frac{1}{N} \sum_{i=1}^N \delta_{x^N_i}\otimes\mu_{0,x^N_i} ,
$$
the unique solution $t\mapsto\tilde\mu^N(t) = \varphi_{(\mu_0)^{se}_{X^N}}(t,\cdot,\cdot)_* (\mu_0)^{se}_{X^N}$ of the Vlasov equation \eqref{vlasov} such that $\tilde\mu^N(0)=(\mu_0)^{se}_{X^N}$ is of the form
$\tilde\mu^N(t) = \frac{1}{N} \sum_{i=1}^N \delta_{x^N_i}\otimes\tilde\mu^N_{t,x^N_i}$ (it differs from the semi-empirical measure $\mu(t)^{se}_{X^N} = \frac{1}{N} \sum_{i=1}^N \delta_{x^N_i}\otimes\mu_{t,x^N_i}$).
Its marginal on $\Omega$ is the empirical measure $\nu^e_{X^N} = \frac{1}{N} \sum_{i=1}^N \delta_{x^N_i}$.

Lemma \ref{lem_CV_semiempirical} in Appendix \ref{app_semiempirical} provides results on the convergence of $W_p(\mu_0,(\mu_0)^{se}_{X^N})$ to $0$, as well as estimates with a rate of convergence under appropriate assumptions.
\end{remark}

\subsection{Examples}\label{sec_ex_vlasov}
We continue with the examples given in Sections \ref{sec_ex_particle} (particle systems) and \ref{sec_ex_euler} (CGL equation).

\medskip\noindent
-- For the Hegselmann--Krause (opinion propagation) system \eqref{HK_particle}, under Assumption \ref{G} the mean field (not depending on $t$) is 
\begin{equation}\label{HK_meanfield}
\mathcal{X}[\mu](x,\xi) = \int_{\Omega\times\R^d} \sigma(x,x') (\xi'-\xi) \, d\mu(x',\xi')  \qquad \forall (x,\xi) \in \Omega\times\R^d .
\end{equation}
The Vlasov equation has been derived and studied in \cite{BiccariKoZuazua_M3AS2019, PiccoliPouradierDuteilTrelat_SICON2019} (see also \cite[Section 5.2]{BoudinSalvaraniTrelat_SIMA2022}). 

\medskip\noindent
-- For the Kuramoto particle system \eqref{kuramoto_particle}, under Assumption \ref{G} the mean field (not depending on $t$) is given by
$$
\mathcal{X}[\mu](x,\xi) = \alpha + \int_{\Omega\times\R^d} \sigma(x,x') \sin(\xi'-\xi) \, d\mu(x',\xi')  \qquad \forall (x,\xi) \in \Omega\times\R^d .
$$
The corresponding Vlasov equation was proposed in \cite{Sakaguchi} as being a formal mean field limit of \eqref{kuramoto_particle}. The rigorous mean field limit, called the Kuramoto-Sakaguchi equation, was established in \cite{Lancellotti} in the case where $\sigma$ is constant, by following the classical fixed point arguments of \cite{Neunzert, Spohn_book1991}. The general (network) case is treated in \cite{ChibaMedvedev, KaliuzhnyiMedvedev} and the Vlasov equation associated to the above mean field with the general function $\sigma$, is studied in that reference in view of extending the synchronization theory to spatially structured networks. The Vlasov equation \cite[Eq. (16)]{ChibaMedvedev} is of the form \eqref{vlasov_xfixed}, i.e., it consists of an infinite number of coupled Vlasov equations, parametrized (and coupled) by $x=(\alpha,\beta)$.

\medskip\noindent
-- For the first-order system \eqref{FK_particle}, the mean field does not depend on $(t,x)$ and is given by
\begin{equation*}
\begin{split}
\mathcal{X}[\mu](x,\xi) &= F (\xi) + \int_\Omega\int_{\R^d} K(\xi-\xi') \, d\mu_{x'}(\xi')\, d\nu(x') 
= F(\xi) + \int_\Omega K\star\mu_{x'}(\xi)\,d\nu(x')
\\
&= F(\xi) + \int_{\R^d} K(\xi-\xi')\, d\bar\mu(\xi') 
= F(\xi) + K\star\bar\mu(\xi) = \bar{\mathcal{X}}[\bar\mu](\xi)
\qquad \forall (x,\xi) \in \Omega\times\R^d 
\end{split}
\end{equation*}
where $\bar\mu$ is defined by \eqref{def_barmu} and $\bar{\mathcal{X}}[\bar\mu]$ by \eqref{meanfield_indistinguishable}.
The Vlasov equation $\partial_t\bar\mu+\div((F+K\star\bar\mu)\bar\mu)=0$ is used in mathematical biology to model aggregation phenomena (see \cite{CarrilloChoiHauray_2014, CarrilloDiFrancescoFigalliLaurentSlepcev, dif2}), in the study of neural networks (see \cite{RotskoffVanden-Eijnden_CPAM2022}) or, when $K$ is a singular kernel, in fluid mechanics (see \cite{JabinWang_IM2018, Serfaty}). 

\medskip\noindent
-- For the Cucker--Smale model \eqref{CS_particle}, the mean field does not depend on $(t,x)$ and is given by
\begin{equation*}
\begin{split}
\mathcal{X}[\mu](x,\xi) &= \begin{pmatrix} p\\ \int_{\Omega\times\R^r\times\R^r} a(\Vert q-q'\Vert) (p'-p) \, d\mu(x',\xi') \end{pmatrix}  \\
&= \begin{pmatrix} p\\ \int_{\R^r\times\R^r} a(\Vert q-q'\Vert) (p'-p) \, d\bar\mu(\xi') \end{pmatrix}  
= \bar{\mathcal{X}}[\bar\mu](\xi)
\qquad \forall (x,\xi) \in \Omega\times\R^{2r} 
\end{split}
\end{equation*}
where we recall that $\xi=(q,p)$ and $\xi'=(q',p')$.
The kinetic Cucker--Smale equation satisfied by $\bar\mu$ has been derived in \cite{HaTadmor_KRM2008, NataliniPaul_DCDSB_2022}. Convergence to flocking has been studied in \cite{CarrilloFornasierRosadoToscani_2010, HaLiu_CMS2009}.

\medskip\noindent
-- For the second-order model \eqref{second_order_K}, similarly to the Cucker--Smale example, the mean field does not depend on $(t,x)$ and is given by
\begin{equation*}
\begin{split}
\bar{\mathcal{X}}[\bar\mu](\xi) = \begin{pmatrix} p\\ \int_{\R^r\times\R^r} K(q,q') \, d\bar\mu(\xi') \end{pmatrix}  
\qquad \forall \xi \in \R^{2r} .
\end{split}
\end{equation*}

\medskip\noindent
-- As an example of a Hamiltonian system, the mean field associated with the mapping \eqref{fermion_particle} 
is
$$
\bar{\mathcal{X}}[\bar\mu](\xi) = \begin{pmatrix} p-A(q)\\ -\nabla V(q)+dA(q).(p-A(q))-\int_{\R^r\times\R^r} \big(\partial_1 W(q,q')+\partial_2 W(q',q)\big) \, d\bar\mu(\xi') \end{pmatrix}
$$
for every $\xi\in\R^{2r}$.

\section{From microscopic to mesoscopic scale II: mean field  by lifting the particle system {\normalsize (from Liouville to Vlasov)}}\label{sec_liouville}
\subsection{Liouville equation}
The Eulerian viewpoint consists of propagating, for any parameter $X\in\Omega^N$, an initial probability measure in $\R^{dN}$ under the flow of diffeomorphisms 
$\Phi^N(t,X,\cdot)$ of $\R^{dN}$ generated by the time-dependent vector field $Y^N(t,X,\cdot)$ defined by \eqref{def_Y}. 

Given $N\in\N^*$ fixed, we consider the \emph{($N$-body) Liouville equation} associated with the time-dependent vector field $Y^N$ defined by \eqref{def_Y}, depending on the parameter $X^N\in\Omega^N$, given by
\begin{equation}\label{liouville}
\boxed{
\partial_t\rho^N + \mathrm{div}_\Xi(Y^N\rho^N) = 0
}
\end{equation}
This is a usual transport equation on $\R^{dN}$, parametrized by $X^N\in\Omega^N$, where the divergence is considered with respect to $\Xi=(\xi_1,\ldots,\xi_N)$, and we thus have the following standard result. Here, it is understood that $\rho^N(t)$ is a probability Radon measure on $(\Omega\times\R^d)^N\simeq\Omega^N\times\R^{dN}$.

\begin{proposition}\label{prop_existence_liouville}
Let $K$ be a compact subset of $\Omega\times\R^d$.
Let $\rho^N_0\in\mathcal{P}_c(\Omega^N\times\R^{dN})$ be such that all marginals of $\rho^N_0$ on any copy of $\Omega\times\R^d$ are supported in the same compact $K$.
There exists a unique solution $t\mapsto\rho^N(t)$ of the Liouville equation \eqref{liouville} in $\mathscr{C}^0([0,T_{\max}(K)),\mathcal{P}_c(\Omega^N\times\R^{dN}))$, locally Lipschitz with respect to $t$ for the distance $L^1_\theta W_1$ (where $\theta$ is defined below), such that $\rho^N(0)=\rho^N_0$, given by
\begin{equation}\label{Phistar}
\rho^N(t)= \Phi^N(t)_* \rho^N_0 
\end{equation}
i.e., $\rho^N(t)$ is the image (pushforward) of $\rho^N_0$ under the particle flow.
\end{proposition}

The notation \eqref{Phistar} is slightly abusive. To explain it, let us make precise some notations and in particular the disintegration procedure. Given any measure $\rho\in\mathcal{P}(\Omega^N\times\R^{dN})$, denoting by $\pi^{\otimes N}:\Omega^N\times\R^{dN}\rightarrow\Omega^N$ the canonical projection, we will always denote by $\theta$ the probability Radon measure on $\Omega^N$ given by $\theta = (\pi^{\otimes N})_*\rho$ (image of $\rho$ under $\pi^{\otimes N}$), that is the marginal of $\rho$ on $\Omega^N$. By disintegration of $\rho$ with respect to $\theta$, there exists a family $(\rho_X)_{X\in\Omega^N}$ of probability Radon measures on $\R^{dN}$ such that $\rho = \int_{\Omega^N} \rho_X\, d\theta(X)$. 

With these notations, 
$\rho^N_t=\rho^N(t)$ is disintegrated as $\rho^N_t = \int_{\Omega^N} \rho^N_{t,X}\, d\theta^N(X)$ with respect to its marginal $\theta^N = (\pi^{\otimes N})_*\rho^N(t)$ on $\Omega^N$. The marginal $\theta^N$ does not depend on $t$ because \eqref{liouville} can be written as $\partial_t\rho^N+L_{Y^N}\rho^N=0$, with the Lie derivative acting with respect to the variable $\xi$, and we have $(\pi^{\otimes N})_*L_{Y^N}=0$. Finally, \eqref{Phistar} means that
$$
\rho^N_{t,X} = (\Phi^N_{t,X})_*\rho^N_{0,X}
$$
for every $t\in[0,T_{\max}(K))$ and for $\theta^N$-almost every $X\in\Omega^N$.

\begin{remark}\label{rem_Liouville_Dirac}
If $\rho^N_0 = \delta_{X^N}\otimes\delta_{\Xi^N_0}$ for some $(X^N,\Xi^N_0)\in K^N$ then $\rho^N(t)=\delta_{X^N}\otimes\delta_{\Xi^N(t)}$ where $t\mapsto\Xi^N(t)$ is the solution on $[0,T_{\max}(K))$ of the particle system \eqref{system_Y} with parameter $X^N$ such that $\Xi^N(0)=\Xi^N_0$.
In other words, the solutions of the particle system are naturally embedded as Dirac measures solutions of the Liouville system.

Hence, in some sense, the Liouville equation contains all possible solutions of the particle system. But it contains more: considering the particle system \eqref{system_Y}, instead of taking a \emph{deterministic} initial condition $\Xi^N(0)=\Xi^N_0\in\R^{dN}$, one may want to take a \emph{distribution} of initial conditions, for instance one may want to consider all possible initial conditions that are distributed around $\Xi^N_0$ according to a Gaussian law, in order to take into account noise or uncertainties in the initial conditions. In such a way, the Liouville equation \eqref{liouville} has a \emph{probabilistic} interpretation with respect to the particle system \eqref{particle_system}.

If the probability measure $\rho^N(t)$ on $\Omega^N\times\R^{dN}$ has a density $f^N$, then $f^N(t,X,\Xi)$ represents the density of particles with labels $X=(x_1,\dots,x_N)\in\Omega^N$ and respective states $\Xi=(\xi_1,\dots,\xi_N)\in\R^{dN}$.
This is in contrast with the mean field procedure that consists of taking the large $N$ limit of the average over all particles but one. In the next section we show how to derive Vlasov from Liouville by taking marginals.
\end{remark}

\subsection{Deriving Vlasov from Liouville by taking marginals, propagation of chaos}\label{sec_liouville_to_vlasov}
Compared with $\mu(t)$ that is a probability measure on $\Omega\times\R^d$, $\rho^N(t)$ is a probability measure on $(\Omega\times\R^d)^N\simeq\Omega^N\times\R^{dN}$. It is thus tempting to search for a relationship between $\mu(t)$ and $\rho^N(t)$ by taking marginals of $\rho^N(t)$.
This is what has been done in \cite{Sznitman}, in \cite{Jabin_KRM2014, JabinWang_IM2018} or in \cite{Golse_2016, GolseMouhotPaul_CMP2016} in the different context of quantum mechanics.
Adapted to the present situation, the method developed in \cite{GolseMouhotPaul_CMP2016}, which provides an explicit rate of convergence, consists of proving that the marginals of the solutions $\rho^N(t)$ of \eqref{liouville} are close, in Wasserstein topology, to solutions $\mu(t)$ of the Vlasov equation \eqref{vlasov}, as established hereafter. 

As we are going to see, this can be done by taking adequate initial conditions $\rho^N_0$ for the Liouville equation \eqref{liouville}.
We have to perform a symmetrization under permutations for the initial condition $\rho^N_0$ and also for the corresponding solution $\rho^N(t)$, not only with respect to $\Xi$ but also with respect to the parameter variable $X$. 
Note that the symmetrization is not preserved by the flow, so we have to consider the symmetrization $\rho^N(t)^s$ at any time $t$.

Given any $\rho\in\mathcal{P}(\Omega^N\times\R^{dN})$, we define the measure $\rho^s\in\mathcal{P}(\Omega^N\times\R^{dN})$, called the symmetrization under permutations of $\rho$ (see Appendix \ref{app_symm}), by
\begin{equation*}
\int_{\Omega^N\times\R^{dN}} f(X,\Xi)\, d\rho^s(X,\Xi) = 
\frac{1}{N!} \sum_{\sigma\in\mathfrak{S}_N} \int_{\Omega^N\times\R^{dN}} f(\sigma\cdot X,\sigma\cdot\Xi)\, d\rho(X,\Xi)
\end{equation*}
for every $f\in \mathscr{C}^0_c(\Omega^N\times\R^{dN})$,
where $\sigma\cdot X = (x_{\sigma(1)},\dots,x_{\sigma(N)})$ and $\sigma\cdot\Xi = (\xi_{\sigma(1)},\ldots,\xi_{\sigma(N)})$ for all $X\in\Omega^N$ and $\Xi\in\R^{dN}$, and where $\mathfrak{S}_N$ is the group of permutations of $N$ elements.

Now, given any $k\in\{1,\ldots,N\}$, we denote by $\rho^s_{N:k}$ the $k^\textrm{th}$-order marginal of $\rho^s$
(not to be confused with the symmetrization under permutations of the marginal, which we do not use), which is, by definition, the image of $\rho^s$ under the projection of 
$\Omega^N\times\R^{dN}$ onto the product $\Omega^k\times\R^{dk}$ of the $k$ first copies of $\Omega$ with the $k$ first copies of $\R^d$.

Since we are going to compute Wasserstein distances in $(\Omega\times\R^d)^k\simeq\Omega^k\times\R^{dk}$, we have to choose a distance in that space. Recall that $\Omega\times\R^d$ is equipped with the distance $\mathrm{d}_{\Omega\times\R^d} = \mathrm{d}_{\Omega} + \mathrm{d}_{\R^d}$ where $\mathrm{d}_{\R^d}$ is the distance on $\R^d$ induced by the norm $\Vert\cdot\Vert$ (which is arbitrary).
Let $q\in[1,+\infty]$ be arbitrary. 
Given any $k\in\N^*$, we endow $(\Omega\times\R^d)^k$ with the $\ell^q$ distance based on $\mathrm{d}_{\Omega\times\R^d}$, defined by
\begin{equation}\label{def_distq}
\begin{split}
\mathrm{d}^{[q]}_{(\Omega\times\R^d)^k}((X,\Xi),(X',\Xi'))
&= \left\Vert ( \mathrm{d}_{\Omega\times\R^d}((x_1,\xi_1),(x'_1,\xi'_1)), \ldots, \mathrm{d}_{\Omega\times\R^d}((x_k,\xi_k),(x'_k,\xi'_k)) ) \right\Vert_{\ell^q} \\
&= \left\{ \begin{array}{ll}
\displaystyle \bigg( \sum_{i=1}^k \left( \mathrm{d}_\Omega(x_i,x'_i) + \Vert\xi_i-\xi'_i\Vert \right)^q  \bigg)^{1/q} & \textrm{if}\ q\in[1,+\infty) \\[4mm]
\displaystyle \max_{1\leq i\leq k} \left( \mathrm{d}_\Omega(x_i,x'_i) + \Vert\xi_i-\xi'_i\Vert \right) & \textrm{if}\ q=+\infty
\end{array}\right.
\end{split}
\end{equation}
for all $X=(x_1,\ldots,x_k)$ and $X'=(x'_1,\ldots,x'_k)$ in $\Omega^k$
and for all $\Xi=(\xi_1,\ldots,\xi_k)$ and $\Xi'=(\xi'_1,\ldots,\xi'_k)$ in $\R^{dk}$.
Note that, when $k=1$, we have $\mathrm{d}^{[q]}_{\Omega\times\R^d} = \mathrm{d}^{[1]}_{\Omega\times\R^d} = \mathrm{d}_{\Omega} + \mathrm{d}_{\R^d}$. 

Given any $p,q\in[1,+\infty]$, we denote by $W_p^{[q]}$ the Wasserstein distance $W_p$ on $\mathcal{P}(\Omega^k\times(\R^d)^k)$ with respect to the distance $\mathrm{d}^{[q]}_{(\Omega\times\R^d)^k}$.

We refer to the beginning of Appendix \ref{app_useful_lemmas} and in particular to Remark \ref{rem_choicedist} for comments on the importance of choosing a distance on the product space $(\Omega\times\R^d)^k$ and for remarks on the Wasserstein distance $W_p^{[q]}$.
In particular, by \eqref{inegWpq}, we have $W_p^{[q_2]} \leq W_p^{[q_1]} \leq k^{\frac{1}{q_1}-\frac{1}{q_2}} \, W_p^{[q_2]}$ if $1\leq q_1\leq q_2\leq +\infty$
for any $p\in[1,+\infty]$.

\medskip

In this section, we establish two ways of deriving Vlasov from Liouville by taking marginals. 

Let $\mu_0\in\mathcal{P}_c(\Omega\times\R^d)$, disintegrated as $\mu_0=\int_{\Omega}\mu_{0,x}\, d\nu(x)$ with respect to its marginal $\nu=\pi_*\mu_0$ on $\Omega$. Setting $T_0 = T_{\max}(\supp(\mu_0))$ (as given by Lemma \ref{lem_Tmax}), we consider the unique solution $t\mapsto\mu(t)=\varphi_{\mu_0}(t)_*\mu_0$ in $\mathscr{C}^0_{\mathrm{comp}}([0,T_0),\mathcal{P}_c(\Omega\times\R^d))$ of the Vlasov equation \eqref{vlasov} such that $\mu(0)=\mu_0$, as given by Theorem \ref{thm_vlasov}.
Recall that $\mu_{t,x} = \varphi_{\mu_0}(t,x,\cdot)_*\mu_{0,x}$ for $\nu$-almost every $x\in\Omega$.

Hereafter, we propose two possible choices of $\rho^N_0\in\mathcal{P}_c(\Omega^N\times\R^{dN})$, generating by Proposition \ref{prop_existence_liouville} the solution $\rho^N(t)=\Phi^N(t)_*\rho^N_0$ of the Liouville equation \eqref{liouville} from which we recover in the large-$N$ limit the solution $\mu(t)$ of the Vlasov equation \eqref{vlasov} by taking marginals.

In Theorem \ref{thm_CV_liouville_empirical}, we take $\rho^N_0$ Dirac; in Theorem \ref{thm_CV_liouville}, we take $\rho^N_0$ ``semi-Dirac''. In both cases, we prove that $\rho^N(t)^s_{N:k}$ converges to $\mu(t)^{\otimes k}$ as $N\rightarrow +\infty$ and we establish convergence estimates in Wasserstein distance $W_p^{[q]}$. The fact that the $k^\textrm{th}$-order marginal $\rho^N(t)^s_{N:k}$ of the symmetrization of $\rho^N(t)$, which is far from being a tensor product at time $t=0$, becomes nevertheless the tensor product $\mu(t)^{\otimes k}$ at the limit $N\rightarrow +\infty$, is usually referred to as \emph{propagation of chaos} (formalized in the pioneering articles \cite{Kac, McKean}, see also \cite{Golse_2016, MischlerMouhot_IM2013, Spohn_book1991, Sznitman}).

\subsubsection{First way, with $\rho^N_0$ Dirac}
For fixed $N\in\N^*$, let $(X^N,\Xi^N_0)\in \supp(\mu_0)^N\subset\Omega^N\times\R^{dN}$ be arbitrary. 
Typically we may want that the empirical measure $\mu^e_{(X^N,\Xi^N_0)}=\frac{1}{N}\sum_{i=1}^N\delta_{x^N_i}\otimes\delta_{\xi^N_{0,i}}$ 
converges to $\mu_0$ in Wasserstein distance as $N\rightarrow+\infty$ (see Appendix \ref{app_kreinmilman} for such conditions). Let $t\mapsto\Xi^N(t)=(\xi^N_1(t),\ldots,\xi^N_N(t))$ be the solution on $[0,T_0)$ of the particle system \eqref{system_Y} such that $\Xi^N(0)=\Xi^N_0$. If $\mu^e_{(X^N,\Xi^N_0)}$ converges to $\mu_0$ then, by Corollary \ref{cor_empirical_X}, the empirical measure 
$$
\mu^e_{(X^N,\Xi^N(t))}=\frac{1}{N}\sum_{i=1}^N\delta_{x^N_i}\otimes\delta_{\xi^N_i(t)}
$$
converges to $\mu(t)$ in Wasserstein distance as $N\rightarrow+\infty$. 

Defining $\rho^N_0\in\mathcal{P}_c(\Omega^N\times\R^{dN})$ as the Dirac measure $\rho^N_0 = \delta_{X^N}\otimes\delta_{\Xi^N_0}$, by Remark \ref{rem_Liouville_Dirac}, the unique solution of the Liouville equation \eqref{liouville} such that $\rho^N(0) = \rho^N_0$, is given by the Dirac measure 
$$
\rho^N(t)=\Phi^N(t)_*\rho^N_0 = \delta_{X^N}\otimes\delta_{\Xi^N(t)} 
\qquad\forall t\in [0,T_0).
$$
It is then easy to see that $\rho^N(t)^s_{N:1} = \mu^e_{(X^N,\Xi^N(t))}$ (see the proof of the theorem below).
Therefore, if $\mu^e_{(X^N,\Xi^N_0)}$ converges weakly to $\mu_0$ then $\rho^N(t)^s_{N:1}$ converges weakly to $\mu(t)$ as $N\rightarrow+\infty$.
Actually, this first fact re-expresses results seen in Sections \ref{sec_vlasov_def} and \ref{sec_relationship_particle_vlasov}.
The convergence is less obvious for the marginals of order $k\geq 2$. 

Recall that $G$ satisfies Assumption \ref{G}.

\begin{theorem}\label{thm_CV_liouville_empirical}
We have the following statements, for any $p\in[1,+\infty)$ and $q\in[1,+\infty]$.
\begin{enumerate}[leftmargin=*,label=$\bf (\Alph*)$]
\item\label{thm_CV_liouville_empirical_A} 
If $\mu^e_{(X^N,\Xi^N_0)}$ converges weakly (equivalently, in Wasserstein distance $W_p$) to $\mu_0$ as $N\rightarrow+\infty$, then, for every $k\in\N^*$, $\rho^N(t)^s_{N:k}$ converges weakly (equivalently, in Wasserstein distance $W_p^{[q]}$) to $\mu(t)^{\otimes k}$ as $N\rightarrow+\infty$, uniformly with respect to $t$ on any compact interval of $[0,T_0)$.

\item\label{thm_CV_liouville_empirical_B}
Assuming that $G$ is locally Lipschitz with respect to $(x,x',\xi,\xi')$ (uniformly with respect to $t$ on any compact), setting
\begin{equation}\label{def_Smu}
S_\mu^N(\tau)=\supp(\mu(\tau)) \cup \{ (x^N_i,\xi^N_i(\tau))\ \mid\ i\in\{1,\ldots,N\} \}
\end{equation}
and defining
\begin{equation}\label{defCmuN}
C_\mu^N(t) = \exp\bigg( 2 \int_0^t \Lip( G(\tau,\cdot,\cdot,\cdot,\cdot)_{\vert S_\mu^N(\tau)^2} ) \, d\tau \bigg) ,
\end{equation}
for every $N\in\N^*$ and for every $t\in[0,T_0)$ we have 
$$
\rho^N(t)^s_{N:1} = \mu^e_{(X^N,\Xi^N(t))} = \frac{1}{N}\sum_{i=1}^N\delta_{x^N_i}\otimes\delta_{\xi^N_i(t)}
$$
and
\begin{equation}\label{thm_CV_liouville_empirical_1}
W_p\big( \rho^N(t)^s_{N:1} , \mu(t) \big) \leq C_\mu^N(t)\, W_p\big( \mu^e_{(X^N,\Xi^N_0)}, \mu_0 \big)
\end{equation}
and, for every $k\in\N^*$ such that $k^2\leq N\ln\big(1+\frac{1}{2^p}\big)$,
\begin{multline}\label{thm_CV_liouville_empirical_n}
W_p^{[q]}\big( \rho^N(t)^s_{N:k} , \mu(t)^{\otimes k} \big) 
\leq 3 k^{1/q} \left( \frac{k^2}{N} \right)^{1/p} 
\left( \diam_\Omega(\supp(\nu)) + \diam_{\R^d}(\Xi^N(t)) \right) \\
+ k^{1/q} C_\mu^N(t)\, W_p\big( \mu^e_{(X^N,\Xi^N_0)}, \mu_0 \big) 
\end{multline}
where $\displaystyle\diam_{\R^d}(\Xi^N(t)) = \max_{1\leq i,j\leq N}\Vert\xi^N_i(t)-\xi^N_j(t)\Vert$.
\end{enumerate}
\end{theorem}

Theorem \ref{thm_CV_liouville_empirical} is proved in Appendix \ref{app_proof_thm_CV_liouville_empirical}.

In \eqref{thm_CV_liouville_empirical_n}, $\displaystyle\diam_\Omega(\supp(\nu))=\max_{x,x'\in\supp(\nu)}\mathrm{d}_\Omega(x,x')$, and the Wasserstein distance $W_p^{[q]}$ on $\Omega^k\times(\R^d)^k$ is computed with respect to the distance $\mathrm{d}^{[q]}_{\Omega^k\times(\R^d)^k}$ defined by \eqref{def_distq}.
Since $W_p^{[q]} \leq k^{\frac{1}{q}} \, W_p^{[\infty]}$ (by \eqref{inegWpq}), the strongest inequality \eqref{thm_CV_liouville_empirical_n} is obtained when $q=+\infty$.

Lemmas \ref{lem_kreinmilman} and \ref{lem_villani} in Appendix \ref{app_kreinmilman} show that there always exists a sequence of empirical measures $\mu^e_{(X^N,\Xi^N_0)}$ converging weakly to $\mu_0$.
As alluded above, to obtain an interesting convergence estimate from Item \ref{thm_CV_liouville_empirical_B} of this theorem, we apply Lemma \ref{lem_rate_CV_empirical} in Appendix \ref{app_kreinmilman}, which yields the estimate $W_p(\mu^e_{(X^N,\Xi^N_0)},\mu_0) \leq \frac{1}{N^{r/p}} \, C_{\Omega\times\R^d}^{1/p} \, \diam_{\Omega\times\R^d}(\supp(\mu_0))^{1-1/p} $ under the assumption of the existence of a family of tagged partitions.
As noted in this appendix, there exist plenty of results quantifying the convergence of empirical measures to a given measure (see, e.g., \cite{FournierGuillin_PTRF2015}). Lemma \ref{lem_rate_CV_empirical} is a rough result.

\begin{corollary}\label{cor_CV_liouville_empirical}
In the context of Item \ref{thm_CV_liouville_empirical_B} of Theorem \ref{thm_CV_liouville_empirical}, we assume moreover that there exists a family of tagged partitions of $\supp(\mu_0)$ associated with $\mu_0$ (see Section \ref{sec_general_notations}), i.e., for every $N\in\N^*$ there exists a partition of $\supp(\mu_0) = \cup_{i=1}^N F^N_i$ such that all subsets $F^N_i\subset\Omega\times\R^d$ are $\mu_0$-measurable, pairwise disjoint, satisfy $\mu_0(F^N_i)=\frac{1}{N}$ and $\diam_{\Omega\times\R^d}(F^N_i)\leq C_{\Omega\times\R^d}/N^r$ for some $C_{\Omega\times\R^d}>0$ and $r>0$ not depending on $N$, and $N$-tuples $X^N=(x^N_1,\ldots,x^N_N)\in\Omega^N$ and $\Xi^N_0=(\xi^N_{0,1},\ldots,\xi^N_{0,N})\in(\R^d)^N$ such that $(x^N_i,\xi^N_{0,i})\in F^N_i$ for every $i\in\{1,\ldots,N\}$. Then, for every $t\in[0,T_0)$,
$$
W_p\left( \rho^N(t)^s_{N:1} , \mu(t) \right) \leq \frac{1}{N^{r/p}} \, C_{\Omega\times\R^d}^{1/p} \, \diam_E(\supp(\mu_0))^{1-1/p} \,  C_\mu^N(t)
$$
and, for every $k\in\N^*$ such that $k^2\leq 2N\ln\big(1+\frac{1}{2^p}\big)$,
\begin{multline}\label{cor_CV_liouville_empirical_Wpinfty}
W_p^{[\infty]}\left( \rho^N(t)^s_{N:k} , \mu(t)^{\otimes k} \right) 
\leq 2 \left( \frac{k^2}{N} \right)^{1/p} 
\left( \diam_\Omega(\supp(\nu)) + \diam_{\R^d}(\Xi^N(t)) \right) \\
+ \frac{C_{\Omega\times\R^d}^{1/p}}{N^{r/p}} \, \diam_{\Omega\times\R^d}(\supp(\mu_0))^{1-1/p} \, C_\mu^N(t)  .
\end{multline}
\end{corollary}

When $\Omega$ is a $n$-dimensional manifold (thus $\dim(\Omega\times\R^d)=n+d$), we have $r=1/(n+d) < 1$.

According to the estimate \eqref{cor_CV_liouville_empirical_Wpinfty}, $\rho^N(t)^s_{N:k}$ converges to $\mu(t)^{\otimes k}$ in Wasserstein distance $W_p^{[\infty]}$ as $N\rightarrow+\infty$, uniformly with respect to $t$ on compact intervals of $[0,T_0)$, at rate $1/N^{r/p}$ if $k\ll N^{(1-r)/2}$ and at rate $k^{2/p}/N^{1/p}$ if $N^{(1-r)/2} \ll k \ll N^{1/2}$.
The rate of convergence can be improved if one uses better results for convergence of empirical measures.

Note that the assumption of a family of tagged partitions in Corollary \ref{cor_CV_liouville_empirical} essentially entails that $\mu_0$ be absolutely continuous with respect to a Lebesgue measure on $\Omega\times\R^d$.

\paragraph{Particular case where $G$ does not depend on $(x,x')$.}
When $G$ does not depend on $(x,x')$, particles are indistinguishable and the mean field is given by \eqref{meanfield_indistinguishable}.
We have the following corollary of Theorem \ref{thm_CV_liouville_empirical}. 

\begin{corollary}\label{cor_CV_liouville_empirical_indistinguishable}
Let $\bar\mu_0\in\mathcal{P}_c(\R^d)$ and let $t\mapsto\bar\mu(t)$ be the unique solution on $[0,T_0)$, with $T_0=T_{\max}(\supp(\bar\mu_0))$, of the Vlasov equation \eqref{vlasov_without_x} such that $\bar\mu(0)=\bar\mu_0$ (see Corollary \ref{cor_vlasov_indistinguishable}).
Besides, let $\bar\rho^N_0=\delta_{\Xi^N_0}$ and let $t\mapsto\bar\rho^N(t)=\delta_{\Xi^N(t)}$ be the unique solution on $[0,T_0)$ of the Liouville equation \eqref{liouville} (without dependence on $X$) such that $\bar\rho^N(0)=\bar\rho^N_0$.
Then, for every $t\in[0,T_0)$,
\begin{equation}\label{thm_CV_liouville_empirical_indistinguishable_1}
W_p\big( \bar\rho^N(t)^s_{N:1} , \bar\mu(t) \big)
=
W_p\big( \bar\mu^e_{\Xi^N(t)} , \bar\mu(t) \big) \leq C_{\bar\mu}^N(t)\, W_p\big( \bar\mu^e_{\Xi^N_0}, \bar\mu_0 \big)
\end{equation}
and we have $\bar\rho^N(t)^s_{N:1} = \mu^e_{\Xi^N(t)} = \frac{1}{N}\sum_{i=1}^N \delta_{\xi^N_i(t)}$ (empirical measure), where
$C_{\bar\mu}^N(t)$ is defined by \eqref{defCmuN} (without dependence on $x,x'$), and, for every $k\in\{2,\ldots,N\}$, 
\begin{equation}\label{thm_CV_liouville_empirical_indistinguishable_n}
W_p^{[\infty]}\big( \bar\rho^N(t)^s_{N:k} , \bar\mu(t)^{\otimes k} \big) 
\leq 
2 \left( \frac{k^2}{N} \right)^{1/p} 
\diam_{\R^d}(\Xi^N(t)) + C_{\bar\mu}^N(t)\, W_p\big( \bar\mu^e_{\Xi^N_0}, \bar\mu_0 \big) .
\end{equation}
\end{corollary}

\begin{proof}
Following the proof of Corollary \ref{cor_vlasov_indistinguishable} and choosing $\bar\nu=\delta_{\bar x}$ for some arbitrary $\bar x\in\Omega$, when $G$ does not depend on $(x,x')$, 
$\bar\mu(\cdot)$ is a solution of the Vlasov equation \eqref{vlasov_without_x} (without dependence on $x$) if and only if $\mu(\cdot)=\delta_{\bar x}\otimes\bar\mu(\cdot)$ is a solution of the Vlasov equation \eqref{vlasov}. We now define $X^N=(\bar x,\ldots,\bar x)\in\Omega^N$, and we take $\rho^N_0=\delta_{X^N}\otimes\delta_{\Xi^N_0}$ as initial condition for the Liouville equation in Theorem \ref{thm_CV_liouville_empirical}, so that $\rho^N(t)=\delta_{X^N}\otimes\bar\rho^N(t)$ where $\bar\rho^N(t) = \delta_{\Xi^N(t)}$. 
With these choices, we obviously have $\rho^N(t)^s_{N:k} = \delta_{\bar x}^{\otimes k}\otimes\bar\rho^N(t)^s_{N:k}$,
and then \eqref{thm_CV_liouville_empirical_indistinguishable_1} and \eqref{thm_CV_liouville_empirical_indistinguishable_n} follow directly from \eqref{thm_CV_liouville_empirical_1} and \eqref{thm_CV_liouville_empirical_n}, by applying Remark \ref{rem_Wp_tensor} in Appendix \ref{app_tensor}. 
\end{proof}

\subsubsection{Second way, with $\rho^N_0$ ``semi-Dirac''}
For fixed $N\in\N^*$, let $X^N=(x^N_1,\ldots,x^N_N)\in\Omega^N$ be arbitrary. 
We set $\delta_{X^N} = \delta_{x^N_1}\otimes\cdots\delta_{x^N_N}$ and $\rho^N_{0,X^N} = \mu_{0,x^N_1}\otimes\cdots\otimes\mu_{0,x^N_N}$. Defining $\rho^N_0\in\mathcal{P}_c(\Omega^N\times\R^{dN})$ as the ``semi-Dirac'' measure $\rho^N_0 = \delta_{X^N}\otimes\rho^N_{0,X^N}$, we consider the unique solution on $[0,T_0)$ of the Liouville equation \eqref{liouville} such that $\rho^N(0) = \rho^N_0$, given by the ``semi-Dirac'' measure
$$
\rho^N(t)=\Phi^N(t)_*\rho^N_0 = \delta_{X^N}\otimes \Phi(t,X^N,\cdot)_*\rho^N_{0,X^N} = \delta_{X^N}\otimes\rho^N_{t,X^N}  .
$$
Note indeed that the marginal $\theta^N=(\pi^{\otimes N})_*\rho^N_t$ of $\rho^N_t=\rho^N(t)$ on $\Omega^N$ is $\theta^N=\delta_{X^N}$, and that $\rho^N_{t,X^N} = \Phi^N(t,X^N,\cdot)_*\rho^N_{0,X^N}$.

As a preliminary remark, we claim that, at $t=0$, we have
\begin{equation}\label{rhoN0}
(\rho^N_0)^s_{N:1} = \frac{1}{N}\sum_{i=1}^N\delta_{x^N_i}\otimes\mu_{0,x^N_i} = (\mu_0)^{se}_{X^N}  
\end{equation}
(semi-empirical measure),
which converges weakly to $\mu_0$ as $N\rightarrow +\infty$ under slight assumptions on $\mu_0$, by Lemma \ref{lem_CV_semiempirical} in Appendix \ref{app_semiempirical}. More generally, $(\rho^N_0)^s_{N:k}$ converges weakly to $\mu_0^{\otimes k}$ (in the proof of the theorem hereafter, we give an explicit expression for $(\rho^N_0)^s_{N:k}$, using \eqref{prhosN:k_eps} in Appendix \ref{app_technical_lemma}).
In the theorem below, we establish that this convergence is propagated in time. 

\begin{theorem}\label{thm_CV_liouville}
We assume that the norm $\Vert\cdot\Vert$ on $\R^d$ is induced by a scalar product on $\R^d$. Let $p\in[1,2]$ and $q\in[1,+\infty]$ be such that $p\leq q$.
\begin{enumerate}[leftmargin=*,label=$\bf (\Alph*)$]
\item\label{thm_CV_liouville_A} 
Assume that $x\mapsto\mu_{0,x}$ is $\nu$-almost everywhere continuous for the Wasserstein distance $W_p$.
Then, for every $k\in\N^*$, $\rho^N(t)^s_{N:k}$ converges weakly to $\mu(t)^{\otimes k}$ (equivalently, in Wasserstein distance $W_p^{[q]}$) as $N\rightarrow+\infty$, uniformly with respect to $t$ on any compact interval of $[0,T_0)$.

\item\label{thm_CV_liouville_B} Assuming that $G$ is locally Lipschitz with respect to $(x,x',\xi,\xi')$ (uniformly with respect to $t$ on any compact), defining $S_\mu^N(\tau)$ by \eqref{def_Smu} and
\begin{equation}\label{def_Cmu}
C_\mu(t) = 11 \Big(1+70\max_{0\leq\tau\leq t}\diam_{\Omega\times\R^d} ( \supp(\mu(t)) ) \Big)^{1/2} \exp\Big( 2 t \max_{0\leq\tau\leq t} \Vert G(\tau,\cdot,\cdot,\cdot,\cdot)_{\vert S_\mu^N(\tau)^2}\Vert_{\mathscr{C}^{0,1}}  \Big) ,
\end{equation}
we have, for every $N\in\N^*$, for every $k\in\{1,\ldots,N\}$ such that $k^2\leq \frac{N}{2}$,
\begin{multline}\label{estim_W1_k}
W_p^{[q]}\left( \rho^N(t)^s_{N:k} , \mu(t)^{\otimes k} \right) \\
\leq 
k^{1/q} C_\mu(t) \max \left( \bigg( \frac{k^2}{N} \bigg)^{1/p} , \frac{1}{N^{\frac{1}{q}-\frac{1}{2}}} , N^{1-\frac{1}{q}} \sqrt{W_1\left( (\mu_0)^{se}_{X^N} , \mu_0 \right)} , W_p\left( (\mu_0)^{se}_{X^N} , \mu_0 \right) \right) 
\end{multline}
for every $t\in[0,T_0)$ (for $k=1$, without the first term in the above parenthesis).
\end{enumerate}
\end{theorem}

Theorem \ref{thm_CV_liouville} is proved in Appendix \ref{app_proof_thm_CV_liouville}.
Note that the $p$-Wasserstein distance at the left-hand side of \eqref{estim_W1_k} is considered with $p\leq 2$, because in the proof we use in an instrumental way a variance-type estimate, measuring the $L^2$ discrepancy between the mean field and the particle vector field (see Appendix \ref{app_mean_field}). Besides, $q\in[1,+\infty]$ is arbitrary, but only the values $q\in[1,2)$ are meaningful, noting that it is also required that $p\leq q$. The strongest estimate inferred from \eqref{estim_W1_k} is when $q=1$ (thus, also $p=1$), i.e., when one takes the $\ell^1$ distance on $\Omega^k\times(\R^d)^k$. This choice has no importance while $k$ is small, but becomes significant if one takes for instance $k=N^{1/4}$.

To obtain an interesting convergence result from this theorem, we apply the second item of Lemma \ref{lem_CV_semiempirical} of Appendix \ref{app_semiempirical}, which yields $W_1((\mu_0)^{se}_{X^N},\mu_0) \leq \frac{(L+1)C_\Omega}{N^r}$
and $W_p((\mu_0)^{se}_{X^N},\mu_0) \leq \diam_{\Omega\times\R^d}(\supp(\mu_0))^{1-1/p}((L+1)C_\Omega / N^r)^{1/p}$ under a regularity assumption on $\mu_0$.

\begin{corollary}\label{cor_CV_liouville}
In the context of Item \ref{thm_CV_liouville_B} of Theorem \ref{thm_CV_liouville}, we assume moreover that, for every $N\in\N^*$, there exists a tagged partition $(\mathcal{A}^N,X^N)$ of $\Omega$ associated with $\nu$ satisfying \eqref{def_tagged} (see Section \ref{sec_general_notations}), 
and that $x\mapsto\mu_{0,x}$ is Lipschitz for the Wasserstein distance $W_1$, i.e., that there exists $L>0$ such that $W_1(\mu_{0,x},\mu_{0,y})\leq L\, \mathrm{d}_\Omega(x,y)$ for $\nu$-almost all $x,y\in\Omega$. 
Then 
\begin{equation}\label{cor_estim_W1_k}
W_p^{[q]}\left( \rho^N(t)^s_{N:k} , \mu(t)^{\otimes k} \right) 
\leq 
k^{1/q} (L+1)C_\Omega C_\mu(t) \max \bigg( \bigg( \frac{k^2}{N} \bigg)^{1/p} , \frac{1}{N^{\frac{1}{q}-\frac{1}{2}}} , \frac{1}{N^{\frac{r}{2}+\frac{1}{q}-1}} , \frac{1}{N^{r/p}} \bigg)  
\end{equation}
for every $t\in[0,T_0)$.
\end{corollary}

When $\Omega$ is a $n$-dimensional manifold, we have $r=1/n$, hence, if we take $q=1$ and $p=1$, the rate of convergence provided by \eqref{cor_estim_W1_k} is $\frac{k}{N^{1/2n}}$.

Note that the assumption of a family of tagged partitions in Corollary \ref{cor_CV_liouville} essentially entails that $\nu$ be absolutely continuous with respect to a Lebesgue measure on $\Omega$.

\paragraph{Particular case where $G$ does not depend on $(x,x')$.}
When $G$ does not depend on $(x,x')$, we have the following corollary of Theorem \ref{thm_CV_liouville} (still assuming that the norm $\Vert\cdot\Vert$ on $\R^d$ is induced by a scalar product on $\R^d$, that $p\in[1,2]$ and that $q\in[1,+\infty]$, with $p\leq q$).

\begin{corollary}\label{cor_CV_liouville_indistinguishable}
Let $\bar\mu_0\in\mathcal{P}_c(\R^d)$ and let $t\mapsto\bar\mu(t)$ be the unique solution on $[0,T_0)$, with $T_0=T_{\max}(\supp(\bar\mu_0))$, of the Vlasov equation \eqref{vlasov_without_x} such that $\bar\mu(0)=\bar\mu_0$ (see Corollary \ref{cor_vlasov_indistinguishable}).
Besides, let $\bar\rho^N_0=\bar\mu_0^{\otimes N}$ and let $t\mapsto\bar\rho^N(t)=\Phi(t,\cdot)_*\bar\rho^N_0$ be the unique solution on $[0,T_0)$ of the Liouville equation \eqref{liouville} (without dependence on $X$) such that $\bar\rho^N(0)=\bar\rho^N_0$.
Then, for every $N\in\N^*$, for every $k\in\{1,\ldots,N\}$, we have
\begin{equation}\label{estim_W1_k_indistinguishable}
W_p^{[q]}\left( \bar\rho^N(t)_{N:k} , \bar\mu(t)^{\otimes k} \right) 
\leq 
k^{1/q} C_{\bar\mu}(t) \max \bigg( \bigg( \frac{k^2}{N} \bigg)^{1/p} , \frac{1}{N^{\frac{1}{q}-\frac{1}{2}}} \bigg) 
\end{equation}
for every $t\in[0,T_0)$, 
where $C_{\bar\mu}(t)$ is defined as in \eqref{def_Cmu} (without dependence on $x$).
\end{corollary}

\begin{proof}
The proof is the same as the one of Corollary \ref{cor_CV_liouville_empirical_indistinguishable}: we take $\bar\nu=\delta_{\bar x}$ for an arbitrary $\bar x\in\Omega$. Then $(\mu_0)^{se}_{X^N}=\delta_{\bar x}\otimes\bar\mu_0$ and thus $W_1((\mu_0)^{se}_{X^N},\mu_0)=W_p((\mu_0)^{se}_{X^N},\mu_0)=0$. We conclude the proof by noticing that, since the particle dynamics are invariant under permutations, we have $\bar\rho^N(t)_{N:k}=\bar\rho^N(t)_{N:k}^s$.
\end{proof}

\begin{remark}
It is interesting to observe that, in \eqref{estim_W1_k_indistinguishable}, we have not taken the symmetrization of the measure $\bar\rho^N(t)$ (in contrast to Corollary \ref{cor_CV_liouville_empirical_indistinguishable}).
\end{remark}

\begin{remark}
Applying Corollary \ref{cor_CV_liouville_indistinguishable} to the kinetic plus potential Hamiltonian case
where we have $G(t,(q_i,p_i),(q_j,p_j))=\begin{pmatrix} p_i , -\nabla V(q_i-q_j)\end{pmatrix}$,
we recover \cite[Theorem 3.1]{GolseMouhotPaul_CMP2016}.
The corollary can also be applied to more general Hamiltonian systems, for example, $G(t,(q_i,p_i),(q_j,p_j))=( p_i , -\nabla (V(q_i-q_j)+(p_i-A(q_i))^2) )$,
where $A:\R^d\to\R^d$ is a vector potential associated to a magnetic field; or to Cucker--Smale systems, 
for which $G(t,(q_i,p_i),(q_j,p_j))=\begin{pmatrix} p_i , F(\vert q_i-q_j\vert)(p_i-p_j) \end{pmatrix}$, 
and generalizations introduced in \cite{NataliniPaul_DCDSB_2022}. 
\end{remark}

\section{From mesoscopic to macroscopic scale: hydrodynamic moments and closures {\normalsize (from Vlasov to CGL when closure holds)}}\label{sec_hydro}
\subsection{Averaged dynamical quantities defined on $\Omega$}\label{sec_ave}

Given any $\mu\in\mathcal{P}(\Omega\times\R^d)$, disintegrated as $\mu=\int_{\Omega}\mu_x\, d\nu(x)$, the three macroscopic quantities traditionally considered in the \emph{hydrodynamic limit} procedure are the three first moments of the measure $\mu$ with respect to $\xi$ (see, e.g., \cite{Spohn_book1991}), leading to define, for $\nu$-almost every $x\in\Omega$:
\begin{itemize}
\item the total mass $\rho(x)\geq 0$ of $\mu_x$ by
$$
\rho(x) = \int_{\R^d}d\mu_x(\xi) = \mu_x(\R^d) ,
$$
(moment of order $0$). In our setting, since $\mu_x$ is a probability measure for $\nu$-almost every $x$, we have $\rho(x)=1$ for $\nu$-almost every $x\in\Omega$, and we will use this normalization throughout this section. We nevertheless keep the symbol $\rho(x)$ in the formul\ae\ below in order to make the underlying mass$\,\times\,$velocity structure visible and to facilitate comparison with the classical hydrodynamic literature;
\item the ``speed" $y(x)\in\R^d$ by
$$
\rho(x)\, y(x) = \int_{\R^d} \xi \, d\mu_x(\xi) ,
$$
(moment of order $1$)
which, since $\rho(x)=1$, is also the expectation of any random variable with law $\mu_x$;
\item and the ``temperature" $T(x)\geq 0$ by
$$
d\, \rho(x)\, T(x) = \int_{\R^d} \Vert \xi - y(x)\Vert^2\, d\mu_x(\xi) 
$$
(moment of order $2$), which is a variance; equivalently, if $\Vert\cdot\Vert$ is the Euclidean norm,
$$
\frac{1}{2}\rho(x) \Vert y(x)\Vert^2 + \frac{d}{2} \rho(x) T(x) = \frac{1}{2} \int_{\R^d} \Vert\xi\Vert^2\, d\mu_x(\xi) .
$$
\end{itemize}

Let $t\mapsto\mu(t)$ be a fixed solution of the Vlasov equation \eqref{vlasov} (recall that the mean field $\mathcal{X}[\mu]$ is defined by \eqref{def_mean_field}). 
According to Remark \ref{rem_notdependt}, its marginal $\nu(t)=\nu$ on $\Omega$ does not depend on $t$. 
Following the hydrodynamic limit procedure recalled above (see also, e.g., \cite{CarrilloChoi_ARMA2021, FigalliKang_APDE2019, NataliniPaul_DCDSB_2022}), for every $t\in\R$ and for $\nu$-almost every $x\in\Omega$, we define the three first moments $\rho(t,x)$, $y(t,x)$ and $T(t,x)$ of $\mu(t)$.
The moment $\rho(t,x)$ of order $0$ does not depend on $t$ and is equal to $1$ for $\nu$-almost every $x\in\Omega$ and $0$ otherwise. We now study the moments of order one and two. 

\subsection{Moment of order $1$ and CGL closure}
Given any solution $t\mapsto\mu(t)$ of the Vlasov equation \eqref{vlasov} on $[0,T]$ (for some $T>0$), of marginal $\nu$ on $\Omega$, using the disintegration of $\mu$ with respect to $\nu$ we define
\begin{equation}\label{ytx}
\boxed{
y(t,x) = \int_{\R^d} \xi \, d\mu_{t,x}(\xi) 
}
\end{equation}
for $\nu$-almost every $x\in\Omega$, and $y(t,x)=0$ for every $x\in\Omega\setminus\supp(\nu)$, for every $t\in[0,T]$.
As a preliminary remark, using \eqref{vlasov} (or, rather, \eqref{vlasov_xfixed}), we have
\begin{equation*}
\partial_t y(t,x)
= \left\langle \partial_t \mu_{t,x} , \xi\mapsto\xi \right\rangle
= \left\langle \mu_{t,x} , L_{\mathcal{X}[\mu_t](t,x,\cdot)} (\xi\mapsto\xi) \right\rangle
= \int_{\R^d} \mathcal{X}[\mu_t](t,x,\xi)\, d\mu_{t,x}(\xi)
\end{equation*}
which is the mean field $\mathcal{X}[\mu_t]$ averaged under the conditional measure $\mu_{t,x}$.
Hence
\begin{equation}\label{calcyt}
\partial_t y(t,x) = \int_{\R^d} \int_{\Omega\times\R^d} G(t,x,x',\xi,\xi') \, d\mu_t(x',\xi')\, d\mu_{t,x}(\xi)  .
\end{equation}
It is remarkable that, for some classes of functions $G$, and for some classes of initial data, the right-hand side of \eqref{calcyt} can be expressed in terms of $y(t,x)$ only: we thus obtain a ``closed" equation in $y$, as seen next.

\subsubsection{Linear CGL closure}\label{sec_op_propag}

\begin{proposition}\label{prop1}
Assume that $G$ is linear with respect to $(\xi,\xi')$, i.e.,
$$
G(t,x,x',\xi,\xi') = a_1(t,x,x')\xi + a_2(t,x,x')\xi' \qquad\forall (t,x,x',\xi,\xi') \in\R\times\Omega\times\Omega\times\R^d\times\R^d.
$$
For any solution $t\mapsto\mu(t)$ of the Vlasov equation \eqref{vlasov}, the mapping $t\mapsto y(t,\cdot)$, where $y(t,x)$ is defined by \eqref{ytx}, is a solution of the continuum / graph limit equation \eqref{Euler_general}, in which the operator $A$ is linear, given by
$$
A(t,y)(x) = \int_\Omega a_1(t,x,x')\, d\nu(x')\, y(x) + \int_\Omega a_2(t,x,x')y(x')\, d\nu(x')\qquad \forall y\in L^\infty_\nu(\Omega,\R^d),
$$
where $\nu$ is the marginal of $\mu(t)$ on $\Omega$ (not depending on $t$).
\end{proposition}

\begin{proof}
Using the disintegration of the measure, we infer from \eqref{calcyt} and from the specific expression of $G$ that
\begin{multline*}
\partial_t y(t,x) = 
\underbrace{\int_{\R^d} \xi\, d\mu_{t,x}(\xi)}_{=y(t,x)} \int_{\Omega} a_1(t,x,x') \underbrace{\int_{\R^d} d\mu_{t,x'}(\xi')}_{=1}\, d\nu(x') \\
+ \underbrace{\int_{\R^d} d\mu_{t,x}(\xi)}_{=1} \int_{\Omega} a_2(t,x,x') \underbrace{\int_{\R^d} \xi'\, d\mu_{t,x'}(\xi')}_{=y(t,x')}\, d\nu(x')
\end{multline*}
and the result follows. 
\end{proof}

\begin{remark}\label{rem_micromacro}
If $\mu(0)=\frac{1}{N}\sum_{i=1}^N \delta_{x^N_i}\otimes\delta_{\xi^N_i(0)}=\mu^e_{(X^N,\Xi^N_0)}$ as in Proposition \ref{prop_empirical_x}, then $\mu(t) = \mu^e_{(X^N,\Xi^N(t))}$ 
whose marginal on $\Omega$ is $\nu=\nu^e_{X^N}=\frac{1}{N}\sum_{i=1}^N\delta_{x^N_i}$ and whose disintegration with respect to $\nu$ is $\mu_{t,x} = \delta_{\xi^N_i(t)}$ if $x=x^N_i$ for $i\in\{1,\ldots,N\}$ and $0$ otherwise. In this case, in the context of Proposition \ref{prop1}, we have then $y(t,x) = \xi^N_i(t)$ if $x=x^N_i$ for $i\in\{1,\ldots,N\}$ and $0$ otherwise, and the differential equation \eqref{HK_euler} exactly coincides with the particle system \eqref{HK_particle}.
\end{remark}

Proposition \ref{prop1} applies, for example, to the Hegselmann--Krause system: under Assumption \ref{G}, if $t\mapsto\mu(t)$ is a solution of the Vlasov equation associated to the mean field \eqref{HK_meanfield} then $t\mapsto y(t,\cdot)$, with $y(t,x)$ defined by \eqref{ytx}, is a solution of the continuum / graph limit equation \eqref{HK_euler}.

However, the conclusion of Proposition \ref{prop1} no longer holds if $G$ is not linear with respect to $(\xi,\xi')$: in that case, $\partial_t y$ cannot be expressed solely in terms of $y$.

\subsubsection{Open issue: how to obtain a closed equation?}
An open question is to characterize the mappings $G$ such that, for any solution $\mu$ of \eqref{vlasov}, the function $y$ defined by \eqref{ytx} satisfies the nonlinear CGL equation \eqref{Euler_general}, $\partial_t y(t,\cdot) = A(t,y(t,\cdot))$.
We face here the classical problem in kinetic theory of closing the moment system: the equations for the three first moments depend a priori on higher-order moments, and suitable closure assumptions are not known in general (see \cite{CarrilloChoi_ARMA2021} for further discussion, see also Section \ref{sec_coupled} below). This motivates the monokinetic ansatz for $\mu$ described next.

\subsubsection{The $\nu$-monokinetic case}\label{sec_nu-monokinetic}
In this section, we assume that $\Omega$ is compact (for the non-compact case, see Remark \ref{rem_Omega_not_compact}).
Let us consider specific solutions $\mu$ of the Vlasov equation \eqref{vlasov}, that are \emph{$\nu$-monokinetic}, meaning that $\mu$ is delta-valued in the $\xi$ variable and has the marginal $\nu$ on $\Omega$.
Given any $\nu\in\mathcal{P}(\Omega)$ and any measurable function $y:\Omega\rightarrow\R^d$, we define the $\nu$-monokinetic measure $\mu^\nu_y$ on $\Omega\times\R^d$ by
\begin{equation}\label{def_monokinetic_measure}
\mu^\nu_y  = \nu\otimes \delta_{y(\cdot)}  .   
\end{equation} 
We have $y(x) = \int_\Omega \xi\, d(\mu^\nu_y)_x(\xi)$ for $\nu$-almost every $x\in\Omega$ (as in \eqref{ytx}), where the disintegration of $\mu^\nu_y$ with respect to its marginal $\nu$ on $\Omega$ is given by the family of conditional measures defined by $(\mu^\nu_y)_x = \delta_{y(x)}$.

\begin{proposition}\label{prop_monokinetic}
Let $\nu\in\mathcal{P}(\Omega)$. Let $T>0$ and let $t\mapsto y(t,\cdot)\in L^\infty_\nu(\Omega,\R^d)$ be a locally Lipschitz mapping on $[0,T]$. 

The mapping $t\mapsto\mu(t)=\mu^\nu_{y(t,\cdot)}\in\mathcal{P}_c(\Omega\times\R^d)$, of marginal $\nu$ on $\Omega$, is a ($\nu$-monokinetic) solution on $[0,T]$ of the Vlasov equation \eqref{vlasov} with the general mean field \eqref{def_mean_field} if and only if the mapping $t\mapsto y(t,\cdot)\in L^\infty_\nu(\Omega,\R^d)$ 
is a solution on $[0,T]$ of the (nonlinear) CGL equation \eqref{Euler_general}.
\end{proposition}

\begin{proof}
When $\mu_t=\mu^\nu_{y(t,\cdot)}$, \eqref{def_mean_field} gives $\mathcal{X}[\mu_t](t,x,\xi) = \int_{\Omega} G(t,x,x',\xi,y(t,x'))\, d\nu(x')$. The proof is straightforward; note that $A(t,y)(x) = \mathcal{X}[\mu^\nu_y](t,x,y(x))$ (where the nonlinear operator $A$ is defined by \eqref{def_A_general}). 
\end{proof}

\begin{remark}\label{rem_Tmax_Euler}
Proposition \ref{prop_monokinetic} implies Theorem \ref{thm_euler} (existence and uniqueness for the CGL equation \eqref{Euler_general}). Indeed, assume that $\Omega$ is compact, let $\nu\in\mathcal{P}(\Omega)$, let $y^0\in L^\infty_\nu(\Omega,\R^d)$, set $K'=\mathrm{ess.im}(y^0)$ its essential range (compact subset of $\R^d$) and $K=\Omega\times K'$ (compact).
Since the unique solution of the Vlasov equation \eqref{vlasov} such that $\mu(0)=\mu^\nu_{y^0}=\nu\otimes\delta_{y^0(\cdot)}$ is well defined on $[0,T_{\max}(K))$ (by Theorem \ref{thm_vlasov}) and is given by $\mu(t)=\mu^\nu_{y(t,\cdot)}$ (by Proposition \ref{prop_monokinetic}), it follows that the nonlinear CGL equation \eqref{Euler_general} has a unique solution on $[0,T_{\max}(K))$ such that $y(0,\cdot)=y^0(\cdot)$.
\end{remark}

When $\mu_t$ is not of the form $\mu^\nu_{y(t,\cdot)}$, $t\mapsto y(t,\cdot)$ fails in general to satisfy a ``closed" equation (i.e., $\partial_t y(t,\cdot)$ may not be expressible only in function of the first moment $y(t,\cdot)$).
Instead, there may be a full hierarchy of equations coupling all moments of $\mu_{t,x}$ (see Section \ref{sec_coupled}).
However, when convergence to consensus holds, one may expect that any solution $\mu$ of \eqref{vlasov} is asymptotically of the form $\mu^\nu_{y(t,\cdot)}$. 

\begin{remark}
In Remark \ref{rem_embedding_particle_Euler} in Section \ref{sec_euler_equation}, we have seen how to embed the solutions of the particle system to solutions of the CGL equation by considering an empirical measure $\nu$.
This embedding works because, when $\nu=\nu^e_{X^N}=\frac{1}{N}\sum_{i=1}^N \delta_{x^N_i}$ and $y(t,x^N_i)=\xi^N_i(t)$, the $\nu$-monokinetic measure $\mu^\nu_{y(t,\cdot)}$ coincides with the empirical measure $\mu^e_{(X^N,\Xi^N(t))}$. 
Indeed, 
$$
\mu^\nu_{y(t,\cdot)} = \frac{1}{N}\sum_{i=1}^N \delta_{x^N_i} \otimes \delta_{y(t,\cdot)} = \frac{1}{N}\sum_{i=1}^N \delta_{x^N_i} \otimes \delta_{y(t,x^N_i)} 
= \frac{1}{N}\sum_{i=1}^N \delta_{x^N_i} \otimes \delta_{\xi^N_i(t)} = \mu^e_{(X^N,\Xi^N(t))} .
$$

\end{remark}

\begin{remark}
In Appendix \ref{app_discrepancy_empirical_monokinetic}, we provide estimates on the discrepancy between empirical measures and $\nu$-monokinetic measures.
Lemma \ref{lem_discrepancy} of that appendix, combined with Theorem \ref{thm_estim_graph} and with the proof of that theorem, yields estimates on the discrepancy of the empirical measure $\mu^e_{(X^N,\Xi^N(t))}$ 
with respect to the $\nu$-monokinetic measures $\mu^\nu_{y(t,\cdot)}$ or $\mu^\nu_{y^N(t,\cdot)}$. 
\end{remark}

\begin{remark}
The proofs of Theorems \ref{thm_estim_graph} and \ref{thm_estim_graph_2} that we provide in Appendices \ref{app_proof_thm_estim_graph} and \ref{app_proof_thm_estim_graph_2} are direct, but actually one can also prove these propositions by applying Corollary \ref{cor_CV_liouville_empirical} with $\mu(t)=\mu^\nu_{y(t,\cdot)}=\nu\otimes\delta_{y(t,\cdot)}$ (the $\nu$-monokinetic measure) and use 
Lemma \ref{lem_discrepancy} of Appendix \ref{app_discrepancy_empirical_monokinetic}.
\end{remark}

\subsection{Moment of order $2$}\label{sec_moment_of_order_2}
In this section, we assume that the norm $\Vert\cdot\Vert$ on $\R^d$ is induced by a scalar product $\langle\ ,\ \rangle_{\R^d}$ on $\R^d$. 
We define
$$
\boxed{
T(t,x) = \frac{1}{d} \int_{\R^d} \Vert \xi - y(t,x)\Vert^2\, d\mu_{t,x}(\xi) 
}
$$
for $\nu$-almost every $x\in\Omega$.
Note that $T(t,x)=0$ for $\nu$-almost every $x\in\Omega\setminus\supp(\nu)$.
Using \eqref{vlasov} (or, rather, \eqref{vlasov_xfixed}) and noting that $\int_{\R^d} \langle \xi-y(t,x),\partial_t y(t,x)\rangle_{\R^d}\, d\mu_{t,x}(\xi) = 0$, we compute
\begin{equation}\label{dtTtx}
\partial_t T(t,x) = \frac{2}{d} \int_{\R^d} \left\langle \xi-y(t,x), \mathcal{X}[\mu_t](t,x,\xi) \right\rangle_{\R^d} \, d\mu_{t,x}(\xi) .
\end{equation}

\begin{proposition}[Cooling rate in the Hegselmann--Krause model]\label{prop_op_consensus}
In the Hegselmann--Krause model \eqref{HK_particle}, we have $G(t,x,x',\xi,\xi')=\sigma(x,x')(\xi'-\xi)$ and
$$
\boxed{
\partial_t T(t,x) = -2 S(x) T(t,x)
}
$$
where $S(x) = \int_{\Omega} \sigma(x,x')\, d\nu(x')$ for $\nu$-almost every $x\in\Omega$. Hence $t\mapsto T(t,x) = T(0,x) e^{-2tS(x)}$ decreases exponentially to $0$ as $t\rightarrow+\infty$ for $\nu$-almost every $x\in\Omega$ such that $S(x)>0$.
\end{proposition}

\begin{proof}
In the Hegselmann--Krause model (see \eqref{HK_meanfield}), we have $\mathcal{X}[\mu_t](t,x,\xi) = \int_\Omega\int_{\R^d} \sigma(x,x')(\xi'-\xi)\, d\mu_{t,x'}(\xi')\, d\nu(x')$, which expands as
$$
\mathcal{X}[\mu_t](t,x,\xi) 
= -S(x)(\xi-y(t,x)) + \int_\Omega \sigma(x,x')(y(t,x')-y(t,x))\, d\nu(x') .
$$
The second term is independent of $\xi$, hence has zero average against the centered measure $(\xi-y(t,x))\,d\mu_{t,x}(\xi)$. Combined with $\int_{\R^d} (\xi-y(t,x))\, d\mu_{t,x}(\xi)=0$ and \eqref{dtTtx}, the result follows.
\end{proof}

\begin{remark}\label{rem_consensus}
We will see in Remark \ref{rem_HK_moment_k} in Section \ref{sec_coupled} that, in the Hegselmann--Krause model, all moments of order $\geq 2$ satisfy the same differential equation, and thus, decrease exponentially to $0$ as $t\rightarrow+\infty$ as soon as $S(x)>0$ for $\nu$-almost every $x\in\Omega$. 
This shows that, under the latter assumption, the solution $t\mapsto\mu(t)$ of the Vlasov equation \eqref{vlasov} is such that $\mu_{t,x}$ is exponentially close (in Wasserstein distance) to the Dirac measure $\delta_{y(t,x)}$ as $t\rightarrow+\infty$. 

In \cite{BoudinSalvaraniTrelat_SIMA2022}, convergence to consensus is proved for the Euler equation under the assumptions that $d\nu(x)=dx$, that $S(x)\geq\delta>0$ for almost every $x\in\Omega$ and that the (infinite-dimensional) graph associated with $\sigma$ be strongly connected.
This remark shows that the result of \cite{BoudinSalvaraniTrelat_SIMA2022} can be generalized by relaxing the assumption on $S$ to: $S(x)>0$ for $\nu$-almost every $x\in\Omega$.
\end{remark}

For general mappings $G$, the question of whether or not $T$ is the solution of some ``closed" equation is open.

In the $\nu$-monokinetic case, i.e., assuming that $\mu$ is of the form \eqref{def_monokinetic_measure} and is a solution of \eqref{vlasov}, we have $T(t,x) = 0$. This is expected since $T(t,x)$ is the variance and thus measures the distance to the average $y(t,x)$.

\subsection{Generalization: coupled equations of moments}\label{sec_coupled}
More generally, assuming $d=1$ to simplify the notation, let us perform a formal expansion of $G$ around the first moment $y(t,\cdot)$. Setting $G_0(t,x,x') = G(t,x,x',y(t,x),y(t,x'))$, we have
$$
G(t,x,x',\xi,\xi') = G_0(t,x,x') + \sum_{i+j\geq 1} g_{ij}(t,x,x') (\xi-y(t,x))^i (\xi'-y(t,x'))^j
$$
where $y(t,x) = \int_{\R} \xi \, d\mu_{t,x}(\xi)$ is the moment of order $1$ of $\mu_{t,x}$ (recall that the moment of order $0$ is $y_0(t,x) = \int_{\R} d\mu_{t,x}(\xi) = 1$). Defining the central moment of order $i$ by
$$
y_i(t,x) = \int_{\R} (\xi-y(t,x))^i\, d\mu_{t,x}(\xi) \qquad \forall i\in\N
$$
(note that $y_0(t,x)=1$ and $y_1(t,x)=0$), 
we have
\begin{equation*}
\begin{split}
\mathcal{X}[\mu_t](t,x,\xi) &= \int_{\Omega\times\R} G(t,x,x',\xi,\xi') \, d\mu_t(x',\xi') \\
&= \int_{\Omega} G_0(t,x,x')\, d\nu(x') \\
&\qquad\qquad + \sum_{i+j\geq 1} (\xi-y(t,x))^i \int_{\Omega} g_{ij}(t,x,x') y_j(t,x') \, d\nu(x') 
\end{split}
\end{equation*}
and thus, using \eqref{def_A_general}, 
$$
\boxed{
\mathcal{X}[\mu_t](x,\xi) = A(t,y(t))(x) + \sum_{i+j\geq 1} (\xi-y(t,x))^i \int_{\Omega} g_{ij}(t,x,x') y_j(t,x') \, d\nu(x')
}
$$
It is interesting to see that, in the above formal expansion of $\mathcal{X}[\mu_t](x,\xi)$ using the centered moments, the first term is $A(t,y(t))(x)$.

Therefore, we have
\begin{equation*}
\begin{split}
\partial_t y(t,x) &= \int_{\R} \mathcal{X}[\mu_t](x,\xi)\, d\mu_{t,x}(\xi) \\
&= A(t,y(t))(x) + \sum_{i+j\geq 1} \left( \int_{\Omega} g_{ij}(t,x,x') y_j(t,x') \, d\nu(x') \right) y_i(t,x) 
\end{split}
\end{equation*}
(since $y_1=0$ and $y_1(t,x')=0$, the sum can be taken over all pairs $(i,j)$ with $i+j\geq 2$)
and, for every $k\in\N\setminus\{0,1\}$,
\begin{equation*}
\begin{split}
\partial_t y_k(t,x) &= \langle \mu_{t,x}, L_{\mathcal{X}[\mu_t]}.(\xi\mapsto(\xi-y(t,x))^k)\rangle -  \langle \mu_{t,x}, k(\xi-y(t,x))^{k-1} \partial_t y(t,x) \rangle \\
&= k \int_\R (\xi-y(t,x))^{k-1} \left( \mathcal{X}[\mu_t](x,\xi) - \partial_t y(t,x) \right) d\mu_{t,x}(\xi) \\
&= k \int_\R (\xi-y(t,x))^{k-1} \left( \mathcal{X}[\mu_t](x,\xi) - \int_{\R} \mathcal{X}[\mu_t](x,\xi')\, d\mu_{t,x}(\xi') \right) d\mu_{t,x}(\xi) \\
&= k \sum_{i+j\geq 1} \left( \int_{\Omega} g_{ij}(t,x,x') y_j(t,x')\, d\nu(x') \right) \int_\R (\xi-y(t,x))^{k-1} \left( (\xi-y(t,x))^i - y_i(t,x) \right) d\mu_{t,x}(\xi) \\
&= k \sum_{i+j\geq 1} \left( \int_{\Omega} g_{ij}(t,x,x') y_j(t,x')\, d\nu(x') \right) \Big( y_{k-1+i}(t,x) - y_{k-1}(t,x) y_i(t,x) \Big) 
\end{split}
\end{equation*}
(since $y_1=0$, the pair $(i,j)=(0,1)$ contributes zero).
In general, the equations for all moments are coupled and the system is not closed.

Closing the hierarchy of equations for all moments $y_i(t,x)$, $i\in\N^*$, by adding a small parameter $\varepsilon$ is an open question.

\begin{remark}\label{rem_HK_moment_k}
In the Hegselmann--Krause model \eqref{HK_particle}, we have $G(t,x,x',\xi,\xi')=\sigma(x,x')(\xi'-\xi)$ and thus $g_{ij}=0$ if $i+j\geq 2$ and $g_{01}=-g_{10}=\sigma$. We recover the facts that the equation in $y$ is closed and that $\partial_t y_2(t,x) = -2 S(x) y_2(t,x)$.
Moreover, a straightforward computation shows that
$$
\frac{1}{k} \partial_t y_k(t,x) = -S(x) y_k(t,x) \qquad \forall k\in\N\setminus\{0,1\},
$$
thus generalizing the case $k=2$ studied in Proposition \ref{prop_op_consensus}. Therefore, $y_k(t,x) = y_k(0,x) e^{-ktS(x)}$.
\end{remark}

\subsection{Directly from Liouville to CGL}\label{sec_liouville_to_euler}
In this section we record a direct Liouville-to-CGL statement at the level of first moments. The question is the following: if the Liouville equation contains the full $N$-particle probability distribution, which moment of this distribution converges, as $N\rightarrow +\infty$, to the first moment of the Vlasov equation, and hence to the CGL solution when a closure mechanism is available?

The answer is obtained by combining the marginal convergence of Section \ref{sec_liouville_to_vlasov} with the moment estimate of Lemma \ref{lem_moment}. Rather than introducing a new limiting procedure, we identify the Liouville-level observable whose limit is the macroscopic field.

Let $\rho^N(t)$ be a solution of the Liouville equation with one of the initializations used in Theorem \ref{thm_CV_liouville_empirical} or Theorem \ref{thm_CV_liouville}. For $i\in\{1,\ldots,N\}$ we define the first-moment marginal $\mathcal{M}_1^i[\rho^N(t)]$, a vector-valued measure on $\Omega^N\times\R^{d(N-1)}$, by duality: for every scalar test function $\varphi$ on $\Omega^N\times\R^{d(N-1)}$,
$$
\left\langle \mathcal{M}_1^i[\rho^N(t)], \varphi \right\rangle
= \int_{\Omega^N\times\R^{dN}} \varphi(X,\Xi_{-i})\xi_i\, d\rho^N(t)(X,\Xi),
$$
where $\Xi_{-i}$ denotes the vector $\Xi$ with the $i^\textrm{th}$ component removed. Given a solution $\mu(t)$ of the Vlasov equation, we also set
$$
y(t,x) = \int_{\R^d}\xi\, d\mu_{t,x}(\xi)
$$
whenever the first moment is finite.

\begin{proposition}\label{prop_liouville_to_cgl_moment}
Under the assumptions of Corollary \ref{cor_CV_liouville}, the first-moment marginal of the symmetrized Liouville solution converges to the Vlasov first moment in the bounded Lipschitz dual norm \eqref{def_BL_norm}. More precisely, uniformly with respect to $i\in\{1,\ldots,N\}$ and with respect to $t$ on compact subintervals of the common interval of existence,
$$
\Vert (\mathcal{M}_1^i[\rho^N(t)])^s_{N:1} - y(t,\cdot)\nu \Vert_{\mathrm{BL}^*}
\leq R_N(t) ,
$$
where $R_N(t)$ is controlled by the right-hand side of \eqref{cor_estim_W1_k} with $k=p=q=1$ (up to a multiplicative constant depending only on $\sup_t\diam_{\R^d}(\supp(\bar\mu(t)))$). In particular, $R_N(t)\rightarrow 0$ as $N\rightarrow+\infty$.
\end{proposition}

The $\mathrm{BL}^*$ norm used above is defined by \eqref{def_BL_norm} in Appendix \ref{app_lem_moment}.

\begin{corollary}[Monokinetic and linear closed cases]
If $\mu(t)=\nu\otimes\delta_{y(t,\cdot)}$, then $y$ is the solution of the nonlinear CGL equation \eqref{Euler_general}. Hence the observable in Proposition \ref{prop_liouville_to_cgl_moment} converges to the CGL solution. The same conclusion holds for the Hegselmann--Krause model, and more generally for kernels linear in $(\xi,\xi')$, because the first moment is then closed.
\end{corollary}

\begin{proof}
Symmetrization commutes with taking first-order marginals. Therefore the first-moment marginal above is the first moment of $\rho^N(t)^s_{N:1}$. The result follows from Corollary \ref{cor_CV_liouville}, applied with $k=1$, and from Lemma \ref{lem_moment}. The corollary follows from Proposition \ref{prop_monokinetic} and Proposition \ref{prop1}.
\end{proof}

\subsection{Hydrodynamic limit in the second-order case}\label{sec_second_order_hydrodynamics}

For second-order particle systems \eqref{second-order_particle} where $q$ is interpreted as a position and $p$ as a speed (or a momentum), in the classical kinetic literature where particles are assumed to be indistinguishable, the hydrodynamic quantities that are most often considered are the three first moments of $\mu$ integrated with respect to $p$ and kept as functions of $q$.

Recall that the ``hydrodynamics" replaces a scalar equation on ``phase space" (the one of $q$ and $p$), the Vlasov one, by a system of equations on the ``configuration space" (the one of $q$ only). The interest of this approach is twofold: firstly at the conceptual level, as it casts the dynamics on a physical, directly observable space --- a feature particularly important for situations not naturally embedded in the classical physics paradigm, e.g., biology, economy, social sciences --- and secondly from a numerical viewpoint, since the increased number of variables in phase-space PDEs is very costly (see \cite{npm} for a comparison between numerics for Vlasov and Euler). 

More precisely, in the second-order case, any solution $t\mapsto\mu(t)$ of the Vlasov equation (not depending on $x$) is such that $\mu(t)\in\mathcal{P}_c(\R^r\times\R^r)$, where $d=2r$ and $\xi=(q,p)$, and, assuming that $\frac{d\mu_t(q,p)}{dq\, dp}=f(t,q,p)$, the three first (marginal) moments of $\mu(t)$ under consideration are:
\smallskip
\begin{itemize}
\item the mass $m(q) = \int_{\R^r} f(t,q,p)\, dp$,
\item the momentum $m(q)v(q) = \int_{\R^r} p\, f(t,q,p)\, dp$,
\item the energy (or temperature) $m(q)E(q) = \int_{\R^r} \Vert p-v(q)\Vert^2\, f(t,q,p)\, dp$,
\end{itemize}
as defined in \cite{Spohn_book1991}.
This is different from what we did in Section \ref{sec_ave}.

These quantities are introduced for example in \cite{HaTadmor_KRM2008} in the Cucker--Smale model.
In general, they provide alternative objects of investigation in the specific case of second-order models. Note however that they do not satisfy a closed system of equations; the equations for the mass and the momentum close under a monokinetic ansatz on the Vlasov solution.

We assume that $d=2r$, $\xi=(q,p)\in\R^r\times\R^r$, and that the system of particles is an indistinguishable second-order system of the form
$$
\dot q_i(t) = p_i(t),\qquad
\dot p_i(t) = \frac{1}{N}\sum_{j=1}^N K(t, q_i(t),p_i(t),q_j(t),p_j(t))
$$
for some continuous mapping $K$ of class $C^1$ with respect to its four last variables. 
Here, $q_i(t)$ is the position and $p_i(t)$ is the velocity of agent $i$.
Assumption \ref{G} in Section \ref{sec_micro} is satisfied with 
$G(t,x,x',\xi,\xi') = ( p, K(t,\xi,\xi') )$.
Since $G$ does not depend on $(x,x')$, particles are indistinguishable. When considering the limit equations 
(Euler or Vlasov), we can however choose to distinguish them, as already discussed.

\paragraph{CGL equation.}
Choosing a set $\Omega$ of labels and a probability measure $\nu$ on $\Omega$, the CGL equation \eqref{Euler_general} is
\begin{equation}\label{euler_2ndorder}
\begin{split}
\partial_t y_1(t,x) &= y_2(t,x) \\
\partial_t y_2(t,x) &= \int_\Omega K(t,y_1(t,x),y_2(t,x),y_1(t,x'),y_2(t,x'))\, d\nu(x')
\end{split}
\end{equation}
where it is understood that $y_1(t,x_i)\simeq q_i(t)$ and $y_2(t,x_i)\simeq p_i(t)$
(see Section \ref{sec_estimates_graph_limit}).


\paragraph{Vlasov equation.}
As we have seen before in \eqref{meanfield_indistinguishable} and \eqref{vlasov_without_x}, since in the present case the mean field $\mathcal{X}[\mu]$ does not depend on the variable $x$, we have $\mathcal{X}[\mu](t,x,\xi)=\bar{\mathcal{X}}[\bar\mu](t,\xi)$ where the mean field $\bar{\mathcal{X}}[\bar\mu]$ is given by
$$
\bar{\mathcal{X}}[\bar\mu](t,\xi) = \begin{pmatrix} p\\ \int_{\R^r\times\R^r} K(t,\xi,\xi')\, d\bar\mu(\xi') \end{pmatrix}
$$
and the Vlasov equation is $\partial_t\bar\mu + \div(\bar{\mathcal{X}}[\bar\mu]\bar\mu) = 0$.

\paragraph{Hydrodynamic moments and the standard Euler equation.}
We now consider the hydrodynamic variables in the sense usually considered in the literature: moments of the measure $\bar\mu$ where the integration is performed with respect to $p$, but keeping $q$ as a parameter.

\medskip
\textbf{Approach by disintegration.}
If we proceed by disintegration, the resulting moments are not particularly informative.
Given any probability measure $\bar\mu$ on $\R^r\times\R^r$ (where the variables are $(q,p)$), let $\theta$ be the marginal of $\bar\mu$ on the first copy of $\R^r$ (where the variable is $q$). The disintegration of $\bar\mu$ with respect to $\theta$ is $\bar\mu = \int_{\R^r}\mu_q\, d\theta(q)$ where the measures $\mu_q$ are probability measures.
The moment of order $0$ is then $\rho(q) = \int_{\R^r} d\mu_q(p) = 1$ and is therefore uninformative; we could consider the moments of order $1$ and $2$, but we proceed instead in the way usually done in the literature.

\medskip
\textbf{The usual approach.}
We assume that, for every $t$, $\bar\mu(t)$ (solution of Vlasov) is absolutely continuous, i.e., $d\bar\mu_t(q,p) = f(t,q,p)\, dq\, dp$, and we define
\begin{itemize}
\item the moment of order $0$: $\rho(t,q) = \int_{\R^r} f(t,q,p)\, dp$;
\item the moment of order $1$: $\rho(t,q) u(t,q) = \int_{\R^r} p\, f(t,q,p)\, dp$.
\end{itemize}
This is different from the disintegration approach above, because now $\rho(t,q)$ is not constant.
Actually, comparing with the disintegration of $\bar\mu_t$, we have 
$\frac{d\bar\mu_{t,q}}{dp}(p) = f(t,q,p) / \int_{\R^r} f(t,q,p')\, dp'$. 
So, here, the way we consider the hydrodynamic variables is different. 

We recall how to compute the equations satisfied by $\rho$ and $u$ (this is well known in the existing literature).
The Vlasov equation for $f$ is
$$
\partial_t f + \langle p,\nabla_q f\rangle + \mathrm{div}_p(\mathcal{X}_K[f]f) = 0 
$$
where $\mathcal{X}_K[f](t,q,p) = \int_{\R^r}\int_{\R^r} K(t,q,p,q',p') f(t,q',p')\, dq'\, dp'$.
Multiplying by $\varphi_1(q)\varphi_2(p)$ and integrating, we get
\begin{multline*}
\int_{\R^r}\int_{\R^r} \varphi_1(q)\varphi_2(p) \partial_t f(t,q,p)\, dq\, dp \\
= \int_{\R^r}\int_{\R^r} \Big( \varphi_2(p)\, d\varphi_1(q).p + \varphi_1(q)\, d\varphi_2(p).\mathcal{X}_K[f](t,q,p) \Big) \, f(t,q,p)\, dq\, dp
\end{multline*}
First, taking $\varphi_2(p)=1$ and integrating by parts in the equation above, the resulting identity holds for every test function $\varphi_1$; localizing in $q$ yields
$$
\boxed{
\partial_t\rho + \mathrm{div}_q(\rho u)=0
}
$$
which is the classical continuity equation. It is obtained without any specific assumption, in contrast to the next one.

Taking $\varphi_2=p_i$ for any $i\in\{1,\ldots,r\}$ (where $p=(p_1,\ldots,p_r)$), replacing above and integrating by parts does not suffice in producing a ``closed" equation. As done usually, we assume that the velocity distribution is monokinetic: $f(t,q,p) = \rho(t,q)\delta(p-u(t,q))$. In this way we obtain the equation
$$
\boxed{
\partial_t (\rho u) + \mathrm{div}_q(\rho u\otimes u) = \int_{\R^r} K(t,q,u(t,q),q',u(t,q'))\, \rho(t,q)\, \rho(t,q')\, dq'  
}
$$
which is the pressureless Euler equation, as obtained by \cite{CarrilloFornasierToscaniVecil_2010, FigalliKang_APDE2019} and \cite{NataliniPaul_SIMA2023} for more general systems including chemiotaxis.
Although we have used the same name, this Euler equation (and the way it has been obtained) is of course completely different from the one given in \eqref{euler_2ndorder}.

\section{Further comments and perspectives}\label{sec_further_comments}
In this paper, we have studied various ways to pass to the limit in finite systems of (possibly distinguishable) particles, and described precise relationships between the various limits: Vlasov, CGL, hydrodynamic Euler and Liouville equations. 
This has been done under the standing assumption \ref{G} on the interaction mapping $G$ modeling the particle dynamics. 

As already said, we have restricted our study to regular mappings $G$, because our objective was to highlight in the simplest possible way the basic relationships between the microscopic, mesoscopic and macroscopic scales, with the most possible general viewpoint. The study of singular kernels is much more challenging and requires the development of other techniques. 

Related to this issue, we give in the section hereafter a surprising consequence of our results.

\subsection{Approximation of PDEs by finite particle systems}\label{sec_approx_PDE}

We end this article with a perspective showing that the CGL viewpoint can also be read in the reverse direction: starting from a (linear or quasilinear) PDE on a bounded domain, one can build a family of finite particle systems whose pointwise Riemann-sum graph limit approximates the PDE. Such reverse derivations have been investigated in specific contexts, in particular for scalar conservation laws obtained as macroscopic limits of follow-the-leader traffic models \cite{dif0, dif2}.
The full theory, with sharp algebraic rates and a careful treatment of boundary conditions, would require different discretization tools than the pointwise Riemann sums used throughout this article. 
The purpose of the present subsection is more modest: we explain, within the framework of pointwise Riemann sums developed in this paper, how a regularized PDE is naturally approximated by a graph-limit particle system, and what convergence rate one can extract directly from the theory of Section~\ref{sec_graphlimit}. This reveals both the power and the intrinsic limitations of pointwise Riemann-sum discretizations of singular Schwartz kernels.

\paragraph{Setting and notation.}
Let $\Omega$ be the closure of a bounded open subset of $\R^n$, of Lebesgue measure $1$, let $\nu$ be the Lebesgue measure on $\Omega$, and let $(\mathcal A^N,X^N)_{N\in\N^*}$ be a family of tagged partitions of $\Omega$ with cells $\Omega_i^N$ and tags $x_i^N$, satisfying
\begin{equation*}
h_N = \max_{1\leq i\leq N} \diam_\Omega(\Omega_i^N) \leq \frac{C_\Omega}{N^{1/n}}.
\end{equation*}
We consider an evolution equation of the form
\begin{equation}\label{PDE_eq}
\partial_t y(t,x) = \sum_{\vert\alpha\vert\leq p} a_\alpha(t,x,y(t,x))\, D^\alpha y(t,x),
\qquad y(0,\cdot)=y^0,
\end{equation}
of order $p\geq 1$, with $\R^d$-valued unknown $y$ and with $d$-by-$d$ matrix coefficients $a_\alpha$ that we assume sufficiently smooth. Appropriate boundary conditions are understood and, for clarity of exposition, we leave them implicit.
For brevity we keep the same notation in the linear case, where the coefficients $a_\alpha=a_\alpha(t,x)$ do not depend on $y$.
We assume that \eqref{PDE_eq} has a sufficiently regular solution on $[0,T]$.

\paragraph{Kernel regularization.}
Equation \eqref{PDE_eq} cannot be put in the form of the CGL equation \eqref{Euler_general} directly, because the differential operator on its right-hand side is unbounded and corresponds to a distributional Schwartz kernel.
The strategy is to first \emph{kernelize} this operator at scale $\varepsilon>0$ and then apply the graph-limit theory of Section~\ref{sec_graphlimit} to the resulting regularized equation.

Let $\eta\in\mathscr{C}^\infty_c(\R^n)$ be a smooth mollifier with $\int_{\R^n}\eta=1$, and let $\eta_\varepsilon(x)=\varepsilon^{-n}\eta(x/\varepsilon)$. A model regularization of the operator in \eqref{PDE_eq} is given by the smooth kernel
\begin{equation}\label{PDE_sigma_model}
\sigma_\varepsilon(t,x,x',\xi) = \sum_{\vert\alpha\vert\leq p} a_\alpha(t,x,\xi)\, D_x^\alpha\eta_\varepsilon(x-x').
\end{equation}
The formula \eqref{PDE_sigma_model} is only a local model, or a model without boundary. On a bounded domain with boundary conditions, a boundary-compatible regularization (such as Burenkov's variable-step mollifiers) would be required; this is deliberately not developed here.
The corresponding regularized equation
\begin{equation}\label{PDE_regularized}
\partial_t y_\varepsilon(t,x) = \int_\Omega \sigma_\varepsilon(t,x,x',y_\varepsilon(t,x))\, y_\varepsilon(t,x')\, dx',
\qquad y_\varepsilon(0,\cdot)=y^0,
\end{equation}
is exactly the CGL equation \eqref{Euler_general} associated with the interaction mapping
\begin{equation}\label{PDE_Geps}
G_\varepsilon(t,x,x',\xi,\xi') = \sigma_\varepsilon(t,x,x',\xi)\,\xi'.
\end{equation}
We assume that the regularization is consistent with the original PDE in the sense that, for some $\alpha>0$ (typically $\alpha=1$ for first-order mollification of sufficiently regular solutions),
\begin{equation}\label{PDE_regularization_error}
\Vert y-y_\varepsilon\Vert_{\mathscr{C}^0([0,T],X)} \leq C_T\,\varepsilon^\alpha,
\end{equation}
where $X=L^\infty(\Omega,\R^d)$ or $L^2(\Omega,\R^d)$.

\begin{example}[1D transport equation]\label{ex_transport_PDE}
The simplest illustration is the 1D linear transport equation $\partial_t y(t,x)=-\partial_x y(t,x)$ on $\Omega=[0,1]$, for which $p=1$, $n=1$, and the only non-zero coefficient is $a_1\equiv -1$. The smooth kernel \eqref{PDE_sigma_model} reduces to
\begin{equation*}
\sigma_\varepsilon(x,x') = -\,\partial_x\eta_\varepsilon(x-x') = -\,\frac{1}{\varepsilon^2}\,\eta'\left(\frac{x-x'}{\varepsilon}\right),
\end{equation*}
which (for an even mollifier $\eta$) is antisymmetric in $x-x'$ and is a smooth approximation of the distributional kernel of $-\partial_x$. Convolution with this kernel gives the regularized transport operator $A_\varepsilon f(x)=\int_\Omega \sigma_\varepsilon(x,x')f(x')\,dx'$. The associated particle system \eqref{PDE_particle} below then reads
$$
\dot\xi_{\varepsilon,i}^N(t) = -\,\frac{1}{N}\sum_{j=1}^N \frac{1}{\varepsilon^2}\,\eta' \left(\frac{x_i^N-x_j^N}{\varepsilon}\right) \xi_{\varepsilon,j}^N(t),
$$
i.e., a finite-difference-type interacting system in which each agent's velocity is a weighted antisymmetric finite difference of the values of the other agents.
\end{example}

For the model kernel \eqref{PDE_sigma_model}, on bounded ranges of $\xi$, the pointwise sizes of $\sigma_\varepsilon$ scale like
\begin{equation}\label{PDE_kernel_bounds}
\Vert\sigma_\varepsilon\Vert_{L^\infty} \leq \frac{C}{\varepsilon^{n+p}},
\qquad
\Lip_{(x,x')}(\sigma_\varepsilon) \leq \frac{C}{\varepsilon^{n+p+1}},
\qquad
\Lip_\xi(\sigma_\varepsilon) \leq \frac{C}{\varepsilon^{n+p}}.
\end{equation}
Crucially however, after integration in $x'$, both the kernel and its $\xi$-derivative satisfy the much milder \emph{row-sum bounds}
\begin{equation}\label{PDE_rowsum_continuous}
\sup_{t,x,\xi}\int_\Omega \Vert\sigma_\varepsilon(t,x,x',\xi)\Vert\, dx' \leq \frac{C}{\varepsilon^p},
\qquad
\sup_{t,x,\xi}\int_\Omega \Vert\partial_\xi\sigma_\varepsilon(t,x,x',\xi)\Vert\, dx' \leq \frac{C}{\varepsilon^p}.
\end{equation}
This reflects the natural scaling of an operator of order $p$: only $p$ derivatives of the mollifier survive integration in $x'$, while $n+p$ pointwise derivatives are needed to localize the kernel. The contrast between \eqref{PDE_kernel_bounds} and \eqref{PDE_rowsum_continuous} is what allows one to gain an exponential factor in the estimates below, turning a double-exponential into a single exponential in $\varepsilon$.

\paragraph{The associated particle system.}
Applying the graph-limit construction of Section~\ref{sec_graphlimit} to the CGL equation~\eqref{PDE_regularized}, with kernel $G_\varepsilon$ given by~\eqref{PDE_Geps}, leads to the family of particle systems
\begin{equation}\label{PDE_particle}
\boxed{
\dot\xi_{\varepsilon,i}^N(t) = \frac{1}{N}\sum_{j=1}^N \sigma_\varepsilon\big(t,x_i^N,x_j^N,\xi_{\varepsilon,i}^N(t)\big)\, \xi_{\varepsilon,j}^N(t),
\qquad \xi_{\varepsilon,i}^N(0)=y^0(x_i^N),
\qquad 1\leq i\leq N,
}
\end{equation}
indexed by both scales $\varepsilon$ and $N$. According to \eqref{def_y_Xi}, the reconstructed piecewise constant function is
\begin{equation*}
y_\varepsilon^N(t,x) = \sum_{i=1}^N \xi_{\varepsilon,i}^N(t)\,\mathds{1}_{\Omega_i^N}(x).
\end{equation*}
The question is now to estimate $\Vert y-y_\varepsilon^N\Vert_X$ as a function of $\varepsilon$ and $N$, and to balance the two scales.

\paragraph{Convergence estimate.}
For every fixed $\varepsilon>0$, the kernel $G_\varepsilon$ in \eqref{PDE_Geps} satisfies Assumption~\ref{G} and is locally Lipschitz, so Theorem~\ref{thm_estim_graph} applies. However, applying it as a black box to $G_\varepsilon$ is far from optimal: the full pointwise Lipschitz constant of $G_\varepsilon$, which is of order $\varepsilon^{-(n+p+1)}$ by \eqref{PDE_kernel_bounds}, enters the exponent of the Gronwall factor, and a rough $L^\infty$ estimate of the regularized trajectories produces a further exponential, which yields the safe but pessimistic bound
\begin{equation}\label{PDE_loglog_estimate}
\Vert y(t)-y_\varepsilon^N(t)\Vert_X
\leq
C_T\left(
\varepsilon^\alpha + \frac{1}{N^{1/n}}\,\exp\left(\frac{C_T}{\varepsilon^{n+p+1}}\,\exp\left(\frac{C_T}{\varepsilon^p}\right)\right)\right).
\end{equation}
Optimizing in $\varepsilon$ gives at best a double-logarithmic rate (of the form $(\ln\ln N)^{-\alpha/p}$ up to lower-order logarithmic corrections, since the inner exponential $\exp(C_T/\varepsilon^p)$ is the dominant term to balance against $1/N^{1/n}$).
As we now show, both exponentials in \eqref{PDE_loglog_estimate} can be tamed at once by comparing the discrete and the regularized continuous dynamics directly, instead of invoking Theorem~\ref{thm_estim_graph} as a black box. This will lead to a logarithmic rate.

Set $e_i(t) = \xi_{\varepsilon,i}^N(t) - y_\varepsilon(t,x_i^N)$.
Subtracting the equation satisfied by $y_\varepsilon(\cdot,x_i^N)$ from \eqref{PDE_particle}, we obtain
\begin{multline}\label{PDE_error_eq}
\dot e_i(t) = \frac{1}{N}\sum_{j=1}^N \sigma_\varepsilon\big(t,x_i^N,x_j^N,\xi_{\varepsilon,i}^N(t)\big)\, e_j(t)
\\
+ \underbrace{\frac{1}{N}\sum_{j=1}^N \Big(\sigma_\varepsilon(t,x_i^N,x_j^N,\xi_{\varepsilon,i}^N(t))-\sigma_\varepsilon(t,x_i^N,x_j^N,y_\varepsilon(t,x_i^N))\Big) y_\varepsilon(t,x_j^N)}_{\text{linearization defect}}
+ r_i^N(t),
\end{multline}
where the Riemann residual is
\begin{equation}\label{PDE_residual}
r_i^N(t) = \frac{1}{N}\sum_{j=1}^N \sigma_\varepsilon(t,x_i^N,x_j^N,y_\varepsilon(t,x_i^N))\, y_\varepsilon(t,x_j^N) \;-\; \int_\Omega \sigma_\varepsilon(t,x_i^N,x',y_\varepsilon(t,x_i^N))\, y_\varepsilon(t,x')\, dx'.
\end{equation}

Let us bound the three terms in \eqref{PDE_error_eq}.
For the linearization defect, the mean-value theorem in $\xi$ gives, with $\xi_\theta^i(t)=y_\varepsilon(t,x_i^N) + \theta\, e_i(t)$,
$$
\sigma_\varepsilon(t,x_i^N,x_j^N,\xi_{\varepsilon,i}^N(t)) - \sigma_\varepsilon(t,x_i^N,x_j^N,y_\varepsilon(t,x_i^N))
= \left(\int_0^1 \partial_\xi\sigma_\varepsilon\big(t,x_i^N,x_j^N,\xi_\theta^i(t)\big)\,d\theta\right) e_i(t).
$$
Taking the norm and summing in $j$, the linearization defect is bounded by
$$
\Vert y_\varepsilon\Vert_{L^\infty}\,\Vert e_i(t)\Vert\,\sup_{\xi}\frac{1}{N}\sum_{j=1}^N \big\Vert\partial_\xi\sigma_\varepsilon(t,x_i^N,x_j^N,\xi)\big\Vert.
$$
The same row-sum quantity also bounds the contribution of the first term in \eqref{PDE_error_eq}, hence both are controlled by the discrete analogue of \eqref{PDE_rowsum_continuous},
\begin{equation}\label{PDE_rowsum}
\max_{1\leq i\leq N}\, \frac{1}{N}\sum_{j=1}^N \Big( \Vert\sigma_\varepsilon(t,x_i^N,x_j^N,\xi)\Vert + \Vert\partial_\xi\sigma_\varepsilon(t,x_i^N,x_j^N,\xi)\Vert \Big) \leq \frac{C_T}{\varepsilon^p},
\end{equation}
which holds for $N$ large enough relative to $\varepsilon$ along the scales chosen below (it follows from \eqref{PDE_rowsum_continuous} by the elementary Riemann argument as soon as $h_N \varepsilon^{-(n+p+1)}\leq C\varepsilon^{-p}$).
The crucial point is that, in the Gronwall step below, $\Vert e_j(t)\Vert$ is multiplied only by the row sum \eqref{PDE_rowsum} of order $\varepsilon^{-p}$, not by the pointwise Lipschitz constant $\varepsilon^{-(n+p+1)}$ used in the black-box estimate \eqref{PDE_loglog_estimate}.

The Riemann residual~\eqref{PDE_residual} is a quadrature error for the function $x'\mapsto \sigma_\varepsilon(t,x_i^N,x',y_\varepsilon(t,x_i^N))\, y_\varepsilon(t,x')$, whose spatial Lipschitz norm is controlled by $\Lip_{x'}(\sigma_\varepsilon)\,\Vert y_\varepsilon\Vert_{L^\infty} + \Vert\sigma_\varepsilon\Vert_{L^\infty}\,\Lip(y_\varepsilon)$, so the elementary Riemann estimate gives
\begin{equation}\label{PDE_residual_bound}
\max_{1\leq i\leq N} \Vert r_i^N(t)\Vert \leq C_T\,h_N\,\frac{1}{\varepsilon^{n+p+1}}\,\big(\Vert y_\varepsilon\Vert_{L^\infty}+\Lip(y_\varepsilon)\big),
\end{equation}
where $\Lip(y_\varepsilon)$ denotes the spatial Lipschitz constant of $y_\varepsilon(t,\cdot)$ on $[0,T]\times\Omega$.

Setting $E(t)=\max_i \Vert e_i(t)\Vert$ and combining the three bounds above with \eqref{PDE_error_eq}, we obtain the differential inequality
$$
\dot E(t) \leq \frac{C_T}{\varepsilon^p}\,E(t) + C_T\,h_N\,\frac{\Vert y_\varepsilon\Vert_{L^\infty}+\Lip(y_\varepsilon)}{\varepsilon^{n+p+1}}.
$$
Since $E(0)=0$, Gronwall's lemma yields
\begin{equation*}
E(t) \leq C_T\,h_N\,\frac{\Vert y_\varepsilon\Vert_{L^\infty}+\Lip(y_\varepsilon)}{\varepsilon^{n+p+1}}\,\exp\Big(\frac{C_T}{\varepsilon^p}\Big).
\end{equation*}
Since $y_\varepsilon^N$ is piecewise constant on the cells $\Omega_i^N$ of diameter at most $h_N$, the spatial Lipschitz bound on $y_\varepsilon$ gives, for any $x\in\Omega_i^N$,
$$
\Vert y_\varepsilon^N(t,x)-y_\varepsilon(t,x)\Vert \leq \Vert\xi_{\varepsilon,i}^N(t)-y_\varepsilon(t,x_i^N)\Vert + h_N\,\Lip(y_\varepsilon) \leq E(t) + h_N\,\Lip(y_\varepsilon),
$$
where the second term is absorbed in the first up to a multiplicative constant. We conclude that
\begin{equation}\label{PDE_reconstruction_error}
\Vert y_\varepsilon^N(t)-y_\varepsilon(t)\Vert_{L^\infty(\Omega,\R^d)} \leq C_T\,h_N\,\frac{\Vert y_\varepsilon\Vert_{L^\infty}+\Lip(y_\varepsilon)}{\varepsilon^{n+p+1}}\,\exp \left(\frac{C_T}{\varepsilon^p}\right).
\end{equation}
To make \eqref{PDE_reconstruction_error} explicit in $\varepsilon$, suppose that, on $[0,T]$, the regularized solution satisfies the spatial Lipschitz bound
\begin{equation}\label{PDE_yeps_Lip_bound}
\Vert y_\varepsilon\Vert_{L^\infty}+\Lip(y_\varepsilon) \leq C_T\,\exp\left(\frac{C_T}{\varepsilon^\lambda}\right)
\end{equation}
for some $\lambda\geq 0$. Plugging \eqref{PDE_yeps_Lip_bound} into \eqref{PDE_reconstruction_error} and combining with \eqref{PDE_regularization_error} via the triangle inequality, we obtain
\begin{equation}\label{PDE_main_estimate}
\boxed{
\Vert y_\varepsilon^N(t)-y(t)\Vert_X \leq C_T\left( \varepsilon^\alpha + \frac{1}{N^{1/n}}\,\frac{1}{\varepsilon^{n+p+1}}\,\exp\left(\frac{C_T}{\varepsilon^\kappa}\right)\right)
\qquad\forall t\in[0,T],
}
\end{equation}
where $\kappa=\max(p,\lambda)$. Compared with the black-box estimate \eqref{PDE_loglog_estimate}, the inner exponential of the double-exponential has been collapsed into the polynomial prefactor; what remains is a single exponential in a negative power of $\varepsilon$.

Optimizing \eqref{PDE_main_estimate} now leads to a logarithmic, rather than double-logarithmic, balance. Choosing $\varepsilon_N = (C/\ln N)^{1/\kappa}$ with $C$ large enough that $\exp(C_T/\varepsilon_N^\kappa)\leq N^{1/(n+1)}$, the polynomial prefactor $\varepsilon^{-(n+p+1)}$ contributes only a power of $\ln N$ and the second term is dominated by the first; we obtain the rate
\begin{equation}\label{PDE_log_rate}
\boxed{
\Vert y_{\varepsilon_N}^N(t)-y(t)\Vert_X \leq \frac{C_T}{(\ln N)^{\alpha/\kappa}}
\qquad\forall t\in[0,T].
}
\end{equation}
The exponent $\kappa=\max(p,\lambda)$ has a clear interpretation:
\begin{itemize}[leftmargin=*]
\item \emph{Linear, uniformly stable case.} If \eqref{PDE_eq} is linear and the regularized linear flow $y_\varepsilon$ is uniformly bounded in spatial Lipschitz norm on $[0,T]$ (i.e., \eqref{PDE_yeps_Lip_bound} holds with $\lambda=0$), then $\kappa=p$ and the rate is $C_T(\ln N)^{-\alpha/p}$. For Example~\ref{ex_transport_PDE} and $\alpha=1$, this is $\mathrm{O}((\ln N)^{-1})$.
\item \emph{Generic linear or quasilinear case.} Without uniform stability, differentiating \eqref{PDE_regularized} in space generically yields \eqref{PDE_yeps_Lip_bound} with $\lambda=p+1$, hence $\kappa=p+1$ and the rate is $C_T(\ln N)^{-\alpha/(p+1)}$. This applies in particular to quasilinear equations \eqref{PDE_eq} with $\xi$-dependent coefficients $a_\alpha(t,x,\xi)$.
\end{itemize}
The improvement over the naive $(\ln\ln N)^{-\alpha/(n+p)}$ rate is purely a matter of proof: the same particle system \eqref{PDE_particle} is used, but its discrepancy with the PDE solution is estimated by a direct comparison rather than by invoking the general graph-limit theorem as a black box.

\begin{remark}[Pointwise Riemann sums versus cell averages]
The estimate \eqref{PDE_log_rate} cannot in general be improved to an algebraic rate within the pointwise Riemann-sum framework of this paper. The reason is structural: the residual~\eqref{PDE_residual} is a pointwise quadrature error for the regularized kernel, and therefore inherits the Lipschitz constant $\Lip_{x'}(\sigma_\varepsilon)\sim\varepsilon^{-(n+p+1)}$. Since $\sigma_\varepsilon$ approximates a derivative of order $p$ of a Dirac mass, this Lipschitz constant must blow up as $\varepsilon\to 0$.

By contrast, cell-average or projection-reconstruction discretizations estimate the regularized drift as a whole, rather than the pointwise quadrature of the regularized singular kernel. This avoids differentiating the kernel in the quadrature error and may lead to algebraic rates; this approach lies beyond the scope of the present article.
%
%
\end{remark}

\subsection{Some open questions}
We provide hereafter several further comments and open issues.

\paragraph{Closing the hierarchy of moments.}
In Section \ref{sec_hydro}, we have defined the moments of the measure $\mu$ solution of the Vlasov equation. We have seen that the hierarchy of moments is closed if $G$ is linear with respect to $(\xi,\xi')$ but is not closed in general otherwise. Proposition \ref{prop1} is elementary but conceptually useful. A natural open problem is to characterize all mappings $G$ for which the hierarchy closes at a finite level. 

In cases where the hierarchy is not closed (like for fluid equations), we wonder whether it is possible to add a small parameter $\varepsilon$ which would be used to close the hierarchy by taking adequate limits (see \cite{Golse_2016} for similar comments).

\paragraph{Convergence to consensus.}
We have shown in Section \ref{sec_moment_of_order_2} that, surprisingly, the ``temperature" always decreases exponentially, pointwisely, as soon as $S(x)>0$, for the Hegselmann--Krause model. This fact allows one to easily recover (and improve) some known results on convergence to consensus. 
We do not know to what extent this observation may be generalized but we think that it can be used to derive consensus results under weaker assumptions.

More generally, for nonlinear systems enjoying consensus properties (like nonlinear Hegselmann--Krause models) or synchronization properties (like the Kuramoto model), following Section \ref{sec_nu-monokinetic}, we expect that, under appropriate assumptions, any solution $t\mapsto\mu(t)$ of the Vlasov equation \eqref{vlasov} is asymptotically of the form $\mu^\nu_{y(t,\cdot)}$ where $t\mapsto y(t,\cdot)$ is a solution of the CGL equation \eqref{Euler_general}. 

In any case, establishing a consensus result at the level of the Vlasov equation is interesting because it should a priori imply consensus at the level of the (macroscopic) CGL equation and for the (microscopic) particle system.

\paragraph{Improving error estimates.}
In our results, we establish error estimates between solutions of the particle system and a limit equation (CGL or Vlasov) on compact intervals of time, and the error grows exponentially in time, essentially due to a Gronwall argument. Such errors can certainly be much improved for some classes of mappings $G$, maybe under consensus convergence properties, in order to obtain uniform in time estimates.

\paragraph{Use of the measure $\nu$.}
In the existing literature, the measure $\nu$ used in the definition \eqref{def_A_general} of the operator $A$ of the CGL equation \eqref{Euler_general} (graph limit) is always the Lebesgue measure. We have shown in this paper the interest of considering other measures, in particular empirical measures, to obtain (trivial) relationships with the particle system. But more generally, it is certainly of interest to use other measures $\nu$, depending on the context. For example, in social sciences, each agent could have a probability of decision, but several agents could have the same probability of opinion.

\paragraph{Weakly regular (but not singular) mapping $G$.}
It is likely that, in our results and in particular in Theorem \ref{thm_vlasov}, we can weaken the continuity assumption on $G$, provided we consider the solutions of the Vlasov equation in a weaker sense. This is what is done in \cite{JabinPoyatoSoler} for some classes of opinion propagation models: the authors do not assume that there exists a limit mapping $G$ (as we do in \ref{G}) but to take the mean field limit, in a weaker sense, they make another assumption of uniform boundedness on their dynamics. However, at the limit they lose the distinguishability of the particles.

It can be noted that when $G$ is weakly regular (for instance $L^\infty$), we do not have, a priori, an existence and uniqueness result for the particle system. But, following \cite{JabinWang_JFA2016}, we can study the Liouville equation \eqref{liouville}, for which we can have existence (but not uniqueness) for rough vector fields, and then derive the Vlasov equation by taking marginals.

\paragraph{Multiple-wise interactions.}
As mentioned at the end of Section \ref{sec_ex_particle}, we have considered particle systems with pairwise interactions only. The case of \emph{multiple-wise} interactions, in which a group of $m\geq 2$ agents jointly generates a force on any given agent, is the subject of our recent work \cite{PaulRossiTrelat}, where propagation of chaos, the mesoscopic Vlasov equation and the macroscopic limit are studied for fixed $m$, and a further limit $m\to+\infty$ is investigated. To illustrate the framework, we briefly indicate here the case of ``triple-wise'' interactions:
$$
\dot\xi_i(t) = \frac{1}{N^2} \sum_{j,k=1}^N G(t,x_i,x_j,x_k,\xi_i(t),\xi_j(t),\xi_k(t)) .
$$
For such dynamics, the mean field is then formally obtained as
$$
\mathcal{X}[\mu](t,x,\xi) = \int_{\Omega\times\R^d} \int_{\Omega\times\R^d} G(t,x,x',x'',\xi,\xi',\xi'')\, d\mu(x',\xi')\, d\mu(x'',\xi'')  
$$
and the Vlasov equation \eqref{vlasov} remains formally the same.
Note that the above mean field $\mathcal{X}[\mu]$ is now quadratic in $\mu$, which complicates significantly the analysis. A detailed treatment of these dynamics, including the existence and uniqueness of solutions to the Vlasov equation and the analysis of the macroscopic limit, is carried out in \cite{PaulRossiTrelat}.

\paragraph{Similar models.}
There exist in the literature some interesting models that are not covered by our analysis but are nevertheless close to it.
A first example is the more general Hegselmann--Krause model 
$$
\dot\xi^N_i(t) = \frac{1}{N}\sum_{j=1}^N \frac{\phi(\Vert\xi^N_i(t)-\xi^N_j(t)\Vert)}{\displaystyle\frac{1}{N}\sum_{k=1}^N\phi(\Vert\xi^N_i(t)-\xi^N_k(t)\Vert)} (\xi^N_j(t)-\xi^N_i(t)),\qquad i\in\{1,\ldots,N\},
$$
studied in \cite{MotschTadmor_SIREV2014}.
Its graph limit is the CGL equation
$$
\partial_t y(t,x) = \int_{\R^d} \frac{\phi(\Vert y(t,x)-y(t,x')\Vert)}{\displaystyle\int_{\R^d}\phi(\Vert y(t,x)-y(t,x'')\Vert)\, dx''} (y(t,x')-y(t,x))\, dx' 
$$
and the mean field is
$$
\mathcal{X}[\mu](\xi) = \int_{\R^d} \frac{\phi(\Vert \xi-\xi'\Vert)}{\displaystyle\int_{\R^d}\phi(\Vert \xi-\xi''\Vert)\, d\mu(\xi'')} (\xi'-\xi)\, d\mu(\xi') .
$$
The corresponding Vlasov equation is studied in \cite{MotschTadmor_SIREV2014}.

A second example is the Transformers model studied in \cite{GeshkovskiLetrouitPolyanskiyRigollet}
$$
\dot\xi^N_i(t) = \frac{1}{N}\sum_{j=1}^N \frac{\exp\langle Q\xi^N_i(t),K\xi^N_j(t)\rangle}{\displaystyle\frac{1}{N}\sum_{k=1}^N\exp\langle Q\xi^N_i(t),K\xi^N_k(t)\rangle} V\xi^N_j(t),\qquad i\in\{1,\ldots,N\}, 
$$
where $Q$, $K$ and $V$ are matrices. This interacting particle model seems to be particularly relevant in artificial intelligence.
Although the above dynamics cannot be written in the form of the particle system \eqref{particle_system}, it is quite evident that all the theory developed in this paper extends to such cases and that its graph limit is the CGL equation
$$
\partial_t y(t,x) = \int_{\R^d} \frac{\exp\langle Qy(t,x),Ky(t,x')\rangle}{\displaystyle\int_{\R^d}\exp\langle Qy(t,x),Ky(t,x'')\rangle\, dx''} Vy(t,x')\, dx' 
$$
and that the Vlasov equation (studied in \cite[Section 6.3]{GeshkovskiLetrouitPolyanskiyRigollet}) is \eqref{vlasov} with the mean field
$$
\mathcal{X}[\mu](\xi) = \int_{\R^d} \frac{\exp\langle Q\xi,K\xi'\rangle}{\displaystyle\int_{\R^d}\exp\langle Q\xi,K\xi''\rangle\, d\mu(\xi'')} V\xi'\, d\mu(\xi')  .
$$
There exist also other variants of particle systems, involving some delays, or some coupling with other equations (like in the Keller-Segel model). We think that many of them can be covered by slight extensions of the analysis done in this paper.

\paragraph{Stochastic particle systems.}
Throughout this article, we have focused on \emph{deterministic} finite systems of particles. Since many stochastic systems of interacting particles, involving noise, can be relevant in modeling collective behavior, it is of interest to extend the results of this paper to the stochastic context. For example, using Ito calculus, it is proved in \cite{BolleyCanizoCarrillo_M3AS2011} that taking the stochastic mean field limit in a kinetic McKean-Vlasov type finite particle system, involving some Brownian motion, leads to a kinetic Fokker-Planck equation. At a different level, deterministic many-body systems can themselves generate stochastic limit equations: in \cite{BodineauGallagherSaint-Raymond_IM2016}, Brownian motion is rigorously obtained as the macroscopic limit of a deterministic system of hard-spheres. 
The general picture drawn on Figure \ref{fig_embeddings} remains to be investigated in the stochastic setting.

\paragraph{Control at the various scales.} 
From the control theory viewpoint, it is natural to add a control term in the particle system \eqref{particle_system} and, accordingly, in the limit CGL equation \eqref{Euler_general} and in the Vlasov equation \eqref{vlasov}. This was done in \cite{AlbiChoiFornasierKalise_AMO2017, Aydogdu2017, CaponigroPiccoliRossiTrelat_M3AS2017, FornasierSolombrino_COCV2014, PiccoliPouradierDuteilTrelat_SICON2019, PiccoliRossiTrelat_SIMA2015} (just to cite a few). The main objective is then to obtain ``commutative diagrams" in the following sense: if $u$ is a control for a limit equation then one wants that there is an explicit sequence of controls $u^N$ for the particle system, converging to $u$ (this is the easy part); conversely, and much more difficultly, one wants to design controls $u^N$ for the family of particle systems, indexed by $N$, converging to a control $u$ for the limit equation.
This question can be settled in various contexts: exact control, optimal control, stabilization. This is a major challenge.

\paragraph{Numerical consequences.}
All convergence results and error estimates established in this paper show that solutions of particle systems provide good approximations of solutions of the CGL or Vlasov equation.
In numerical analysis, particle methods, or particle-in-cell methods, have been used extensively, in particular to approximate solutions of fluid equations (see, e.g., \cite{Cottet, Raviart}). Of course, these equations involve unbounded operators. But the results announced in Section \ref{sec_approx_PDE} open a new perspective regarding numerical issues, to be explored.

\appendix
\section{Appendix}\label{sec_appendix}

Let $E$ be a Polish space, endowed with a distance $\mathrm{d}_E$.

\subsection{Some general facts on the Wasserstein distance}\label{app_useful_lemmas}

\paragraph{Choice of a distance on $E^k$.}
Let $q\in[1,+\infty]$ be arbitrarily fixed.
Given any $k\in\N^*$, we endow $E^k$ with the $\ell^q$ distance based on $d_E$, defined by
\begin{equation}\label{def_distq_app}
\mathrm{d}^{[q]}_{E^k}(y,y')
= \left\Vert ( \mathrm{d}_E(y_1,y'_1), \ldots, \mathrm{d}_E(y_k,y'_k) ) \right\Vert_{\ell^q}
= \left\{ \begin{array}{ll}
\displaystyle\bigg( \sum_{i=1}^k \mathrm{d}_E(y_i,y'_i)^q \bigg)^{1/q} & \textrm{if}\ 1\leq q<+\infty \\[4mm]
\displaystyle\ \max_{1\leq i\leq k} \mathrm{d}_E(y_i,y'_i) & \textrm{if}\ q=+\infty
\end{array}\right.
\end{equation}
for all $y=(y_1,\ldots,y_k)$ and $y'=(y'_1,\ldots,y'_k)$ in $E^k$. 

Fixing such a choice has an impact on the computation of the Wasserstein distance $W_p$ between two probability measures on $E^k$. Indeed, this means that the distance \eqref{def_distq_app} is used in the definition \eqref{def_Wp} of $W_p$, and that, in the definition \eqref{def_W1} of $W_1$, the Lipschitz constants must be computed with the distance \eqref{def_distq_app}. 
The lemma below is thus important to compute Lipschitz constants.

\begin{lemma}\label{lem_lipp}
Let $f\in\Lip(E^k)$. Then, for any $y_2,\ldots,y_k\in E$, the mapping $y_1\mapsto f(y_1,y_2,\ldots,y_k)$ is Lipschitz, of Lipschitz constant less than $\Lip(f)$. We set $\Lip_{y_1}(f) = \max\{ \Lip(f(\cdot,y_2,\ldots,y_k))\ \mid\ y_2,\ldots,y_k\in E\}$. All other $\Lip_{y_i}(f)$ are defined similarly, for $i=2,\ldots,k$.
We have
\begin{equation*}
\Lip(f) = 
\Vert ( \Lip_{y_1}(f), \ldots, \Lip_{y_k}(f) ) \Vert_{\ell^{q'}}
= \left\{ \begin{array}{ll}
\displaystyle\bigg( \sum_{i=1}^k \Lip_{y_i}(f)^{q'} \bigg)^{1/q'} & \textrm{if}\ q'<+\infty \\[4mm]
\displaystyle\ \max_{1\leq i\leq k} \Lip_{y_i}(f) & \textrm{if}\ q'=+\infty
\end{array}\right.
\end{equation*}
where $q'\in[1,+\infty]$ is defined by $\frac{1}{q}+\frac{1}{q'}=1$. 
\end{lemma}

\begin{proof}
It suffices to write
\begin{equation*}
\begin{split}
& \vert f(y_1,\ldots,y_k)-f(y'_1,\ldots,y'_k)\vert \\
\leq\ & \vert f(y_1,y_2\ldots,y_k)-f(y'_1,y_2,\ldots,y_k)\vert + \cdots + \vert f(y'_1,\ldots,y'_{k-1},y_k)-f(y'_1,\ldots,y'_{k-1},y'_k)\vert \\
\leq\ & \sum_{i=1}^k \Lip_{y_i}(f) \, \mathrm{d}_E(y_i,y'_i)
\end{split}
\end{equation*}
and to use the H\"older inequality.
\end{proof}

\begin{remark}\label{rem_choicedist}
The choice of a distance $\mathrm{d}^{[q]}_{E^k}$ on the tensor product $E^k$ (i.e., the choice of $q\in[1,+\infty]$) is far from being insignificant because, although all norms are equivalent in $E^k$, comparing them gives constants depending on $k$. 
The choice thus becomes particularly meaningful when $k$ is large.

Another remark is that the definition \eqref{def_distq_app} is based on the usual $\ell^q$ norm, for $q\in[1,+\infty]$. Other choices are possible, but in order to preserve many of the statements that follow, the convexity of the norm is important.
\end{remark}

\paragraph{Notation $W_p^{[q]}$.}
For all $p,q\in[1,+\infty]$, following Remark \ref{rem_choicedist}, hereafter we denote by $W_p^{[q]}$ the Wasserstein distance $W_p$ on $\mathcal{P}(E^k)$ (defined by \eqref{def_Wp}) with respect to the distance $\mathrm{d}^{[q]}_{E^k}$ on $E^k$.

It follows from the usual inequalities for $\ell^q$ norms in $\R^k$ that $q\mapsto \mathrm{d}^{[q]}_{E^k}$ is decreasing and
\begin{equation}\label{classical_ellq}
1\leq q_1\leq q_2\leq +\infty \ \Rightarrow\ \mathrm{d}^{[q_2]}_{E^k} \leq \mathrm{d}^{[q_1]}_{E^k} \leq k^{\frac{1}{q_1}-\frac{1}{q_2}} \, \mathrm{d}^{[q_2]}_{E^k} 
\end{equation}
and thus
\begin{equation}\label{inegWpq}
1\leq q_1\leq q_2\leq +\infty \ \Rightarrow\ W_p^{[q_2]} \leq W_p^{[q_1]} \leq k^{\frac{1}{q_1}-\frac{1}{q_2}} \, W_p^{[q_2]}  
\end{equation}
for any $p\in[1,+\infty]$.
These inequalities complement \eqref{inegWp1Wp2}.
For $p$ fixed, in the family of distances $W_p^{[q]}$, for $q\in[1,+\infty]$, the $\ell^1$ distance $W_p^{[1]}$ is the weakest one. 
This is an important point because, in the existing literature, the $\ell^2$ distance $W_p^{[2]}$ is most often used, whereas in this work, the use of $q=1$ is crucial in several places.

\medskip

In all subsections hereafter, we fix an arbitrary $p\in[1,+\infty)$. The case $p=+\infty$ is obtained by taking the limit when it makes sense. We also fix an arbitrary $q\in[1,+\infty]$.

\subsubsection{Convexity}\label{app_conv}

\begin{lemma}[$(W_p)^p$ is convex]\label{lem_Wp_convex}
Given any $\mu_1,\mu_2,\mu_1',\mu_2' \in\mathcal{P}(E)$ and any $\lambda\in[0,1]$, we have
$$
W_p(\lambda\mu_1+(1-\lambda)\mu_2,\lambda\mu_1'+(1-\lambda)\mu_2')^p \leq \lambda W_p(\mu_1,\mu_1')^p + (1-\lambda) W_p(\mu_2,\mu_2')^p.
$$
\end{lemma}

\begin{proof}
This result is a particular case of \cite[Part I, Chapter 4, Theorem 4.8]{Villani_2009}. 
Let $\Pi_i$ be an optimal coupling between $\mu_i$ and $\mu_i'$, for $i=1,2$. Then $\Pi=\lambda\Pi_1+(1-\lambda)\Pi_2$ couples $\lambda\mu_1+(1-\lambda)\mu_2$ and $\lambda\mu_1'+(1-\lambda)\mu_2'$ (maybe not optimally). Hence
\begin{multline*}
W_p(\lambda\mu_1+(1-\lambda)\mu_2,\lambda\mu_1'+(1-\lambda)\mu_2')^p \leq \int_E \mathrm{d}_E(x,x')^p\, d\Pi(x,x') \\
= \lambda \int_E \mathrm{d}_E(x,x')^p\, d\Pi_1(x,x') + (1-\lambda) \int_E \mathrm{d}_E(x,x')^p\, d\Pi_2(x,x')
= \lambda W_p(\mu_1,\mu_1')^p + (1-\lambda) W_p(\mu_2,\mu_2')^p
\end{multline*}
and the lemma follows.
\end{proof}

\begin{lemma}\label{lem_Wp_eps}
Let $\mu_1,\mu_2,\beta \in\mathcal{P}(E)$ and let $\varepsilon\in(0,1]$ be such that $\mu_1 = (1+\varepsilon) \mu_2 - \varepsilon \beta$. Then
$$
W_p(\mu_1,\mu_2) \leq \varepsilon^{1/p} W_p(\mu_1,\beta)
$$
and, assuming that $\varepsilon<1$, 
$$
W_p(\mu_1,\mu_2) \leq \frac{\varepsilon^{1/p}}{1-\varepsilon^{1/p}} W_p(\mu_2,\beta) .
$$
In the particular case $p=1$, we have $W_1(\mu_1,\mu_2) = \varepsilon W_1(\mu_2,\beta)$.
\end{lemma}

\begin{proof}
We have $\mu_2=\frac{1}{1+\varepsilon}\mu_1+\frac{\varepsilon}{1+\varepsilon}\beta$ (convex combination), and applying Lemma \ref{lem_Wp_convex} we get $W_p(\mu_1,\mu_2)^p \leq \frac{\varepsilon}{1+\varepsilon} W_p(\mu_1,\beta)^p \leq \varepsilon W_p(\mu_1,\beta)^p$, and the first inequality follows.
The second inequality is obtained by using the triangular inequality $W_p(\mu_1,\beta) \leq W_p(\mu_1,\mu_2) + W_p(\mu_2,\beta)$.
When $p=1$, given any $f\in \mathscr{C}^0_c(E)$, we have $\int_{E} f\, d(\mu_1-\mu_2) = \varepsilon \int_{E} f\, d(\mu_2-\beta)$, and taking (in two steps) the supremum over all $f$ such that $\Lip(f)\leq 1$, 
we get $W_1(\mu_1,\mu_2) = \varepsilon W_1(\mu_2,\beta)$.
\end{proof}

\subsubsection{Symmetrization}\label{app_symm}
Let $N\in\N^*$ be arbitrary. 
Given any $\mu\in\mathcal{P}(E^N)$, the measure $\mu^s\in\mathcal{P}(E^N)$, called the symmetrization under permutations of $\mu$, is defined by
\begin{equation}\label{def_compacte_symmetrization_E}
\mu^s=\frac{1}{N!}\sum_{\sigma\in\mathfrak{S}_N}\sigma_*\mu
\end{equation}
where the measure $\sigma_*\mu$ is defined by 
$\langle\sigma_*\mu,f\rangle=\langle\mu,\sigma^*f\rangle$ and $(\sigma^*f)(y)=f(\sigma\cdot y)$, with $\sigma\cdot y=(y_{\sigma(1)},\ldots,y_{\sigma(N)})$ for every $y\in E^N$ and for every $\sigma\in\mathfrak{S}_N$, where $\mathfrak{S}_N$ is the group of permutations of $N$ elements. Here, $\langle\ ,\ \rangle$ is the duality bracket.
Equivalently,
$$
\int_{E^N} f(y)\, d\mu^s(y) = 
\frac{1}{N!} \sum_{\sigma\in\mathfrak{S}_N} \int_{E^N} f(\sigma\cdot y)\, d\mu(y)
\qquad \forall f\in \mathscr{C}^0_c(E^N) .
$$

\begin{lemma}\label{lem_Wp_s}
Given any $\mu_1,\mu_2 \in \mathcal{P}(E^N)$, 
we have
$$
W_p^{[q]}(\mu_1^s,\mu_2^s)\leq W_p^{[q]}(\mu_1,\mu_2) .
$$
\end{lemma}

In this lemma, the Wasserstein distance $W_p$ is computed with respect to the $\ell^q$ distance $\mathrm{d}^{[q]}_{E^N}$.

\begin{proof}
This follows from Lemma \ref{lem_Wp_convex}, since $\mu^s$ is written as the convex combination \eqref{def_compacte_symmetrization_E}, noting that $W_p^{[q]}(\sigma_*\mu_1,\sigma_*\mu_2)=W_p^{[q]}(\mu_1,\mu_2)$ for any $\sigma\in\mathfrak{S}_N$ because the distance $\mathrm{d}^{[q]}_{E^N}$ defined by \eqref{def_distq_app} is itself symmetric and because, for any $\Pi$ coupling $\mu_1$ and $\mu_2$ and for any $\sigma\in\mathfrak{S}_N$, $(\sigma\otimes\sigma)_*\Pi$ couples $\sigma_*\mu_1$ and $\sigma_*\mu_2$.
\end{proof}

\subsubsection{Marginals}\label{app_marginals}
Let $N\in\N^*$ be arbitrary. 
Given any $\mu\in\mathcal{P}(E^N)$ and any $k\in\{1,\ldots,N\}$, the $k^\textrm{th}$-order marginal $\mu_{N:k}\in\mathcal{P}(E^k)$ of $\mu$ is the image of $\mu$ under the canonical projection $\pi_k:E^N=E^k\times E^{N-k}\rightarrow E^k$.

\begin{lemma}\label{lem_Wp_marginal}
Given any $\mu_1,\mu_2 \in \mathcal{P}(E^N)$ and any $k\in\{1,\ldots,N\}$, we have
\begin{equation}\label{lem_Wp_marginal_ineq}
W_p^{[q]}((\mu_1)_{N:k},(\mu_2)_{N:k}) \leq W_p^{[q]}(\mu_1,\mu_2)  .
\end{equation}
\end{lemma}

The Wasserstein distance at the left-hand (resp., right-hand) side of \eqref{lem_Wp_marginal_ineq} is computed with respect to the $\ell^q$ distance $\mathrm{d}^{[q]}_{E^k}$ (resp., $\mathrm{d}^{[q]}_{E^N}$).
When $p\leq q$ and $\mu_1$ and $\mu_2$ are symmetric, a stronger estimate is given in Lemma \ref{lem_Wp_marginal_symm} (Appendix \ref{app_marginal_symm}).

\begin{proof}
Let $\Pi$ be an optimal coupling between $\mu_1$ and $\mu_2$. Then, obviously, $(\pi_k\otimes\pi_k)_*\Pi$ couples (maybe not optimally) $(\pi_k)_*\mu_1=(\mu_1)_{N:k}$ and $(\pi_k)_*\mu_2=(\mu_2)_{N:k}$. Therefore
\begin{equation*}
\begin{split}
W_p^{[q]}((\mu_1)_{N:k},(\mu_2)_{N:k})^p &\leq \int_{E^k} \mathrm{d}^{[q]}_{E^k}((y_1,\ldots,y_k),(y'_1,\ldots,y'_k))^p\, d(\pi_k\otimes\pi_k)_*\Pi((y_1,\ldots,y_k),(y'_1,\ldots,y'_k)) \\
&\leq \int_{E^N} \mathrm{d}^{[q]}_{E^k}(\pi_k(y),\pi_k(y'))^p\, d\Pi(y,y') \\
& \leq \int_{E^N} \mathrm{d}^{[q]}_{E^N}(y,y')^p\, d\Pi(y,y') = W_p^{[q]}(\mu_1,\mu_2)^p
\end{split}
\end{equation*}
where we have used that $ \mathrm{d}^{[q]}_{E^k}(\pi_k(y),\pi_k(y')) \leq  \mathrm{d}^{[q]}_{E^N}(y,y')$.
\end{proof}

\subsubsection{Tensor product}\label{app_tensor}

Let $k\in\N^*$. For every $i\in\{1,\ldots,k\}$, let $E_i$ be a Polish space, endowed with a distance $\mathrm{d}_{E_i}$. We endow the product space $E_1\times\cdots\times E_k$ with the distance
\begin{equation}\label{def_dist_E1Ek}
\mathrm{d}^{[q]}_{E_1\times\cdots\times E_k}(y,y') = \left\Vert ( \mathrm{d}_{E_1}(y_1,y'_1), \ldots, \mathrm{d}_{E_k}(y_k,y'_k) ) \right\Vert_{\ell^q}
= \left\{ \begin{array}{ll}
\displaystyle\bigg( \sum_{i=1}^k \mathrm{d}_{E_i}(y_i,y'_i)^q \bigg)^{1/q} & \textrm{if}\ 1\leq q<+\infty \\[4mm]
\displaystyle\ \max_{1\leq i\leq k} \mathrm{d}_{E_i}(y_i,y'_i) & \textrm{if}\ q=+\infty
\end{array}\right.
\end{equation}
for all $y=(y_1,\ldots,y_k)$, $y'=(y'_1,\ldots,y'_k)\in E_1\times\cdots\times E_k$.

\begin{lemma}\label{lem_Wp_tensor} 
Given any $\mu_1,\mu'_1\in\mathcal{P}(E_1)$, $\ldots$, $\mu_k,\mu'_k\in\mathcal{P}(E_k)$, we have, for every $j\in\{1,\ldots,k\}$,
\begin{equation}\label{Wp_tensor_k}
W_p(\mu_j,\mu'_j) \leq W_p^{[q]}\left( \overset{k}{\underset{i=1}{\otimes}} \mu_i \, ,  \overset{k}{\underset{i=1}{\otimes}} \mu'_i \right) \leq \max\big( k^{\frac{1}{q}-\frac{1}{p}}, 1 \big)  \bigg( \sum_{i=1}^k W_p(\mu_i,\mu'_i)^p \bigg)^{1/p} 
\end{equation}
and the right-hand side inequality in \eqref{Wp_tensor_k} is an equality if $p=q$.

Taking $E_i=E$, $\mathrm{d}_{E_i}=\mathrm{d}_E$, $\mu_i=\mu$ and $\mu'_i=\mu'$ for every $i\in\{1,\ldots,k\}$, we have the slightly stronger inequality
\begin{equation}\label{lem_Wp_tensor_ineq}
W_p^{[q]}(\mu^{\otimes k},(\mu')^{\otimes k}) \leq 
k^{1/q} \, W_p(\mu,\mu') 
\end{equation}
and the inequality is an equality if $p=q$.
\end{lemma}

See \cite{MariucciReiss_EJS2018} for Lemma \ref{lem_Wp_tensor}.

The Wasserstein distance $W_p$ at the left-hand side of \eqref{Wp_tensor_k} is computed with respect to the distance $d_{E_j}$.
The Wasserstein distance $W_p^{[q]}$ in the middle of \eqref{Wp_tensor_k} is computed with respect to the distance $\mathrm{d}^{[q]}_{E_1\times\cdots\times E_k}$ defined by \eqref{def_dist_E1Ek}.

The Wasserstein distance $W_p^{[q]}$ at the left-hand side of \eqref{lem_Wp_tensor_ineq} is computed with respect to the distance $\mathrm{d}^{[q]}_{E^k}$ defined by \eqref{def_distq_app}. Recall that $q\in[1,+\infty]$ has been chosen arbitrarily to define this distance. At the right-hand side of \eqref{lem_Wp_tensor_ineq}, if $q=+\infty$ then $k^{1/q}=1$.

\begin{remark}\label{rem_Wp_tensor}
As a particular case of \eqref{Wp_tensor_k}, taking $k=2$ and $\mu_2=\mu'_2=\mu$, we have 
$$
W_p(\mu_1,\mu'_1)\leq W_p^{[q]}(\mu_1\otimes\mu,\mu'_1\otimes\mu)=W_p^{[q]}(\mu\otimes\mu_1,\mu\otimes\mu'_1) \leq \max\big( 2^{\frac{1}{q}-\frac{1}{p}}, 1 \big) W_p(\mu_1,\mu'_1) .
$$
In particular, if $p\leq q$ then $W_p(\mu_1,\mu'_1)=W_p^{[q]}(\mu_1\otimes\mu,\mu'_1\otimes\mu)=W_p^{[q]}(\mu\otimes\mu_1,\mu\otimes\mu'_1)$.
\end{remark}

\begin{proof}
We have $W_p(\mu_j,\mu'_j) \leq W_p^{[q]}\Big( \overset{k}{\underset{i=1}{\otimes}} \mu_i \, ,  \overset{k}{\underset{i=1}{\otimes}} \mu'_i\Big)$ for every $i\in\{1,\ldots,k\}$: this is proved like in Lemma \ref{lem_Wp_marginal} because $\mu_j$ is the marginal on $E_j$ of the measure $\overset{k}{\underset{i=1}{\otimes}} \mu_i$ on $E$, and similarly for $\mu'_j$. Therefore the left-hand side inequality in \eqref{Wp_tensor_k} follows.

Let us now establish the right-hand side inequality in \eqref{Wp_tensor_k}, for $q<+\infty$.
For every $i\in\{1,\ldots,k\}$, let $\Pi_i$ be an optimal coupling between $\mu_i$ and $\mu'_i$. Then, obviously, $\Pi = \overset{k}{\underset{i=1}{\otimes}} \Pi_i$ couples (maybe not optimally) $\overset{k}{\underset{i=1}{\otimes}} \mu_i$ and $\overset{k}{\underset{i=1}{\otimes}} \mu'_i$. Therefore
$$
W_p^{[q]}\left( \overset{k}{\underset{i=1}{\otimes}} \mu_i \, ,  \overset{k}{\underset{i=1}{\otimes}} \mu'_i \right)^p 
\leq \int_{E_1\times E_1} \cdots \int_{E_k\times E_k} \bigg( \sum_{i=1}^k \mathrm{d}_{E_i}(y_i,y'_i)^q \bigg)^{p/q}\, d\Pi_k(y_k,y'_k)\, \cdots\, d\Pi_1(y_1,y'_1) .
$$
If $p\geq q$, using the convexity inequality $(\vert a_1\vert+\cdots+\vert a_k\vert)^r \leq k^{r-1}\left( \vert a_1\vert^r+\cdots+\vert a_k\vert^r \right)$ for $r\geq 1$ (with equality for $r=1$), coming from \eqref{classical_ellq}, we obtain
$$
W_p^{[q]}\left( \overset{k}{\underset{i=1}{\otimes}} \mu_i \, ,   \overset{k}{\underset{i=1}{\otimes}} \mu'_i \right)^p 
\leq k^{\frac{p}{q}-1} \sum_{i=1}^k W_p(\mu_i,\mu'_i)^p  
$$
and the inequality is an equality if $p=q$ because in this case $\Pi$ is an optimal coupling. 
If $p\leq q$, using the inequality $(\vert a_1\vert+\cdots+\vert a_k\vert)^{1/r} \leq \vert a_1\vert^{1/r}+\cdots+\vert a_k\vert^{1/r}$ for $r\geq 1$ (coming from \eqref{classical_ellq}), we obtain
$$
W_p^{[q]}\left( \overset{k}{\underset{i=1}{\otimes}} \mu_i \, ,   \overset{k}{\underset{i=1}{\otimes}} \mu'_i \right)^p 
\leq \sum_{i=1}^k W_p(\mu_i,\mu'_i)^p  .
$$
All in all, we have established \eqref{Wp_tensor_k}.

To prove \eqref{lem_Wp_tensor_ineq}, using the definition \eqref{def_Wp_randomlaws} of $W_p$, we note that
$$
W_p^{[q]}(\mu^{\otimes k},(\mu')^{\otimes k})^p
\leq \mathbb{E} \bigg( \sum_{i=1}^k \mathrm{d}_E(Y,Y')^q \bigg)^{p/q} %
= k^{p/q} \, \mathbb{E} \mathrm{d}_E(Y,Y')^p 
= k^{p/q}\, W_p(\mu,\mu')^p
$$
where $Y$ and $Y'$ are random variables (with values in $E$) of laws $\mu$ and $\mu'$, such that $W_p(\mu,\mu')^p = \mathbb{E} \mathrm{d}_E(Y,Y')^p$. 
\end{proof}

\subsubsection{Diameter of the support}\label{app_supp}


\begin{lemma}\label{lem_Wp_supp}
Given any $\mu_1,\mu_2\in\mathcal{P}_c(E)$, we have
$$ 
W_p(\mu_1,\mu_2) \leq \diam_E(\supp(\mu_1)\cup\supp(\mu_2)) 
= \max\{ \mathrm{d}_E(y,y')\ \mid\ y,y'\in \supp(\mu_1)\cup\supp(\mu_2) \} .
$$
\end{lemma}

\begin{proof}
By \eqref{def_Wp}, since $W_p(\mu_1,\mu_2)^p$ is the infimum of $\int_{E^2} \mathrm{d}_E(y,y')^p \, d\Pi(y,y')$ over all probability measures $\Pi$ on $E^2$ coupling $\mu_1$ and $\mu_2$, 
we have 
$W_p(\mu_1,\mu_2) \leq \max \{ \mathrm{d}_E(y_1,y_2)\ \mid\ y_1\in\supp(\mu_1), y_2\in\supp(\mu_2) \}$, and the result follows.
\end{proof}

\subsubsection{Propagation}\label{app_propag}
In this section, we assume that $E$ is a Banach space, endowed with a norm $\Vert\cdot\Vert_E$.
Let also $\Lambda$ (space of parameters) be a Polish space, endowed with a distance $\mathrm{d}_\Lambda$.
The space $\Lambda\times E$ is endowed with the distance $\mathrm{d}_{\Lambda\times E}=\mathrm{d}_\Lambda+\mathrm{d}_E$, where $\mathrm{d}_E$ is the distance on $E$ induced by the norm $\Vert\cdot\Vert_E$.

\begin{lemma}\label{lem_propagflow}
For $i=1,2$, let $Y^i(t,\lambda,\cdot)$ be a continuous time-varying vector field on $E$, depending on the parameter $\lambda\in\Lambda$, locally Lipschitz with respect to $(\lambda,y)\in\Lambda\times E$ uniformly with respect to $t$ on any compact interval, generating a flow $(\Phi^i(t,t_0,\lambda,\cdot))_{t\in\R}$ (assumed to be well defined for every $t\in\R$) for any $t_0\in\R$, that is,
\begin{equation*}
\begin{split}
\partial_t\Phi^i(t,t_0,\lambda,y) &= Y^i(t,\lambda,\Phi^i(t,t_0,\lambda,y)) \\
\Phi^i(t_0,t_0,\lambda,y) &= y
\end{split}
\end{equation*}
for all $t,t_0\in\R$, $y\in E$ and $\lambda\in\Lambda$.
Given any $t_0\in\R$ and any $\mu^1(t_0),\mu^2(t_0) \in\mathcal{P}_c(\Lambda\times E)$, we set $\mu^i_t=\mu^i(t)=\Phi^i(t,t_0)_*\mu^i(t_0)$ for every $t\geq t_0$, for $i=1,2$;
this notation means, denoting by $\nu^i$ the (constant in time) marginal of $\mu^i(t)$ on $\Lambda$ and disintegrating $\mu^i_t = \int_\Lambda \mu^i_{t,\lambda}\, d\nu^i(\lambda)$, that $\mu^i_{t,\lambda} = \Phi^i(t,t_0,\lambda,\cdot)_*\mu^i(t_0)$ for $\nu^i$-almost every $\lambda\in\Lambda$.
For every $p\in[1,+\infty)$, we have
\begin{equation}\label{propagflow1} 
W_p(\mu^1(t), \mu^2(t)) 
\leq e^{(t-t_0)L([t_0,t])} W_p(\mu^1(t_0),\mu^2(t_0)) + M([t_0,t]) \frac{e^{(t-t_0)L([t_0,t])}-1}{L([t_0,t])}  
\end{equation}
for every $t\geq t_0$,
where\footnote{Note that $S(t)$ is compact and that $\Phi^i(t,t_0,\supp(\mu^i(t_0)))=\supp(\mu^i(t))$.}
\begin{equation}\label{defL}
L([t_0,t]) = \max_{t_0\leq\tau\leq t} \Lip\left(Y^1(\tau,\cdot,\cdot)_{\vert S(\tau)}\right) ,
\end{equation}
$$
S(t) = (\supp(\nu^1)\cup\supp(\nu^2)) \times \Phi^1(t,t_0,\supp(\mu^1(t_0))\cup\supp(\mu^2(t_0))) \ \cup\ \supp(\mu^2(t)) ,
$$
\begin{equation}\label{defMmax}
M([t_0,t]) = \max \{ \Vert Y^1(\tau,\lambda,y)-Y^2(\tau,\lambda,y)\Vert_E \ \mid\ t_0\leq\tau\leq t,\ (\lambda,y)\in\supp(\mu^2(\tau)) \} .
\end{equation}
Alternatively, the second term at the right-hand side of \eqref{propagflow1} can be replaced by
\begin{equation}\label{propagflow2} 
M_p([t_0,t]) (t-t_0)^{1/p} \left( \frac{e^{(t-t_0)p' L([t_0,t])}-1}{p' L([t_0,t])} \right)^{1/p'}
\end{equation}
where $\frac{1}{p}+\frac{1}{p'}=1$ and
\begin{equation}\label{defMp}
M_p([t_0,t]) = \max_{t_0\leq\tau\leq t} \left( \int_{\Lambda\times E} \Vert Y^1(\tau,\lambda,y)-Y^2(\tau,\lambda,y)\Vert_E^p \, d\mu^2_\tau(\lambda,y) \right)^{1/p}  .
\end{equation}
\end{lemma}

\noindent Some remarks are in order:
\begin{itemize}[label=-,leftmargin=3.5mm,parsep=0cm,itemsep=0cm,topsep=0cm]
\item In \eqref{propagflow1} (and in \eqref{propagflow2}), it is understood that if $L([t_0,t])=0$ then $\frac{e^{(t-t_0)L([t_0,t])}-1}{L([t_0,t])}$ is replaced by $t-t_0$.
Lemma \ref{lem_propagflow} extends \cite[Proposition 4]{PiccoliRossi_ARMA2016} to the case with parameters and to the local Lipschitz case; also, 
the alternative (not usual) estimate with \eqref{propagflow2} is useful to derive some results of this paper.

\item If $Y^1=Y^2$ then $M(\cdot)=0$.

\item When $t_0=0$, we denote $\Phi^i(t,\lambda,y)=\Phi^i(t,0,\lambda,y)$, $L(t)=L([0,t])$ and $M(t)=M([0,t])$.

\item
Note also that, in Lemma \ref{lem_propagflow}, only the first vector field $Y^1$ needs to be locally Lipschitz; for $Y^2$, it suffices that \eqref{defMp} is well defined and that the flow $\Phi^2$ is well defined.
\end{itemize}

\begin{proof}
Given any $(\lambda_1,y_1)\in\supp(\mu^1(t_0))$ and $(\lambda_2,y_2) \in\supp(\mu^2(t_0))$, using \eqref{defL} we have 
\begin{equation*}
\begin{split}
& \partial_t \left\Vert \Phi^1(t,t_0,\lambda_1,y_1) - \Phi^1(t,t_0,\lambda_2,y_2) \right\Vert_E \\
\leq\ & \Vert Y^1(t, \lambda_1,\Phi^1(t,t_0,\lambda_1,y_1)) - Y^1(t, \lambda_2,\Phi^1(t,t_0,\lambda_2,y_2)) \Vert_E \\
\leq\ & L([t_0,t]) \left( \mathrm{d}_\Lambda(\lambda_1,\lambda_2) + \left\Vert \Phi^1(t,t_0,\lambda_1,y_1) - \Phi^1(t,t_0,\lambda_2,y_2) \right\Vert_E \right)
\end{split}
\end{equation*}
because $\Phi^1(t,t_0,\lambda_1,y_1)\in\Phi^1(t,t_0,\supp(\mu^1(t_0)))$ and $\Phi^1(t,t_0,\lambda_2,y_2))\in\Phi^1(t,t_0,\supp(\mu^2(t_0)))$ (this motivates the definition of $S(t)$),
and by integration we get that 
\begin{equation}\label{Phii_Lip}
\mathrm{d}_\Lambda(\lambda_1,\lambda_2) + \left\Vert \Phi^1(t,t_0,\lambda_1,y_1) - \Phi^1(t,t_0,\lambda_2,y_2) \right\Vert_E
\leq e^{(t-t_0)L([t_0,t])} \left( \mathrm{d}_\Lambda(\lambda_1,\lambda_2) + \Vert y_1-y_2\Vert_E \right)
\end{equation}
for every $t\geq t_0$, by monotonicity of $t\mapsto L([t_0,t])$. 

Taking an optimal coupling $\Pi_{t_0}\in\mathcal{P}((\Lambda\times E)^2)$ between $\mu^1(t_0)$ and $\mu^2(t_0)$, the probability measure $\Pi_t = (\Phi^1(t,t_0)\otimes\Phi^2(t,t_0))_* \Pi_{t_0}$ couples (maybe not optimally) $\mu^1(t)$ with $\mu^2(t)$.\footnote{Indeed, denoting by $\pi_i$ the projection of $(\Lambda\times E)^2$ onto the $i^\textrm{th}$-copy of $\Lambda\times E$, we have $\pi_i \circ (\Phi^1\otimes\Phi^2) = \Phi^i\circ\pi_i$.}
Therefore, using the definition \eqref{def_Wp} of $W_p$, 
\begin{equation*}
\begin{split}
&\ W_p(\mu^1(t), \mu^2(t))^p \\
\leq&\  \int_{(\Lambda\times E)^2} \left( \mathrm{d}_\Lambda(\lambda_1,\lambda_2) + \Vert y_1-y_2\Vert_E \right)^{p} d\Pi_t(\lambda_1,y_1,\lambda_2,y_2) \\
=&\  \int_{(\Lambda\times E)^2} \left( \mathrm{d}_\Lambda(\lambda_1,\lambda_2) + \Vert \Phi^1(t,t_0,\lambda_1,y_1)-\Phi^2(t,t_0,\lambda_2,y_2)\Vert_E \right)^p  d\Pi_{t_0}(\lambda_1,y_1,\lambda_2,y_2) \\
\leq&\  \int_{(\Lambda\times E)^2} \Big( \mathrm{d}_\Lambda(\lambda_1,\lambda_2) 
+ \Vert \Phi^1(t,t_0,\lambda_1,y_1)-\Phi^1(t,t_0,\lambda_2,y_2)\Vert_E \\
& \qquad\qquad\qquad\qquad\qquad
+ \Vert \Phi^1(t,t_0,\lambda_2,y_2)-\Phi^2(t,t_0,\lambda_2,y_2)\Vert_E
\Big)^p d\Pi_{t_0}(\lambda_1,y_1,\lambda_2,y_2) 
\end{split}
\end{equation*}
and thus, and using the triangular inequality in $L^p$, we get
\begin{equation}\label{ineg_coupl}
\begin{split}
& W_p(\mu^1(t), \mu^2(t)) \\
\leq& \ \bigg( \int_{(\Lambda\times E)^2} \left( \mathrm{d}_\Lambda(\lambda_1,\lambda_2) + \Vert \Phi^1(t,t_0,\lambda_1,y_1)-\Phi^1(t,t_0,\lambda_2,y_2)\Vert_E \right)^p d\Pi_{t_0}(\lambda_1,y_1,\lambda_2,y_2) \bigg)^{1/p} \\
&\qquad + \bigg( \int_{(\Lambda\times E)^2} \Vert \Phi^1(t,t_0,\lambda_2,y_2)-\Phi^2(t,t_0,\lambda_2,y_2)\Vert_E^p \, d\Pi_{t_0}(\lambda_1,y_1,\lambda_2,y_2) \bigg)^{1/p}
\end{split}
\end{equation}
Using \eqref{Phii_Lip}, the first term of the sum at the right-hand side of \eqref{ineg_coupl} is less than or equal to
\begin{multline*}
e^{(t-t_0)L([t_0,t])} \bigg( \int_{(\Lambda\times E)^2} \left( \mathrm{d}_\Lambda(\lambda_1,\lambda_2) + \Vert y_1-y_2\Vert_E \right)^p d\Pi_{t_0}(\lambda_1,y_1,\lambda_2,y_2) \bigg)^{1/p} \\
= e^{(t-t_0)L([t_0,t])} W_p(\mu^1(t_0),\mu^2(t_0)) ,
\end{multline*}
the latter equality being because $\Pi_{t_0}$ is an optimal coupling between $\mu^1(t_0)$ and $\mu^2(t_0)$. 

To treat the second term, we first observe that, for $(\lambda,y)\in \supp(\mu^2(t_0))$,
\begin{equation*}
\begin{split}
\partial_t \Vert \Phi^1(t,t_0,\lambda,y)-\Phi^2(t,t_0,\lambda,y) \Vert_E 
\leq\ & \Vert Y^1(t,\lambda,\Phi^1(t,t_0,\lambda,y)) - Y^1(t,\lambda,\Phi^2(t,t_0,\lambda,y)) \Vert_E \\
& + \Vert Y^1(t,\lambda,\Phi^2(t,t_0,\lambda,y)) - Y^2(t,\lambda,\Phi^2(t,t_0,\lambda,y)) \Vert_E \\
\leq\ & L([t_0,t]) \Vert \Phi^1(t,t_0,\lambda,y)-\Phi^2(t,t_0,\lambda,y) \Vert_E \\
& + \Vert Y^1(t,\lambda,\Phi^2(t,t_0,\lambda,y)) - Y^2(t,\lambda,\Phi^2(t,t_0,\lambda,y)) \Vert_E 
\end{split}
\end{equation*}
where we have used \eqref{defL}, noting that $(\lambda,\Phi^1(t,t_0,\lambda,y))\in S(t)$ and $(\lambda,\Phi^2(t,t_0,\lambda,y))\in S(t)$,
and thus, using the Gronwall lemma and the fact that $\tau\mapsto L([t_0,\tau])$ is nondecreasing,
\begin{multline}\label{ox_18:44}
\Vert \Phi^1(t,t_0,\lambda,y)-\Phi^2(t,t_0,\lambda,y) \Vert_E \\
\leq \int_{t_0}^t e^{(t-\tau) L([t_0,t])} \Vert Y^1(\tau,\lambda,\Phi^2(\tau,t_0,\lambda,y)) - Y^2(\tau,\lambda,\Phi^2(\tau,t_0,\lambda,y)) \Vert_E \, d\tau  .
\end{multline}
Using the definition \eqref{defMmax} of $M([t_0,t])$ and the fact that $\Phi^2(\tau,t_0,\supp(\mu^2(t_0))) = \supp(\mu^2(\tau))$, we get
$$
\Vert \Phi^1(t,t_0,\lambda,y)-\Phi^2(t,t_0,\lambda,y) \Vert_E 
\leq M([t_0,t]) \frac{e^{(t-t_0)L([t_0,t])}-1}{L([t_0,t])} .
$$
%
Therefore, 
the second term of the sum at the right-hand side of \eqref{ineg_coupl} is 
estimated by
\begin{multline*}
\bigg( \int_{(\Lambda\times E)^2} \Vert \Phi^1(t,t_0,\lambda_2,y_2)-\Phi^2(t,t_0,\lambda_2,y_2)\Vert_E^p \, d\Pi_{t_0}(\lambda_1,y_1,\lambda_2,y_2) \bigg)^{1/p} \\
= \bigg( \int_{\Lambda\times E} \Vert \Phi^1(t,t_0,\lambda,y)-\Phi^2(t,t_0,\lambda,y)\Vert_E^p \, d\mu^2_{t_0}(\lambda,y) \bigg)^{1/p}  
\leq M([t_0,t]) \ \frac{e^{(t-t_0)L([t_0,t])}-1}{L([t_0,t])}
\end{multline*}
where we have used that the second marginal of $\Pi_{t_0}$ is $\mu^2_{t_0}=\mu^2(t_0)$. 
The estimate \eqref{propagflow1} follows.

To obtain the alternative estimate with the term \eqref{propagflow2}, we apply the H\"older inequality to the right-hand side of \eqref{ox_18:44}, obtaining
\begin{equation*}
\begin{split}
& \Vert \Phi^1(t,t_0,\lambda,y)-\Phi^2(t,t_0,\lambda,y) \Vert_E \\
\leq\ & \bigg( \frac{e^{p' (t-t_0)L([t_0,t])}-1}{p' L([t_0,t])} \bigg)^{1/p'}  \bigg( \int_{t_0}^t \Vert Y^1(\tau,\lambda,\Phi^2(\tau,t_0,\lambda,y)) - Y^2(\tau,\lambda,\Phi^2(\tau,t_0,\lambda,y)) \Vert_E^p \, d\tau \bigg)^{1/p} .
\end{split}
\end{equation*}
Therefore, 
the second term of the sum at the right-hand side of \eqref{ineg_coupl} is 
estimated by
\begin{equation*}
\begin{split}
& \bigg( \int_{(\Lambda\times E)^2} \Vert \Phi^1(t,t_0,\lambda_2,y_2)-\Phi^2(t,t_0,\lambda_2,y_2)\Vert_E^p \, d\Pi_{t_0}(\lambda_1,y_1,\lambda_2,y_2) \bigg)^{1/p} \\
=\ & \bigg( \int_{\Lambda\times E} \Vert \Phi^1(t,t_0,\lambda,y)-\Phi^2(t,t_0,\lambda,y)\Vert_E^p \, d\mu^2_{t_0}(\lambda,y) \bigg)^{1/p}  \\
\leq\ & \left( \frac{e^{p' (t-t_0)L([t_0,t])}-1}{p' L([t_0,t])} \right)^{1/p'} \bigg( \int_{t_0}^t  \int_{\Lambda\times E} \Vert Y^1(\tau,\lambda,\Phi^2(\tau,t_0,\lambda,y)) \\
& \qquad\qquad\qquad\qquad\qquad\qquad\qquad\qquad\qquad\qquad   - Y^2(\tau,\lambda,\Phi^2(\tau,t_0,\lambda,y)) \Vert_E^p \, d\mu^2_{t_0}(\lambda,y) \, d\tau  \bigg)^{1/p} \\
\leq\ & \left( \frac{e^{p' (t-t_0)L([t_0,t])}-1}{p' L([t_0,t])} \right)^{1/p'} \bigg( \int_{t_0}^t  \int_{\Lambda\times E} \Vert Y^1(\tau,\lambda,y) - Y^2(\tau,\lambda,y) \Vert_E^p \, d\mu^2_\tau(\lambda,y) \, d\tau \bigg)^{1/p} \\
\leq\ & \left( \frac{e^{p' (t-t_0)L([t_0,t])}-1}{p' L([t_0,t])} \right)^{1/p'}  (t-t_0)^{1/p}  M_p([t_0,t]) 
\end{split}
\end{equation*}
The lemma is proved.
\end{proof}

\begin{lemma}\label{lem_flowid}
Let $Y(t,\lambda,\cdot)$ be a continuous time-varying vector field on $E$, depending on the parameter $\lambda\in\Lambda$, locally Lipschitz with respect to $y\in E$ uniformly with respect to $(t,\lambda)$ on any compact, generating a flow $(\Phi(t,t_0,\lambda,\cdot))_{t\in\R}$ (assumed to be well defined for every $t\in\R$) for any $t_0\in\R$ (as in Lemma \ref{lem_propagflow}).
Given any $t_0\in\R$ and any $\mu_{t_0} \in\mathcal{P}_c(\Lambda\times E)$, we set $\mu(t)=\Phi(t,t_0)_*\mu_{t_0}$ for every $t\geq t_0$.
For every $p\in[1,+\infty)$, we have
$$
W_p(\mu(t), \mu(t_0)) \leq M([t_0,t]) \vert t-t_0\vert \qquad \forall t\geq t_0
$$
where
$M([t_0,t]) = \max \left\{ \Vert Y(\tau,\lambda,y)\Vert \ \mid\ t_0\leq\tau\leq t,\ (\lambda,y)\in\supp(\mu(\tau)) \right\}$.
\end{lemma}

\begin{proof}
One could apply Lemma \ref{lem_propagflow} with $Y^1=0$ and $Y^2=Y$ provided $Y$ were also Lipschitz with respect to $\lambda$; we give instead a direct proof that does not require this assumption.
We first establish the following general result.

\begin{lemma}\label{lemphistar}
Let $F$ be a Polish space, endowed with a distance $\mathrm{d}_F$, let $\mu\in\mathcal{P}_c(F)$ and let $\phi:F\rightarrow F$ be a measurable mapping. For every $p\in[1,+\infty)$, we have
$$
W_p(\phi_*\mu,\mu) \leq \left( \int_F \mathrm{d}_F(y,\phi(y))^p\, d\mu(y) \right)^{1/p}
$$
\end{lemma}

\begin{proof}[Proof of Lemma \ref{lemphistar}.]
We define $\Pi\in\mathcal{P}(F\times F)$ as the pushforward of $\mu$ under the mapping $y\mapsto(y,\phi(y))$. Then $\Pi$ couples $\mu$ and $\phi_*\mu$, and
$W_p(\phi_*\mu,\mu)^p \leq \int_{F^2} \mathrm{d}_F(y,y')^p\, d\Pi(y,y') = \int_F \mathrm{d}_F(y,\phi(y))^p\, d\mu(y)$.
\end{proof}

Applying Lemma \ref{lemphistar} with $F=\Lambda\times E$, $\mu=\mu_{t_0}$ and $\phi=\Phi(t,t_0)$, we have
$$
W_p(\mu(t), \mu(t_0))^p \leq \int_{\Lambda\times E} \mathrm{d}_{\Lambda\times E}((\lambda,y),(\lambda,\Phi(t,t_0,\lambda,y)))^p\, d\mu_{t_0}(\lambda,y) 
$$
and we note that $\mathrm{d}_{\Lambda\times E}((\lambda,y),(\lambda,\Phi(t,t_0,\lambda,y))) = \Vert\Phi(t,t_0,\lambda,y)-y\Vert$. Now, since $\Phi(t,t_0,\lambda,y) = y + \int_{t_0}^t Y(\tau,\lambda,\Phi(\tau,t_0,\lambda,y))\, d\tau$, we have $\Vert \Phi(t,t_0,\lambda,y)-y\Vert_E \leq (t-t_0)\,M([t_0,t])$ for $(\lambda,y)\in\supp(\mu_{t_0})$, since $\supp(\mu(\tau))=\Phi(\tau,t_0,\supp(\mu_{t_0}))$. Lemma \ref{lem_flowid} follows.
\end{proof}

\subsubsection{Moment of order one}\label{app_lem_moment}
Let $\Omega$ be a Polish space and let $d\in\N^*$.

For any vector-valued bounded Borel measure $m$ on $\Omega$ taking values in $\R^d$, we define its \emph{bounded Lipschitz dual norm} by
\begin{equation}\label{def_BL_norm}
\Vert m\Vert_{\mathrm{BL}^*} 
= \sup\bigg\{ \bigg\vert \int_\Omega \langle\varphi(x),dm(x)\rangle \bigg\vert  : \varphi\in\Lip(\Omega,\R^d),\ \Lip(\varphi)\leq 1,\ \Vert\varphi\Vert_\infty\leq 1 \bigg\}.
\end{equation}
For scalar signed measures (case $d=1$), this reduces to the standard bounded Lipschitz distance, which is equivalent to $W_1$ on probability measures.

\begin{lemma}\label{lem_moment}
For $i=1,2$, let $\mu_i\in\mathcal{P}_c(\Omega\times\R^d)$, disintegrated as $\mu_i = \int_\Omega (\mu_i)_x\, d\nu_i(x)$ with respect to its marginal $\nu_i$ on $\Omega$, and let $y_i$ be the moment of order one of $\mu_i$, defined by $y_i(x) = \int_{\R^d}\xi\, d(\mu_i)_x(\xi)$ for $\nu_i$-almost every $x\in\Omega$. Let $R>0$ be such that $\supp(\mu_i)\subset\Omega\times\bar B_R$ for $i=1,2$. Then $y_1\nu_1$ and $y_2\nu_2$ are bounded Borel measures on $\Omega$ taking values in $\R^d$, and
\begin{equation*}
\Vert y_1\nu_1 - y_2\nu_2\Vert_{\mathrm{BL}^*}\leq \max(R,1)\, W_1(\mu_1,\mu_2).
\end{equation*}
\end{lemma}

\begin{proof}
For any $\varphi\in\Lip(\Omega,\R^d)$ with $\Lip(\varphi)\leq 1$ and $\Vert\varphi\Vert_\infty\leq 1$, we have, by Fubini, 
$$
\bigg\vert \int_\Omega \langle\varphi,d(y_1\nu_1-y_2\nu_2)\rangle \bigg\vert
= \bigg\vert \int_{\Omega\times\R^d} F\, d(\mu_1-\mu_2) \bigg\vert
$$
where $F(x,\xi)=\langle\varphi(x),\xi\rangle$. On $\Omega\times\bar B_R$, $F$ is Lipschitz with $\Lip(F)\leq\max(R,1)$ (using $\vert\xi\vert\leq R$, $\Lip(\varphi)\leq 1$ and $\Vert\varphi\Vert_\infty\leq 1$). The Kantorovich--Rubinstein duality \eqref{def_W1} yields the result.
\end{proof}

\subsection{More precise facts on the marginals of a symmetrization}\label{app_marginals_symmetric}
Let $N\in\N^*$ be arbitrary.
Recall that the symmetrization of a measure is defined by \eqref{def_compacte_symmetrization_E} (see Appendix \ref{app_symm}).
\subsubsection{First marginal of the symmetrization}\label{app_first_marginal_symmetrization}
For every $i\in\{1,\ldots,N\}$, we denote by $p^i$ the projection of $E^N$ onto the $i^\textrm{th}$ copy of $E$, i.e., in coordinates, $p^i(y)=y_i$. 

\begin{lemma}\label{first_marginal_symmetrization}
Let $\mu\in\mathcal{P}(E^N)$ be arbitrary. 
\begin{itemize}
\item The first marginal $\mu^s_{N:1} = p^1_*\mu^s$ of the symmetrization $\mu^s$ of $\mu$ is given by
$$
\mu^s_{N:1} = \frac{1}{N}\sum_{i=1}^N p^i_*\mu
$$
where $p^i_*\mu$ is the image of $\mu$ under the projection $p^i$. In other words, $\mu^s_{N:1}$ is the average of the marginals of $\mu$ on the copies of $E$.
\item We have $\displaystyle p^i_*\mu^s = \frac{1}{N}\sum_{j=1}^N p^j_*\mu^s$ for every $i\in\{1,\ldots,N\}$ and thus $p^i_*\mu^s$ does not depend on $i$. In other words, the marginals of a symmetric measure on the copies of $E$ are all equal; the same is true for the marginals of higher order. 
\end{itemize}
\end{lemma}

\begin{proof}
Given any $f\in \mathscr{C}^0_c(E)$, we have
\begin{multline*}
\langle\mu^s_{N:1},f\rangle = \langle p^1_*\mu^s, f\rangle
= \langle \mu^s, (p^1)^*f\rangle
= \frac{1}{N!} \sum_{\sigma\in\mathfrak{S}_N} \langle \sigma_*\mu,(p^1)^*f\rangle
= \frac{1}{N!} \sum_{\sigma\in\mathfrak{S}_N} \langle \mu,\sigma^*(p^1)^*f\rangle \\
= \frac{1}{N!} \sum_{\sigma\in\mathfrak{S}_N} \int_{E^N} f\circ p^1(\sigma\cdot y)\, d\mu(y)
= \frac{1}{N!} \sum_{\sigma\in\mathfrak{S}_N} \int_{E^N} f(y_{\sigma(1)})\, d\mu(y)
\end{multline*}
When designing a permutation $\sigma\in\mathfrak{S}_N$, we have $N$ choices for $\sigma(1)$, among $\{1,\ldots,N\}$, and the rest is a permutation of $N-1$ elements. Since $\card(\mathfrak{S}_{N-1})=(N-1)!$, we get 
that
$$
\langle\mu^s_{N:1},f\rangle = \frac{1}{N} \sum_{i=1}^N \int_{E^N} f(y_i)\, d\mu(y)
= \frac{1}{N} \sum_{i=1}^N \int_{E^N} f\circ p^i(y)\, d\mu(y) 
= \frac{1}{N} \sum_{i=1}^N \langle p^i_*\mu, f\rangle
$$
whence the first item.

The second item is proved in the same way, replacing $\mu$ by $\mu^s$.
\end{proof}

\subsubsection{Marginals of symmetric measures}\label{app_marginal_symm}
We have seen in Lemma \ref{lem_Wp_marginal} (Appendix \ref{app_marginals}) that 
$W_p^{[q]}((\mu_1)_{N:k},(\mu_2)_{N:k}) \leq W_p^{[q]}(\mu_1,\mu_2)$, for every $k\in\{1,\ldots,N\}$, for any $\mu_1,\mu_2\in\mathcal{P}(E^N)$.
When $\mu_1$ and $\mu_2$ are symmetric (i.e., $\mu_j=\mu_j^s$ for $j=1,2$), a stronger estimate holds.

\begin{lemma}\label{lem_Wp_marginal_symm}
Let $\mu_1,\mu_2\in\mathcal{P}(E^N)$ be symmetric measures. 
Then, for any $p,q\in[1,+\infty]$ such that $p\leq q$,
\begin{equation}\label{Wp_marginal_symm}
W_p^{[q]}((\mu_1)_{N:k},(\mu_2)_{N:k}) \leq \left( \frac{k}{N} \right)^{1/q} W_p^{[q]}(\mu_1,\mu_2)
\qquad \forall k\in\{1,\ldots,N\}.
\end{equation}
\end{lemma}

In \eqref{Wp_marginal_symm}, the $W_p^{[q]}$ distances are computed with respect to the $\ell^q$ distances $\mathrm{d}^{[q]}_{E^k}$ and $\mathrm{d}^{[q]}_{E^N}$ defined by \eqref{def_distq_app}.

\begin{proof}
Assume that $q<+\infty$ (for $q=+\infty$, it suffices to take limits).
Using the definition \eqref{def_Wp} of $W_p^{[q]}$, we have
\begin{equation}\label{wpqp}
W_p^{[q]}(\mu_1,\mu_2)^p
= \int_{(E^N)^2} \mathrm{d}_{E^N}^{[q]}(y^1,y^2)^p \, d\Pi(y^1,y^2) 
 = \int_{(E^N)^2} \bigg( \sum_{i=1}^N \mathrm{d}_E(y^1_i,y^2_i)^q \bigg)^{p/q}  d\Pi(y^1,y^2) 
\end{equation}
where $\Pi\in\mathcal{P}((E^N)^2)$ is an optimal coupling between $\mu_1$ and $\mu_2$ for the $W_p^{[q]}$ distance, and $y^j=(y^j_1,\ldots,y^j_N)$ for $j\in\{1,2\}$. 

Since $\mu_j$ is symmetric, by Lemma \ref{first_marginal_symmetrization} in Appendix \ref{app_first_marginal_symmetrization}, all marginals of $\mu_j$ on the copies of $E$ are equal, for $j\in\{1,2\}$; the same holds for marginals of higher order. 
For any $k\in\{1,\ldots,N\}$, $\Pi_{N:k}$ couples (maybe not optimally) $(\mu_1)_{N:k}$ and $(\mu_2)_{N:k}$, and this, for any choice of a $k$-tuple $(j_1,\ldots,j_k)$ of distinct elements of $\{1,\ldots,N\}$ with respect to which we take the marginals (because all marginals are equal). Therefore
\begin{equation}\label{kupletpq}
\begin{split}
\int_{(E^N)^2} \bigg( \sum_{i\in\{j_1,\ldots,j_k\}} \mathrm{d}_E(y^1_i,y^2_i)^q \bigg)^{p/q}  d\Pi(y^1,y^2) 
&= \int_{(E^k)^2} \bigg( \sum_{i=1}^k \mathrm{d}_E(y^1_i,y^2_i)^q \bigg)^{p/q}  d\Pi_{N:k}(y^1,y^2) \\
&\geq W_p^{[q]}((\mu_1)_{N:k},(\mu_2)_{N:k})^p 
\end{split}
\end{equation}
for any $(j_1,\ldots,j_k)\in\mathcal{J}_k$, where $\mathcal{J}_k$ is the set of all $k$-tuples $(j_1,\ldots,j_k)$ of distinct elements of $\{1,\ldots,N\}$. Noting that $\card(\mathcal{J}_k) = \binom{N}{k} = \frac{N!}{k!(N-k)!}$, it follows from \eqref{kupletpq} that
\begin{equation}\label{j1j2jk}
W_p^{[q]}((\mu_1)_{N:k},(\mu_2)_{N:k})^p 
\leq \frac{k!(N-k)!}{N!} \sum_{(j_1,\ldots,j_k)\in\mathcal{J}_k} \int_{(E^N)^2} \bigg( \sum_{i\in\{j_1,\ldots,j_k\}} \mathrm{d}_E(y^1_i,y^2_i)^q \bigg)^{p/q}  d\Pi(y^1,y^2)
\end{equation}

Now, \eqref{Wp_marginal_symm} follows from \eqref{wpqp}, \eqref{j1j2jk} and from the following general inequality:
\begin{multline}\label{combineqalpha}
\frac{k!(N-k)!}{N!} \sum_{(j_1,\ldots,j_k)\in\mathcal{J}_k}  \bigg( \frac{1}{k} \sum_{i\in\{j_1,\ldots,j_k\}} z_i \bigg)^\alpha \leq \bigg( \frac{1}{N} \sum_{i=1}^N z_i \bigg)^\alpha \qquad
\forall z_1,\ldots,z_N\geq 0\quad \forall \alpha\in(0,1],
\end{multline}
applied with $z_i = \mathrm{d}_E(y^1_i,y^2_i)^q$ for every $i\in\{1,\ldots,N\}$, and that we establish hereafter. 
Since $\card(\mathcal{J}_k) = \binom{N}{k} = \frac{N!}{k!(N-k)!}$, by concavity of the function $s\mapsto s^\alpha$, we have
\begin{equation}\label{doublesumk}
\frac{k!(N-k)!}{N!} \sum_{(j_1,\ldots,j_k)\in\mathcal{J}_k}  \bigg( \frac{1}{k} \sum_{i\in\{j_1,\ldots,j_k\}} z_i \bigg)^\alpha \leq
\left( \frac{k!(N-k)!}{N!} \frac{1}{k} \sum_{(j_1,\ldots,j_k)\in\mathcal{J}_k}\sum_{i\in\{j_1,\ldots,j_k\}} z_i \right)^\alpha .
\end{equation}
Now, in the double sum appearing at the right-hand side of \eqref{doublesumk}, given any fixed $i\in\{1,\ldots,N\}$, the term $z_i$ appears $\binom{N-1}{k-1} = \frac{(N-1)!}{(k-1)!(N-k)!}$ times in this double sum, because there are $\binom{N-1}{k-1}$ $k$-tuples of $\mathcal{J}_k$ containing $i$. Therefore
\begin{equation}\label{doublesumk_simplifie}
\sum_{(j_1,\ldots,j_k)\in\mathcal{J}_k}\sum_{i\in\{j_1,\ldots,j_k\}} z_i = \frac{(N-1)!}{(k-1)!(N-k)!} \sum_{i=1}^N z_i,
\end{equation}
and \eqref{combineqalpha} follows from \eqref{doublesumk} and \eqref{doublesumk_simplifie}.

\end{proof}

\subsubsection{A combinatorial lemma towards propagation of chaos}\label{app_technical_lemma}

\begin{lemma}[Marginals of a symmetrized tensor product]\label{technical_lemma}
Let $\mu_1,\ldots,\mu_N \in \mathcal{P}(E)$, and let $\rho\in\mathcal{P}(E^N)$ be defined by
$$
\rho = \mu_1\otimes\cdots\otimes\mu_N . 
$$
The symmetrization of $\rho$ is given by
\begin{equation}\label{rhos}
\rho^s = \frac{1}{N!} \sum_{\sigma\in\mathfrak{S}_N} \mu_{\sigma(1)} \otimes\cdots\otimes \mu_{\sigma(N)} . 
\end{equation}
The first marginal $\rho^s_{N:1}\in\mathcal{P}(E)$ of $\rho^s$ is
\begin{equation}\label{first_marginal_rhos}
\rho^s_{N:1} = \frac{1}{N} \sum_{i=1}^N \mu_i 
\end{equation}
and, for every $k\in\{2,\ldots,N\}$, its $k^\textrm{th}$-order marginal $\rho^s_{N:k}\in\mathcal{P}(E^k)$ is
\begin{equation}
\rho^s_{N:k} = (1+\varepsilon_k) \left( \rho^s_{N:1} \right)^{\otimes k} - \varepsilon_k \beta_k \label{prhosN:k_eps}
\end{equation}
where
\begin{equation}\label{def_epsilon_k}
\varepsilon_k = \frac{N^k(N-k)!}{N!} - 1 \in \left[ 0, e^{\frac{k^2}{N}}-1 \right]
\end{equation}
(the latter upper bound being valid if $k\leq\frac{N}{2}$)
and 
\begin{equation}\label{def_betak}
\beta_k = \frac{1}{\varepsilon_k} \frac{(N-k)!}{N!} \sum \mu_{i_1} \otimes\cdots\otimes \mu_{i_k} 
\in\mathcal{P}(E^k) 
\end{equation}
where the sum in \eqref{def_betak} is taken over all $k$-tuples $(i_1,\ldots,i_k)\in \{1,\ldots,N\}^k$ for which at least two elements are equal.
For every $p\in[1,+\infty)$, for every $k\in\N^*$ such that 
$k^2\leq N\ln\big(1+\frac{1}{2^p}\big)$, we have
\begin{equation}\label{dist_almost_tensor}
W_p^{[q]} \left( \rho^s_{N:k} , ( \rho^s_{N:1} )^{\otimes k} \right) 
\leq 
3 \left( \frac{k^2}{N} \right)^{1/p} \, 
W_p^{[q]} \left( ( \rho^s_{N:1} )^{\otimes k},\beta_k \right) 
\end{equation}
and therefore, assuming moreover that $\mu_1,\ldots,\mu_N\in\mathcal{P}_c(E)$, 
\begin{equation}\label{dist_almost_tensor_2}
W_p^{[q]} \left( \rho^s_{N:k} , ( \rho^s_{N:1} )^{\otimes k} \right) 
\leq 
3 k^{1/q} \left( \frac{k^2}{N} \right)^{1/p} \diam_E\left(\bigcup_{i=1}^N\supp(\mu_i)\right)  .
\end{equation}
\end{lemma}

In \eqref{dist_almost_tensor} and \eqref{dist_almost_tensor_2}, the Wasserstein distance $W_p^{[q]}$ is computed with respect to the $\ell^q$ distance $\mathrm{d}^{[q]}_{E^k}$.
The estimate \eqref{dist_almost_tensor_2} is instrumental in our main proofs: it quantifies how close the $k^\textrm{th}$-order marginal $\rho^s_{N:k}$ of the symmetric measure $\rho^s$ is to the tensor power $(\rho^s_{N:1})^{\otimes k}$ for $N$ large.

The first part of the lemma --- in particular the formulas \eqref{first_marginal_rhos} and \eqref{prhosN:k_eps} --- is well known to experts; see, e.g., \cite[Section 3]{PulvirentiSimonella} for the case of Dirac measures.

\begin{proof}
The formula \eqref{rhos} follows directly from \eqref{def_compacte_symmetrization_E}, and the formula \eqref{first_marginal_rhos} follows from Lemma \ref{first_marginal_symmetrization} in Appendix \ref{app_first_marginal_symmetrization} because $p^i_*\rho=\mu_i$.

Let us now compute the $k^\textrm{th}$-order marginal $\rho^s_{N:k}$ of $\rho^s$, for every $k\in\{2,\ldots,N\}$. Let $I^N_k$ be the set of all $k$-tuples $(i_1,\ldots,i_k)$ consisting of distinct integers chosen in $\{1,\ldots,N\}$. We have $\card(I^N_k)=\frac{N!}{(N-k)!}$. Denoting by $\mathfrak{S}_N^{i_1,\ldots,i_k}$ the set of all $\sigma\in\mathfrak{S}_N$ such that $(\sigma(1),\ldots,\sigma(k)) = (i_1,\ldots,i_k)$, we have $\card(\mathfrak{S}_N^{i_1,\ldots,i_k})=(N-k)!$. Now, since
$$
\sum_{\sigma\in\mathfrak{S}_N} \mu_{\sigma(1)} \otimes\cdots\otimes \mu_{\sigma(N)}
= \sum_{(i_1,\ldots,i_k)\in I^N_k} \mu_{i_1} \otimes\cdots\otimes \mu_{i_k} \otimes \sum_{\sigma\in\mathfrak{S}_N^{i_1,\ldots,i_k}} \mu_{\sigma(k+1)} \otimes\cdots\otimes \mu_{\sigma(N)}
$$
we infer that
\begin{equation}\label{rhosN:n}
\rho^s_{N:k} = \frac{(N-k)!}{N!} \sum_{(i_1,\ldots,i_k)\in I^N_k} \mu_{i_1} \otimes\cdots\otimes \mu_{i_k} . 
\end{equation}
Now, writing $I^N_k = \{1,\ldots,N\}^k \setminus \left( \{1,\ldots,N\}^k \setminus I^N_k \right)$, we write the sum in \eqref{rhosN:n} as a sum over $\{1,\ldots,N\}^k$ minus a sum over $\{1,\ldots,N\}^k \setminus I^N_k$ (where at least two of the indices are equal).
For the first sum, we have
\begin{equation}\label{muEXXiotimesN}
\sum_{(i_1,\ldots,i_k)\in \{1,\ldots,N\}^k} \mu_{i_1} \otimes\cdots\otimes \mu_{i_k} 
= \bigg( \sum_{i=1}^N \mu_i \bigg)^{\otimes k}
= N^k \left( \rho^s_{N:1} \right)^{\otimes k} .
\end{equation}
We infer from \eqref{rhosN:n} and \eqref{muEXXiotimesN} that
$$
\rho^s_{N:k} = \frac{N^k(N-k)!}{N!} \left( \rho^s_{N:1} \right)^{\otimes k} - \frac{(N-k)!}{N!} \beta
$$
where 
$$
\beta = \sum_{(i_1,\ldots,i_k)\in \{1,\ldots,N\}^k \setminus I^N_k} \mu_{i_1} \otimes\cdots\otimes \mu_{i_k} 
$$
is a nonnegative Radon measure of total mass $\vert\beta\vert = \card(\{1,\ldots,N\}^k \setminus I^N_k) = N^k-\frac{N!}{(N-k)!}$.
Besides, we have
$$
1 \leq \frac{N^k(N-k)!}{N!} = \frac{N^k}{N(N-1)\cdots(N-k+1)} 
= \frac{1}{\prod_{i=1}^{k-1}\left(1-\frac{i}{N}\right)} \leq e^{\frac{k^2}{N}}
$$
if $k\leq\frac{N}{2}$, where we have used the inequalities $\ln(1-s)\geq -2s$ for $0\leq s\leq\frac{1}{2}$ and
$$
\ln \prod_{i=1}^{k-1}\left(1-\frac{i}{N}\right) = \sum_{i=1}^{k-1} \ln\left( 1-\frac{i}{N}\right) \geq -\frac{2}{N}\sum_{i=1}^{k-1}i = - \frac{(k-1)k}{N} \geq -\frac{k^2}{N} .
$$
Therefore, defining $\varepsilon_k$ by \eqref{def_epsilon_k} and
$$
\beta_k = \frac{1}{\varepsilon_k} \frac{(N-k)!}{N!} \beta \in\mathcal{P}(\R^{dk}),
$$
we obtain $\rho^s_{N:k} = (1+\varepsilon_k) \left( \rho^s_{N:1} \right)^{\otimes k} - \varepsilon_k \beta_k$, which is \eqref{prhosN:k_eps}.
Then, applying Lemma \ref{lem_Wp_eps} in Appendix \ref{app_conv}, using that $\varepsilon_k<1$ if $e^{k^2/N}-1<1$, or equivalently, $k^2<N\ln(2)$, we obtain that
\begin{equation}\label{dist_almost_tensor_0}
W_p^{[q]}( ( \rho^s_{N:1} )^{\otimes k},\rho^s_{N:k} ) \leq \frac{\varepsilon_k^{1/p}}{1-\varepsilon_k^{1/p}}  W_p^{[q]} ( ( \rho^s_{N:1} )^{\otimes k},\beta_k ) \qquad \textrm{if}\ \ k^2<N\ln(2) .
\end{equation}
The estimate \eqref{dist_almost_tensor} is now inferred from \eqref{dist_almost_tensor_0} as follows: if $k^2\leq N\ln(1+\frac{1}{2^p})$ then $e^{k^2/N}-1\leq\frac{1}{2^p}$, hence $\varepsilon_k^{1/p}\leq\frac{1}{2}$ (using \eqref{def_epsilon_k}) and thus $\frac{\varepsilon_k^{1/p}}{1-\varepsilon_k^{1/p}} \leq 2 \varepsilon_k^{1/p}$, and it follows from \eqref{dist_almost_tensor_0} that
$$
W_p^{[q]}( ( \rho^s_{N:1} )^{\otimes k},\rho^s_{N:k} ) \leq 2 \big( e^{\frac{k^2}{N}} - 1\big)^{1/p} \, 
W_p^{[q]} ( ( \rho^s_{N:1} )^{\otimes k},\beta_k ) .
$$
Using that $\frac{e^x-1}{x}\leq 1/(2^p\ln(1+\frac{1}{2^p}))$ for every $x\in(0,\ln(1+\frac{1}{2^p})]$ and $\ln(\frac{3}{2})\geq \frac{1}{3}$, we obtain \eqref{dist_almost_tensor}.

Let us finally establish \eqref{dist_almost_tensor_2}.
Using \eqref{muEXXiotimesN} and \eqref{def_betak}, which express $( \rho^s_{N:1} )^{\otimes k}$ and $\beta_k$ as linear combinations,
applying two times Lemma \ref{lem_Wp_convex} in Appendix \ref{app_conv} and then Lemma \ref{lem_Wp_supp} in Appendix \ref{app_supp}, we infer that
$$
W_p^{[q]} ( ( \rho^s_{N:1} )^{\otimes k},\beta_k )
\leq \max \left( \sum_{i=1}^k \mathrm{d}_E(y_i,y'_i)^q \right)^{1/q} 
\leq k^{1/q}\, \diam_E\left(\bigcup_{i=1}^N\supp(\mu_i)\right)
$$
where, above, the maximum has been taken over all possible $y_i,y'_i\in\supp(\mu_i)$, for $i\in\{1,\ldots,k\}$. 
Then, \eqref{dist_almost_tensor_2} follows from \eqref{dist_almost_tensor} combined with the above inequality.
\end{proof}

\subsection{Density of empirical measures in the set of probability measures}\label{app_kreinmilman}
Let $E$ be a Polish space, endowed with a distance $\mathrm{d}_E$. 
For every $N\in\N^*$, let $Y^N=(y^N_1,\ldots,y^N_N)\in E^N$, and define the \emph{empirical measure} $\mu^e_{Y^N}\in\mathcal{P}(E)$ by
$$
\mu^e_{Y^N} = \frac{1}{N} \sum_{i=1}^N \delta_{y^N_i} .
$$
The points $y^N_i$ are not required to be distinct, so that the empirical measure $\mu^e_{Y^N}$ is equivalently a \emph{convex combination of Dirac masses with rational coefficients}.
Note that
$$
\int_{E} f\, d\mu^e_{Y^N} = \frac{1}{N} \sum_{i=1}^N f(y^N_i)  \qquad \forall f\in \mathscr{C}^0(E).
$$

A sequence $(\mu_j)_{j\in\N^*}$ of $\mathcal{P}(E)$ converges weakly to $\mu\in\mathcal{P}(E)$ if $\int_{E} f\, d\mu_j \rightarrow \int_{E} f\, d\mu$ as $j\rightarrow+\infty$ for any 
$f\in \mathscr{C}_b(E)$ (narrow convergence), where $\mathscr{C}_b(E)$ is the Banach space of bounded functions on $E$.

\begin{lemma}\label{lem_kreinmilman}
When $E$ is compact, the set $\{\mu^e_{Y^N}\ \mid\ N\in\N^*,\ Y^N\in E^N\}$ is weakly dense in $\mathcal{P}(E)$. In other words, any probability measure on $E$ is the weak limit of a sequence of empirical measures. 
\end{lemma}

\begin{proof}
This is a well-known consequence of the Krein-Milman theorem (see, e.g., \cite[Lemma 7]{McCann_DMJ1995}); we recall the proof for completeness. The set $\mathcal{P}(E)$ is convex and weak star compact, and its extreme points are Dirac masses. The Krein-Milman theorem implies that any $\mu\in\mathcal{P}(E)$ is the limit of a finite convex combination $\sum_i \lambda_i \delta_{y_i}$ of Dirac masses. By density of rationals, without loss of generality we can moreover assume that $\lambda_i\in\mathbb{Q}$. The statement follows.
\end{proof}

Recall that, for any $p\in[1,+\infty)$, the Wasserstein distance $W_p$ metrizes the weak convergence in $\mathcal{P}_p(E)$ together with convergence of moments of order $p$.
We have then the following variant of the above lemma (see \cite[Theorem 6.18]{Villani_2009}).

\begin{lemma}\label{lem_villani}
The set $\{\mu^e_{Y^N}\ \mid\ N\in\N^*,\ Y^N\in E^N\}$ is dense in $\mathcal{P}_p(E)$ for the Wasserstein distance $W_p$. In other words, any $\mu\in\mathcal{P}_p(E)$ is the limit of a sequence of empirical measures for the Wasserstein distance $W_p$.
\end{lemma}

\begin{proof}
It suffices to consider $R>0$ sufficiently large such that $\int_{E\setminus B(y_0,R)} \mathrm{d}_E(y_0,y)^p\, d\mu(y)<\varepsilon$, for $\varepsilon>0$ small enough, so that the argument can be performed in the compact set $\overline{B}(y_0,R)$, and the statement readily follows (see also \cite[Chap. 5]{Santambrogio_2015}).
\end{proof}

Several results in the literature quantify the convergence of empirical measures $\mu^e_{Y^N}$ to $\mu\in\mathcal{P}(E)$, mostly in a probabilistic context (see, e.g., \cite{FournierGuillin_PTRF2015}, where $Y$ consists of $N$ i.i.d. random variables of distribution $\mu$). In the lemma below, $Y$ is deterministic and the rate is that obtained from Riemann integration.

\begin{lemma}\label{lem_rate_CV_empirical}
Let $\mu\in\mathcal{P}_c(E)$ and let $N\in\N^*$.
We assume that there exists a family of tagged partitions of $\supp(\mu)$ associated with $\mu$ (see \eqref{def_tagged}), i.e., for every $N\in\N^*$ there exists a partition of $\supp(\mu) = \cup_{i=1}^N F^N_i$ such that all subsets $F^N_i$ are $\mu$-measurable, pairwise disjoint, $\mu(F^N_i)=\frac{1}{N}$, and satisfy $\diam_E(F^N_i)\leq\frac{C_E}{N^r}$ for some $C_E>0$ not depending on $N$, and a $N$-tuple $Y^N=(y^N_1,\ldots,y^N_N)\in E^N$ such that $y^N_i\in F^N_i$ for every $i\in\{1,\ldots,N\}$. Then
$$
W_1(\mu^e_{Y^N},\mu) \leq \frac{C_E}{N^r} 
$$
and thus also, using \eqref{inegWp1Wp2}, 
$$
W_p(\mu^e_{Y^N},\mu) \leq \diam_E(\supp(\mu))^{1-1/p}\frac{C_E^{1/p}}{N^{r/p}}
\qquad \forall p\in[1,+\infty) .
$$
\end{lemma}

Note that, when $E$ is a finite-dimensional manifold, $r=1/\dim(E)$.

When one wants the assumption on the tagged partition to hold for every $N\in\N^*$, this requires the mass of $\mu$ to be sufficiently uniformly distributed; for instance, it is satisfied if $\mu$ is absolutely continuous with respect to a Lebesgue measure with a density bounded above and below on $\supp(\mu)$. This is an elementary deterministic statement, and is unrelated to the much deeper probabilistic results of \cite{FournierGuillin_PTRF2015}.

\begin{proof}
For every $i\in\{1,\ldots,N\}$, we have $\int_{F_i} f(y^N_i)\, d\mu(y) = f(y^N_i) \mu(F^N_i) = \frac{1}{N} f(y^N_i)$ because $\mu(F^N_i)=\frac{1}{N}$ and thus, for every $f\in\Lip(E)$ such that $\Lip(f)\leq 1$,
\begin{multline*}
\left\vert \int_{E} f\, d(\mu-\mu^e_{Y^N}) \right\vert
= \left\vert \sum_{i=1}^N \int_{F^N_i} f(y)\, d\mu(y) - \frac{1}{N}\sum_{i=1}^N f(y^N_i) \right\vert
= \left\vert \sum_{i=1}^N \int_{F^N_i} ( f(y)-f(y^N_i) )\, d\mu(y) \right\vert \\
\leq \sum_{i=1}^N \int_{F^N_i} \vert f(y)-f(y^N_i) \vert\, d\mu(y)
\leq \sum_{i=1}^N \int_{F^N_i} \mathrm{d}_E(y,y^N_i) \, d\mu(y)
\leq \sum_{i=1}^N \mu(F^N_i) \, \diam_E(F^N_i)
\leq \frac{C_E}{N^r} 
\end{multline*}
and the conclusion follows by taking the supremum over all $f$.
\end{proof}

\subsection{Convergence of empirical and semi-empirical measures}\label{app_cv_emp_semiemp}
Let $(\Omega,\mathrm{d}_\Omega)$ be a complete metric space and let $\nu\in\mathcal{P}_c(\Omega)$.
We assume that there exists a family of tagged partitions $(\mathcal{A}^N,X^N)$ of $\supp(\nu)$ associated with $\nu$ satisfying \eqref{def_tagged} (see Section \ref{sec_general_notations}), with $\mathcal{A}^N=(\Omega^N_1,\ldots,\Omega^N_N)$ and $X^N=(x^N_1,\ldots,x^N_N)$. 
We define the \emph{empirical measure} $\nu^e_{X^N}\in\mathcal{P}(\Omega)$ by
$$
\nu^e_{X^N} = \frac{1}{N} \sum_{i=1}^N \delta_{x^N_i} .
$$

\subsubsection{Convergence of empirical measures on $\Omega$}

\begin{lemma}\label{lem_CV_empirical}
\begin{itemize}
\item Let $f$ be a bounded and $\nu$-almost everywhere continuous (i.e., $\nu$-Riemann integrable) function on $\Omega$, of compact support. Then
\begin{equation}\label{CV_sum_riemann}
\int_{\Omega} f\, d(\nu-\nu^e_{X^N}) = \int_{\Omega} f\, d\nu - \frac{1}{N} \sum_{i=1}^N f(x^N_i) = \mathrm{o}(1)
\end{equation}
as $N\rightarrow+\infty$.
As a consequence, $\nu^e_{X^N}$ converges weakly to $\nu$ as $N\rightarrow+\infty$; equivalently, $W_p \left( \nu^e_{X^N},\nu \right) = \mathrm{o}(1)$ as $N\rightarrow+\infty$. 
\item Given any $\alpha\in(0,1]$ and any $N\in\N^*$, we have
\begin{equation}\label{lem_CV_empirical_1}
\left\vert \int_{\Omega} f\, d\nu - \frac{1}{N} \sum_{i=1}^N f(x^N_i) \right\vert \leq \frac{C_\Omega^\alpha}{N^{r\alpha}} \Hol_\alpha(f) 
\end{equation}
for every $f\in \mathscr{C}^{0,\alpha}_c(\Omega)$. As a consequence of \eqref{lem_CV_empirical_1} for $\alpha=1$, we have
\begin{equation}\label{lem_CV_empirical_2}
W_1 \left( \nu^e_{X^N},\nu \right) \leq \frac{C_\Omega}{N^r} 
\end{equation}
and thus also, using \eqref{inegWp1Wp2}, $W_p(\nu^e_{X^N},\nu) \leq \diam_\Omega(\supp(\nu))^{1-1/p}\frac{C_\Omega^{1/p}}{N^{r/p}}$, for any $p\in[1,+\infty)$.
\end{itemize}
\end{lemma}

\begin{proof}
In the first item, \eqref{CV_sum_riemann} follows from the theorem of convergence of Riemann sums, as already recalled in \eqref{CV_Riemann_sum_2}. Interpreted in terms of the empirical measure $\nu^e_{X^N}$, this means that $\nu^e_{X^N}$ converges weakly to $\nu$ as $N\rightarrow+\infty$. In accordance with the Portmanteau theorem (see, e.g., \cite[Chapter 1, Section 2, Theorem 2.1]{Billingsley}), since $W_p$ metrizes the weak convergence, we have $W_p \left( \nu^e_{X^N},\nu \right) = \mathrm{o}(1)$ as $N\rightarrow+\infty$ since $\supp(\nu)$ is compact. 

Writing $\int_{\Omega} f\, d\nu = \sum_{i=1}^N \int_{\Omega^N_i}f\, d\nu$ and using that $\nu(\Omega^N_i)=\frac{1}{N}$ (thus $\frac{1}{N} f(x^N_i) = \int_{\Omega^N_i} f(x^N_i)\, d\nu(x)$) and that $\diam_\Omega(\Omega^N_i)\leq\frac{C_\Omega}{N^r}$ (see \eqref{def_tagged}), we have
\begin{multline*}
\left\vert \int_{\Omega} f\, d\nu - \frac{1}{N} \sum_{i=1}^N f(x^N_i) \right\vert 
=\left\vert \sum_{i=1}^N \int_{\Omega^N_i} ( f(x) - f(x^N_i) ) \, d\nu(x) \right\vert 
\leq \sum_{i=1}^N \int_{\Omega^N_i} \vert f(x)-f(x^N_i)\vert\, d\nu(x) \\
\leq \Hol_\alpha(f) \sum_{i=1}^N \int_{\Omega^N_i} \mathrm{d}_\Omega(x,x^N_i)^\alpha\, d\nu(x) 
\leq \Hol_\alpha(f) \sum_{i=1}^N \nu(\Omega^N_i) \, \diam_\Omega(\Omega^N_i)^\alpha
\leq \frac{C_\Omega^\alpha}{N^{r\alpha}} \Hol_\alpha(f)
\end{multline*}
which gives \eqref{lem_CV_empirical_1}. Taking $\alpha=1$, \eqref{lem_CV_empirical_2} follows by the definition \eqref{def_W1} of $W_1$. 
\end{proof}

\subsubsection{Convergence of semi-empirical measures}\label{app_semiempirical}
Let $d\in\N^*$.
Let $\mu\in\mathcal{P}_c(\Omega\times\R^d)$, disintegrated as $\mu = \int_{\Omega} \mu_x\, d\nu(x)$ with respect to its marginal $\nu=\pi_*\mu$ on $\Omega$.
We define the \emph{semi-empirical measure} $\mu^{se}_{X^N}\in\mathcal{P}(\Omega\times\R^d)$ by
$$
\mu^{se}_{X^N} = \frac{1}{N} \sum_{i=1}^N \delta_{x^N_i} \otimes \mu_{x^N_i} = \int_{\Omega} \mu_x\, d\nu^e_{X^N}(x) .
$$
Its marginal on $\Omega$ is the empirical measure $\nu^e_{X^N}$.
In other words, the disintegration of $\mu^{se}_{X^N}$ with respect to $\nu^e_{X^N}$ is the family of probability measures given by $\mu_{x^N_i}$ when $x=x^N_i$ for some $i\in\{1,\ldots,N\}$ and $0$ otherwise.


\begin{lemma}\label{lem_CV_semiempirical}
\ 
\begin{itemize}
\item We assume that $x\mapsto\mu_x$ is $\nu$-almost everywhere continuous for the Wasserstein distance $W_1$ (equivalently, $W_p$). Let $f$ be a bounded and $\mu$-almost everywhere continuous (i.e., $\mu$-Riemann integrable) function on $\Omega\times\R^d$, of compact support, Lipschitz with respect to $\xi\in\R^d$ with a Lipschitz constant that is uniform with respect to $x\in\Omega$. Then
\begin{equation}\label{lem_CV_sum_riemann_SE}
\int_{\Omega\times\R^d} f\, d(\mu-\mu^{se}_{X^N}) = \mathrm{o}(1)
\end{equation}
as $N\rightarrow+\infty$.
As a consequence, $\mu^{se}_{X^N}$ converges weakly to $\mu$; 
equivalently, $W_p(\mu^{se}_{X^N},\mu) = \mathrm{o}(1)$ as $N\rightarrow+\infty$.

\item
We assume that $x\mapsto\mu_x$ is Lipschitz for the Wasserstein distance $W_1$, i.e., that there exists $L>0$ such that $W_1(\mu_x,\mu_y) \leq L\, \mathrm{d}_\Omega(x,y)$ for $\nu$-almost all $x,y\in\Omega$. Then, given any $N\in\N^*$,
\begin{equation}\label{lem_CV_semiempirical_1}
\left\vert \int_{\Omega\times\R^d} f\, d(\mu-\mu^{se}_{X^N}) \right\vert  \leq   \frac{(L+1)C_\Omega}{N^r}\Lip(f) 
\end{equation}
for every $f\in \mathscr{C}^0_0(\Omega\times\R^d)\cap\Lip(\Omega\times\R^d)$.
As a consequence,
\begin{equation}\label{lem_CV_semiempirical_2}
W_1 \left( \mu^{se}_{X^N},\mu \right) \leq  \frac{(L+1)C_\Omega}{N^r} ,
\end{equation}
and thus also, using \eqref{inegWp1Wp2}, 
$$
W_p(\mu^{se}_{X^N},\mu) \leq \diam_{\Omega\times\R^d}(\supp(\mu))^{1-1/p}\frac{((L+1)C_\Omega)^{1/p}}{N^{r/p}}
\qquad \forall p\in[1,+\infty) .
$$
\end{itemize}
\end{lemma}

\begin{proof}
Let $f:\Omega\times\R^d\rightarrow\R$ be a bounded and $\mu$-almost everywhere continuous function, of compact support, Lipschitz with respect to $\xi\in\R^d$.
The function $F$ defined by $F(x) = \int_{\R^d} f(x,\xi)\, d\mu_x(\xi)$ is bounded on $\Omega$, and
\begin{equation}\label{Fxy}
\begin{split}
\vert F(x)-F(x')\vert &\leq \int_{\R^d} \vert f(x,\xi)-f(x',\xi)\vert\, d\mu_x(\xi) + \left\vert \int_{\R^d} f(x',\xi)\, d(\mu_x-\mu_{x'})(\xi) \right\vert \\
&\leq \int_{\R^d} \vert f(x,\xi)-f(x',\xi)\vert\, d\mu_x(\xi) + W_1(\mu_x,\mu_{x'}) \Lip(f(x',\cdot))  
\end{split}
\end{equation}
for all $x,x'\in\Omega$. Now:
\begin{itemize}
\item First, if moreover $x'\mapsto\Lip(f(x',\cdot))$ is bounded on $\Omega$ and if $x\mapsto\mu_x$ is $\nu$-almost everywhere continuous for the Wasserstein distance $W_1$, then we infer from \eqref{Fxy} that $F$ is $\nu$-almost everywhere continuous. Therefore
$$
\int_{\Omega\times\R^d} f\, d(\mu-\mu^{se}_{X^N}) = \int_{\Omega} F\, d(\nu-\nu^e_{X^N}) = \int_{\Omega} F\, d\nu - \frac{1}{N} \sum_{i=1}^N F(x^N_i) = \mathrm{o}(1)
$$
as $N\rightarrow+\infty$ by convergence of Riemann sums ($f$ and thus $F$ being fixed), which gives \eqref{lem_CV_sum_riemann_SE}. 
%
\item Second, if $f\in\Lip(\Omega\times\R^d)$ and if $x\mapsto\mu_x$ is $L$-Lipschitz for the Wasserstein distance $W_1$ then we infer from \eqref{Fxy} that
$$
\vert F(x)-F(x')\vert \leq \Lip(f)\, \mathrm{d}_\Omega(x,x') + W_1(\mu_x,\mu_{x'}) \Lip(f) 
\leq \Lip(f) (1+L)\, \mathrm{d}_\Omega(x,x')
$$
and thus, using \eqref{lem_CV_empirical_1} with $\alpha=1$, that 
$$
\left\vert \int_\Omega F\, d(\nu-\nu^e_{X^N}) \right\vert \leq \frac{C_\Omega}{N^r}\Lip(F) \leq \frac{(L+1)C_\Omega}{N^r}\Lip(f),
$$
whence \eqref{lem_CV_semiempirical_1}; the estimate \eqref{lem_CV_semiempirical_2} then follows by the definition \eqref{def_W1} of $W_1$.
\end{itemize}
\end{proof}

\begin{remark}
In the first item of Lemma \ref{lem_CV_semiempirical}, the boundedness assumption on $f$ can be slightly weakened to: $x\mapsto f(x,0)$ bounded and $\mu\in\mathcal{P}_1(\Omega\times\R^d)$. Indeed, writing $\vert f(x,\xi)\vert\leq \vert f(x,0)\vert + \Lip(f(x,\cdot))\vert\xi\vert$, we infer that $F$ is bounded. The rest of the proof is the same.
\end{remark}


\subsection{Discrepancy between empirical and $\nu$-monokinetic measures}\label{app_discrepancy_empirical_monokinetic}
Recall that:
\begin{itemize}
\item given any $X^N=(x^N_1,\ldots,x^N_N)\in\Omega^N$ and any $\Xi^N=(\xi^N_1,\dots,\xi^N_N)\in\R^{dN}$, the empirical measure $\mu^e_{(X^N,\Xi^N)}$ on $\Omega\times\R^d$ 
is defined by \eqref{def_empirical_measure};
\item given any $\nu\in\mathcal{P}(\Omega)$ and any measurable function $y:\Omega\rightarrow\R^d$, the $\nu$-monokinetic measure $\mu^\nu_y$ on $\Omega\times\R^d$ is defined by \eqref{def_monokinetic_measure}.
\end{itemize}

\begin{lemma}\label{lem_discrepancy}
Let $\nu\in\mathcal{P}(\Omega)$ and let $(\mathcal{A}^N,X^N)_{N\in\N^*}$ be a family of tagged partitions associated with $\nu$ (see \eqref{def_tagged}), with $\mathcal{A}^N = (\Omega^N_1,\ldots,\Omega^N_N)$ and $X^N=(x^N_1,\ldots,x^N_N)$.
\begin{enumerate}[label=(\roman*)]
\item\label{lemi} Let $y\in \Lip(\Omega,\R^d)$. For every $N\in\N^*$, taking $\Xi^N=(\xi^N_1,\ldots,\xi^N_N)$ with $\xi^N_i=y(x^N_i)$ for every $i\in\{1,\ldots,N\}$, we have
$$
\left\vert \left\langle \mu_y^\nu - \mu^e_{(X^N,\Xi^N)}, f \right\rangle \right\vert \leq \frac{C_\Omega}{N^r} \Lip\left( x\mapsto f(x,y(x)) \right)
\qquad\forall f\in \Lip_c(\Omega\times\R^d) .
$$
\item\label{lemii} For every $N\in\N^*$, let $\Xi^N=(\xi^N_1,\ldots,\xi^N_N)\in\R^{dN}$. Defining the piecewise continuous function 
$$
y^N(x) = \sum_{i=1}^N \xi^N_i \, \mathds{1}_{\Omega^N_i}(x)\qquad \forall x\in\Omega ,
$$
so that $y^N(x^N_i)=\xi^N_i$ for every $i\in\{1,\ldots,N\}$, we have
$$
\left\vert \left\langle \mu_{y^N}^\nu - \mu^e_{(X^N,\Xi^N)}, f \right\rangle \right\vert \leq \frac{C_\Omega}{N^r} \max_{1\leq i\leq N} \Lip(f(\cdot,\xi^N_i))
\qquad\forall f\in \Lip_c(\Omega\times\R^d) .
$$
\end{enumerate}
\end{lemma}

\begin{proof}
Let us prove \ref{lemi}. We have
$
\left\langle \mu_y^\nu, f \right\rangle = \int_\Omega f(x,y(x))\, d\nu(x) = \sum_{i=1}^N \int_{\Omega^N_i}f(x,y(x))\, d\nu(x)
$
and (using that $\nu(\Omega^N_i)=\frac{1}{N}$)
$$
\left\langle \mu^e_{(X^N,\Xi^N)}, f \right\rangle = \frac{1}{N} \sum_{i=1}^N  f(x^N_i,y(x^N_i)) = \sum_{i=1}^N \int_{\Omega^N_i} f(x^N_i,y(x^N_i))\, d\nu(x)
$$
hence
\begin{equation*}
\begin{split}
\left\vert \left\langle \mu_y^\nu - \mu^e_{(X^N,\Xi^N)}, f \right\rangle \right\vert 
&\leq \sum_{i=1}^N \int_{\Omega^N_i}  \left\vert  f(x,y(x)) - f(x^N_i,y(x^N_i)) \right\vert d\nu(x) \\
&\leq \Lip\left( x\mapsto f(x,y(x)) \right) \sum_{i=1}^N \int_{\Omega^N_i} \mathrm{d}_\Omega(x,x^N_i)\, d\nu(x)
\end{split}
\end{equation*}
and \ref{lemi} follows because $\int_{\Omega^N_i} \mathrm{d}_\Omega(x,x^N_i)\, d\nu(x) \leq \nu(\Omega^N_i) \,  \diam_\Omega(\Omega^N_i) \leq \frac{C_\Omega}{N^{1+r}}$ (using \eqref{def_tagged}).

The estimate of \ref{lemii} is proved similarly: we have
$
\left\langle \mu_{y^N}^\nu, f \right\rangle = \sum_{i=1}^N \int_{\Omega^N_i} f(x,\xi^N_i)\, d\nu(x)
$
and thus
\begin{multline*}
\left\vert \left\langle \mu_{y^N}^\nu - \mu^e_{(X^N,\Xi^N)}, f \right\rangle \right\vert 
\leq \sum_{i=1}^N \int_{\Omega^N_i} \left\vert f(x,\xi^N_i) - f(x^N_i,\xi^N_i) \right\vert d\nu(x)  \\
\leq \sum_{i=1}^N \Lip(f(\cdot,\xi^N_i)) \int_{\Omega^N_i} \mathrm{d}_\Omega(x,x^N_i)\, d\nu(x) 
\end{multline*}
and \ref{lemii} follows.
\end{proof}

\begin{remark}\label{remLip}
The proof shows that the estimates remain valid if we only require the relevant functions to be Lipschitz on each $\Omega^N_i$ separately; in particular, they may be discontinuous across the boundaries between cells. Under this weaker assumption, item \ref{lemii} becomes a consequence of \ref{lemi}.
\end{remark}

\subsection{Mean field and variance}\label{app_mean_field}
Let $\mu\in\mathcal{P}_c(\Omega\times\R^d)$ be arbitrary.
Recall that the mean field $\mathcal{X}[\mu](t,x,\xi)$ is defined by \eqref{def_mean_field}, which is the expectation of $G(t,x,\cdot,\xi,\cdot)$ for the measure $\mu$, performed with respect to $(x',\xi')\in\Omega\times\R^d$:
$$
\mathcal{X}[\mu](t,x,\xi) = \int_{\Omega\times\R^d} G(t,x,x',\xi,\xi') \, d\mu(x',\xi')  
= \mathbb{E}^\mu G(t,x,\cdot,\xi,\cdot)
$$
Given any $t\in\R$, any $x,x'\in\Omega$ and any $\xi,\xi'\in\R^d$, we set
$$
e_t[\mu](x,x',\xi,\xi') = G(t,x,x',\xi,\xi') - \mathcal{X}[\mu](t,x,\xi) .
$$
Of course, we have $\mathbb{E}^\mu e_t[\mu](x,\cdot,\xi,\cdot) = 0$ and thus also
$$
\mathbb{E}^{\mu\otimes\mu} e_t[\mu] = 0 .
$$
This naturally leads to consider the variance of $e_t[\mu]$ with respect to $\mu\otimes\mu$:
$$
\mathrm{Var}(e_t[\mu]) = \mathbb{E}^{\mu\otimes\mu} \Vert e_t\Vert^2
= \int_{\Omega^2\times\R^{2d}} \Vert G(t,x,x',\xi,\xi') - \mathcal{X}[\mu](t,x,\xi)\Vert^2\, d\mu(x',\xi')\, d\mu(x,\xi) .
$$
Note that
\begin{equation}\label{estimVar}
\mathrm{Var}(e_t[\mu]) \leq 4\Vert G(t,\cdot,\cdot,\cdot,\cdot)_{\vert\supp(\mu)^2}\Vert_{\mathscr{C}^0}^2 .
\end{equation}

Let $N\in\N^*$ be fixed.
Recall that the particle (time-dependent) vector field $Y^N=(Y^N_1,\ldots,Y^N_N)$ is defined by \eqref{def_Y} with $Y^N_i$ defined by \eqref{def_Yi}, i.e., $Y^N_i(t,X,\Xi) = \frac{1}{N} \sum_{j=1}^N G(t,x_i,x_j,\xi_i,\xi_j)$, where we use the notations $X=(x_1,\ldots,x_N)\in\Omega^N$ and $\Xi=(\xi_1,\ldots,\xi_N)\in(\R^d)^N$.

\begin{lemma}\label{lem_variance_Y}
We assume that the norm $\Vert\cdot\Vert$ on $\R^d$ is induced by a scalar product $\langle\ ,\ \rangle$ on $\R^d$.
For every $i\in\{1,\ldots,N\}$ we have
$$
\int_{\Omega^N\times\R^{dN}} \Vert Y^N_i(t,X,\Xi)-\mathcal{X}[\mu](t,x_i,\xi_i)\Vert^2 \, d\mu^{\otimes N}(X,\Xi)
= \frac{1}{N} \mathrm{Var}(e_t[\mu])
\leq \frac{4}{N}\Vert G(t,\cdot,\cdot,\cdot,\cdot)_{\vert\supp(\mu)^2}\Vert_{\mathscr{C}^0}^2 .
$$
\end{lemma}

\begin{proof}[Proof of Lemma \ref{lem_variance_Y}.]
By definition of $e_t$, we have, for every $i\in\{1,\ldots,N\}$,
$$
Y^N_i(t,X,\Xi)-\mathcal{X}[\mu](t,x_i,\xi_i) = \frac{1}{N}\sum_{j=1}^N e_t[\mu](x_i,x_j,\xi_i,\xi_j) .
$$
Therefore,
\begin{equation}\label{expansion_var}
\begin{split}
& \int_{\Omega^N\times\R^{dN}} \Vert Y^N_i(t,X,\Xi)-\mathcal{X}[\mu](t,x_i,\xi_i)\Vert^2 \, d\mu^{\otimes N}(X,\Xi) \\
=&\ \frac{1}{N^2} \int_{\Omega^N\times\R^{dN}} \sum_{j=1}^N \Vert e_t[\mu](x_i,x_j,\xi_i,\xi_j) \Vert^2 \, d\mu^{\otimes N}(X,\Xi)  \\
& \qquad\qquad + \frac{1}{N^2} \int_{\Omega^N\times\R^{dN}} \sum_{\substack{j,k=1 \\ j\neq k}}^N \langle e_t[\mu](x_i,x_j,\xi_i,\xi_j), e_t[\mu](x_i,x_k,\xi_i,\xi_k)\rangle \, d\mu^{\otimes N}(X,\Xi) .
\end{split}
\end{equation}
The first term at the right-hand side of \eqref{expansion_var} is equal to
\begin{equation}\label{var10:35}
\frac{1}{N}\int_{\Omega^2\times\R^{2d}} \Vert e_t[\mu](x,x',\xi,\xi') \Vert^2 \, d\mu(x',\xi') \, d\mu(x,\xi) = \frac{1}{N} \mathrm{Var}(e_t[\mu]) .
\end{equation}
The second term at the right-hand side of \eqref{expansion_var} is equal to 
\begin{multline}\label{var10:36}
\frac{N^2-N}{N^2} \int_{\Omega^3\times\R^{3d}} \langle e_t[\mu](x,x',\xi,\xi'), e_t[\mu](x,x'',\xi,\xi'')\rangle \, d\mu(x,\xi)\, d\mu(x',\xi')\, d\mu(x'',\xi'') \\
= \frac{N^2-N}{N^2} \int_{\Omega\times\R^d} \left\Vert \int_{\Omega\times\R^d} e_t[\mu](x,x',\xi,\xi') \, d\mu(x',\xi') \right\Vert^2 \, d\mu(x,\xi)
= 0
\end{multline}
because the expectation of $e_t[\mu](x_i,\cdot,\xi_i,\cdot)$ is equal to $0$.
The lemma is proved, using \eqref{estimVar}.
\end{proof}

Although Lemma \ref{lem_variance_Y} is not used directly in this article, we have included it because of its independent interest. In the proof of Theorem \ref{thm_CV_liouville} (Appendix \ref{app_proof_thm_CV_liouville}), we shall need a more technical variant, given below.

\begin{lemma}\label{lem_variance_Y_2}
As in Lemma \ref{lem_variance_Y}, we assume that the norm $\Vert\cdot\Vert$ on $\R^d$ is induced by a scalar product $\langle\ ,\ \rangle$ on $\R^d$.
Let $\bar X=(\bar x_1,\ldots,\bar x_N)\in\Omega^N$ be arbitrary, and let
$$
\rho = \delta_{\bar x_1}\otimes\cdots\otimes\delta_{\bar x_N} \otimes \mu_{\bar x_1}\otimes\cdots\otimes\mu_{\bar x_N} .
$$
Then
\begin{multline}\label{varianceell1}
M(t) = \left( \int_{\Omega^N\times\R^{dN}} \bigg( \sum_{i=1}^N \Vert Y^N_i(t,X,\Xi)-\mathcal{X}[\mu](t,x_i,\xi_i)\Vert \bigg)^2 \, d\rho(X,\Xi) \right)^{1/2} \\
\leq 2 \Vert G(t,\cdot,\cdot,\cdot,\cdot)_{\vert\supp(\mu)^2}\Vert_{\mathscr{C}^{0,1}} \left( 
\sqrt{N}\sqrt{1+70\,\diam_{\Omega\times\R^d} ( \supp(\mu) )}  + N \sqrt{5\, W_1\left( \mu, \mu^{se}_{\bar X} \right) } \right) 
\end{multline}
where $\mu^{se}_{\bar X} = \frac{1}{N}\sum_{i=1}^N \delta_{\bar x_i}\otimes\mu_{\bar x_i} = \rho^s_{N:1}$
(semi-empirical measure).
\end{lemma}

\begin{proof}
As a first remark, we note that, since the function inside the integral at the left-hand side of the inequality \eqref{varianceell1} is symmetric, we can replace $\rho$ by the symmetrization $\rho^s$ in the integral (indeed, when $F:\Omega^N\times\R^{dN}\rightarrow\R$ is symmetric, we have $\int F\, d\rho=\int F\, d\rho^s$). As a second remark, since $M(t)$ is defined as the $L^2$ norm of a sum, we infer from the triangular inequality that
\begin{equation*}
\begin{split}
M(t) &\leq \sum_{i=1}^N \left( \int_{\Omega^N\times\R^{dN}} \Vert Y^N_i(t,X,\Xi)-\mathcal{X}[\mu](t,x_i,\xi_i)\Vert^2 \, d\rho^s(X,\Xi) \right)^{1/2} \\
&\leq N \max_{1\leq i\leq N} \sqrt{I_i(t)}
\qquad\textrm{with}\qquad
I_i(t) = \int_{\Omega^N\times\R^{dN}} \Vert Y^N_i(t,X,\Xi)-\mathcal{X}[\mu](t,x_i,\xi_i)\Vert^2 \, d\rho^s(X,\Xi)
\end{split}
\end{equation*}
Note that it is essential to symmetrize $\rho$ \emph{before} applying the triangular inequality: without symmetrization, the integrals $I_i(t)$ would depend on $i$, and the bound $N\max_i \sqrt{I_i(t)}$ would be much weaker.
Let us now estimate $I_i(t)$, for any fixed $i\in\{1,\ldots,N\}$. We cannot apply directly Lemma \ref{lem_variance_Y} because in the integral $I_i(t)$ the integration is performed with respect to $\rho^s$, and not with respect to $\mu^{\otimes N}$. However, following the proof of Lemma \ref{lem_variance_Y}, we expand $I_i(t)$ similarly as in \eqref{expansion_var}; replacing $\mu^{\otimes N}$ by $\rho^s$ and thus the second-order marginal $\rho^s_{N:2}$ and third-order marginal $\rho^s_{N:3}$ appear. Note that, by Lemma \ref{first_marginal_symmetrization} in Appendix \ref{app_marginals_symmetric}, since $\rho^s$ is symmetric, all its second-order (resp., third-order) marginals on the various copies of $(\Omega\times\R^d)^2$ (resp., of $(\Omega\times\R^d)^3$) are equal. We obtain
\begin{equation}\label{def_I_i}
I_i(t) = \frac{1}{N}  \int_{\Omega^2\times\R^{2d}} \Vert e_t[\mu] \Vert^2 \, d\rho^s_{N:2} 
+ \frac{N^2-N}{N^2} \int_{\Omega^3\times\R^{3d}} F_t[\mu] \, d\rho^s_{N:3}
\end{equation}
with
$$
F_t[\mu](x,x',x'',\xi,\xi',\xi'') = \langle e_t[\mu](x,x',\xi,\xi'), e_t[\mu](x,x'',\xi,\xi'')\rangle
$$
To estimate the first term at the right-hand side of \eqref{def_I_i}, we observe that (using the definition \eqref{def_W1} of the Wasserstein distance $W_1$)
\begin{equation*}
\begin{split}
\int_{\Omega^2\times\R^{2d}} \Vert e_t[\mu] \Vert^2 \, d(\rho^s_{N:2}-\mu^{\otimes 2})
&\leq \Lip(\Vert e_t[\mu]_{\vert\supp(\mu)^2} \Vert^2) W_1^{[1]}(\rho^s_{N:2},\mu^{\otimes 2}) \\
&\leq 4\Vert G(t,\cdot,\cdot,\cdot,\cdot)_{\vert\supp(\mu)^2} \Vert_{\mathscr{C}^{0,1}}^2 W_1^{[1]}(\rho^s_{N:2},\mu^{\otimes 2})
\end{split}
\end{equation*}
(the choice of $q=1$, above, has little importance; other choices would change the constant $4$, see Lemma \ref{lem_lipp})
and that, using \eqref{estimVar} and \eqref{var10:35},
$$
\int_{\Omega^2\times\R^{2d}} \Vert e_t[\mu] \Vert^2 \, d\mu^{\otimes 2} = \mathrm{Var}(e_t[\mu]) \leq 4\Vert G(t,\cdot,\cdot,\cdot,\cdot)_{\vert\supp(\mu)^2}\Vert_{\mathscr{C}^0}^2 \leq 4\Vert G(t,\cdot,\cdot,\cdot,\cdot)_{\vert\supp(\mu)^2}\Vert_{\mathscr{C}^{0,1}}^2 .
$$
To estimate the second term at the right-hand side of \eqref{def_I_i}, similarly, we observe that
\begin{equation*}
\begin{split}
\int_{\Omega^3\times\R^{3d}} F_t[\mu] \, d(\rho^s_{N:3}-\mu^{\otimes 3})
&\leq \Lip(F_t[\mu]_{\vert\supp(\mu)^2}) W_1^{[1]}(\rho^s_{N:3},\mu^{\otimes 3}) \\
&\leq 4\Vert G(t,\cdot,\cdot,\cdot,\cdot)_{\vert\supp(\mu)^2} \Vert_{\mathscr{C}^{0,1}}^2 W_1^{[1]}(\rho^s_{N:3},\mu^{\otimes 3})
\end{split}
\end{equation*}
and that, as in \eqref{var10:36},
$$
\int_{\Omega^3\times\R^{3d}} F_t[\mu] \, d\mu^{\otimes 3} = 0 .
$$
It follows that
$$
I_i(t) \leq 4 \Vert G(t,\cdot,\cdot,\cdot,\cdot)_{\vert\supp(\mu)^2}\Vert_{\mathscr{C}^{0,1}}^2 \left( \frac{1}{N}
+ W_1^{[1]} \left( \rho^s_{N:2},\mu^{\otimes 2} \right) + W_1^{[1]} \left(\rho^s_{N:3},\mu^{\otimes 3} \right) \right) .
$$
Now, applying Lemma \ref{technical_lemma} (Appendix \ref{app_technical_lemma}), 
we infer from \eqref{first_marginal_rhos} that
$\rho^s_{N:1} = \frac{1}{N}\sum_{i=1}^N \delta_{\bar x_i}\otimes\mu_{\bar x_i} = \mu^{se}_{\bar X}$
(semi-empirical measure) and from \eqref{dist_almost_tensor_2} (taking $k=2,3$) that
$$
W_1^{[1]}\left( \rho^s_{N:2}, (\mu^{se}_{\bar X})^{\otimes 2} \right) + W_1^{[1]}\left( \rho^s_{N:3}, (\mu^{se}_{\bar X})^{\otimes 3} \right)  \leq \frac{70}{N} \diam_{\Omega\times\R^d} ( \supp(\mu) )
$$
and thus, using the triangular inequality,
$$
\sum_{k=2,3} W_1^{[1]} \left( \rho^s_{N:k},\mu^{\otimes k} \right) 
\leq 
\frac{70}{N} \diam_{\Omega\times\R^d} ( \supp(\mu) ) + \sum_{k=2,3} W_1^{[1]}\left( \mu^{\otimes k}, (\mu^{se}_{\bar X})^{\otimes k} \right) .
$$
Applying Lemma \ref{lem_Wp_tensor} in Appendix \ref{app_tensor}, we have 
$$
W_1^{[1]}\left( \mu^{\otimes k}, (\mu^{se}_{\bar X})^{\otimes k} \right) 
\leq k\, W_1\left( \mu, \mu^{se}_{\bar X} \right) .
$$
Finally,
$$
I_i(t) \leq 4 \Vert G(t,\cdot,\cdot,\cdot,\cdot)_{\vert\supp(\mu)^2}\Vert_{\mathscr{C}^{0,1}}^2 \left( \frac{1}{N}
+ \frac{70}{N} \diam_{\Omega\times\R^d} ( \supp(\mu) )  + 5\, W_1\left( \mu, \mu^{se}_{\bar X} \right)  \right) 
$$
and the estimate \eqref{varianceell1} follows.
\end{proof}

Related considerations can be found in \cite{GolseMouhotPaul_CMP2016, NataliniPaul_DCDSB_2022, NataliniPaul_SIMA2023}.

\section{Proofs}\label{app_proofs}

\subsection{Proof of Theorem \ref{thm_estim_graph}}\label{app_proof_thm_estim_graph}
We start by proving the second item of Theorem \ref{thm_estim_graph}.
Hence, we assume that $G$ is locally $\alpha$-H\"older continuous with respect to $(x,x',\xi,\xi')$ (uniformly with respect to $t$ on any compact).

\begin{lemma}\label{lem_xxprime}
Let $x,x'\in\Omega$ be arbitrary. We have
\begin{equation}\label{estimyy0}
\Vert y(t,x)-y(t,x') \Vert \leq e^{tL_y(t)}\left( \Vert y^0(x)-y^0(x') \Vert + \mathrm{d}_\Omega(x,x')^\alpha \right) 
\qquad\forall t\geq 0 .
\end{equation}
\end{lemma}

\begin{proof}[Proof of Lemma \ref{lem_xxprime}.]
By definition, we have
$\partial_t y(t,z) = \int_\Omega G(t,z,x'',y(t,z),y(t,x''))\, d\nu(x'')$ for every $z\in\Omega$, 
hence
\begin{multline}\label{diff_yxx'}
\partial_t y(t,x) - \partial_t y(t,x')
= \int_\Omega G(t,x,x'',y(t,x),y(t,x''))\, d\nu(x'') - \int_\Omega G(t,x',x'',y(t,x),y(t,x''))\, d\nu(x'') \\
+ \int_\Omega G(t,x',x'',y(t,x),y(t,x''))\, d\nu(x'') - \int_\Omega G(t,x',x'',y(t,x'),y(t,x''))\, d\nu(x'')
\end{multline}
and using the definition of $L_y(t)$ we obtain
\begin{equation*}
\Vert \partial_t ( y(t,x)-y(t,x') ) \Vert \leq L_y(t) \left( \mathrm{d}_\Omega(x,x')^\alpha + \Vert y(t,x)-y(t,x') \Vert \right)
\end{equation*}
and \eqref{estimyy0} follows by integration (noting that $\tau\mapsto L_y(\tau)$ is nondecreasing).
\end{proof}

\begin{remark}\label{rem_lem_xxprime}
If $G$ does not depend on $(x,x')$ then one can remove the term $\mathrm{d}_\Omega(x,x')^\alpha$ in \eqref{estimyy0}.
\end{remark}

By assumption, $\Vert y^0(x)-y^0(x') \Vert \leq \Hol_\alpha(y^0)\, \mathrm{d}_\Omega(x,x')^\alpha$ for all $x,x'\in\Omega$. Combined with \eqref{estimyy0} in Lemma \ref{lem_xxprime}, this shows that $y(t,\cdot)$ is $\alpha$-H\"older continuous and yields \eqref{holalpha}.

Let us establish \eqref{estim_graph_2}. 
We set $r_i^N(t) = y(t,x_i^N)-\xi_i^N(t)$, for $i=1,\ldots,N$. By definition, we have
\begin{equation}\label{defridot}
\dot r_i^N(t)
= \frac{1}{N} \sum_{j=1}^N \left( G(t,x_i^N,x_j^N,y(t,x_i^N),y(t,x_j^N)) - G(t,x_i^N,x_j^N,\xi_i^N(t),\xi_j^N(t)) \right)  + \epsilon_i^N(t)
\end{equation}
where
\begin{equation}\label{defepsiloni}
\epsilon_i^N(t) = \int_{\Omega} G(t,x_i^N,x',y(t,x_i^N),y(t,x'))\, d\nu(x') - \frac{1}{N} \sum_{j=1}^N G(t,x_i^N,x_j^N,y(t,x_i^N),y(t,x_j^N))
\end{equation}
with $r_i^N(0)=0$, for every $i\in\{1,\ldots,N\}$. 
On the one hand, we have
\begin{equation}\label{diffGLij}
\left\Vert G(t,x_i^N,x_j^N,y(t,x_i^N),y(t,x_j^N)) - G(t,x_i^N,x_j^N,\xi_i^N(t),\xi_j^N(t)) \right\Vert
\leq L_y^N(t) ( \Vert r_i^N(t)\Vert + \Vert r_j^N(t)\Vert )
\end{equation}
where $L_y^N(t)$ is defined by \eqref{hypGL1}.
On the other hand, using \eqref{lem_CV_empirical_1} in Lemma \ref{lem_CV_empirical} (see Appendix \ref{app_cv_emp_semiemp}), we have
\begin{equation}\label{estimepsi}
\Vert\epsilon_i^N(t)\Vert \leq \frac{C_\Omega^\alpha}{N^{r\alpha}} \Hol_\alpha ( x'\mapsto G(t,x_i^N,x',y(t,x_i^N),y(t,x')) )
\end{equation}
and we claim that
\begin{equation}\label{HolGy}
\Hol_\alpha ( x'\mapsto G(t,x_i^N,x',y(t,x_i^N),y(t,x')) )
\leq L_y(t)(1+e^{tL_y(t)}(\Hol_\alpha(y^0)+1)) .
\end{equation}
Indeed, writing for short $g(x',y(t,x')) = G(t,x_i^N,x',y(t,x_i^N),y(t,x'))$, we have
\begin{equation*}
\begin{split}
& \Vert g(x'_1,y(t,x'_1))-g(x'_2,y(t,x'_2))\Vert \\
\leq\ & \Vert g(x'_1,y(t,x'_1)) - g(x'_2,y(t,x'_1))\Vert + \Vert g(x'_2,y(t,x'_1)) - g(x'_2,y(t,x'_2)) \Vert \\
\leq\ & L_y(t) \mathrm{d}_\Omega(x'_1,x'_2)^\alpha + L_y(t) \Vert y(t,x'_1)-y(t,x'_2)\Vert \\
\leq\ & L_y(t) \mathrm{d}_\Omega(x'_1,x'_2)^\alpha + L_y(t) \Hol_\alpha(y(t,\cdot)) \mathrm{d}_\Omega(x'_1,x'_2)^\alpha \\
\end{split}
\end{equation*}
and \eqref{HolGy} follows by using \eqref{holalpha}.
Finally, setting $R^N(t)=(r_1^N(t),\ldots,r_N^N(t))$, noting that $L_y(t)\leq L_y^N(t)$, we infer from \eqref{defridot}, \eqref{diffGLij}, \eqref{estimepsi} and \eqref{HolGy} that 
\begin{equation*}
\frac{d}{dt} \Vert R^N(t)\Vert_\infty \leq \Vert \dot R^N(t)\Vert_\infty \leq L_y^N(t) \left( 2\Vert R^N(t)\Vert_\infty + \frac{C_\Omega^\alpha}{N^{r\alpha}}  (1+e^{tL_y^N(t)}(\Hol_\alpha(y^0)+1))  \right)
\end{equation*}
and, noting that $\tau\mapsto L_y^N(\tau)$ (defined by \eqref{hypGL1}) is nondecreasing and by integration, we obtain \eqref{estim_graph_2}.

Let us establish \eqref{estim_graph_3}. For every $x\in\Omega$ there exists $i\in\{1,\ldots,N\}$ such that $x\in\Omega_i^N$, and thus $\mathrm{d}_\Omega(x,x_i^N)\leq \diam_\Omega(\Omega_i^N)\leq\frac{C_\Omega}{N^r}$ (by \eqref{def_tagged}). It follows from \eqref{holalpha} that
$$
\Vert y(t,x)-y(t,x_i^N) \Vert \leq \Hol_\alpha(y(t,\cdot)) \mathrm{d}_\Omega(x,x_i^N)^\alpha
\leq \frac{C_\Omega^\alpha}{N^{r\alpha}} e^{tL_y^N(t)}\left( \Hol_\alpha(y(0,\cdot))+1 \right)
$$
and, noting that $y^N(t,x)=\xi_i^N(t)$, \eqref{estim_graph_3} follows by the triangular inequality, using \eqref{estim_graph_2}.

\medskip

Let us now prove the first item of Theorem \ref{thm_estim_graph}.
Starting as in the proof of Lemma \ref{lem_xxprime}, by continuity of $G$, we infer from \eqref{diff_yxx'} that, for any $\varepsilon>0$, if $x$ and $x'$ are sufficiently close then
$$
\Vert \partial_t ( y(t,x)-y(t,x') ) \Vert \leq L_y(t) \left( \varepsilon + \Vert y(t,x)-y(t,x') \Vert \right)
$$
and by integration we obtain
\begin{equation}\label{diff_yxx'_eps}
\Vert y(t,x)-y(t,x') \Vert \leq e^{tL_y(t)}\left( \Vert y^0(x)-y^0(x') \Vert + \varepsilon \right) .
\end{equation}
By assumption, $y^0$ is continuous $\nu$-almost everywhere on $\Omega$. It follows from \eqref{diff_yxx'_eps} that, for every $t\geq 0$, $y(t,\cdot)$ is continuous $\nu$-almost everywhere on $\Omega$ with the same continuity set as $y^0$ (thus, not depending on $t$). 

Let us finally establish \eqref{estim_graph_1}.
By the Riemann integration theorem (see \eqref{CV_Riemann_sum_1}) and the assumed continuity properties of $G$, we have $\max_{1\leq i\leq N}\Vert\varepsilon_i^N(t)\Vert = \mathrm{o}(1)$ (where $\varepsilon_i^N(t)$ is defined by \eqref{defepsiloni}) as $N\rightarrow+\infty$, uniformly with respect to $t$ on every compact.
Besides, we still have the inequality \eqref{diffGLij}, but with $L_y^N(t)$ now defined by
$$
L_y^N(t) = \max_{\substack{x,x'\in\Omega \\ 0\leq\tau\leq t}} \Lip\big(G(\tau,x,x',\cdot,\cdot)_{\vert S_y^N(\tau)^2}\big) ,
$$
i.e., like in \eqref{hypGL1} but without the first term involving $\Hol_\alpha(G)$.
With this substitution, we obtain
$$
\frac{d}{dt} \Vert R^N(t)\Vert_\infty \leq \Vert \dot R^N(t)\Vert_\infty \leq L_y^N(t) \left( 2\Vert R^N(t)\Vert_\infty + \mathrm{o}(1) \right)
$$
and integrating we get $\Vert R^N(t)\Vert_\infty \leq e^{2tL_y^N(t)} \mathrm{o}(1)$, which yields \eqref{estim_graph_1}, noting that $L_y^N(t)$ is uniformly bounded with respect to $t\in[0,T]$ and to $N\in\N^*$ (as a consequence of Lemma \ref{lem_Tmax}).
Then, \eqref{estim_graph_1gen} follows by the triangular inequality, using the $\nu$-almost everywhere continuity of $y(t,\cdot)$.

\subsection{Proof of Theorem \ref{thm_estim_graph_2}}\label{app_proof_thm_estim_graph_2}
The proof is a slight adaptation of that of Theorem \ref{thm_estim_graph}.
We start by establishing \eqref{estim_graph_2_2}, assuming that $G$ is locally $\alpha$-H\"older continuous with respect to $(x,x',\xi,\xi')$ (uniformly with respect to $t$ on any compact).

\begin{lemma}\label{lem_xxprime_2}
Let $i\in\{1,\ldots,N\}$ and $x,x'\in\Omega_i^N$ be arbitrary. We have
\begin{equation}\label{estimyy0_2}
\Vert y_N(t,x)-y_N(t,x') \Vert \leq e^{tL_{y_N}(t)} \mathrm{d}_\Omega(x,x')^\alpha 
\end{equation}
where $L_{y_N}(t)$ is defined as $L_y(t)$ in Theorem \ref{thm_estim_graph} with $y$ replaced by $y_N$.
\end{lemma}

\begin{proof}
Following the proof of Lemma \ref{lem_xxprime}, we arrive at
\begin{equation*}
\Vert \partial_t ( y_N(t,x)-y_N(t,x') ) \Vert \leq L_{y_N}(t) \left( \mathrm{d}_\Omega(x,x')^\alpha + \Vert y_N(t,x)-y_N(t,x') \Vert \right)
\end{equation*}
and \eqref{estimyy0_2} follows by integration, noting that $y_N(0,x)-y_N(0,x')=0$ if $x,x'\in\Omega_i^N$.
\end{proof}

It follows from Lemma \ref{lem_xxprime_2} that $y_N(t,\cdot)$ is $\alpha$-H\"older continuous on each $\Omega_i^N$, with H\"older constant $e^{tL_{y_N}(t)}$.

We set $r^N(t,x) = y_N(t,x)-y^N(t,x)$ for every $x\in\Omega$. Given any $x\in\Omega$, denote by $i\in\{1,\ldots,N\}$ the (a.e.\ unique) index such that $x\in\Omega_i^N$; then $y^N(t,x)=\xi_i^N(t)$ and
\begin{equation*}
\partial_t r^N(t,x)
= \frac{1}{N} \sum_{j=1}^N \left( G(t,x,x_j^N,y_N(t,x),y_N(t,x_j^N)) - G(t,x_i^N,x_j^N,\xi_i^N(t),\xi_j^N(t)) \right)  + \epsilon^N(t,x)
\end{equation*}
where
\begin{equation*}
\epsilon^N(t,x) = \int_{\Omega} G(t,x,x',y_N(t,x),y_N(t,x'))\, d\nu(x') - \frac{1}{N} \sum_{j=1}^N G(t,x,x_j^N,y_N(t,x),y_N(t,x_j^N))
\end{equation*}
with $r^N(0,x)=0$. 
We have, on the one hand,
\begin{multline*}
\left\Vert G(t,x,x_j^N,y_N(t,x),y_N(t,x_j^N)) - G(t,x_i^N,x_j^N,\xi_i^N(t),\xi_j^N(t)) \right\Vert \\
\leq L_{y_N}^N(t) ( \mathrm{d}_\Omega(x,x_i^N)^\alpha + \Vert r^N(t,x)\Vert + \Vert r^N(t,x_j^N)\Vert ) \qquad \forall x\in\Omega_i^N 
\end{multline*}
and, on the other hand, proceeding as in the proof of Theorem \ref{thm_estim_graph} (see Appendix \ref{app_proof_thm_estim_graph}), 
$$
\Vert\epsilon^N(t,x)\Vert \leq \frac{C_\Omega^\alpha}{N^{r\alpha}} L_{y_N}^N(t) ( 1+e^{tL_{y_N}^N(t)} ) \qquad \forall x\in\Omega_i^N .
$$
Using that $\mathrm{d}_\Omega(x,x_i^N)\leq\diam_\Omega(\Omega_i^N)\leq\frac{C_\Omega}{N^r}$ (see \eqref{def_tagged}), we finally obtain
$$
\frac{d}{dt} \Vert r^N(t,\cdot)\Vert_{L^\infty(\Omega)} \leq \Vert \partial_t r^N(t,\cdot)\Vert_{L^\infty(\Omega)} \leq  L_{y_N}^N(t) \left( 2\Vert r^N(t,\cdot)\Vert_{L^\infty(\Omega)} + \frac{C_\Omega^\alpha}{N^{r\alpha}} (2+e^{tL_{y_N}^N(t)}) \right) 
$$
and by integration, noting that $\tau\mapsto L_{y_N}^N(\tau)$ is nondecreasing, \eqref{estim_graph_2_2} follows. 

Finally, \eqref{estim_graph_1gen_2} is established as in the proof of Theorem \ref{thm_estim_graph}.

\subsection{Proof of Theorem \ref{thm_vlasov}}\label{app_proof_thm_vlasov}
Let $T>0$ be arbitrary.
Let $F$ be either equal to $\Omega\times\R^d$, or a compact subset of $\Omega\times\R^d$ that is the closure of an open set. Let $\mathscr{C}_0(F)$ be the Banach space of continuous functions on $F$ vanishing at infinity (when $F$ is compact we have $\mathscr{C}_0(F)=\mathscr{C}^0(F)$), and let $\mathcal{M}^1(F) = \mathscr{C}^0(F)'$ be the Banach space of Radon measures on $F$, endowed with the total variation norm $\Vert\ \Vert_{TV}$ (which is the dual norm). We have $\mathcal{P}_c(F) \subset \mathcal{M}^1(F)$ and $\mathscr{C}^0([0,T],\mathcal{P}_c(F)) \subset L^\infty([0,T],\mathcal{M}^1(F))$.

The Banach space $L^\infty([0,T],\mathcal{M}^1(F))$ is endowed with its strong topology, induced by the $L^\infty$ norm in time of the total variation in space, but can also be endowed with a weak star topology, as follows.
Recall the general fact of Bochner integral theory that $L^1([0,T],E)' = L^\infty([0,T],E')$ (isometric isomorphism) for any separable Banach space $E$, where the prime denotes the topological dual. Applying this fact to $E = \mathscr{C}_0(F)$ (which is separable, since $F$ is Polish locally compact), we identify $L^\infty([0,T],\mathcal{M}^1(F))$ with the dual of $L^1([0,T],\mathscr{C}_0(F))$, and we endow it with the corresponding weak star topology.

We have the following preliminary lemma.

\begin{lemma}\label{lem_prelim_thm_vlasov}
Let $K$ be a compact subset of $\Omega\times\R^d$, let
$\mu_0\in\mathcal{P}_c(K)$ 
and let $T>0$ be arbitrary. 
Assume that there exists a sequence of measures $\mu^k\in\mathscr{C}^0([0,T],\mathcal{P}_c(K))$ solutions of the Vlasov equation $\partial_t\mu^k+L_{\mathcal{X}[\mu^k]}\mu^k=0$ in the sense \eqref{vlasov_meaning}, such that:
\begin{itemize}
\item $\mu^k(0)$ converges weakly to $\mu_0$ in $\mathcal{P}_c(K)$,
\item $\mu^k$ converges to $\mu\in L^\infty([0,T],\mathcal{M}^1(\Omega\times\R^d))$ for the weak star topology,
\end{itemize}
as $k\rightarrow+\infty$. Then 
$\mu\in \mathscr{C}^0([0,T],\mathcal{P}_c(K))$ and $t\mapsto\mu(t)$ is Lipschitz continuous in $W_p$ distance (for any $p\in[1,+\infty)$) and is a solution of the Vlasov equation $\partial_t\mu+L_{\mathcal{X}[\mu]}\mu=0$ (in the sense \eqref{vlasov_meaning}) such that $\mu(0)=\mu_0$.
Moreover, $\mu^k(t)$ converges weakly to $\mu(t)$ as $k\rightarrow+\infty$ (equivalently, $W_p(\mu^k(t),\mu(t))\rightarrow 0$), uniformly with respect to $t\in[0,T]$.
\end{lemma}

\begin{proof}
For every $k\in\N^*$, since $\mu^k\in\mathscr{C}^0([0,T],\mathcal{P}_c(K))$ is a solution of the Vlasov equation in the sense \eqref{vlasov_meaning} and thus in the distributional sense, we have, for every $f\in \mathscr{C}^\infty_c([0,T)\times K)$, after integration by parts,
\begin{multline}\label{distrib_form_muk}
\int_K f(0,x,\xi)\, d\mu^k(0,x,\xi) + \int_0^T \int_K \partial_t f(t,x,\xi)\, d\mu^k(t,x,\xi)\, dt \\
+ \int_0^T \int_{K^2} \langle \nabla_\xi f(t,x,\xi), G(t,x,x',\xi,\xi') \rangle\, d\mu^k(t,x',\xi')\, d\mu^k(t,x,\xi)\, dt = 0 .
\end{multline}
\smallskip

\emph{Step 1: support and uniform Lipschitz continuity in $W_p$.}
By assumption, $\supp(\mu^k)\subset[0,T]\times K$ for every $k\in\N^*$. By the Prokhorov theorem, a subsequence of $\mu^k$ converges weakly in $\mathscr{C}^0([0,T]\times K)'$ to some measure on $[0,T]\times K$, which must coincide with $\mu$. The support of $\mu$ is then contained in the Kuratowski liminf of $\supp(\mu^k)$ (see, e.g., \cite[Proposition~5.1.8]{AmbrosioGigliSavare}), and hence $\supp(\mu)\subset[0,T]\times K$. Since $\mu\in L^\infty([0,T],\mathcal{M}^1(\Omega\times\R^d))$, we have $\supp(\mu(t))\subset K$ for almost every $t\in[0,T]$.

For every $k\in\N^*$, consider on $[0,T]\times K$ the continuous time-dependent vector field $v^k(t,x,\xi)=\mathcal{X}[\mu^k(t)](t,x,\xi)$, which is Lipschitz with respect to $\xi$ thanks to Assumption~\ref{G}. Since $\supp(\mu^k)\subset[0,T]\times K$, we have $\Vert v^k\Vert_{\mathscr{C}^0([0,T]\times K)}\leq C$ for some $C>0$ not depending on $k$. Let $(\Phi_{v^k}(t))_{t\in[0,T]}$ be the flow on $[0,T]\times K$ generated by $v^k$. Since $\mu^k$ solves the transport equation $\partial_t\mu^k + L_{v^k}\mu^k=0$, the standard existence and uniqueness theorem for linear transport equations (see, e.g., \cite[Theorem~5.34]{Villani_2003}) yields $\mu^k(t)=\Phi_{v^k}(t)_*\mu^k(0)$ for every $t\in[0,T]$. By Lemma~\ref{lem_flowid} (Appendix~\ref{app_propag}), 
$$
W_p(\mu^k(t_1),\mu^k(t_2)) \leq C\,\vert t_1-t_2\vert \qquad\forall t_1,t_2\in[0,T] ,
$$
i.e., $t\mapsto\mu^k(t)$ is $C$-Lipschitz on $[0,T]$ in the $W_p$ distance, uniformly in $k$.

\smallskip

\emph{Step 2: uniform convergence in $W_p$ via Ascoli.}
The sequence $(\mu^k(\cdot))_{k\in\N^*}$ is equicontinuous on $[0,T]$ for the $W_p$ distance (by the Lipschitz bound) and pointwise relatively compact (since $\bigcup_k\supp(\mu^k(t))\subset K$ is compact, by the Prokhorov theorem). By the Ascoli theorem applied to maps from $[0,T]$ to the metric space $(\mathcal{P}_c(K),W_p)$, a subsequence of $\mu^k(\cdot)$ converges in $\mathscr{C}^0([0,T],(\mathcal{P}_c(K),W_p))$ to some $\tilde\mu(\cdot)$. By the Portmanteau theorem, the corresponding pointwise weak convergence implies weak-star convergence in $L^\infty([0,T],\mathcal{M}^1(\Omega\times\R^d))$, so $\tilde\mu=\mu$. Hence $\mu\in\mathscr{C}^0([0,T],\mathcal{P}_c(K))$, $t\mapsto\mu(t)$ is $C$-Lipschitz in $W_p$, and 
\begin{equation}\label{CV_unif_W_p_muk}
\sup_{t\in[0,T]} W_p(\mu^k(t),\mu(t)) \xrightarrow[k\to+\infty]{}\, 0 .
\end{equation}

\smallskip

\emph{Step 3: passage to the limit in the Vlasov equation.}
We pass to the limit as $k\to+\infty$ in \eqref{distrib_form_muk}. The convergence \eqref{CV_unif_W_p_muk} (which metrizes weak convergence on $\mathcal{P}_c(K)$) implies $\mu^k(0)\rightharpoonup\mu_0$ and, for the time-integral term, that $\int_K \partial_t f\, d\mu^k(t)\to\int_K\partial_t f\, d\mu(t)$ pointwise in $t\in[0,T]$, with a uniform bound by $\Vert\partial_t f\Vert_\infty$; dominated convergence gives the convergence of the second term in \eqref{distrib_form_muk}.

For the bilinear term, observe that \eqref{CV_unif_W_p_muk} implies $\mu^k(t)\otimes\mu^k(t)\rightharpoonup\mu(t)\otimes\mu(t)$ on $K^2$ for every $t\in[0,T]$; since $f$ and $G$ are continuous and bounded on $[0,T]\times K$ and $[0,T]\times K^2$, respectively, we have, pointwise in $t$,
$$
\int_{K^2} \langle\nabla_\xi f(t,x,\xi),G(t,x,x',\xi,\xi')\rangle\, d(\mu^k(t)\otimes\mu^k(t))(x,\xi,x',\xi')
\to \int_{K^2}\langle\nabla_\xi f,G\rangle\, d(\mu(t)\otimes\mu(t)),
$$
with a uniform bound by $\Vert\nabla_\xi f\Vert_\infty\Vert G\Vert_{L^\infty([0,T]\times K^2)}$; dominated convergence concludes.

We thus obtain that $\mu$ satisfies \eqref{distrib_form_muk} with $\mu^k$ replaced by $\mu$, for every $f\in\mathscr{C}^\infty_c([0,T)\times K)$. Hence $\mu$ is a solution of the Vlasov equation $\partial_t\mu+L_{\mathcal{X}[\mu]}\mu=0$ in the distributional sense, with $\mu(0)=\mu_0$. Since $t\mapsto\mu(t)$ is continuous in $W_p$ distance, $\mu$ is in fact a solution in the sense \eqref{vlasov_meaning}. Lemma~\ref{lem_prelim_thm_vlasov} is proved.
\end{proof}

In view of establishing Item \ref{A}, let us first prove the existence of a solution of the Vlasov equation. Given $\mu_0\in\mathcal{P}_c(\Omega\times\R^d)$, we consider a sequence of empirical measures $\mu^e_{(X^N,\Xi^N_0)} = \frac{1}{N}\sum_{i=1}^N \delta_{x^N_i}\otimes\delta_{\xi^N_{0,i}}$ converging weakly to $\mu_0$ as $N\rightarrow+\infty$. 
Setting $X^N=(x^N_1,\ldots,x^N_N)\in\Omega^N$ and $\Xi^N_0=(\xi^N_{0,1},\ldots,\xi^N_{0,N})\in(\R^d)^N$, let $t\mapsto\Xi^N(t)=(\xi^N_1(t),\ldots,\xi^N_N(t))$ be the unique solution of the particle system \eqref{system_Y} with parameter $X^N$ such that $\Xi^N(0)=\Xi^N_0$. It is well defined on $[0,T]$ for any $T\in(0,T_{\max}(\supp(\mu_0)))$ thanks to Assumption \ref{G} and Lemma \ref{lem_Tmax}.
Using the first part of Proposition \ref{prop_empirical_x}, which does not use anything from Theorem \ref{thm_vlasov} (see its proof), $t\mapsto \mu^e_{(X^N,\Xi^N(t))} = \frac{1}{N}\sum_{i=1}^N \delta_{x^N_i}\otimes\delta_{\xi^N_i(t)}$ is a solution of the Vlasov equation \eqref{vlasov} in the sense \eqref{vlasov_meaning}. 

Without loss of generality, we can assume that $(X^N,\Xi^N_0)\in(\supp(\mu_0))^N$ for every $N\in\N^*$, where we recall that $\supp(\mu_0)$ is compact.
Since $\mu^e_{(X^N,\Xi^N(t))}$ is supported on the corresponding solutions of the particle system, it follows from Lemma \ref{lem_Tmax} that there exists a compact subset $K\subset\Omega\times\R^d$ such that $\supp(\mu^e_{(X^N,\Xi^N(t))})\subset K$ for every $t\in[0,T]$ and for every $N\in\N^*$, i.e., the measures $\mu^e_{(X^N,\Xi^N(\cdot))}$ are equi-compactly supported on $[0,T]$, uniformly with respect to $N$.

Besides, since $\mu^e_{(X^N,\Xi^N(t))}$ is a probability measure, we have $\Vert\mu^e_{(X^N,\Xi^N(t))}\Vert_{TV} = 1 <+\infty$ for every $t\in[0,T]$, and thus the sequence $(\mu^e_{(X^N,\Xi^N(\cdot))})_{N\in\N^*}$ is bounded in $L^\infty([0,T],\mathcal{M}^1(\Omega\times\R^d))$ for the strong topology, i.e., in $(L^1([0,T],\mathscr{C}_0(\Omega\times\R^d)))'$ for the strong (dual norm) topology. By the Banach-Alaoglu theorem, there exists a subsequence of $(\mu^e_{(X^N,\Xi^N(\cdot))})_{N\in\N^*}$ converging to some $\mu\in L^\infty([0,T],\mathcal{M}^1(\Omega\times\R^d))$ for the weak star topology.

Therefore, a subsequence of the sequence of measures $\mu^e_{(X^N,\Xi^N(\cdot))}\in\mathscr{C}^0_{\mathrm{comp}}([0,T],\mathcal{P}_c(K))$ satisfies all assumptions of Lemma \ref{lem_prelim_thm_vlasov}.
It follows from that lemma that $\mu\in\mathscr{C}^0([0,T],\mathcal{P}_c(K))$ 
and that $\mu$ is a solution on $[0,T]$ of the Vlasov equation $\partial_t\mu+L_{\mathcal{X}[\mu]}\mu=0$ (in the sense \eqref{vlasov_meaning}) such that $\mu(0)=\mu_0$, and is Lipschitz continuous with respect to $t$ in $W_p$ distance. 

At this stage we have obtained existence of solutions in $\mathscr{C}^0_{\mathrm{comp}}([0,T_{\max}(\supp(\mu_0))),\mathcal{P}_c(\Omega\times\R^d))$ (uniqueness will be proved below). 

\begin{remark}
In \cite{PiccoliRossi_ACAP2013, PiccoliRossi_ARMA2014, PiccoliRossi_ARMA2016, PiccoliRossiTrelat_SIMA2015}, existence is established by constructing a sequence of piecewise constant measures converging to a solution, under the stronger assumption that $G$ be globally Lipschitz continuous. The proof given above relies on approximation by empirical measures and propagation of their supports, in the spirit of \cite{HaLiu_CMS2009} (see also \cite{Neunzert} and \cite[Part~I, Theorem~5.1]{Spohn_book1991}); this approach makes Lemma \ref{lem_Tmax} more readily exploitable.
For the Cucker--Smale model, the proof in \cite{CanizoCarrilloRosado_M3AS2011} relies on a fixed-point argument in the metric space of solutions, with a careful estimate of the propagation of supports.
\end{remark}

We record the following observation, used repeatedly below. When $G$ is locally Lipschitz with respect to all variables $(x,x',\xi,\xi')$, we have, for all $\mu^1,\mu^2\in\mathcal{P}_c(\Omega\times\R^d)$ and every $(t,x,\xi)\in[0,T]\times\Omega\times\R^d$,
\begin{multline}\label{Xmu1mu2_0}
\left\Vert \mathcal{X}[\mu^1](t,x,\xi)-\mathcal{X}[\mu^2](t,x,\xi) \right\Vert = \left\Vert \int_{\Omega\times\R^d} G(t,x,x',\xi,\xi')\, d(\mu^1(x',\xi')-\mu^2(x',\xi')) \right\Vert \\
\leq \Lip(G(t,x,\cdot,\xi,\cdot)_{\vert S})\, W_1(\mu^1,\mu^2) 
\leq \Lip(G(t,x,\cdot,\xi,\cdot)_{\vert S})\, W_p(\mu^1,\mu^2) ,
\end{multline}
where $S=\supp(\mu^1)\cup\supp(\mu^2)$ (a compact set), and we have used that $W_1\leq W_p$.

In case \ref{A}, however, $G$ is locally Lipschitz only with respect to $(\xi,\xi')$ under Assumption~\ref{G}, so the classical Wasserstein distance $W_1$ cannot be used as above. The main difference comes from the following observation: given any $\mu^1,\mu^2\in\mathcal{P}_c(\Omega\times\R^d)$ \emph{having the same marginal} $\nu\in\mathcal{P}_c(\Omega)$ on $\Omega$, we have, by disintegration,
$$
\mathcal{X}[\mu^1](t,x,\xi)-\mathcal{X}[\mu^2](t,x,\xi) = \int_{\Omega} \int_{\R^d} G(t,x,x',\xi,\xi')\, d(\mu^1_{x'}(\xi')-\mu^2_{x'}(\xi'))\, d\nu(x')
$$
and thus
\begin{equation}\label{Xmu1mu2}
\begin{split}
\left\Vert \mathcal{X}[\mu^1](t,x,\xi)-\mathcal{X}[\mu^2](t,x,\xi) \right\Vert  
& \leq 
\max_{x'\in\supp(\nu)} \Lip(G(t,x,x',\xi,\cdot)_{\vert S_{x'}}) \ L^1_\nu W_1(\mu^1,\mu^2)   \\
& \leq \max_{x'\in\supp(\nu)} \Lip(G(t,x,x',\xi,\cdot)_{\vert S_{x'}}) \ L^1_\nu W_p(\mu^1,\mu^2) 
\end{split}
\end{equation}
where 
$S_{x'}=\supp(\mu^1_{x'})\cup\supp(\mu^2_{x'})$ (compact) 
and $L^1_\nu W_p(\mu^1,\mu^2) = \int_{\Omega} W_p(\mu^1_{x'},\mu^2_{x'})\, d\nu(x')$ is defined by \eqref{def_L1Wp}.

Let us now establish \eqref{L1Wpmu1mu2} in item \ref{A2} (which also entails uniqueness).
Let $\mu^1,\mu^2\in \mathscr{C}^0_{\mathrm{comp}}([0,T],\mathcal{P}_c(\Omega\times\R^d))$ be two solutions of the Vlasov equation for some $T>0$, having the same (constant in time) marginal $\nu=\pi_*\mu^i\in\mathcal{P}_c(\Omega)$ on $\Omega$. Let $K\subset\Omega\times\R^d$ be a compact subset containing $\supp(\mu^i(t))$ for $i=1,2$ and for every $t\in[0,T]$.

For $i=1,2$, we consider on $[0,T]\times K$ the continuous time-dependent vector field $v^i(t,x,\xi) = \mathcal{X}[\mu^i_t](t,x,\xi)$ (which is $\mathscr{C}^1$ with respect to $\xi$), so that $\mu^i$ is a solution of the transport equation $\partial_t\mu^i+L_{v^i}\mu^i=0$.
Since we shall apply Lemma \ref{lem_propagflow} (in Appendix \ref{app_propag}) at varying initial times, for every $t_0\in[0,T]$ we consider the flow $(\Phi_{v^i}(t,t_0))_{t\in[0,T]}$ on $[0,T]\times K$ generated by $v^i$, i.e., defined as the unique solution of $\partial_t\Phi_{v^i}(t,t_0,x,\cdot) = v^i(t,x,\cdot)\circ\Phi_{v^i}(t,t_0,x,\cdot)$ such that $\Phi_{v^i}(t_0,t_0,x,\cdot)=\mathrm{id}_{\R^d}$ for $\nu$-almost every $x\in\supp(\nu)$.
Then, we have $\mu^i(t)=\Phi_{v^i}(t,t_0)_*\mu^i(t_0)$ for every $t\in[0,T]$.
This means, disintegrating $\mu^i_t=\mu^i(t)=\int_\Omega\mu^i_{t,x}\, d\nu(x)$, that $\mu^i_{t,x}=\Phi_{v^i}(t,t_0,x,\cdot)_*\mu^i_{t_0,x}$ for $\nu$-almost every $x\in\supp(\nu)$, for every $t\in[0,T]$.

It follows from Lemma \ref{lem_propagflow} (in Appendix \ref{app_propag}), applied with $\Lambda=\emptyset$ and $E=\R^d$ to the vector fields $v^i(t,x,\cdot)$ for any fixed $x\in\supp(\nu)$, that
$$
W_p(\mu^1_{t,x},\mu^2_{t,x}) \leq e^{(t-t_0)L([t_0,t])} W_p(\mu^1_{t_0,x},\mu^2_{t_0,x}) + M([t_0,t]) \frac{e^{(t-t_0)L([t_0,t])}-1}{L([t_0,t])} \qquad \forall t\in[t_0,T]
$$
where, setting 
$S(t) = \supp(\nu)\times\Phi_{v^1}(t,t_0,\supp(\mu^1_{t_0})\cup\supp(\mu^2_{t_0})) \cup \supp(\mu^2(t))$ (compact),
$$
L([t_0,t]) = \mathrm{ess\,sup} \left\{ \Vert (\partial_\xi G, \partial_{\xi'}G) (\tau,x,x',\xi,\xi') \Vert  \ \mid\ t_0\leq\tau\leq t,\ (x,\xi), (x',\xi')\in S(\tau) \right\}  ,
$$
$$
M([t_0,t]) = \max \left\{ \Vert \mathcal{X}[\mu^1_\tau](\tau,x,\xi) - \mathcal{X}[\mu^2_\tau](\tau,x,\xi) \Vert \ \mid\ t_0\leq\tau\leq t,\ (x,\xi)\in\supp(\mu^2_\tau) \right\} .
$$
Since $\mu^1_\tau$ and $\mu^2_\tau$ have the same marginal $\nu$ on $\Omega$, it follows from \eqref{Xmu1mu2} and from the above definition of $L([t_0,t])$ and of $S(t)$ that 
$$
M([t_0,t]) \leq L([t_0,t]) \displaystyle\max_{t_0\leq\tau\leq t} L^1_\nu W_p(\mu^1_\tau,\mu^2_\tau) .
$$
Therefore
$$
W_p(\mu^1_{t,x},\mu^2_{t,x}) \leq e^{(t-t_0)L([t_0,t])} W_p(\mu^1_{t_0,x},\mu^2_{t_0,x}) \\
+ \big( e^{(t-t_0)L([t_0,t])}-1 \big) \max_{t_0\leq\tau\leq t} L^1_\nu W_p(\mu^1_\tau,\mu^2_\tau) .
$$
Integrating with respect to $x\in\Omega$ for the measure $\nu$, we obtain
\begin{multline}\label{L1W1tt0}
L^1_\nu W_p(\mu^1(t),\mu^2(t)) \leq e^{(t-t_0)L([t_0,t])} L^1_\nu W_p(\mu^1(t_0),\mu^2(t_0)) \\
+ \big( e^{(t-t_0)L([t_0,t])}-1 \big) \max_{t_0\leq\tau\leq t} L^1_\nu W_p(\mu^1(\tau),\mu^2(\tau)) .
\end{multline}
We have the following general lemma.

\begin{lemma}\label{lem_gronwall_max}
For every $t_0\in\R$, let $a_{t_0}:[t_0,+\infty)\rightarrow[0,+\infty)$ be a nondecreasing function, continuous at $t_0$, depending continuously on $t_0$.
Let $h:\R\rightarrow[0,+\infty)$ be an absolutely continuous function such that 
$$
h(t)\leq e^{(t-t_0)a_{t_0}(t)} h(t_0) + \big( e^{(t-t_0)a_{t_0}(t)}-1 \big) \max_{t_0\leq\tau\leq t} h(\tau) 
\qquad \forall t\geq t_0\qquad \forall t_0\in\R .
$$
Then
$$
h(t) \leq h(0) \exp \left( 2\int_0^t a_\tau(\tau) \, d\tau \right) \qquad \forall t\in\R.
$$
\end{lemma}

\begin{proof}
Taking $t_0<t<t_1$, writing
$$
\frac{h(t)-h(t_0)}{t-t_0} \leq \frac{e^{(t-t_0)a_{t_0}(t_1)}-1}{t-t_0} h(t_0) + \frac{e^{(t-t_0)a_{t_0}(t_1)}-1}{t-t_0} \max_{t_0\leq\tau\leq t} h(\tau) 
$$
and taking the limit as $t\rightarrow t_0$, since $t_1$ is arbitrary, we obtain $h'(t_0)\leq 2 h(t_0) a_{t_0}(t_0)$, for almost every $t_0\in\R$. The lemma follows by integration.
\end{proof}

Applying Lemma \ref{lem_gronwall_max} to $h(t) = L^1_\nu W_p(\mu^1(t),\mu^2(t))$ and $a_{t_0}(t)=L([t_0,t])$, and using \eqref{L1W1tt0}, we obtain \eqref{L1Wpmu1mu2}.
In particular, the uniqueness statement follows.

At this step, we have proved existence and uniqueness of solutions of the Vlasov equation in the space $\mathscr{C}^0_{\mathrm{comp}}([0,T],\mathcal{P}_c(\Omega\times\R^d))$.
We can thus now define the Vlasov flow by \eqref{vlasov_flow}, and we obtain \eqref{mu_image_varphi} by uniqueness.

The estimate \eqref{W1mu1mu2} in item \ref{B} is established along the same lines, by applying Lemma \ref{lem_propagflow} (Appendix~\ref{app_propag}) with $\Lambda=\Omega$ and $E=\R^d$, and using \eqref{Xmu1mu2_0} instead of \eqref{Xmu1mu2}. We omit the details.

\medskip

It remains to establish the item \ref{A1}. For $K\subset\Omega\times\R^d$ compact and $T\in(0,T_{\max}(K))$, we consider a sequence of measures $\mu^k\in \mathscr{C}^0([0,T],\mathcal{P}_c(K))$ solutions of the Vlasov equation such that $\mu^k_0=\mu^k(0)$ converges weakly to $\mu_0=\mu(0)$ as $k\rightarrow+\infty$. Our objective is to prove that $\mu^k(t)$ converges weakly to $\mu(t)$, uniformly with respect to $t\in[0,T]$.

Since $\mu^k(t)$ is a probability measure, we have $\Vert\mu^k(t)\Vert_{TV} = 1 <+\infty$ for every $t\in[0,T]$, and thus the sequence $(\mu^k(\cdot))_{k\in\N^*}$ is bounded in $L^\infty([0,T],\mathcal{M}^1(\Omega\times\R^d))$ (for the strong topology), i.e., in $(L^1([0,T],\mathscr{C}_0(\Omega\times\R^d)))'$ for the strong (dual norm) topology. By the Banach-Alaoglu theorem, a subsequence of $(\mu^k(\cdot))_{k\in\N^*}$ converges to some $\tilde\mu(\cdot)\in L^\infty([0,T],\mathcal{M}^1(\Omega\times\R^d))$ for the weak star topology.

It follows from Lemma \ref{lem_prelim_thm_vlasov} that $\tilde\mu\in\mathscr{C}^0([0,T],\mathcal{P}_c(K))$, that $\tilde\mu$ is a solution of the Vlasov equation with $\tilde\mu(0)=\mu_0$, and that $\mu^k(t)$ converges weakly to $\tilde\mu(t)$ uniformly in $t\in[0,T]$. By uniqueness, $\tilde\mu=\mu$. Since every weak-star limit point of $(\mu^k)$ coincides with $\mu$, the whole sequence $\mu^k(t)$ converges weakly to $\mu(t)$, uniformly in $t\in[0,T]$. This concludes the proof of the theorem.

\subsection{Proof of Theorem \ref{thm_CV_liouville_empirical}}\label{app_proof_thm_CV_liouville_empirical}
We have $\rho^N(t)=\delta_{X^N}\otimes\delta_{\Xi^N(t)}$. 
By \eqref{first_marginal_rhos} in Lemma~\ref{technical_lemma} (Appendix~\ref{app_technical_lemma}), applied with $\mu_i=\delta_{x^N_i}\otimes\delta_{\xi^N_i(t)}$, we have
$$
\rho^N(t)^s_{N:1} = \frac{1}{N} \sum_{i=1}^N \delta_{x^N_i} \otimes \delta_{\xi^N_i(t)}  = \mu^e_{(X^N,\Xi^N(t))} ,
$$
which yields the preliminary observation stated before Theorem \ref{thm_CV_liouville_empirical}.
For $k=1$, statement \ref{thm_CV_liouville_empirical_A} follows from item \ref{A1} of Theorem \ref{thm_vlasov}, and the estimate \eqref{thm_CV_liouville_empirical_1} from item \ref{B}.

For any $k\in\{2,\ldots,N\}$, the $k^\textrm{th}$-order marginal $\rho^N(t)^s_{N:k}$ is given by \eqref{prhosN:k_eps} in Appendix \ref{app_technical_lemma} (applied with $\mu_i=\delta_{x^N_i}\otimes\delta_{\xi^N_i(t)}$).
By the triangular inequality, we have
\begin{equation}\label{w1_1755}
W_p^{[q]}\left( \rho^N(t)^s_{N:k} , \mu(t)^{\otimes k} \right)
\leq 
W_p^{[q]}\left( \rho^N(t)^s_{N:k} , (\mu^e_{(X^N,\Xi^N(t))})^{\otimes k} \right) 
+ W_p^{[q]} \left( (\mu^e_{(X^N,\Xi^N(t))})^{\otimes k}, \mu(t)^{\otimes k} \right)
\end{equation}
For the first term in the right-hand side of \eqref{w1_1755}, since $(X^N,\Xi^N_0)\in(\supp(\mu_0))^N$, we have
$$
\diam_{\Omega\times\R^d}\bigg( \bigcup_{i=1}^N \supp(\delta_{x^N_i}\otimes\delta_{\xi^N_i(t)}) \bigg)
\leq \diam_\Omega(\supp(\nu)) + \diam_{\R^d}(\Xi^N(t)) ;
$$
applying \eqref{dist_almost_tensor_2} of Lemma \ref{technical_lemma} (Appendix~\ref{app_technical_lemma}) with $\mu^e_{(X^N,\Xi^N(t))} = \rho^N(t)^s_{N:1}$, we obtain
\begin{equation}\label{w1_1756}
W_p^{[q]}\left( \rho^N(t)^s_{N:k} , (\mu^e_{(X^N,\Xi^N(t))})^{\otimes k} \right) 
\leq 3 k^{1/q} \left( \frac{k^2}{N} \right)^{1/p} \left(  \diam_\Omega(\supp(\nu)) + \diam_{\R^d}(\Xi^N(t))  \right) .
\end{equation}
For the second term in the right-hand side of \eqref{w1_1755}, applying first the estimate \eqref{lem_Wp_tensor_ineq} of Lemma \ref{lem_Wp_tensor} (Appendix~\ref{app_tensor}) and then Theorem \ref{thm_vlasov}, we have
\begin{multline}\label{w1_1757}
W_p^{[q]} \left( (\mu^e_{(X^N,\Xi^N(t))})^{\otimes k}, \mu(t)^{\otimes k} \right)
\leq k^{1/q}\, W_p \left( \mu^e_{(X^N,\Xi^N(t))} ,  \mu(t)  \right)  \\
\leq k^{1/q} \, C_\mu^N(t) \, W_p \left( \mu^e_{(X^N,\Xi^N_0)} ,  \mu(0)  \right) 
\end{multline}
where the constant $C_\mu^N(t)$ is defined by \eqref{defCmuN} (or equivalently by $C_{\mu,\mu^e_{(X^N,\Xi^N)}}(t)$, with the notation used in \eqref{def_C} in Theorem \ref{thm_vlasov}). 
Therefore, \eqref{thm_CV_liouville_empirical_n} follows from \eqref{w1_1755}, \eqref{w1_1756} and \eqref{w1_1757}. Note that, for $k=1$, the first term in the right-hand side of \eqref{w1_1755} vanishes, which gives \eqref{thm_CV_liouville_empirical_1} again.

The convergence statement \ref{thm_CV_liouville_empirical_A} for any $k\in\N^*$ is then obtained by combining \eqref{w1_1755}, \eqref{w1_1756} and the qualitative convergence given by item~\ref{A} of Theorem \ref{thm_vlasov} (in place of the quantitative estimate used above).

\subsection{Proof of Theorem \ref{thm_CV_liouville}}\label{app_proof_thm_CV_liouville}
Since $\rho^N_0=\delta_{X^N}\otimes\rho^N_{0,X^N}$ with $\delta_{X^N} = \delta_{x^N_1}\otimes\cdots\otimes\delta_{x^N_N}$ and $\rho^N_{0,X^N} = \mu_{0,x^N_1}\otimes\cdots\otimes\mu_{0,x^N_N}$, formula \eqref{first_marginal_rhos} of Lemma~\ref{technical_lemma} (Appendix~\ref{app_technical_lemma}), applied with $\mu_i=\delta_{x^N_i}\otimes\mu_{0,x^N_i}$, yields $(\rho^N_0)^s_{N:1} = \frac{1}{N}\sum_{i=1}^N\delta_{x^N_i}\otimes\mu_{0,x^N_i} = (\mu_0)^{se}_{X^N}$ (the semi-empirical measure), which is \eqref{rhoN0}. The weak convergence to $\mu_0$ stated in item~\ref{thm_CV_liouville_A} for $k=1$ then follows from Lemma~\ref{lem_CV_semiempirical} (Appendix~\ref{app_semiempirical}); this proves the preliminary observation stated before the theorem. 

Recall that $\rho^N(t) = \Phi^N(t)_*\rho^N_0$ and $\mu(t) = \varphi_{\mu_0}(t)_*\mu_0$.
Setting
\begin{equation*}
\tilde\rho^N(t) = \varphi_{\mu_0}(t)^{\otimes N}_* \rho^N_0 
= \delta_{x^N_1}\otimes\cdots\otimes\delta_{x^N_N} \otimes \mu_{t,x^N_1}\otimes\cdots\otimes\mu_{t,x^N_N}
\end{equation*}
(the latter equality is because $\varphi_{\mu_0}(t,x_i,\cdot)_*\mu_{0,x^N_i} = \mu_{t,x^N_i}$),
we note that $\tilde\rho^N(t)^s = \varphi_{\mu_0}(t)^{\otimes N}_* (\rho^N_0)^s$ and that
$$
\tilde\rho^N(t)^s_{N:k} = \varphi_{\mu_0}(t)^{\otimes k}_* (\rho^N_0)^s_{N:k}  \qquad \forall k\in\{1,\ldots,N\} .
$$
Indeed, this follows from the following obvious lemma.

\begin{lemma}\label{lem_margsym}
Let $E$ be a measure space, $\varphi:E\rightarrow E$ be a measurable mapping, $N\in\N^*$ and $\rho\in\mathcal{P}(E^N)$. Then
$$
\left(\varphi^{\otimes N}_*\rho\right)_{N:k} = \varphi_*(\rho_{N:k})
\qquad\forall k\in\{1,\ldots,N\}.
$$
\end{lemma}

\begin{proof}[Proof of Lemma \ref{lem_margsym}.]
Denoting by $\pi_k:E^N=E^k\times E^{N-k}\rightarrow E^k$ the canonical projection, the lemma follows directly from the identity $\pi_k\circ\varphi^{\otimes N} = \varphi^{\otimes k}\circ\pi_k$.
\end{proof}

In particular, we have
$$
\tilde\rho^N(t)^s_{N:1} = \varphi_{\mu_0}(t)_* (\rho^N_0)^s_{N:1} = \varphi_{\mu_0}(t)_* (\mu_0)^{se}_{X^N} = \frac{1}{N}\sum_{i=1}^N\delta_{x^N_i}\otimes\mu_{t,x^N_i} = \mu(t)^{se}_{X^N} .
$$
In order to establish 
\eqref{estim_W1_k}, we start by applying the triangular inequality:
\begin{equation}\label{estim_Wp_k=1_triangineq}
W_p^{[q]}\left( \rho^N(t)^s_{N:k} , \mu(t)^{\otimes k} \right) 
\leq 
W_p^{[q]}\left( \rho^N(t)^s_{N:k} , \tilde\rho^N(t)^s_{N:k} \right) 
+ W_p^{[q]}\left( \tilde\rho^N(t)^s_{N:k} , \mu(t)^{\otimes k} \right) ,
\end{equation}
and we next show how to estimate each of the two terms of the sum at the right-hand side of \eqref{estim_Wp_k=1_triangineq}. 

\paragraph{First term.}
Applying successively Lemma \ref{lem_Wp_marginal_symm} in Appendix \ref{app_marginal_symm} and Lemma \ref{lem_Wp_s} in Appendix \ref{app_symm}, and using that $W_p^{[q]}\leq W_2^{[q]}\leq W_2^{[1]}$ (see \eqref{inegWp1Wp2} and \eqref{inegWpq}) because $p\leq 2$, we have
\begin{equation}\label{firstterm_1}
W_p^{[q]}\left( \rho^N(t)^s_{N:k} , \tilde\rho^N(t)^s_{N:k} \right)
\leq \Big(\frac{k}{N}\Big)^{1/q} W_2^{[1]}\left( \rho^N(t) , \tilde\rho^N(t)  \right) 
\end{equation}
note that the Wasserstein distance $W_2$ above is computed with respect to the distance $\mathrm{d}^{[1]}_{(\Omega\times\R^d)^N}$ defined by \eqref{def_distq}; the choice $q=1$ is essential.
Recall that $\rho^N(t) = \Phi^N(t)_*\rho^N_0$ and $\tilde\rho^N(t) = \varphi_{\mu_0}(t)^{\otimes N}_* \rho^N_0$. To estimate the right-hand side of \eqref{firstterm_1}, we apply Lemma \ref{lem_propagflow} (Appendix~\ref{app_propag}) in the space $\Omega^N\times\R^{dN}$ endowed with the distance $\mathrm{d}^{[1]}_{(\Omega\times\R^d)^N}$, with $\Lambda=\Omega^N$, $E=\R^{dN}$, and the flows $\Phi^N(t)$ and $\varphi_{\mu_0}(t)^{\otimes N}$ generated, respectively, by the time-dependent vector fields $Y^N$ defined in \eqref{def_Y} and $\mathcal{X}[\mu_t]^{\otimes N}$ (where $\mathcal{X}[\mu_t]$ is defined in \eqref{def_mean_field}). The alternative estimate of that lemma, applied with $p=2$, yields 
\begin{equation}\label{W21N}
W_2^{[1]}\left( \rho^N(t) , \tilde\rho^N(t)  \right)   \leq 
M_2(t) \sqrt{t} \bigg( \frac{e^{tL_2(t)}-1}{L_2(t)} \bigg)^{1/2}
\end{equation}
where
$$
L_2(t) = \max_{0\leq\tau\leq t} \Lip(Y^N(\tau,\cdot,\cdot)_{\vert \supp(\rho^N(\tau)) \cup \supp(\tilde\rho^N(\tau))})
$$
and, using \eqref{defMp},
\begin{equation*}
\begin{split}
M_2(t) &= \max_{0\leq\tau\leq t} \left( \int_{\Omega^N\times\R^{dN}} \Vert Y^N(\tau,\cdot,\cdot)-\mathcal{X}[\mu_\tau](\tau,\cdot,\cdot)^{\otimes N}\Vert_{\ell^1}^2 \, d\tilde\rho^N_\tau \right)^{1/2} \\
&= \max_{0\leq\tau\leq t} \left( \int_{\Omega^N\times\R^{dN}} \bigg( \sum_{i=1}^N \Vert Y^N_i(\tau,X,\Xi)-\mathcal{X}[\mu_\tau](\tau,x_i,\xi_i)\Vert \bigg)^2 d\tilde\rho^N_\tau(X,\Xi) \right)^{1/2} \\
\end{split}
\end{equation*}
where we recall that $\Vert\Xi\Vert_{\ell^1} = \sum_{i=1}^N \Vert\xi_i\Vert$ for any $\Xi=(\xi_1,\ldots,\xi_N)\in(\R^d)^N$. 
Let us estimate $L_2(t)$ and $M_2(t)$.

Since the $\ell^1$ distance $\mathrm{d}^{[1]}_{(\Omega\times\R^d)^N}$ has been used, according to Lemma \ref{lem_lipp} in Section \ref{app_useful_lemmas} we have, using the definition \eqref{def_Yi} of $Y_i$,
\begin{multline*}
L_2(t) = \max_{0\leq\tau\leq t} \max_{1\leq i\leq N} \Lip(Y^N_i(\tau,\cdot,\cdot)_{\vert \supp(\rho^N(t)) \cup \supp(\tilde\rho^N(t))})
= \max_{0\leq\tau\leq t} \Lip(G(\tau,\cdot,\cdot,\cdot,\cdot)_{\vert S_\mu^N(\tau)^2} ) \\
\leq \max_{0\leq\tau\leq t} \Vert G(\tau,\cdot,\cdot,\cdot,\cdot)_{\vert S_\mu^N(\tau)^2} \Vert_{\mathscr{C}^{0,1}} = L(t) 
\end{multline*}
where $S_\mu^N(\tau)$ is defined by \eqref{def_Smu}. 
The choice $q=1$ is crucial: for $q>1$, the exponent of $N$ in \eqref{W21N} would be positive, which would prevent convergence as $N\to+\infty$.

Besides, by Lemma~\ref{lem_variance_Y_2} (Appendix~\ref{app_mean_field})---the choice $p=2$ allows us to apply this lemma---we have
$$
M_2(t) \leq 2 L(t) \left( \sqrt{N}\sqrt{1+70\max_{0\leq\tau\leq t}\diam_{\Omega\times\R^d} ( \supp(\mu(\tau)) )}  + N \sqrt{5\, W_1\left( \mu(t)^{se}_{X^N}, \mu(t) \right) } \right) .
$$
Since the map $s\mapsto\frac{e^{ts}-1}{s}$ is increasing for $s>0$, 
and since $\sqrt{y(e^y-1)}\leq e^y$ for every $y\geq 0$,
we infer from \eqref{firstterm_1} and \eqref{W21N} that
$$
W_p^{[q]}\left( \rho^N(t)^s_{N:k} , \tilde\rho^N(t)^s_{N:k} \right)
\leq 2 \Big(\frac{k}{N}\Big)^{1/q} 
\left( \sqrt{N} C'_\mu(t)  + N \sqrt{5\, W_1\left( \mu(t)^{se}_{X^N}, \mu(t) \right) } \right)  e^{tL(t)}    .
$$
where $C'_\mu(t) = (1+70\max_{0\leq\tau\leq t}\diam_{\Omega\times\R^d} ( \supp(\mu(\tau)) ))^{1/2}$, for every $t\geq 0$.
Applying Lemma \ref{lem_propagflow} (in Appendix \ref{app_propag}) with $\Lambda=\Omega$ and $E=\R^d$ to the Vlasov flow $\varphi_{\mu_0}(t)$ in $\Omega\times\R^d$ generated by the vector field $\mathcal{X}[\mu_t](t,\cdot,\cdot)$, we obtain 
$$
W_1\left( \mu(t)^{se}_{X^N}, \mu(t) \right) \leq e^{tL(t)} W_1\left( (\mu_0)^{se}_{X^N} , \mu_0 \right) . 
$$
Finally,
\begin{equation}\label{estim_firstterm}
W_p^{[q]}\left( \rho^N(t)^s_{N:k} , \tilde\rho^N(t)^s_{N:k} \right)
\leq 2 k^{1/q} \left(  \frac{C'_\mu(t)}{N^{\frac{1}{q}-\frac{1}{2}}} + N^{1-\frac{1}{q}} \sqrt{5\, W_1\left( (\mu_0)^{se}_{X^N} , \mu_0 \right)} \right) e^{2tL(t)} .
\end{equation}

\paragraph{Second term.}
Applying Lemma \ref{lem_propagflow} (in Appendix \ref{app_propag}) with $\Lambda=\Omega^k$ and $E=(\R^d)^k$ to the Vlasov flow $\varphi_{\mu_0}(t)^{\otimes k}$ in the space $\Omega^k\times(\R^d)^k$ endowed with the distance $\mathrm{d}^{[q]}_{(\Omega\times\R^d)^k}$ defined by \eqref{def_distq}, 
generated by the vector field $\mathcal{X}[\mu_t](t,\cdot,\cdot)^{\otimes k}$, we obtain 
$$
W_p^{[q]}\left( \tilde\rho^N(t)^s_{N:k} , \mu(t)^{\otimes k} \right) 
= W_p^{[q]}\left( \varphi_{\mu_0}(t)^{\otimes k}_* (\rho^N_0)^s_{N:k} , \varphi_{\mu_0}(t)^{\otimes k}_* \mu_0^{\otimes k} \right) \\
\leq e^{tL(t)} W_p^{[q]} \left( (\rho_0)^s_{N:k} ,\mu_0^{\otimes k} \right) 
$$
where $L(t)$ is defined as before.

As in the proof of Theorem~\ref{thm_CV_liouville_empirical} (Appendix~\ref{app_proof_thm_CV_liouville_empirical}), for any $k\in\{2,\ldots,N\}$ the measure $(\rho^N_0)^s_{N:k}$ is given by formula \eqref{prhosN:k_eps} of Lemma~\ref{technical_lemma} (Appendix~\ref{app_technical_lemma}), applied with $\mu_i=\delta_{x^N_i}\otimes\mu_{0,x^N_i}$ and $\beta_k$ given by \eqref{def_betak}.
It then follows from \eqref{dist_almost_tensor_2} in Lemma~\ref{technical_lemma} that, since $(\rho^N_0)^s_{N:1} = (\mu_0)^{se}_{X^N}$, whenever $k^2\leq N\ln\big(1+\frac{1}{2^p}\big)$,
$$
W_p^{[q]} \left( (\rho^N_0)^s_{N:k} , ((\mu_0)^{se}_{X^N})^{\otimes k} \right)
\leq 3 k^{1/q} \bigg( \frac{k^2}{N} \bigg)^{1/p} \diam_{\Omega\times\R^d}(\supp(\mu_0)) 
$$
(the above term is zero and thus does not appear in the final estimate when $k=1$).
Therefore, by the triangular inequality and by \eqref{lem_Wp_tensor_ineq} in Lemma \ref{lem_Wp_tensor} (Appendix \ref{app_tensor}),
\begin{multline}\label{estim_secondterm}
W_p^{[q]}\left( \tilde\rho^N(t)^s_{N:k} , \mu(t)^{\otimes k} \right)
\leq \left( W_p^{[q]} \left( (\rho^N_0)^s_{N:k} , ((\mu_0)^{se}_{X^N})^{\otimes k} \right) + W_p^{[q]} \left( ((\mu_0)^{se}_{X^N})^{\otimes k} ,\mu_0^{\otimes k}  \right) \right) e^{tL(t)} \\
\leq k^{1/q} \bigg( 3 \bigg( \frac{k^2}{N} \bigg)^{1/p} \diam_{\Omega\times\R^d}(\supp(\mu_0)) + W_p \left( (\mu_0)^{se}_{X^N} ,\mu_0 \right) \bigg) e^{tL(t)} .
\end{multline}

\paragraph{Conclusion.}
From \eqref{estim_Wp_k=1_triangineq}, \eqref{estim_firstterm} and \eqref{estim_secondterm},
we conclude that, for every $t\geq 0$, 
\begin{multline*}
W_p^{[q]}\left( \rho^N(t)^s_{N:k} , \mu(t)^{\otimes k} \right) 
\leq 
2 k^{1/q}  \bigg( \bigg( \frac{k^2}{N} \bigg)^{1/p} C'_\mu(0) 
+ \frac{C'_\mu(t)}{N^{\frac{1}{q}-\frac{1}{2}}} \\
+ N^{1-\frac{1}{q}} \sqrt{5\, W_1\left( (\mu_0)^{se}_{X^N} , \mu_0 \right)} + W_p \left( (\mu_0)^{se}_{X^N} ,\mu_0 \right) \bigg) e^{2tL(t)} 
\end{multline*}
and
\eqref{estim_W1_k} finally follows.

To establish item~\ref{thm_CV_liouville_A} for any $k\in\N^*$, the above arguments must be adapted to the weaker setting where $G$ is locally Lipschitz only with respect to $(\xi,\xi')$. The adaptation follows the same lines as in the proof of item~\ref{A} of Theorem~\ref{thm_vlasov} (Appendix~\ref{app_proof_thm_vlasov}), to which we refer; we omit the details.

\section*{Acknowledgment}
We are indebted to Claude Bardos, Julien Barr\'e, Arnaud Debussche, Nicolas Fournier, Isabelle Gallagher, Thierry Gallay, Fran\c{c}ois Golse, Alain Joye, Beno\^{\i}t Perthame, David Poyato, Mario Pulvirenti, Laure Saint-Raymond, Alain-Sol Sznitman and Eitan Tadmor for useful discussions.

\end{document}